\documentclass[reqno,12pt]{article}

\usepackage[a4paper]{geometry}
\usepackage{lmodern} 
\usepackage{amsmath,amsfonts,amsthm,amscd,amssymb,graphicx}
\usepackage{subfigure}
\usepackage{cancel}
\numberwithin{equation}{section}

\usepackage[colorlinks=true,linkcolor=blue,citecolor=blue,urlcolor=blue]{hyperref} 
\renewcommand{\eqref}[1]{\textcolor{blue}{\href{#1}{(\ref{#1})}}} 

\usepackage{fancyhdr} 
\let\hide\iffalse

\usepackage{hyperref}
\usepackage{verbatim} 
\usepackage{color}
\usepackage{stackengine,scalerel}

\newcommand\obullet[1]{\ThisStyle{\ensurestackMath{%
  \stackon[1pt]{\SavedStyle#1}{\SavedStyle\kern.6\LMpt\bullet}}}}
\newcommand\ocirc[1]{\ThisStyle{\ensurestackMath{%
  \stackon[1pt]{\SavedStyle#1}{\SavedStyle\kern.6\LMpt\circ}}}}

\usepackage{todonotes}
\newtheorem{theorem}{Theorem}[section]

\newtheorem{lemma}[theorem]{Lemma}

\newtheorem{proposition}[theorem]{Proposition}

\newtheorem{corollary}[theorem]{Corollary}

\newtheorem{remark}[theorem]{Remark}
\newtheorem{definition}[theorem]{Definition}

\def\eps{\varepsilon }

\renewcommand{\div}{{\rm div}}

\newcommand{\II}{{\mathbf I}}

\newcommand{\e}{{\varepsilon}}

\newcommand{\Be}{\begin{equation}}
	\newcommand{\Ee}{\end{equation}}

\newcommand{\R}{\mathbb{R}}
\renewcommand{\P}{\mathbf{P}}
\renewcommand{\S}{\mathbb{S}}
\newcommand{\be}{\begin{equation}}
\newcommand{\bm}{\begin{multline}}
\newcommand{\ee}{\end{equation}}

\newcommand{\vp}{v_{\/_\perp}}

\newcommand{\xb}{x_{\mathbf{b}}}
\newcommand{\tb}{t_{\mathbf{b}}}

\newcommand{\Bes}{\begin{eqnarray*}}
\newcommand{\Ees}{\end{eqnarray*}}

\newcommand{\Bs}{\begin{split}}

\def\beq{\begin{equation}}
\def\eeq{\end{equation}}
\def\bb1{{1\!\!1}}

\def\mB{\mathbb{B}}
\def\R{\mbox{Re }}

\def\w{{\omega}}

\def\pt{\partial}

\newcommand{\T}{{\mathbb T}}

\def\ff{ {\bf{f_2}}}
\def\IP{(\II-\P)}
\def\NS{\textbf{NS}}

\def\eps{\varepsilon}

\def\triangle{\Delta}
\def\bega{\begin{aligned}}
\def\enda{\end{aligned}}
\def\R{\mathbb{R}^2}

\def\lw{\left}
\def\rw{\right}

\def\R{\mathbb{R}}
\def\wtd{\widetilde}
\def\la{\langle}
\def\ra{\rangle}

\def\bcase{\begin{cases}}
\def\ecase{\end{cases}}
\def\al{\alpha}
\def\bmx{\begin{bmatrix}}
\def\emx{\end{bmatrix}}

\def\vp{\varphi}

\def\nl{\left\|}
\def\nr{\right\|}

\def\mI{\mathcal I}
\def\mF{\mathcal F}
\def\mB{\mathcal B}
\def\mM{\mathcal M}
\def\mmA{\mathcal A}

\def\bk{\textbf{k}}
\def\vr{\varrho}

\def\AA{\mathcal A}
\begin{document}

\title{
Validity of Prandtl's boundary layer from the Boltzmann theory
} 
\author{
Chanwoo Kim, Trinh T. Nguyen\footnotemark[1]
}
\maketitle

 \renewcommand{\thefootnote}{\fnsymbol{footnote}}

\footnotetext[1]{Department of Mathematics, University of Wisconsin-Madison, WI 53706, USA. Emails: chanwoo.kim@wisc.edu; txn5114@gmail.com.}
\begin{abstract} We justify Prandtl equations and higher order Prandtl expansion from the hydrodynamic limit of the Boltzmann equations. Our fluid data is of the form 
$\text{shear flow}$, plus $\sqrt\kappa$ order term in analytic spaces in $x_\parallel \in\mathbb T^2$ and Sobolev in $x_3\in\R_+$. This work is the first to rigorously justify the Prandtl equations from the hydrodynamic limits of the Boltzmann equations. The novelty lies in obtaining estimates for the linearized Boltzmann equation with a diffusive boundary condition around a Prandtl layer shear flow in analytic spaces. The key techniques involve delicate commutator estimates and the use of local conservation law.
\end{abstract}

\tableofcontents
\section{Introduction}
In this paper, we consider the Boltzmann equation
\begin{equation}\label{F-eq}
\partial_t F + \frac{1}{\varepsilon} \xi \cdot \nabla_x F = \frac{1}{\kappa\varepsilon^2} Q(F, F),
\end{equation}
subject to the boundary condition
\begin{equation}\label{F-bdr}
F(t,x,\xi) = c_\mu \mu_0(\xi) \int_{n(x)\cdot \tilde\xi > 0} F(t,x,\tilde\xi)\, |n(x)\cdot \tilde\xi| \, d\tilde\xi, \quad (x,\xi) \in \{ \partial\Omega \times \mathbb{R}^3 : n(x)\cdot \xi < 0 \},
\end{equation}
where the bilinear hard-sphere collision operator is defined by
\begin{equation}\label{Q-def}
Q(F, G) = \frac{1}{2} \int_{\mathbb{R}^3} \int_{|\omega|=1} |(\xi - \xi_\star) \cdot \omega| \left\{ F(\xi') G(\xi'_\star) + G(\xi') F(\xi'_\star) - F(\xi) G(\xi_\star) - G(\xi) F(\xi_\star) \right\} d\omega\, d\xi_\star,
\end{equation}
and \( c_\mu > 0 \) is a normalization constant given in \eqref{cmu}. The velocities \( (\xi, \xi_\star) \) denote pre-collisional velocities, while \( (\xi', \xi'_\star) \) denote post-collisional velocities.

Understanding the hydrodynamic limit of \eqref{F-eq}--\eqref{F-bdr} in the regime \( \varepsilon, \kappa \to 0 \) is a longstanding and fundamental problem in kinetic theory, dating back to Hilbert’s sixth problem~\cite{HilbertICM}.

The weak convergence of fluid moments was formally established in~\cite{bardos-golse-levermore1} and later rigorously justified for DiPerna--Lions renormalized solutions in~\cite{Golse-Laure} (see also~\cite{bardos-golse-levermore2,bardos-golse-levermore3,bardos-golse-levermore4,Lions-Masmoudi,Masmoudi-SR,Saint-Raymond-book, JiangCPAM17,JiangCPDE2010}). Significant progress on the closeness of Hilbert expansions to solutions of the compressible Euler equations can be found in~\cite{Caflisch1,Lachowicz,Nishida,UkaiAsano}, while further developments include shock profiles~\cite{YuCPAM}, contact discontinuities~\cite{HWT-discontinuity}, vortex solutions~\cite{KimNguyen}, low-regularity Euler flows~\cite{KimLa}, relativistic regimes~\cite{Speck-Strain}, and steady  solutions~\cite{Esposito-Guo-Kim-Marra1,Esposito-Guo-Kim-Marra2,ELM1,ELM2,AEMN1,MR2765739,esposito2023ghosteffectboltzmanntheory,MR4772727}. We also refer to~\cite{YanGuoCPAM06} for the torus setting and~\cite{JangKimAPDE,CaoJangKim} for half-space results.

A central open question is the hydrodynamic limit near singular fluid solutions in bounded domains, where classical Boltzmann solutions may cease to be \( C^1 \) even for smooth convex boundaries. In particular, the derivation of Prandtl-type boundary layers from the Boltzmann equation remains one of the most fundamental and challenging problems in the field. Three major obstacles arise: (1) the Prandtl equations only asymptotically approximate---but do not exactly solve---the Navier--Stokes equations near the boundary; (2) boundary layer instabilities arise even at the linearized level; and (3) there is a lack of general regularity theory for the Boltzmann equation in domains with boundaries, especially under physical boundary conditions such as diffuse reflection.
 These challenges originate from the singular behavior of the transport term \( \varepsilon^{-1} \xi \cdot \nabla_x \) acting on Maxwellians of the form \( e^{-|\xi - \varepsilon U|^2} \) where U denotes the fluid velocities. Similar difficulties persist in vortex-dominated regimes, highlighting the deep interplay between kinetic singularities and fluid instabilities---one of the most important open frontiers in modern kinetic theory.

On the fluid side, a classical and fundamental question in mathematical fluid dynamics is to understand the behavior of solutions to the Navier--Stokes equations in the vanishing viscosity limit. The incompressible Navier--Stokes equations are given by
\begin{equation}\label{NS}
\begin{aligned}
\partial_t u^\kappa + u^\kappa \cdot \nabla u^\kappa + \nabla p^\kappa &= \eta_0 \kappa \Delta u^\kappa, \quad x \in \Omega,\\
\nabla \cdot u^\kappa &= 0,
\end{aligned}
\end{equation}
where \( u^\kappa(t,x) \in \mathbb{R}^3 \) denotes the velocity field, \( p^\kappa \in \mathbb{R} \) the pressure, \( \kappa > 0 \) the kinematic viscosity, and \( \Omega \subset \mathbb{R}^3 \) the fluid domain. In the presence of boundaries, the classical no-slip condition is imposed:
\begin{equation}\label{noslip}
u^\kappa|_{\partial\Omega} = 0,
\end{equation}
which prescribes that the fluid adheres to the boundary, with both tangential and normal components vanishing on \( \partial\Omega \).

Formally setting \( \kappa = 0 \) yields the incompressible Euler equations, i.e., \eqref{NS} with \( \kappa = 0 \), subject to the impermeability condition
\begin{equation}\label{euler}
u^0 \cdot n|_{\partial\Omega} = 0,
\end{equation}
where \( n \) denotes the outward unit normal to \( \partial\Omega \). In contrast to \eqref{noslip}, the Euler boundary condition permits tangential slip along the boundary, leading to a boundary layer mismatch between Navier--Stokes and Euler flows as \( \kappa \to 0 \).

To resolve this mismatch, Prandtl in 1904 \cite{Prandtl1904} proposed a boundary layer corrector. In the canonical half-space setting \( (x_\parallel, x_3) \in \mathbb{T}^2 \times \mathbb{R}_+ \), the approximate solution takes the form
\[
u^\kappa_{\text{app}}(t, x_\parallel, x_3) = u^E(t, x_\parallel, x_3) + u^P\left(t, x_\parallel, \frac{x_3}{\sqrt{\eta_0 \kappa}}\right),
\]
which satisfies the no-slip condition \( u^\kappa_{\text{app}}|_{x_3 = 0} = 0 \). This leads to the celebrated Prandtl boundary layer theory, whose rigorous justification has since attracted extensive attention.

The Prandtl equations are locally well-posed for analytic data in \( x_\parallel \) \cite{SammartinoCaflisch1,SamCafPrandtl,KukavicaVV13}, for monotonic shear flows \cite{MW,PrandtlJAMS,KukavicaVV14}, and for Gevrey-class data \cite{GVMasmoudi15,GV-Pr-Gevrey,Gevrey-3D}. In the absence of such structural conditions, however, the equations are known to be ill-posed in Sobolev spaces \cite{GVDormy,DGVN,GN1}. Even when the Prandtl expansion is well-posed, controlling the remainder \( v := u^\kappa - u^\kappa_{\text{app}} \), which also satisfies homogeneous Dirichlet boundary conditions, poses significant difficulties. The energy estimate for \( v \) contains the convection term
\[
\int_\Omega (v \cdot \nabla u^\kappa_{\text{app}}) \cdot v,
\]
which becomes problematic due to the sharp boundary gradients:
\[
\nabla u^\kappa_{\text{app}} \sim \kappa^{-1/2},
\]
indicating an \( O(\kappa^{-1/2}) \) blow-up in the \( \kappa \to 0^+\) limit. This large vorticity near the boundary is a key source of instability, and whether the boundary layer remains confined or propagates into the interior remains a central and delicate question.

In his seminal work \cite{Kato84}, Kato established a necessary and sufficient condition for the convergence
\begin{equation}\label{IL}
u^\kappa \to u^0 \quad \text{in} \quad L^\infty(0,T; L^2(\Omega))
\end{equation}
as \( \kappa \to 0 \), namely,
\begin{equation}\label{Kato}
\kappa \int_0^T \int_{d(x, \partial\Omega) \lesssim \kappa} |\nabla u^\kappa|^2 \, dx \, dt \to 0.
\end{equation}
Despite the localization of the condition to a vanishingly thin boundary layer of size $\kappa$, verifying \eqref{Kato} remains an open challenge. Importantly, it reveals that any failure of \eqref{IL} must involve subtle concentration phenomena at the boundary, beyond what is captured by formal Prandtl expansions, since \( u^\kappa_{\text{app}} \) does satisfy \eqref{Kato}.

Sammartino and Caflisch \cite{SammartinoCaflisch1,SammartinoCaflisch2} rigorously justified the Prandtl expansion and the inviscid limit \eqref{IL} for analytic initial data in the half-space. Maekawa \cite{Maekawa14} established the inviscid limit under the assumption that the initial vorticity vanishes near the boundary. For Euler shear flows that are monotonic and concave, the inviscid limit is known to hold for Gevrey-regular perturbations of size \( O(\kappa^{1/2+}) \) \cite{GVMM}.

Progress under additional symmetry or structural assumptions can be found in \cite{MR2454590,MR2465261,MR3985543} for symmetric flows, \cite{MR2431354,MR2763344,MR2975337} for channel and pipe geometries, and \cite{MR1842436,MR3877480} for linearized problems. In the absence of any structural assumptions, the inviscid limit has only been established for analytic data near the boundary \cite{2N,KVW,MR4128497,MR4396261,KNVW,BNTT,2N-ext}.

For general smooth data, the formation of boundary vorticity and associated instabilities can lead to multi-scale phenomena that invalidate the Prandtl expansion. This mechanism, rigorously described in \cite{Grenier00CPAM,Toan-Grenier-APDE,GGN3,GreGuoToanAM}, illustrating that small-scale vorticity generation can obstruct the inviscid limit \cite{BardosTiti,ConVV17,MR3927111}. In the stationary setting, significant progress has been made on rigorous boundary layer expansions; see, for instance, \cite{GN2,Iyer1,IyerGlobal1,MR4011041,MR4642817,MR4232771,MR4097332,MR4462474,LinTaoFei,MR4752991,MR4646868,guo2024smallscalesinviscidlimits}.

The goal of the present work is to rigorously justify the Prandtl boundary layer expansion directly from the Boltzmann equation, without passing through the intermediate Navier--Stokes equations. This direct approach is particularly significant: while the Prandtl equations are formally derived from the Navier--Stokes system with no-slip boundary conditions, the validity of this derivation remains delicate due to instability and ill-posedness issues. Moreover, the Navier--Stokes equations themselves do not provide a uniformly valid description of fluid motion in the vanishing viscosity limit, especially near boundaries where sharp velocity gradients and boundary layers form. As a result, the justification of reduced models such as the Prandtl equations—or more generally, asymptotic fluid approximations—from kinetic theory becomes essential. The Boltzmann equation, equipped with diffusive boundary conditions, offers a more physically grounded framework that naturally incorporates microscopic boundary interactions without invoking the controversial no-slip assumption. In this work, we establish the hydrodynamic limit for a class of Prandtl shear flows under diffusive reflection, thereby resolving a longstanding open problem concerning singular fluid limits with large boundary velocity gradients. Our analysis yields the first rigorous derivation of the Prandtl equations from kinetic theory and introduces new techniques to capture boundary layer behavior in singular asymptotic regimes—a decisive step toward understanding this fundamental scaling limit.

To be more precise, we consider the Boltzmann equation \eqref{F-eq}. The goal is to justify that the leading behavior of the solution $F$ is given by 
\[
F=\mu+\eps^2 \ff\sqrt\mu+\eps\delta f\sqrt\mu
\]
where $\mu+\eps^2\ff\sqrt\mu$ is the approximate solution to the Boltzmann equation \eqref{F-eq}
built from the fluid equations and $\eps\delta f\sqrt\mu$ is the remainder. Here, the function $\mu$ is the local Maxwellian corresponding the the velocity $U(t,x)$ of the fluid, given by
\beq\label{mu-def}
\mu(t,x,\xi)=\mu_0(\xi-\eps U).\eeq
Here, $\mu_0$ is the global Maxwellian $\mu_0(\xi)=\frac{1}{(2\pi)^{\frac 32}}e^{-\frac{|\xi|^2}{2}}$.
One can then write the equation for the remainder as 
\beq\label{setup-eq}
\bega 
&\pt_t f+\frac 1 \eps \xi\cdot\nabla_x f+\frac 1 {\kappa\e^2}Lf+\frac{(\pt_t+\frac 1 \eps\xi\cdot\nabla_x)\sqrt\mu}{\sqrt\mu}f\\
&=-\frac 1 \delta \NS(U)\cdot\vp\sqrt\mu+\frac{\eps}{\kappa\delta}\Gamma(\ff,\ff)+\frac{\delta}{\kappa\e}\Gamma(f,f)+\frac 2 \kappa \Gamma(\ff,f)+\text{other terms}.
\enda 
\eeq
Here, $\NS(U)$ is the Navier-Stokes applied to $U$, which will be specified in the next section.
Since $U$ involves a fast scale variable $\frac{x_3}{\sqrt{\eta_0\kappa}}$, we see that $\pt_{x_3}U\sim\kappa^{-\frac 12}$, and so one needs to deal with the singular term 
\[
\frac{(\pt_t+\frac 1 \eps\xi\cdot\nabla_x)\sqrt\mu}{\sqrt\mu}f\sim \frac{1}{\eps}\frac{\xi_3\pt_{x_3}\sqrt\mu}{\sqrt\mu}\sim \pt_{x_3}U_\parallel \xi_\parallel\xi_3 f\sim \kappa^{-\frac 12}\xi_\parallel \xi_3 f.
\]
The problem indeed also appears in the paper  \cite{JangKimAPDE}. One of the crucial ideas of the authors in \cite{JangKimAPDE} is that $U$ solve the \textit{exact} Navier-Stokes equations, giving an error on the right hand side of \eqref{setup-eq} to be $O(\eps)$. Hence, by the energy and Gronwall estimate, formally one gets
\[\bega 
\frac{d}{dt}\nl f\nr_{L^2}&\lesssim\kappa^{-\frac 12}\nl f\nr_{L^2}+O(\eps),\\
\nl f\nr_{L^2}&\lesssim \eps e^{\kappa^{-\frac 12}}.
\enda 
\]
This implies the hydrodynamic limits for the \textit{Navier-Stokes} solutions $\NS(U)=0$ under the regime $\eps e^{\kappa^{-\frac 12}}\to 0$. In our current work, since $U$ does not solve the Navier-Stokes equations and only solve the Prandtl equations, which always leave a remainder in the Navier-Stokes equations, we can expect at most
\[
\NS(U)\lesssim \kappa^{N}
\]
for some $N\ge 1$ large, but finite. A rough $L^2$ estimate for the equation \eqref{setup-eq} gives
\[
\nl f\nr_{L^2}\lesssim \kappa^N e^{\kappa^{-\frac 12}},
\]
which blows up when $\kappa\to 0^+$. In this paper, we settle this problem in the analytic framework for Prandtl shear flows. We consider $\Omega=\T^2\times\R_+$ and denote 
\[
x=(x_\parallel,x_3),\quad x_\parallel=(x_1,x_2)\in\T^2,\quad x_3\ge 0.
\]
We have the formal statement of the main theorem
\begin{theorem} For any Prandtl layers of the form
\[\bega 
U(t,x_1,x_2,x_3)&=(u_e^0(t,x_3)+u_p^0\lw(t,\frac{x_3}{\sqrt{\eta_0\kappa}}\rw)\\
&\quad+\bmx &\sum_{i=0}^N\sqrt{\eta_0\kappa}^i\lw(u_e^i(t,x_\parallel,x_3)+u_p^i(t,x_\parallel,\frac{x_3}{\sqrt{\eta_0\kappa}}
\rw)\\
&\sum_{i=0}^{N-1}\sqrt{\eta_0\kappa}^{i+1}\lw(v_e^{i+1}(t,x_\parallel,\frac{x_3}{\sqrt{\eta_0\kappa}})+v_p^{i+1}(t,x_\parallel,\frac{x_3}{\sqrt{\eta_0\kappa}}
\rw)
\emx 
\enda
\]
that is analytic in $\ x_\parallel \in\T^2$ and analytic in time. There exists a solution $F$ to the Boltzmann equations with the diffusive boundary condition such that 
\[
\frac 1 \eps\nl  \frac{F-\mu}{\sqrt\mu}\nr_{L^2_{x,\xi}}\to 0
\]
as $\eps,\kappa\to 0^+$, in the regime 
\[
\eps\sim \kappa^{N}\quad\text{and}\quad N\gg 1.
\]
\end{theorem}
~\\
\textbf{
Acknowledgment.} CK is partly supported by NSF-DMS 1900923 and NSF-CAREER 2047681. He thanks Professor Yan Guo and Juhi Jang for ealier inspiring discussions. The authors thank Professor Peter Constantin for his interest to this work.~\\~\\
\textbf{
Conflict of Interest Statement.} The authors declare that there are no conflicts of interest related to this work. 
\subsection*{Ideas of the proof}
Our first ingredient is to build an approximate solution to the Boltzmann equation from the Prandtl equations. To this end,
we solve the Prandtl equations, which in turn solves the Navier-Stokes equations up to some high order $\kappa^N$. It is important that our Prandtl velocity satisfies the no-slip boundary condition, so that the local Maxwellian $\mu$ (see \eqref{mu-def}) matches the global Maxwellian $\mu_0$ at the boundary.
Next, we use the Hilbert expansion 
\[
F=\mu+\eps^2 \ff \sqrt\mu+\eps\delta f\sqrt\mu
\]
where $\ff$ is selected so that the remainder has small macroscopic part (see \eqref{ff-def}), and $f$ solves the remainder equation \eqref{setup-eq}.
Multiplying both sides by $f$ and integrating, we get
\[\bega 
&\frac 1 2 \frac{d}{dt}\nl f(t)\nr_{L^2}^2+O(1)\frac 1{\kappa\eps^2}\nl \IP f
\nr_{L^2}^2+\int_{\Omega\times\R^3}\frac{(\pt_t+\frac 1 \eps\xi\cdot\nabla_x)\sqrt\mu}{\sqrt\mu}|f|^2dxd\xi \\
&\lesssim \frac 1 \delta \lw| \int_{\Omega} \NS(U)\cdot b dx
\rw|+\frac{\eps^2}{\delta\sqrt\kappa}\lw|\la \Gamma(\ff,\ff),\frac{1}{\sqrt\kappa\eps}\IP f\ra_{L^2}
\rw|\\
&+\frac \delta{\sqrt\kappa} \lw|\la \Gamma(f,f),\frac{1}{\sqrt\kappa\e}\IP f\ra_{L^2}
\rw|+\frac\eps{\sqrt\kappa} \lw|\la \Gamma(\ff,f),\frac 1 {\sqrt\kappa\e}\IP f
\ra_{L^2}\rw|.
\enda 
\]
Our first key observation is that 
\[
\int_{\Omega\times\R^3}\frac{(\pt_t+\frac 1 \eps\xi\cdot\nabla_x)\sqrt\mu}{\sqrt\mu}|f|^2dxd\xi =\pt_{x_3}U_1b_1b_3+\text{lower order terms}.
\]
We note that $U_1$ involves boundary layer fast variable $z=\frac{x_3}{\sqrt\kappa}$, therefore $\pt_3 U_1\sim \kappa^{-1/2}$. To treat the term $\pt_3 U_1b_1b_3$, we first  note that $x_3\pt_3 U_1=z\pt_{z}U_1\lesssim 1$, therefore
\[
\int_{\Omega} \pt_3 U_1b_1b_3dx =\int x_3\pt_{x_3}U_1\cdot b_1\cdot\frac{b_3}{x_3}\lesssim \|b_1\|_{L^2}\|\pt_3 b_3\|_{L^2}.
\]
Here, we use the Hardy inequality for $b_3$, thanks to the diffusive boundary condition (see Proposition \ref{bc-con12}).
To treat the term $\pt_3 b_3$, we use the key conservation law 
\[
\eps\pt_t a+\pt_1 b_1+\pt_2 b_2+\pt_3 b_3=-\frac{1}{\eps\delta}(\nabla_x \cdot U)+\text{lower order terms}.
\]
Thus we get 
\[\bega 
&\frac 1 2 \frac{d}{dt}\nl f(t)\nr_{L^2}^2+O(1)\frac 1{\kappa\eps^2}\nl \IP f
\nr_{L^2}^2+\int_{\Omega\times\R^3}\frac{(\pt_t+\frac 1 \eps\xi\cdot\nabla_x)\sqrt\mu}{\sqrt\mu}|f|^2dxd\xi \\
&\lesssim \frac 1 \delta \lw| \int_{\Omega} \NS(U)\cdot b dx
\rw|+\frac{\eps^2}{\delta\sqrt\kappa}\lw|\la \Gamma(\ff,\ff),\frac{1}{\sqrt\kappa\eps}\IP f\ra_{L^2}
\rw|\\
&+\frac \delta{\sqrt\kappa} \lw|\la \Gamma(f,f),\frac{1}{\sqrt\kappa\e}\IP f\ra_{L^2}
\rw|+\frac\eps{\sqrt\kappa} \lw|\la \Gamma(\ff,f),\frac 1 {\sqrt\kappa\e}\IP f
\ra_{L^2}\rw|+\cdots.
\enda 
\]
Now in order to close the estimate for $L^2$, we need 
\[
\frac 1 \delta \nl \NS(U)\nr_{L^2}+\frac 1 {\e\delta}\nl\nabla_x\cdot U
\nr_{L^2}\lesssim 1.
\]
To this end, we construct a high order approximation of the Navier-Stokes equations up to order $\kappa^{2N}$. Since we lose $\eps\pt_t,\pt_{x_1},\pt_{x_2}$ of $f$ in the energy estimate, we  work with analyticity norms in these derivatives, see section \ref{func-space-BE}. Fix $\pt\in \{\eps\pt_t,\pt_{x_1},\pt_{x_2}\}$, we have a new commutator term $\frac{1}{\kappa\eps^2}[L,\pt]f$ in studying the equation for $\pt f$. To be more precise, $\pt f$ solves
\[
\lw(\pt_t+\xi\cdot\nabla_x+\frac 1 {\kappa\eps^2}L\rw)(\pt f)=-\frac{1}{\kappa\eps^2}[L,\pt]f+\text{other terms}.\]
By an explicit calculation, we see that
\[
\frac{1}{\kappa\eps^2}[L,\pt]f\sim \frac{1}{\sqrt\kappa} \pt U_1b_3 \frac{1}{\sqrt\kappa\eps}\la \IP f,\hat A_{13}\ra.
\]
Since $U_1\sim u_e^0(t,x_3)+u_p^0(t,\frac{x_3}{\sqrt{\eta_0\kappa}})+O(\sqrt\kappa)$, we have 
\[
\pt_\parallel U_1=\pt_{x_\parallel}\lw\{u_e^0(t,x_3)+u_p^0(t,\frac{x_3}{\sqrt{\eta_0\kappa}})\rw\}+O(\sqrt\kappa) \sim\sqrt\kappa.\]
This gives us the formal estimate $\frac{1}{\kappa\eps^2}\la [L,\pt]f,\pt f\ra\lesssim \nl \pt  f\nr_{L^2}\frac{1}{\sqrt\kappa\eps}\nl \IP f\nr_{L^2}$. Finally, we construct the corresponding analytic fluid solutions and estimate their remainders in the analytic space, see Theorem \ref{thm-ana-fluid}.
\section{Analysis of the Euler and Prandtl equations}
In this paper, we assume that the initial data is of the form 
\beq\label{initial-data}
\bega
u|_{t=0}&=u_e^0( x_3)+u_p^0\lw(\frac {x_3} {\sqrt{\eta_0\kappa}}\rw)+\sum_{i=1}^N {(\eta_0\kappa)}^{\frac i 2}\lw(
u_{e,0}^i( x_\parallel,  x_3)+u_{p,0}^i\lw(x_\parallel ,\frac  {x_3}{\sqrt{\eta_0\kappa}}\rw)
\rw),\\
v|_{t=0}&=\sqrt{\eta_0\kappa} v_{e,0}^1(x_\parallel,{ x_3})+\sum_{i=1}^{N-1} {(\eta_0\kappa)}^{\frac{i+1} 2}\lw(v_{p,0}^{i+1}(x_\parallel ,\frac{ x_3} {\sqrt{\eta_0\kappa}})+v_{e,0}^{i+1}(x_\parallel,{x_3} )\rw)\\
&\quad+{(\eta_0\kappa)}^{\frac{N+1} 2}v_{p,0}^{N+1}(x_\parallel,\frac  { x_3} {\sqrt{\eta_0\kappa}}),
\enda
\eeq
and $\int_{\T^2}u_{e,0}^id{x_\parallel }=\int_{\T^2}v_{e,0}^id{x_\parallel }=0$.
The Navier-Stokes equations on $\Omega=\mathbb T^2\times \R_+$ reads
\[
\bega 
&\pt_t u+u\cdot \nabla_{x_\parallel}u+v\pt_{x_3} u+\nabla_{x_\parallel}p-\eta_0\kappa\triangle u=0,\\
&\pt_t v+u\cdot\nabla_{x_\parallel}v+v\pt_{x_3}v+\pt_{x_3}p-\eta_0\kappa\triangle v=0,\\
&\nabla_{x_\parallel}\cdot u+\pt_{x_3}v=0.
\enda 
\]
Here we denote \[
U(t,x_1,x_2,x_3)=(u(t,x_\parallel,x_3),v(t,x_\parallel,x_3)),\qquad x_\parallel \in \mathbb T^2,\quad x_3\in \mathbb R_+.
\]
Define the stretched variable 
\Be
z = \frac{x_3}{\sqrt {\eta_0\kappa}} \label{def:z}.
\Ee
We build an approximation solution $U=(u_a,v_a)$ of the form
\beq\label{EP-form}
\bega
u_a&=\sum_{i=0}^N\sqrt{\eta_0\kappa}^{i}\lw(u_e^i(t,x_\parallel,x_3)+u_p^i(t,x_\parallel,z)
\rw),\\
v_a&=v_e^0(t,x_\parallel,{x_3})+\sum_{i=0}^{N-1} {\sqrt{\eta_0\kappa}}^{i+1}\lw(v_p^{i+1}(t,x_\parallel,z)+v_e^{i+1}(t,x_\parallel,{x_3})\rw)+\sqrt{\eta_0\kappa}^{N+1}v_p^{N+1}(t,x_\parallel,z),\\
p_a&=\sum_{i=0}^N\sqrt{\eta_0\kappa}^i \lw(p_e^i(t,x_\parallel,{x_3})+p_p^i(t,x_\parallel,z)
\rw).
\enda
\eeq
We have the zero-order Euler equation
\[
\bega
&\pt_t u_e^0+(u_e^0\cdot\nabla_{x_\parallel}+v_e^0\pt_{x_3})u_e^0+\nabla_{x_\parallel} p_e^0=0,\\
&\pt_t v_e^0+(u_e^0\cdot\nabla_{x_\parallel}+v_e^0\pt_{x_3})v_e^0+\pt_{x_3} p_e^0=0,\\
&\nabla_{x_\parallel}\cdot  u_e^0+\pt_{x_3} v_e^0=0,\quad v_e^0|_{{x_3}=0}=0.\\
\enda
\]
Due to the mismatch in the boundary condition, we add a Prandtl corrector $$(u_p^0(t,x_\parallel,z),\sqrt{\eta_0\kappa} v_p^1(t,x_\parallel,z)$$ given by
\[\bega
&\pt_t u_p^0-\pt_z^2 u_p^0+u_p^0\pt_x u_e^0+(u_e^0|_{x_3=0}+u_p^0)\pt_{x_\parallel} u_p^0+(v_e^1|_{z=0}+v_p^1+z\pt_{x_3}v_e^0|_{x_3=0})\pt_z u_p^0=0,\\
&v_p^1=-\int_z^\infty (\nabla_x\cdot u_p^0)dz',\\
&u_p^0|_{z=0}=-u_e^0|_{{x_3}=0},\quad u_p^0|_{z=\infty}=0,\\
&u_p^0|_{t=0}=u_{p,0}.
\enda
\]
The new approximate solution does not satisfy the equation the boundary condition $v_p^0|_{x_3=0}=0$, hence one needs to introduce $(u_e^1,v_e^1)$ solving \[
\bega
&\pt_t u_e^1+(u_e^1\pt_{x_\parallel}+v_e^1\pt_{x_3})u_e^0+(u_e^0\pt_{x_\parallel}+v_e^0\pt_{x_3})u_e^1+\pt_{x_\parallel} p_e^1=0,\\
&\pt_t v_e^1+(u_e^1\pt_{x_\parallel} +v_e^1\pt_{x_3})v_e^0+(u_e^0\pt_{x_\parallel}+v_e^0\pt_{x_3})v_e^1+\pt_{x_3} p_e^1=0,\\
&\pt_{x_\parallel} u_e^1+\pt_{x_3} v_e^1=0,\\
&(u_e^1,v_e^1)|_{t=0}=(u_{e,0}^1,v_{e,0}^1),\\
&v_e^1|_{{x_3}=0}=-v_p^1|_{z=0}.
\enda
\]
For $0\le m\le N$, we have the linearized Euler equation 
\beq\label{Euler}
\bega
&\pt_t u_e^m+(u_e^0\cdot\nabla_{x_\parallel}+v_e^0\pt_{x_3})u_e^m+(u_e^m\cdot\nabla_{x_\parallel}+v_e^m\pt_{x_3})u_e^0+\pt_{x_\parallel} p_e^m\\
&=-\chi_{\{m\ge 1\}}\sum_{1\le i\le m-1} \lw(u_e^{m-i}\pt_x u_e^i+v_e^{m-i}\pt_{x_3} u_e^i\rw)+\chi_{\{m\ge 2\}}\triangle u_e^{m-2},\\
&\pt_t v_e^m+(u_e^m\cdot\nabla_{x_\parallel}+v_e^m\pt_{x_3}) v_e^0+(u_e^0\cdot\nabla_{x_\parallel}+v_e^0\pt_{x_3}) v_e^m+\pt_{x_3} p_e^m,\\
&\quad=-\chi_{\{m\ge 1\}}\sum_{1\le i\le m-1} \lw(u_e^{m-i}\pt_x v_e^i+v_e^{m-i}\pt_{x_3} v_e^i\rw)+\chi_{\{m\ge 2
\}}\triangle v_e^{m-2},\\
&\nabla_{x_\parallel}\cdot u_e^m+\pt_{x_3}v_e^m=0,\qquad v_e^m|_{x_3=0}=-v_p^m|_{z=0},\qquad 1\le m\le N.
\enda 
\eeq
The Prandtl equations $(u_p^m,v_p^m)$ solves
\beq\label{Prandtl}
\bega
&(\pt_t-\pt_z^2)u_p^m+\lw\{u_p^m \pt_{x_\parallel}u_e^0|_{{x_3}=0}+\pt_x u_p^m\lw(u_e^0|_{{x_3}=0}+u_p^0
\rw)+\pt_{x_\parallel}u_p^0(u_e^m|_{{x_3}=0}+u_p^m)\rw\}\\
& +(v_e^1|_{{x_3}=0}+z\pt_{x_3} v_e^0|_{{x_3}=0}+v_p^1)\pt_z u_p^m\\
&+(v_e^{m+1}|_{{x_3}=0}+v_p^{m+1})\pt_z u_p^0\cdot \chi_{\{m\le N-1\}}+\chi_{\{1\le m\le N-1\}}\pt_z p_p^{m+1}=f_{p,\parallel}^m,\\
&u_p^m|_{z=0}=-u_e^m|_{x_3=0},\quad u_p^m|_{z=\infty}=0,\qquad 0\le m\le N,\\
&v_p^{m+1}=-\int_z^\infty \nabla_{x_\parallel}\cdot u_p^m dz',
\enda 
\eeq
where
\[
\bega
f_{p,\parallel}^m&=-\sum_{\substack{i+j+k=m\\i\le m-1\\k\ge 0}}\frac{z^k}{k!}\pt_{x_3}^{k}\nabla_{x_\parallel}\cdot u_e^j|_{{x_3}=0}u_p^i-\sum_{0\le i\le m-1}\pt_{x_\parallel} u_p^i(u_e^{m-i}|_{{x_3}=0}+u_p^{m-i})\\
&-\sum_{\substack{i+j+k=m\\k\ge 1}}\frac{z^k}{k!}\pt_{x_3}^k u_e^i|_{{x_3}=0}\pt_{x_\parallel} u_p^j-\sum_{\substack{1\le i,j\le N\\k\ge 0\\i+j+k=m}}\frac{z^k}{k!}\pt_{x_3}^{k+1}u_e^j|_{{x_3}=0}v_p^i\\
&-\sum_{\substack{
k+j=m+1\\
k\ge 1\\
0\le j\le m-1
}
}\frac{z^k}{k!}\pt_{x_3}^k v_e^0|_{{x_3}=0}\cdot \pt_z u_p^j-\sum_{\substack{i+j=m\\
0\le i\le m-1\\0\le j\le m-1\\
}
}(v_e^{i+1}|_{{x_3}=0}+v_p^{i+1})\pt_z u_p^j \cdot \chi_{\{m\le N-1\}}\\
&-\chi_{\{m\ge 2\}}\cdot \sum_{\substack{
i+j+k=m\\
0\le i\le N-1\\
0\le j\le N\\
k\ge 1
}
}\frac{z^k}{k!}\pt_{x_3}^k v_e^{i+1}|_{{x_3}=0}\pt_zu_p^j+\chi_{\{m\ge 2\}}\triangle_{x_\parallel} u_p^{m-2}.\\
\enda
\]

\subsection{Analytic norms of the fluid equations}
In this section, we first introduce the analytic norms for the fluid solutions, namely the approximate Prandtl solutions on the domain $\mathbb T^2\times\R_+$.
 We adopt the multi-indices notation for the time-space derivatives: 
\beq\label{multi-i}
\pt^{\alpha} =  (\eps\pt_t)^{\alpha_0}\pt_{x_1}^{\alpha_1} \pt_{x_2}^{\alpha_2}, \  \   \alpha = (\alpha_0, \alpha_1, \alpha_2) \in \mathbb N^3_0,
\eeq
where $\mathbb N_0 = \{0 , 1,2, \cdots\}$. We also denote 
\[
\la \al\ra=1+|\al|=1+\al_0+\al_1+\al_2,\qquad \al\in\mathbb N_0^3.
\]
Define, for $\al\in \mathbb N$ and $\pt^\al$ as above, 
\beq\label{coe-def}
\AA_\al(t)=\frac{\tau(t)^{|\al|}}{\al!}\la\al\ra^9,\qquad \la \al\ra=1+|\al|
\eeq
for a continuous function $\tau(t)=\tau_0-M_0t$  for some constants $\tau_0,M_0>0$ depending only on the initial data.
For a function $f=f(t,x)$ that only depends on $x\in \Omega$ and $t\ge 0$, we define
\[
\||f(t)\||_{\tau,p}=\sum_{\al\in \mathbb{N}^3_0}
\AA_\al(t)\nl \pt^\al f(t) \nr_{L^p_x}.
 \]
 The existence and uniqueness of the Prandtl and Euler solutions in the analytic space $\nl \cdot\nr_{\tau_0,2}$ and $\nl
\nr_{\tau_0,\infty}$ are standard, and we refer the readers to \cite{SammartinoCaflisch1,KukavicaVV13,Igor-Vlad-Euler, GV-Pr-Gevrey, KNVW} for the proof.
\begin{theorem} Suppose the initial data $(u_p^m,v_p^m)|_{t=0}$ and $(u_e^m,v_e^m)|_{t=0}$  given in \eqref{initial-data} is bounded in $\nl \cdot\nr_{\tau_0,2}\cap \nl \cdot\nr_{\tau_0,\infty}$ for $0\le m\le N$ for some $\tau_0>0$,  there exists a unique local solution in the space $\nl \cdot\nr_{\tau_0/2,2}\cap \nl \cdot\nr_{\tau_0/2,\infty}$. \end{theorem} 
Our next goal is to show that the approximate solution gives a small error in the analytic spaces for $N$ large. For $(u_a,v_a,p_a)$ defined in \eqref{EP-form}, we define 
\[\bega 
\text{NS}_{\parallel} (u_a)&=\pt_t u_a+(u_a\cdot\nabla_{x_\parallel}+v_a\pt_{x_3})u_a-{\eta_0\kappa} \triangle u_a+\nabla_{x_\parallel} p_a,\\
\text{NS}_{x_3}(v_a)&=\pt_t v_a+(u_a\cdot\nabla_{x_\parallel}+v_a\pt_{x_3})v_a-{\eta_0\kappa} \triangle v_a+ \pt_{x_3} p_a.
\enda \]
 
 \beq\label{Err-ua}
\bega 
&\text{NS}_{\parallel} (u_a)=\sum_{\substack{i+j\ge N+1\\i,j\le N}}\sqrt{\eta_0\kappa}^{i+j}\lw(u_e^i \pt_{x_\parallel} u_e^j+v_e^i\pt_{x_3} u_e^j
\rw)\\
&+\sum_{\substack{
0\le i,j\le N\\k\ge 0\\i+j+k\ge N+1
}} \frac{\sqrt{\eta_0\kappa}^{i+j+k} z^k}{k!}u_p^i \pt_{x_3}^k \pt_{x_\parallel}u_e^j|_{{x_3}=0}\\
&+\sum_{\substack{0\le i,j\le N\\i+j\ge N+1}}\sqrt{\eta_0\kappa}^{i+j}(u_e^i|_{{x_3}=0}+u_p^i)\pt_{x_\parallel} u_p^j+\sum_{\substack{0\le i,j\le N\\i+j+k\ge N+1}}\sqrt{\eta_0\kappa}^{i+j+k}\frac 1 {k!}z^k\pt_{x_3}^k u_e^i|_{{x_3}=0}\pt_{x_\parallel} u_p^j\\
&+\sum_{\substack{
0\le i\le N-1\\
1\le j\le N\\
k\ge 0\\
i+j+k\ge N
}
}
\sqrt{\eta_0\kappa}^{i+j+k+1}v_p^{i+1}\frac{z^k}{k!}\pt_{x_3}^{k+1}u_e^j(t,{x_\parallel},0)+\sum_{j=0}^N\sqrt{\eta_0\kappa}^{N+1+j}v_p^{N+1}\pt_{x_3} u_e^j\\
&+\sum_{\substack{k+j-1\ge N+1\\
0\le j\le N\\
k\ge 1
}
}\sqrt{\eta_0\kappa}^{k+j-1}\frac{z^k}{k!}\pt_{x_3}^k v_e^0|_{{x_3}=0}\pt_z u_p^j+\sum_{\substack{i+j\ge N+1\\
0\le i\le N-1\\
0\le j\le N
}
}\sqrt{\eta_0\kappa}^{i+j}(v_p^{i+1}+v_e^{i+1}|_{{x_3}=0})\pt_z u_p^j\\
&+\sum_{\substack{
i+j+k\ge N+1\\
0\le i\le N-1\\
0\le j\le N\\
k\ge 1
}
}\sqrt{\eta_0\kappa}^{i+j+k}\frac{z^k}{k!}\pt_{x_3}^k v_e^{i+1}|_{{x_3}=0}\pt_zu_p^j.\enda 
\eeq
and 
{\small\beq\label{Err-va}
\bega 
 \text{NS}_{x_3}(v_a) 
&=\sum_{\substack{
i+j\ge N+1\\0\le i,j\le N
}}\sqrt{\eta_0\kappa}^{i+j}\lw(u_e^i\pt_{x_\parallel} v_e^j+v_e^i\pt_{x_3} v_e^j
\rw)\\
&\quad+\sum_{\substack{
0\le i,j\le N\\
k\ge 0\\
i+j+k\ge N+1
}
}\frac{z^k}{k!}u_p^i\pt_{x_3}^k\pt_{x_\parallel}v_e^j|_{{x_3}=0}+\sum_{\substack{0\le i\le N\\0\le j\le N-1\\
i+j\ge N
}
}\sqrt{\eta_0\kappa}^{i+j+1}(u_e^i|_{{x_3}=0}+u_p^i)\pt_{x_\parallel} v_p^{j+1}\\
&\quad+\sum_{\substack{
i\le N, j\le N-1, k\ge 1\\i+j+k\ge N
}
}\sqrt{\eta_0\kappa}^{i+j+k+1}\frac{z^k}{k!}\pt_{x_3}^k u_e^i|_{{x_3}=0}\pt_{x_\parallel} v_p^{j+1}+\sqrt{\eta_0\kappa}^{N+1}u_a \pt_{x_\parallel} v^{N+1}_p \\
&\quad+\sum_{\substack{
i\le N-1\\j\le N\\k\ge 0\\i+j+k\ge N
}
}\sqrt{\eta_0\kappa}^{i+j+k+1}\frac{z^k}{k!}\pt_{x_3}^{k+1}v_e^j|_{{x_3}=0}\cdot v_p^{i+1}+\sqrt{\eta_0\kappa}^{N+1}v_p^{N+1}\sum_{j=0}^N \sqrt{\eta_0\kappa}^j \pt_{x_3} v_e^j\\
&\quad+\sum_{\substack{k+j\ge N\\
j\le N-1\\k\ge 1
}
}\frac{\sqrt{\eta_0\kappa}^{k+j}}{k!}z^k\pt_{x_3}^k v_e^0|_{{x_3}=0}\cdot\pt_z v_p^{j+1}+\sum_{\substack{i+j\ge N\\
0\le i,j\le N-1}
}\sqrt{\eta_0\kappa}^{i+j+1}(v_e^{i+1}|_{{x_3}=0}+v_p^{i+1})\pt_z v_p^{j+1}\\
&\quad+\sum_{\substack{i+j+k\ge N\\
i,j\le N-1\\k\ge 1
}
}\sqrt{\eta_0\kappa}^{i+j+k+1}\frac{z^k}{k!}\pt_{x_3}^{k+1}v_e^{i+1}|_{{x_3}=0}\pt_z v_p^{j+1} +\sqrt{\eta_0\kappa}^{N+1}v_p^{N+1}\sum_{j=0}^{N-1}\sqrt{\eta_0\kappa}^{j}\pt_z v_p^{j+1}\\
&\quad+\sqrt{\eta_0\kappa}^N\pt_zv_p^{N+1}\sum_{\substack{
k\ge 1\\
}
}\frac{z^k}{k!}\sqrt{\eta_0\kappa}^{k}\pt_{x_3}^k v_e^0|_{{x_3}=0}\\
&\quad+\pt_z v_p^{N+1}\sum_{i=0}^{N-1}\sqrt{\eta_0\kappa}^{N+i+1}\lw(v_e^{i+1}+v_p^{i+1}
\rw)+\sqrt{\eta_0\kappa}^{2N+1}v_p^{N+1}\pt_z v_p^{N+1}\\
&\quad+\sqrt{\eta_0\kappa}^{N+1}\pt_t v_p^{N+1}-\sum_{m=N+1}^{N+2}\sqrt{\eta_0\kappa}^m (\triangle v_e^{m-2}+\triangle_{x_\parallel} v_p^{m-2}).
\enda 
\eeq}
 We define 
 \[
 \NS(U)=\bmx \text{NS}_\parallel(u_a)\\\text{NS}_{x_3}(v_a)
 \emx.
 \]
\begin{theorem}\label{thm-ana-fluid} There holds 
 \[
 \nl\NS(U)\nr_{\tau_0,2}+\nl \div U\nr_{\tau_0,2}\lesssim {\kappa}^{\frac N 2}.
 \]
 \end{theorem} 
 \begin{proof} It suffices to bound the infinite sum terms appearing in \eqref{Err-ua} and \eqref{Err-va}. We will give the proof for one term only, as all the other terms can be done similarly. The key idea is that the expansion is derived through the Taylor expansion of the Euler solutions near the boundary, in $x_3=\sqrt{\eta_0\kappa} z$. In fact, to build the approximate solution, we only need to expand up to a finite, but large order, and then use the truncated Taylor expansion. The left over term can be estimated thanks to the Mean Value theorem for the Taylor expansion. To illustrate the method, we give the bound the analytic norm for the term \[
 \sum_{\substack{
0\le i,j\le N\\k\ge 0\\i+j+k\ge N+1
}} \frac{\sqrt{\eta_0\kappa}^{i+j+k} z^k}{k!}u_p^i \pt_{x_3}^k \pt_xu_e^j|_{{x_3}=0} \]
which is the second term in \eqref{Err-ua}. We fix $i,j\le N$. 
 It suffices to consider $k\ge 2N+1$, since if $k\le 2N$, the above sum is a sum of finite terms, and hence it is analytic in $\eps\pt_t$ and $\pt_{x_\parallel}$, using the analyticity of $u_p^i$ and $u_e^j|_{x_3=0}$. Fix $i,j\in [0,N]$, we consider 
 \[
 \sum_{k=2N+1}^\infty \frac 1{k!} {(\eta_0\kappa)}^{\frac {i+j+k}{2}}z^k u_p^i\pt_{x_3}^k \nabla_{x_\parallel}\cdot u_e^j|_{x_3=0}.
 \]
 Define $\wtd u_e^j=\nabla_{x_\parallel}\cdot u_e^j$. The above term can be written as
 \beq\label{ran-s9}
 \sum_{k=2N+1}^\infty \frac 1{k!} {(\eta_0\kappa)}^{\frac {i+j+k}{2}}z^k u_p^i\pt_{x_3}^k \wtd u_e^j|_{x_3=0}=(\eta_0\kappa)^{\frac{i+j}{2}}u_p^i\sum_{k=2N+1}^\infty \frac{1}{k!}x_3^{k}\pt_{x_3}^k\wtd u_e^j|_{x_3=0}.
  \eeq
  By the Mean Value theorem, there exists $x_3^\star\in (0,x_3)$ such that 
 \[\bega 
\sum_{k=2N+1}^\infty \frac 1{k!}x_3^k\pt_{x_3}^k \wtd u_e^j|_{x_3=0} \wtd u_e^j&=x_3^{2N+1}\frac 1{(2N+1)!}\pt_{x_3}^{2N+1}\wtd u_e^j(x_3^\star)\\
&=(z\sqrt{\eta_0\kappa})^{2N+1}\frac 1{(2N+1)!}\pt_{x_3}^{2N+1}\wtd u_e^j(x_3^\star).
\enda \]
Hence, the term \eqref{ran-s9} can be written as
\[
(\eta_0\kappa)^{\frac{i+j}{2}+N+\frac 1 2} z^{N+\frac 1 2 } u_p^i \frac{1}{(2N+1)!}\pt_{x_3}^{2N+1}\wtd u_e^j(x_3^\star).
\]
Since the Euler solution is analytic with radius $\tau_0>0$, the above term can be bounded by
\[
(\eta_0\kappa)^{\frac{i+j+1}{2}+N}\nl
z^{N+\frac 1 2}u_p^i\nr_{L^\infty_z}\tau_0^{-(2N+1)}\lesssim \kappa^{\frac N 2}.
\]
Here, we use also use the exponential decay of the Prandtl solutions, which are standard, see, for example \cite{SammartinoCaflisch1}.
The proof is complete. \end{proof}
 
\subsection{Local Maxwellian of the Prandtl layers}
In this section, we collect and prove estimates on the analytic norm of the local Maxwellian around a velocity $\eps U(t,x)$ where $x\in \Omega=\T^2\times\R_+$.
\begin{proposition} \label{local-1}
Let $s_0>0$ be a positive number. We can make $\tau_0\ll \tau_U$ and $p_0\ll 1$ such that 
\beq\label{local-1-1}
\sum_{|\al|\ge 1}\frac{\tau_0^{|\al|}}{\al!} \nl e^{p_0 (|\xi|^2+|\vp|^2)} \pt^\al (\mu^{s_0})
\nr_{L^\infty_{x,\xi}}\lesssim \eps \kappa^{\frac 12}.
\eeq
\end{proposition}
\begin{proof} First, we show \eqref{local-1-1}.
We note that 
\[
\mu^{s_0}(t,x,\xi)=\frac{1}{(2\pi)^{3s_0/2}}e^{-s_0\frac{|\xi-\eps U|^2}{2}}.
\]
Up to scaling, we may assume $s_0=2$. Since a product of analytic functions is analytic, it suffices to justify that, for fixed $\ell \in \{1,2,3\}$, 
\[
\sum_{|\al|\ge 1}\frac{\tau_0^{|\al|}}{\al!} \lw\{\lw|\pt^\al \lw(e^{-(\xi_{\ell }-\eps U_{\ell })^2}
\rw)
\rw|+\lw| e^{p_0(\xi_\ell-\eps U_\ell)^2} \pt^\al 
\lw( e^{-(\xi_\ell -\eps U_\ell )^2}\rw)\rw|
\rw\}\lesssim \eps \sqrt\kappa.
\]  
Combining Proposition \ref{composite} and Lemma \ref{exp-ana-ele}, we get the result. \end{proof}
\begin{lemma} \label{M-local-2} Assume that 
\[
\sup_{\al\in\mathbb N_0^3}\frac{\tau_0^{|\al|}}{\al!}\nl \pt^\al \pt_{x_3}U
\nr_{L^\infty_x}\lesssim {\kappa}^{-\frac 12},\qquad\sup_{\al\in\mathbb N_0^3}\frac{\tau_0^{|\al|}}{\al!}\nl \pt^\al \pt_{x_i}U
\nr_{L^\infty_x}\lesssim 1 \quad\text{for}\quad i\in\{1,2\}
\]
and 
\[\sup_{\al\in\mathbb N_0^3}\frac{\tau_0^{|\al|}}{\al!}\nl \pt^\al (\eps\pt_t U)
\nr_{L^\infty_x}\lesssim 1,
\]
then for $\rho\ll p_0$, we have 
\[
\sum_{|\al|\ge 1}\frac{\tau_0^{|\al|}}{\al!}\nl e^{\rho|\xi|^2} \pt^\al \pt_{x_3}\lw(\mu^{p_0}\rw)
\nr_{L^\infty_{x,\xi}}\lesssim \eps\kappa^{-\frac 12},\qquad \sum_{|\al|\ge 1}\frac{\tau_0^{|\al|}}{\al!}\nl e^{\rho|\xi|^2}\pt^\al \pt_{x_i}\lw(\mu^{p_0}\rw)
\nr_{L^\infty_{x,\xi}}\lesssim \eps\quad\text{for}\quad i\in\{1,2\}
\]
and 
\[\sum_{|\al|\ge 1}\frac{\tau_0^{|\al|}}{\al!}\nl e^{\rho|\xi|^2} \pt^\al \pt_t \lw(\mu^{p_0}\rw)
\nr_{L^\infty_{x,\xi}}\lesssim 1.
\]
\end{lemma}
\begin{proof} We show the inequalities for $\pt_{x_3}$ only. Up to scaling, we may assume $p_0=2$. We have 
\[
\pt_{x_3}\lw(e^{-|\xi-\eps U|^2}\rw)=-2(\xi-\eps U)\cdot \eps\pt_{x_3}U e^{-|\xi-\eps U|^2}=2\eps \pt_{x_3}U\cdot(\xi-\eps U) e^{-|\xi-\eps U|^2}
\]
Hence 
\[\bega 
\sum_{\al\in\mathbb N_0^3}\frac{\tau_0^{|\al|}}{\al!}\lw|\la \vp\ra^m \pt^\al \pt_{x_3}\lw(e^{-|\xi-\eps U|^2}\rw)\rw|&\lesssim \sum_{\al\in\mathbb N_0^3}\frac{\tau_0^{|\al|}}{\al!} \lw|\eps \pt^\al \pt_{x_3} U\rw|\cdot \sum_{\al\in\mathbb N_0^3}\frac{\tau_0^{|\al|}}{\al!}\lw|\la\vp\ra^m \pt^\al \lw(\vp e^{-|\vp|^2}
\rw)
\rw|\\
&\lesssim \eps\kappa^{-\frac 12} \sum_{\al\in\mathbb N_0^3}\frac{\tau_0^{|\al|}}{\al!}\lw|\pt^\al \lw(\vp e^{-|\vp|^2}
\rw)
\rw|\lesssim 1.\enda
\]
using \eqref{local-1-1}. The proof is complete.
\end{proof}
\section{Analysis of the Boltzmann equations}
\subsection{Preliminaries} 
In this paper, we recall 
 \beq\label{Gamma-def}
 \Gamma(f,g)=\Gamma_+(f,g)-\Gamma_-(f,g)
 \eeq
 where 
 \beq\label{Gamma-pm}
 \begin{cases}
 \Gamma_+(f,g)&=\int_{\R^3\times\S^2}|(\xi-\xi_\star)\cdot\w|\sqrt\mu(\xi_\star)\lw(f(\xi')g(\xi_\star')+g(\xi')f(\xi_\star')
 \rw)d\w d\xi_\star,\\
 \Gamma_-(f,g)&=\int_{\R^3\times\S^2}|(\xi-\xi_\star)\cdot\w|\sqrt\mu(\xi_\star)\lw(f(\xi)g(\xi_\star)+g(\xi)f(\xi_\star)
 \rw)d\w d\xi_\star,
 \end{cases}
 \eeq
 and 
 \beq\label{vec}
\begin{cases}
\xi'&=\xi-\lw\{(\xi-\xi_\star)\cdot \w\rw\} \w\\
\xi_\star'&=\xi_\star+\lw\{(\xi-\xi_\star)\cdot \w\rw\} \w.\\
\end{cases}
\eeq
 We also recall the linear operator around the local Maxwellian $\mu(\xi)=\mu_0(\xi-\eps U)$, given by
\beq\label{L-def}
\bega 
Lf&=-\int_{\R^3\times\S^2}|(\xi-\xi_\star)\cdot \w|\sqrt\mu(\xi_\star)\\
&\quad\times \lw\{\sqrt\mu(\xi')f(\xi_\star')+f(\xi')\sqrt\mu(\xi_\star')-\sqrt\mu(\xi)f(\xi_\star)-f(\xi)\sqrt\mu(\xi_\star)
\rw\}d\w d\xi_\star.
\enda
\eeq
We also define 
\beq\label{L0-def}
\bega 
L_0f(\xi)&=-\int_{\R^3\times\S^2}|(\xi-\xi_\star)\cdot \w|\sqrt\mu_0(\xi_\star)\\
&\quad\times \lw\{\sqrt\mu_0(\xi')f(\xi_\star')+f(\xi')\sqrt\mu_0(\xi_\star')-\sqrt\mu_0(\xi)f(\xi_\star)-f(\xi)\sqrt\mu_0(\xi_\star)
\rw\}d\w d\xi_\star.
\enda
\eeq
For $i,j\in \{1,2,3\}$ and $\vp=\xi-\eps U\in\mathbb R^3 $, we define 
\beq\label{Aij}
\hat A_{ij}(\vp)=\lw(\vp_i\vp_j-\frac 1 3 \delta_{ij}|\vp|^2
\rw)\sqrt\mu_0(\vp),\qquad A_{ij}(\vp)=L^{-1}(\hat A_{ij}).
\eeq

\subsection{Real Analytic Function spaces }\label{func-space-BE}
Recall \eqref{multi-i} and \eqref{coe-def}. For a function $f=f(t,x,\xi)$ where $(t,x,\xi)\in [0,\infty)\times\Omega\times\R^3$ and a function $\tau: t\in [0,\infty)\to \tau(t)\in [0,\infty)$, we define the following real-analytic norms:
\beq\label{EDH}
\bega 
\mathcal E_f(t)&=\nl \AA_\al(t) \nl \pt^\al f(t)\nr_{L^2_{x,\xi}}\nr_{\ell_\al^2},\\
\mathcal D_f(t)&=\eps^{-1}\kappa^{-\frac 12}\nl \nl\nu^{\frac 12}\AA_\al (t)\IP \pt^\al f(t)
\nr_{L^2_{x,\xi}}
\nr_{\ell_\al^2},\\
\mathcal H_{\rho,f}(t)&=\nl \nl \AA_\al(t) e^{\rho|\xi|^2}\pt^\al f(t)\nr_{L^\infty_{t,x,\xi}}\nr_{\ell_\al^2}.\enda 
\eeq
To treat the loss of derivative, we also define the norm 
\[
\mathcal Y_f(t)^2=\sum_{\al\in\mathbb N_0^3} \frac{|\al|+1}{\tau}\AA_\al(t)^2 \nl \pt^\al f
\nr _{L^2_{x,\xi}}^2.
\] For a function $f(t,x,\xi)$ defined on $t\ge 0$, $x\in \Omega$ and $\xi\in\R^3$, we denote 
\beq\label{Pf-def}
\P f(t,x,\xi)=\lw(a+b\cdot\vp+c\frac{|\vp|^2-3}{2}
\rw)\sqrt\mu.
\eeq
For $1\le p\le \infty$, we denote 
\[\bega 
\nl \P f
\nr_{L^p_x}&=\|a\|_{L^p(\Omega)}+\sum_{i=1}^3 \|b_i\|_{L^p(\Omega)}+\|c\|_{L^p(\Omega)},\\
\nl \P f
\nr_{L^p_{x,\xi}}&=\lw\{\int_{\Omega\times\R^3}|\P f(t,x,\xi)|^pdxd\xi\rw\}^{1/p}.
\enda 
\]
We also denote 
\[
f_\al=\pt^\al f\]
\and 
\[
a_\al=\int_{\R^3}f_\al \sqrt\mu d\xi,\quad b_\al=\int_{\R^3}f_\al \vp \sqrt\mu d\xi,\qquad c_\al=\int_{\R^3}f_\al \frac{|\vp|^2-3}{3}\sqrt\mu d\xi.
\]
From \eqref{Pf-def}, it is obvious that 
\beq\label{PFal}\P f_\al =a_\al\sqrt\mu+b_\al \cdot\vp\sqrt\mu+c_\al \frac{|\vp|^2-3}{2}\sqrt\mu.
\eeq
\subsection{Remainder Equations}
We consider the equation \eqref{F-eq}. 
 For a general velocity field $U=U(t,x)\in \R^3$ where $t\ge 0$ and $x\in\T^2\times \R_+$ and a scalar function $P$, 
we define 
\Be\label{def:NS}
\NS(U, P)=\pt_t U+(U\cdot\nabla_x) U+\nabla_x \lw(P-\frac {\eta_0}3 \div U
\rw)-\kappa\eta_0\triangle U.
\Ee
We define 
\beq\label{ff-def}
\ff=P\sqrt\mu-\kappa\sum_{i,j}\pt_i U_jA_{ij}(\vp).\eeq
For $L^2$ estimate, we expand the Boltzmann solution $F$ around $\mu$ as follows
\beq\label{F-exp}
F=\mu+\eps^2\ff\sqrt\mu+\eps\delta \sqrt\mu f,
\eeq
For the $L^\infty$ estimate, we will use the expansion
\beq\label{h-def}
F=\mu+\sqrt\mu\eps^2\ff+\eps\delta\sqrt\mu_M h,
\eeq
where 
\beq\label{mu-M-def}
\mu_M(\xi)=\frac{1}{(2\pi T_M)^{\frac 3 2}}e^{-\frac{|\xi|^2}{2T_M}},\qquad \frac 12<T_M<1.
\eeq
We define the error term
\beq\label{Ra-def}
\textbf{R}_a=(\pt_t+\frac 1\eps \xi\cdot\nabla_x)(\mu+\sqrt\mu\eps^2\ff)-\frac2\kappa Q(\mu,\ff\sqrt\mu)-\frac{\eps^2}{\kappa}Q(\sqrt\mu \ff,\sqrt\mu \ff),
\eeq
the bilinear terms 
 \beq\label{GammaM}
 \bega
  \Gamma_M(f,g)&=\frac{1}{\sqrt{\mu_M}}Q(\sqrt{\mu_M}f,\sqrt{\mu_M}g)=\Gamma_{M,+}(f,g)-\Gamma_{M,-}(f,g),\\ 
 \Gamma_{M,+}(f,g)&=\int_{\R^3\times\S^2}|(\xi-\xi_\star)\cdot\w|\sqrt\mu_M(\xi_\star)\lw(f(\xi')g(\xi_\star')+g(\xi')f(\xi_\star')
 \rw)d\w d\xi_\star,\\
 \Gamma_{M,-}(f,g)&=\int_{\R^3\times\S^2}|(\xi-\xi_\star)\cdot\w|\sqrt\mu_M(\xi_\star)\lw(f(\xi)g(\xi_\star)+g(\xi)f(\xi_\star)
 \rw)d\w d\xi_\star,
 \enda
 \eeq
and the linearized operator 
\[\bega 
L_M h&=-\frac{2}{\sqrt{\mu_M}}Q(\mu,\sqrt{\mu_M}h).\\
\enda\]
We recall the following properties on $L_M$ and $L$.
\begin{lemma} \label{LM-lem}
For any $\eps_0>0$, there exists a constant $C_{\eps_0}>0$ such that 
\[
L_M h=\nu_M(\xi)h-K_Mh,
\]
where
\[
K_Mh=K_{\eps_0} h+K_ch,
\]
and 
\[\begin{cases}
\nu_M(t,x,\xi)&\sim\nu_0(\xi)=\la \xi\ra,\\
K_{\eps_0}h&\lesssim \eps_0\nu_0(\xi) \nl h\nr_{L^\infty_\xi},\\\
K_ch&=\int_{\R^3}k(\xi,\xi_\star)h(\xi_\star)d\xi_\star.
\end{cases}
\] The kernel $k$ satisfies the bound 
\[\bega 
k(\xi,\xi_\star)\lesssim C_{\eps_0}k_{\rho_0}(\xi,\xi_\star),\qquad k_{\rho_0}(\xi,\xi_\star)=|\xi-\xi_\star|^{-1}e^{-\rho_0|\xi-\xi_\star|^2}
\enda 
\]
for some universal constant $\rho_0>0$.
\end{lemma}
\begin{proposition} \label{L-decom-max} Let $f=f(t,x,\xi)$ and $L$ is defined in \eqref{L-def}. There holds 
\[
Lf=\nu f+K_1 f-K_2 f.
\]
Here 
\beq\label{nu-def}
\nu(\xi)=\int_{\R^3\times\S^2}|(\xi-\xi_\star)\cdot\w|\mu(\xi_\star)d\w d\xi_\star\eeq
The linear operator $K_1$ is defined by 
\beq\label{Kk1}
K_1f(\xi)=\int_{\xi_\star}\bk_1(\xi,\xi_\star)f(\xi_\star)d\xi_\star,\qquad \bk_1(\xi,\xi_\star)=\sqrt\mu(\xi)\sqrt\mu(\xi_\star)\int_{|\w|=1}|(\xi-\xi_\star)\cdot\w|d\w.
\eeq
The linear operator $K_2$ is given by
\[
K_2 f=C_0\int_{V_\parallel\in\R^3}\bk_2(\zeta_\perp,\zeta_\parallel,V_\parallel)f(\xi+V_\parallel)d V_\parallel,
\]
where 
\beq\label{Kk2}
\begin{cases}
&\bk_2(\zeta_\perp,\zeta_\parallel,V_\parallel)=\frac 1{|V_\parallel|}e^{-\frac 1 8|V_\parallel|^2-\frac 1 2|\zeta_\parallel|^2}\int_{V_\perp\in\R^2}e^{-\frac 1 2|\zeta_\perp+V_\perp|^2}dV_\perp,\\
&\zeta_\parallel=\frac{(\xi-\eps U)\cdot V_\parallel}{|V_\parallel|^2}V_\parallel+\frac 1 2 V_\parallel,\\
&\zeta_\perp=\xi-\eps U-\frac{(\xi-\eps U)\cdot V_\parallel}{\|V_\parallel\|^2}V_\parallel,
\end{cases}
\eeq
and $C_0$ is a universal constant. Moreover, one has 
\beq\label{L-1}
\nl e^{\rho|\xi|^2} L^{-1}g
\nr _{L^2_\xi}\lesssim \nl e^{\rho|\xi|^2}\nu^{-1}g
\nr_{L^2_\xi}.
\eeq
\end{proposition}
\begin{proof}
We have 
\[\bega 
Lf&=\sqrt\mu(\xi)\int_{\R^3\times\S^2}|(\xi-\xi_\star)\cdot\w|\sqrt\mu(\xi_\star)f(\xi_\star)d\w d\xi_\star+f(\xi)\int_{\R^3\times\S^2}|(\xi-\xi_\star)\cdot\w|\mu(\xi_\star)d\w d\xi_\star\\
&\quad-\int_{\R^3\times\S^2}|(\xi-\xi_\star)\cdot\w|\sqrt\mu(\xi_\star)\lw\{\sqrt\mu (\xi')f(\xi_\star')+f(\xi')\sqrt\mu(\xi_\star')
\rw\}d\w d\xi_\star\\
&=\int_{\R^3}\bk_1(\xi,\xi_\star)f(\xi_\star)d\xi_\star+\nu(\xi) f(\xi)-K_2 f,\\
\enda 
\]
where
\[\bega 
K_2 f&=\int_{\R^3\times\S^2}|(\xi-\xi_\star)\cdot\w|\sqrt\mu(\xi_\star)\lw\{\sqrt\mu(\xi') f(\xi_\star')+f(\xi')\sqrt\mu(\xi_\star')
\rw\}d\w d\xi_\star.\\
\enda 
\]
Let 
\[
V=\xi_\star-\xi.
\]
Then we have 
\[
\bega
\xi_\star&=\xi+V,\\
\xi_\star'&=\xi_\star-(\w\cdot (\xi_\star-\xi))\w=V+\xi-(\w\cdot V)\w=\xi+V_\perp,\\
\xi'&=\xi+((\xi_\star-\xi)\cdot \w) \w=\xi+V_\parallel.
\enda
\]
where $V_\parallel$ is the vector projection of $V$ onto $\w$ and $V_\perp=V-V_\parallel$.
Thus we get
\[\bega 
K_2f&=\int_{\R^3\times\S^2}|V_\parallel|\sqrt\mu(\xi+V)\sqrt\mu(\xi+V_\perp)f(\xi+V_\parallel)d\w dV\\
&\quad+\int_{\R^3\times\S^2}|V_\parallel| \sqrt\mu(\xi+V) \sqrt\mu(\xi+V_\parallel)f(\xi+V_\perp)d\w dV\\
&=K_{2,1}f+K_{2,2}f.
\enda 
\]
We first simplify $K_{2,1}$. We will integrate in $V_\perp$ first, then $V_\parallel$ and then $\w$. For fixed $\w\in \S^2$, we have 
\[
V_\perp\in\R^2,\qquad V_\parallel\in\R.
\]
Namely, $V_\perp$ lies on the plane that is perpendicular to $\w$ and $V_\parallel$ lies on the axis containing $\w$.
Now we have 
\[
\int_{\R^3}dV=\int_{-\infty}^\infty dV_{\parallel}\int_{\R^2}dV_\perp=2\int_0^\infty dV_\parallel \int_{\R^2} dV_\perp.
\]
Hence 
\[
d\w dV=d\w  \cdot 2dV_\perp d |V_\parallel|=2 dV_\perp \frac{|V_\parallel|^2 d|V_\parallel| d\w}{|V_\parallel|^2}=\frac 2 {|V_\parallel|^2}dV_\perp dV_{\parallel},\qquad V_{\parallel}\in \R^3, \quad V_\perp\in\R^2.
\]
This implies 
\[\bega 
K_{2,1}f&=2\int_{V_\perp\in\R^2}\int_{V_\parallel\in\R^3}|V_\parallel|^{-1}\sqrt\mu(\xi+V)\sqrt\mu(\xi+V_\perp)f(\xi+V_\parallel)d V_\parallel dV_{\perp}\\
&=2\int_{V_\perp\in\R^2}\int_{V_\parallel\in\R^3}|V_\parallel|^{-1}\mu_0^{\frac 12}(\xi-\eps U+V)\mu_0^{\frac 1 2} (\xi-\eps U+V_\perp)f(\xi+V_\parallel)d V_\parallel dV_\perp.
\enda
\]
Let $\zeta=\xi-\eps U+\frac 1 2 V_\parallel$. Define $(\zeta_\parallel,\zeta_\perp)$ so that 
\[
\zeta=\zeta_\parallel+\zeta_\perp,\quad \zeta_\parallel \parallel V_\parallel,\quad \zeta_\perp \parallel V_\perp.
\] 
We have the explicit formula
\[\begin{cases}
\zeta_\parallel(t,x,\xi)&=\frac{(\xi-\eps U+\frac 12 V_\parallel)\cdot V_\parallel}{|V_\parallel|^2}V_\parallel=\frac{(\xi-\eps U)\cdot V_\parallel}{|V_\parallel|^2}V_\parallel+\frac 1 2 V_\parallel,\\
\zeta_{\perp}(t,x,\xi)&=(\xi-\eps U+\frac 12V_\parallel)-\zeta_\parallel=\xi-\eps U-\frac{(\xi-\eps U)\cdot V_\parallel}{|V_\parallel|^2}V_\parallel.
\end{cases}
\]
Now we have
\[\begin{cases}
\xi-\eps U+V&=V_\perp+\zeta_\perp+\lw(\zeta_\parallel+\frac 12 V_\parallel
\rw),\\
\xi-\eps U+V_\perp&=V_\perp+\zeta_\perp+\lw(\zeta_\parallel-\frac 12 V_\parallel
\rw).
\end{cases}
\]
Hence 
\[\bega 
&\mu_0^{\frac 12}(\xi-\eps U+V)\mu_0^{\frac 1 2} (\xi-\eps U+V_\perp)\\
&=C_0 \exp\lw\{-\frac 12  \lw|V_\perp+\zeta_\perp\rw|^2-\frac 1 2 |\zeta_\parallel|^2-\frac 1 8 |V_\parallel|^2
\rw\}.
\enda 
\]
Hence we have 
\[
K_{2,1}f=C_0 \int_{\R^3}\frac{1}{|V_\parallel|}\lw\{\int_{V_\perp\in\R^2}e^{-\frac 1 2|V_\perp+\zeta_\perp|^2}dV_\perp\rw\}e^{-\frac 12|\zeta_\parallel|^2-\frac 1 8 |V_\parallel|^2 }f(\xi+V_\parallel) dV_\parallel.
\]
\end{proof} 
\begin{definition}
We define 
\[\bega
\gamma_-&=\{x\in\pt\Omega,\xi\in\R^3:\quad n(x)\cdot \xi<0\},\\
\gamma_+&=\{x\in\pt\Omega,\xi\in\R^3:\quad n(x)\cdot \xi>0\}.
\enda 
\]
When $\Omega=\T^2\times\R_+$, we have 
\[\bega
\gamma_-&=\T^2\times\{x_3=0\}\times\{\xi\in \R^3:\xi_3>0\},\\
\gamma_+&=\T^2\times\{x_3=0\}\times\{\xi\in \R^3:\xi_3<0\}.\\
\enda
\]
We define the boundary integration for a function $f(t,x,\xi)$ where $x\in \pt\Omega$ as follows
\[
\int_{\gamma_{\pm}}fd\gamma=\int_{\gamma_{\pm}}f|n(x)\cdot\xi|dS(x)d\xi,
\]
where $dS(x)$ denotes the standard surface measure on $\pt\Omega$. We also define 
\[
\nl f\nr_{L^p_\gamma}=\nl f\nr_{L^p_{\gamma_+}}+\nl
f\nr_{L^p_{\gamma_-}}
\]
to be the $L^p(\gamma)$ with respect to the measure $|n(x)\cdot\xi|dS(x)d\xi$.
We also define the boundary inner product over $\gamma_{\pm}$ in $\xi$ as 
\[
\la f,g\ra_{\gamma_\pm}(t,x)=\int_{\pm n(x)\cdot\xi>0}f(t,x,\xi)g(t,x,\xi)|n(x)\cdot\xi|d\xi.
\]
The projection operator $\P_{\gamma_+}$ onto the unit vector $\sqrt{c_\mu \mu_0(\xi)}$ with respect to $\la \cdot,\cdot\ra_{\gamma_+}$ is defined as follows
\beq\label{P-gamma}
(\P_{\gamma_+}f)(x,\xi)=c_\mu \sqrt{\mu_0(\xi)}\lw\{\int_{(x,\wtd \xi)\in \gamma_+}f(x,\wtd \xi)\sqrt{ \mu_0(\wtd \xi)}|n(x)\cdot\wtd \xi
|d\wtd \xi\rw\}.
\eeq
\end{definition}~\\
When $\Omega=\T^2\times\R_+$, this can be written as \beq\label{P-gamma-2}
(\P_{\gamma_+}f)(t,x,\xi)=c_\mu \sqrt{\mu_0(\xi)}\lw\{\int_{\wtd\xi_3<0}f(t,x,\wtd \xi)\sqrt{ \mu_0(\wtd \xi)}|\wtd \xi_3
|d\wtd \xi\rw\}.
\eeq
\begin{proposition} Let $F$ solves \eqref{F-eq} and $f,h$ defined in the expansion \eqref{F-exp} and \eqref{h-def}. Then $f$ and $h$ solve
\beq\label{f-eq}
\bega 
&\pt_t f+\frac 1 \eps \xi\cdot\nabla_x f+\frac 1{\kappa\eps^2}Lf+\frac{(\pt_t +\frac 1\eps \xi\cdot\nabla_x)\sqrt\mu}{\sqrt\mu}f\\
&=-\frac{1}{\eps\delta \sqrt\mu}\textbf{R}_a+\frac{\delta}{\kappa\eps}\Gamma(f,f)+\frac{2}{\kappa}\Gamma(\ff,f).\\
\enda
\eeq
with the boundary condition 
\beq\label{bdr-r}
f|_{\gamma_-}=\P_{\gamma_+} f-\frac\eps\delta r,
\eeq
where 
\beq\label{r-def}
r=(1-\P_{\gamma_+})\ff,
\eeq
and 
\beq\label{h-eq}
\pt_th+\frac 1\e \xi\cdot\nabla_x h+\frac{1}{\kappa\e^2}L_Mh=-\frac{1}{ \eps\delta\sqrt\mu_M}\textbf{R}_a
+\frac{\delta}{\kappa\eps}\Gamma_M (h,h)+\frac 2\kappa \frac{Q(\ff\sqrt\mu,h\sqrt{\mu_M})}{\sqrt{\mu_M}}. \eeq
\end{proposition}
\begin{proof}We show \eqref{f-eq}, as the proof of \eqref{h-eq} is similar.
Plugging $F=\mu+\eps^2 \ff\sqrt\mu+\eps\delta f\sqrt\mu$ into the equation \eqref{F-eq}, we get 
\[\bega 
&\eps\delta\sqrt\mu\lw(\pt_t f+\frac 1\e \xi\cdot\nabla_x f+\frac{(\pt_t+\frac1\e\xi\cdot\nabla_x)\sqrt\mu}{\sqrt\mu}f\rw)+(\pt_t+\eps^{-1}\xi\cdot\nabla_x)(\mu+\e^2 \ff\sqrt\mu)\\
&=\frac{2}{\kappa}Q(\sqrt\mu \ff,\mu)+\frac{\e^2}{\kappa}Q(\sqrt\mu \ff,\sqrt\mu \ff)+\frac{2\delta}{\kappa\eps}Q(\mu,\sqrt\mu f)\\
&\quad+\frac{2\delta\e}{\kappa}Q(\sqrt\mu \ff,\sqrt\mu f)+\frac{\delta^2}{\kappa}Q(\sqrt\mu f,\sqrt\mu f).
\enda
\]
Hence we obtain 
\[\bega 
&\pt_t f+\frac1\e \xi\cdot\nabla_x f+\frac{(\pt_t+\frac1\e\xi\cdot\nabla_x)\sqrt\mu}{\sqrt\mu}f+\frac{1}{\kappa\e^2}Lf\\
&=-\frac{1}{\eps\delta \sqrt\mu}\lw\{(\pt_t+\eps^{-1}\xi\cdot\nabla_x)(\mu+\e^2 \ff\sqrt\mu)-\frac 2\kappa Q(\sqrt\mu \ff,\mu)-\frac{\eps^2}{\kappa}Q(\sqrt\mu \ff,\sqrt\mu \ff)\rw\}\\
&\quad+\frac{2}{\kappa}\Gamma(\ff,f)+\frac{\delta}{\kappa\e}\Gamma(f,f).
\enda 
\]
The proof is complete.
\end{proof} 

\begin{proposition} Recalling \eqref{Ra-def}, we have \beq\label{R-a-def} \bega
\textbf{R}_a&=\eps \NS(U,P)\cdot\vp \mu-\frac{\eps^2}{\kappa}\mu^{\frac12}\Gamma(\ff,\ff)+\frac 1 3 \div U |\vp|^2\mu-\eps\kappa \pt_{i\ell}U_j\IP (\vp_\ell A_{ij})\sqrt\mu\\
&\quad+\eps ^2 \mu\lw\{(\pt_t P+U\cdot\nabla_x P)+\eps P(\pt_t U+U\cdot\nabla_x U)\cdot\vp+\frac 1 2 P \pt_i U_\ell \vp_i\vp_\ell
\rw\} \\ 
&\quad-\eps^2 \kappa \mu^{\frac12}\lw\{(\pt_t+U\cdot\nabla_x)(\pt_i U_j)+\eps (\pt_tU+U\cdot\nabla_x U)\cdot \vp \pt_i U_j+\pt_i U_j \pt_k U_\ell \vp_k\vp_\ell
\rw\}A_{ij} \\
&\quad-\eps^2 \kappa \mu^{\frac 12} \pt_i U_j\lw(\eps^2 \pt_t+(\eps^2 U+\vp)\cdot\nabla_x
\rw)A_{ij}.\\
\enda\eeq
\end{proposition}

\begin{proof} We have 
\[
\textbf{R}_a=\lw(\pt_t+\frac 1\eps \xi\cdot\nabla_x
\rw)\mu+\eps^2\lw(\pt_t+\frac 1\eps \xi\cdot\nabla_x
\rw)(\sqrt\mu \ff)-\frac 2 \kappa Q(\mu,\ff\sqrt\mu)-\frac {\eps^2}{\kappa} Q(\sqrt\mu \ff,\sqrt\mu \ff).
\]
First, we have 
\beq\label{mu-term}
\bega 
&\frac{(\pt_t+\eps^{-1}\xi\cdot\nabla_x)\mu}{\sqrt\mu}=\frac{(\pt_t+\eps^{-1} \vp\cdot\nabla_x)\mu}{\sqrt\mu}+\frac{(\pt_t+U\cdot\nabla_x)\mu}{\sqrt\mu}\\
&=\eps (\pt_t U+U\cdot\nabla_x U)\cdot\vp\sqrt\mu+\sum_{i,j}\pt_i U_j\vp_i\vp_j\sqrt\mu\\
&
=\eps (\pt_t U+U\cdot\nabla_x U)\cdot\vp\sqrt\mu+\pt_i U_j\hat A_{ij}(\vp)+\frac 1 3 \div U (|\vp|^2-3)\sqrt\mu+(\div U)\sqrt\mu.\\
\enda 
\eeq
Since $\frac 1 \kappa L\ff+\pt_i U_j\hat A_{ij}(\vp)=0$, we have 
\[\bega 
&\frac{2}{\kappa}Q(\mu,\ff\sqrt\mu)-(\pt_t+\e^{-1}\xi\cdot\nabla_x )\mu\\
&=-\eps (\pt_t U+U\cdot\nabla_x U)\cdot\vp \mu-\frac 1 3 (\nabla_x\cdot U)(|\vp|^2-3)\mu-(\nabla_x\cdot U)\mu. 
\enda \]
Hence 
\[\bega 
&(\pt_t+\frac 1\eps \xi\cdot\nabla_x)(\mu+\sqrt\mu\eps^2\ff)-\frac2\kappa Q(\mu,\ff\sqrt\mu)-\frac{\eps^2}{\kappa}Q(\sqrt\mu \ff,\sqrt\mu \ff)\\
&=\e^2\sqrt\mu\lw\{\pt_t\ff+U\cdot\nabla_x \ff+\frac{(\pt_t+\eps^{-1}\xi\cdot\nabla_x)\sqrt\mu}{\sqrt\mu}\ff
\rw\}\\
&\quad+\eps \lw\{\nabla_x \P \ff+(\pt_tU+U\cdot\nabla_x U)
\rw\}\cdot\vp\mu-\frac{\eps^2}{\kappa}\sqrt\mu \Gamma(\ff,\ff)\\
&\quad+\frac 1 3 \div U |\vp|^2\mu+\eps \vp \cdot\nabla_x \IP \ff\sqrt\mu.
\enda\]
Now we simplify the last term in the above. Now we have
\beq\label{ran-26}
\bega 
&\vp\cdot\nabla_x (\II-\P)\ff\\
&=\vp_k\pt_k\lw\{-\kappa \pt_i U_j A_{ij}(\vp)
\rw\}\\
&=-\kappa \vp_k \pt_{ik}U_j A_{ij}(\vp)-\kappa \pt_i U_j (\vp\cdot\nabla_x)A_{ij}\\
&=-\kappa\P (\pt_{ik}U_j\vp_k A_{ij})-\kappa \pt_{ik}U_j (\II-\P)(\vp_kA_{ij})-\kappa \pt_i U_j (\vp\cdot\nabla_x)A_{ij}.\\\enda 
\eeq
The first term is computed as follows
\beq\label{ran-27}
\bega 
\P(\pt_{ik}U_j\vp_k A_{ij})&=\sum_m \pt_{ik}U_j\la \vp_k\vp_m\sqrt\mu,A_{ij}\ra_{L^2_\vp}\vp_m\sqrt\mu=\pt_{ik}U_j\la \hat A_{km},A_{ij}\ra \vp_m\sqrt\mu\\
&=\eta_0\pt_{ik}U_j\lw(\delta_{ki}\delta_{mj}+\delta_{kj}\delta_{mi}-\frac 2 3 \delta_{km}\delta_{ij}
\rw)\vp_m\sqrt\mu\\
&=\eta_0(\triangle U_m+\frac 1 3 \pt_m(\div U))\vp_m\sqrt\mu.
\enda 
\eeq
Hence
\[\bega 
&\vp\cdot\nabla_x (\II-\P)\ff\\
&=-\kappa\eta_0\triangle U\cdot\vp\sqrt\mu-\frac 1 3\kappa\eta_0\vp\cdot\nabla_x(\div U)\sqrt\mu-\kappa \pt_{ik}U_j (\II-\P)(\vp_kA_{ij})-\kappa \pt_i U_j (\vp\cdot\nabla_x)A_{ij}.
\enda \]
Hence we obtain
\[\bega 
&(\pt_t+\frac 1\eps \xi\cdot\nabla_x)(\mu+\sqrt\mu\eps^2\ff)-\frac2\kappa Q(\mu,\ff\sqrt\mu)-\frac{\eps^2}{\kappa}Q(\sqrt\mu \ff,\sqrt\mu \ff)\\
&=\e^2\mu^{\frac12}\lw\{\pt_t\ff+U\cdot\nabla_x \ff+\frac{(\pt_t+\eps^{-1}\xi\cdot\nabla_x)\sqrt\mu}{\sqrt\mu}\ff
\rw\}+\eps \NS(U,P)\cdot\vp \mu\\
&-\frac{\eps^2}{\kappa}\mu^{\frac12}\Gamma(\ff,\ff)+\frac 1 3 \div U |\vp|^2\mu-\eps\kappa \pt_{ik}U_j\IP (\vp_kA_{ij})\sqrt\mu-\kappa \pt_i U_j(\vp\cdot\nabla_x)A_{ij}\sqrt\mu.
\enda\]
Lastly, we see that 
\[\bega 
&\pt_t \ff+U\cdot\nabla_x \ff+\frac{(\pt_t+\frac 1 \eps \xi\cdot\nabla_x)\sqrt\mu}{\sqrt\mu}\ff\\
&=\lw\{(\pt_t P+U\cdot\nabla_x P)+\eps P(\pt_t U+U\cdot\nabla_x U)\cdot\vp+\frac 1 2 P \pt_i U_k \vp_i\vp_k
\rw\}\sqrt\mu\\
&+A_{ij}(\vp)\lw\{-\kappa(\pt_t+U\cdot\nabla_x)(\pt_i U_j)-\eps\kappa(\pt_tU+U\cdot\nabla_x U)\cdot \vp \pt_i U_j-\kappa \pt_i U_j\pt_k U_l \vp_k\vp_l
\rw\}\\
&-\kappa \pt_i U_j(\pt_t+U\cdot\nabla_x)A_{ij}(\vp).
\enda 
\]
The proof is complete.
\end{proof}
\subsection{Equations and boundary conditions for derivatives}
Fix $\al \in \mathbb N_0^3$. Now we define 
\beq\label{fh-al-def}
f_\al=\pt^\al f,\qquad h_\al=\pt^\al h
\eeq
Applying $\pt^\al$ to the equation \eqref{f-eq} and let $f_\al=\pt^\al f$, we have 
\beq\label{f-al-eq}
\bega 
&\pt_t f_\al+\frac 1 \eps \xi\cdot\nabla_x f_\al+\frac 1 {\kappa\eps^2}Lf_\al+\frac{(\pt_t+\frac 1 \eps \xi\cdot\nabla_x)\sqrt\mu}{\sqrt\mu}f_\al\\
&=\frac 1 {\kappa\eps^2}[L,\pt^\al]f+g_\al+\frac{(\pt_t+\frac 1 \eps \xi\cdot\nabla_x)\sqrt\mu}{\sqrt\mu}f_\al-\pt^\al\lw\{\frac{(\pt_t+\frac 1 \eps \xi\cdot\nabla_x)\sqrt\mu}{\sqrt\mu}f\rw\}.
\enda 
\eeq
where 
\beq\label{g-alpha}
g_\al=\pt^\al\lw(-\frac{1}{\eps\delta \sqrt\mu}\textbf{R}_a+\frac{\delta}{\kappa\eps}\Gamma(f,f)+\frac{2}{\kappa}\Gamma(\ff,f)\rw),
\eeq
along with the boundary condition 
\beq\label{f-al-bdr-con}
f_\al|_{\gamma_-}=\P_\gamma f_\al-\frac \eps\delta \pt^\al r.
\eeq
where $r$ is given in \eqref{r-def}. 
Next, it is clear from \eqref{h-eq} that $h_\al$ and Lemma \ref{LM-lem} that $h_\al=\pt^\al h$ solves the equation
\beq\label{h-al-eq}
\pt_t h_\al+\frac 1 \eps\xi\cdot\nabla_x h_\al+\frac 1 {\kappa\e^2} L_Mh_\al= \wtd g_\al
\eeq
where 
\[
\wtd g_\al=\frac 1 {\kappa\eps^2}[L_M,\pt^\al]h+ \pt^\al\lw(-\frac{1}{ \eps\delta\sqrt\mu_M}\textbf{R}_a
+\frac{\delta}{\kappa\e}\Gamma_M (h,h)+\frac 2\kappa \frac{Q(\ff\sqrt\mu,h\sqrt{\mu_M})}{\sqrt{\mu_M}}\rw).
\]
Recalling the boundary condition for $f$ in \eqref{bdr-r} and the fact that $f=\frac{\sqrt{\mu_M}}{\sqrt\mu_0}h$ on the boundary and the calculation \eqref{ran-76}, we get 
\beq\label{h-al-bdr}
h_\al|_{\{x_3=0,\xi_3>0\}}=c_\mu w_1(\xi)\lw\{\int_{\R^3}h_\al(t,x,\xi_\star)w_2(\xi_\star)d\xi_\star\rw\}-\frac{\eps}\delta w_3(\xi)^{-1}\pt^\al r \eeq
where
\[
w_1(\xi)=e^{-\frac{|\xi|^2}{4}(2-T_M^{-1})},\qquad w_2(\xi)=1_{\{\xi_3<0\}}e^{-\frac{|\xi|^2}{4T_M}}|\xi_3|,\qquad w_3(\xi)=T_M^{\frac 3 2}e^{-\frac{|\xi|^2}{4}(T_M^{-1}-1)}.
\]
We note that $w_1,w_2$ and $w_3$ are decaying functions in $\xi$. ~\\
We now derive the boundary condition for $b$, where we recall \eqref{Pf-def}. We recall that 
\[
F=\mu+\eps^2 \ff\sqrt\mu +\eps \delta f\sqrt\mu.
\]
\begin{proposition} \label{bc-con12} Assume \eqref{Pf-def}. Then we have
\beq\label{bc-con}
b_3|_{\pt\Omega}=0.
\eeq
Moreover, recalling \eqref{PFal}, we have
\[
b_{\al,3}|_{x_3=0}=0.
\]
Here $\pt^\al=(\eps\pt_t)^{\al_0}\pt_{x_1}^{\al_1}\pt_{x_2}^{\al_2}$ for $\al=(\al_0,\al_1,\al_2)\in \mathbb N^3$.
\end{proposition}
\begin{proof}First we will show that $\int_{\R^3}\pt^\al F\xi_3d\xi=0$. Let $z_\gamma(x)=\int_{\xi_3<0}F |\xi_3|d\xi_3$, then we see that 
\[
\pt^\al z_\gamma=\int_{\xi_3<0}\pt^\al F|\xi_3|d\xi_3.
\]
 The boundary condition implies
\[
\pt^\al F|_{x_3=0,\xi_3>0}=c_\mu \mu_0(\xi)\pt^\al z_\gamma(x).
\]
We have
\[\bega 
\int_{\xi_3}\pt^\al F\xi_3 d\xi_3&=\int_{-\infty}^0 \pt^\al F\xi_3 d\xi_3+\int_0^\infty\pt^\al F \xi_3 d\xi_3\\
&=-\int_{-\infty}^0 \pt^\al F|\xi_3|d\xi_3+\int_0^\infty c_\mu \mu_0(\xi)\pt^\al z_\gamma(x)\xi_3d\xi.\\
\enda
\]
Hence
\[\bega
\int_{\R^3}\pt^\al F \xi_3 d\xi_3&=-\pt^\al z_\gamma(x)+\pt^\al z_\gamma(x)c_\mu \int_{\xi_3>0} \mu_0(\xi)\xi_3 d\xi=0.
\enda 
\]
Here we use the fact that
\beq \label{cmu}
c_\mu \int_{\xi_3>0}\mu_0(\xi)|\xi_3|d\xi=1.
\eeq
Now combining $F=\mu+\eps\delta f\sqrt\mu+\eps^2\ff\sqrt\mu$ and the identity $\int_{\R^3}\pt^\al F|_{x_3=0}\xi_3d\xi_3=0$, we get $\pt^\al b_3|_{x_3=0}=0$. The proof is complete. 
\end{proof}
\subsection{Conservation Laws}
To begin, we first show the conservation law for $f$. 
\begin{proposition}  Recalling \eqref{PFal}, we have
\beq\label{law1} \bega 
\eps\pt_t a_\al+\nabla_x\cdot b_\al&=-\eps\pt_j (U_j a_\al)-\frac 1{\eps\delta}\pt^\al (\nabla_x\cdot U)\\
&\quad-\frac{\eps}{\delta}\pt^\al (U_i\pt_i U_i)-\frac{\eps^2}{\delta}\pt^\al(\pt_t P+\pt_{x_j}(U_j P))+\delta^{-1}\kappa\eps^2c_{ij}\pt_i U_j(\nabla_x\cdot U)\\
&\quad-\sum_{\substack{\beta+\gamma=\al\\|\gamma|\ge 1
}
}\frac{\al!}{\beta!\gamma!}\int_{\R^3}(\eps\pt_t+\xi\cdot\nabla_x)f_\beta\cdot \pt^{\gamma}\sqrt\mu d\xi\\
&\quad-\sum_{\substack{\beta+\gamma=\al\\|\gamma|\ge 1
}
}\frac{\al!}{\beta!\gamma!}\int_{\R^3}f_\beta (\eps\pt_t+\xi\cdot\nabla_x)(\pt^{\gamma}\sqrt\mu) d\xi.
\enda 
\eeq
where $c_{ij}$ is defined in \eqref{cij}.
\end{proposition}
\begin{proof}
We use the equation \eqref{F-eq} and the expansion \eqref{F-exp}. Hence we get $\int_{\R^3}(\eps\pt_t F+\xi\cdot\nabla_x F)d\xi=0$. 
We have 
\[\bega 
&\eps\frac{d}{dt}\lw(\int_{\R^3} f\sqrt\mu d\xi\rw)+ \sum_j\pt_{x_j}\lw\{\int_{\R^3}\xi_j f\sqrt\mu\rw\}\\
&=-\frac{1}{\eps\delta}(\e\pt_t+\xi\cdot\nabla_x)\mu d\xi-\frac \e\delta \int_{\R^3}(\eps\pt_t+\xi\cdot\nabla_x)(\ff\sqrt\mu)d\xi.
\enda 
\]
This implies
\[
\eps \pt_t a+\nabla_x\cdot b+\eps\pt_j(U_ja)=-\frac{1}{\eps\delta}(\e\pt_t+\xi\cdot\nabla_x)\mu d\xi-\frac \e\delta \int_{\R^3}(\eps\pt_t+\xi\cdot\nabla_x)(\ff\sqrt\mu)d\xi.
\]
Next we compute the two terms on the right hand side of the above. We have, for any $\pt\in \{\pt_{x_j},\eps\pt_t\}$, 
\[
\pt \mu=-(\xi-\eps U)\cdot (-1)\eps \pt U \mu=\eps \pt U\cdot \vp \mu.
\]
Hence 
\[
\eps\pt_t\mu+\xi_i\pt_{x_i}\mu=\eps (\eps \pt_t U\cdot \vp) \mu+\xi_i\eps\pt_{x_i}U_j\vp_j\mu=\eps^2 \pt_t U\cdot\vp\mu+\vp_i\vp_j\pt_i U_j \eps \mu+\eps^2 U_i\pt_i U_j\vp_i\vp_j \mu.
\]
And thus
\beq\label{ran-37}
\int_{\R^3}\lw(\eps\pt_t\mu+\xi_i\pt_{x_i}\mu\rw) d\xi=(\nabla_x\cdot U)+\eps^2 \sum_{i} U_i\pt_i U_i.
\eeq
Next, we compute $\int_{\R^3}(\eps\pt_t+\xi\cdot\nabla_x)(\ff\sqrt\mu)d\xi$. Recalling \eqref{ff-def}, we have
\beq\label{ran-38}\bega 
&\int_{\R^3}(\eps\pt_t+\xi\cdot\nabla_x)(\ff\sqrt\mu)d\xi\\
&=\eps\pt_t \lw\{\int_{\R^3}P\mu d\xi
\rw\}+\pt_{x_j}\lw\{\int_{\R^3}\xi_j (P\mu)
\rw\}-\kappa \pt_i U_j \int_{\R^3}(\eps\pt_t+\xi\cdot\nabla_x)A_{ij}(\vp)d\xi\\
&=\eps\pt_t P+\pt_{x_j}(\eps U_j P)-\kappa\eps c_{ij}\pt_iU_j (\nabla_x\cdot U).\\
\enda
\eeq
where \beq\label{cij}
c_{ij}=\int_{\R^3}A_{ij}(\vp)d\vp.
\eeq
Now we show \eqref{law1}. We have 
\[\bega
\int_{\R^3} \xi\pt^\al F d\xi&=\eps\delta (b_\al+\eps Ua_\al)+\int_{\R^3}\xi \lw\{\pt^\al\mu+\eps^2\pt^\al(\ff\sqrt\mu)+\eps\delta\sum_{\beta<\al }\binom{\al}{\beta}f_\beta\pt^{\al-\beta}\sqrt\mu\rw\}d\xi,\\
\int_{\R^3}\pt^\al Fd\xi&=\eps\delta a_\al+\int_{\R^3} \lw\{\pt^\al\mu+\eps^2\pt^\al(\ff\sqrt\mu)+\eps\delta\sum_{\beta<\al}\binom{\al}{\beta}f_\beta\pt^{\al-\beta}\sqrt\mu\rw\}d\xi.\\
\enda
\]
Combining the above identities, we have
\[\bega 
\eps\pt_t a_\al+\nabla_x\cdot b_\al&=-\eps \pt_j(U_j a_\al)-\frac{1}{\eps\delta}\int_{\R^3}(\eps\pt_t+\xi\cdot\nabla_x)(\pt^\al \mu)d\xi-\frac{\eps}{\delta}\int_{\R^3}(\eps\pt_t+\xi\cdot\nabla_x)(\pt^\al (\ff\sqrt\mu))\\
&\quad-\sum_{\substack{\beta+\gamma=\al\\|\gamma|\ge 1
}
}\frac{\al!}{\beta!\gamma!}\int_{\R^3}(\eps\pt_t+\xi\cdot\nabla_x)f_\beta\cdot \pt^{\gamma}\sqrt\mu d\xi\\
&\quad-\sum_{\substack{\beta+\gamma=\al\\|\gamma|\ge 1
}
}\frac{\al!}{\beta!\gamma!}\int_{\R^3}f_\beta (\eps\pt_t+\xi\cdot\nabla_x)(\pt^{\gamma}\sqrt\mu) d\xi.
\enda 
\]
Combining the above with \eqref{ran-37} and \eqref{ran-38}, we get the result.\end{proof}
\subsection{$L^2$ estimate for $f$ in analytic spaces}
Our goal is to derive analytic energy estimates for $f$ which solves the equation
\beq\label{fgeq}
\pt_t f+\frac 1 \eps\xi\cdot\nabla_x f+\frac{1}{\kappa\eps^2}Lf+\frac{(\pt_t+\eps^{-1}\xi\cdot\nabla_x)\sqrt\mu}{\sqrt\mu}f=g
\eeq
for a general forcing $g$ with the boundary condition \eqref{bdr-r}. In this section, we take the differential operator $\pt\in \{\eps\pt_t,\pt_{x_1},\pt_{x_2}\}$. 
We define 
\beq\label{dEt-def}
\bega
&d_{E}(t)=\frac{\kappa\e^3}{\delta}\nl \NS(U,P)
\nr_{\tau_0,2}+\frac{1}{\eps\delta}\nl \div U
\nr_{\tau_0,2}+\frac{\kappa\eps}{\delta} \|| \nabla_x U
\||_{\tau_0,2}\\
&+\frac{\eps}{\delta}\lw\{\||\pt_t P+U\cdot\nabla_x P\||_{\tau_0,\infty}+\eps \|| P(\pt_t U+U\cdot\nabla_x U)\||_{\tau_0,\infty}+\|| P\pt_i U_j \||_{\tau_0,\infty}\rw\}\\
&+\frac{\eps\kappa}{\delta}\lw(\||  (\pt_t+U\cdot\nabla_x)(\pt_i U_j)\||_{\tau_0,2}+\eps\||  (\pt_t U+U\cdot\nabla_x U)(\pt_i U_j)\||_{\tau_0,2}+\|| \nabla_x U
\||_{\tau_0,\infty}\|| \nabla_x U
\||_{\tau_0,2}\rw)\\
&+\frac{\eps}{\delta}\||U\cdot\nabla_xU\||_{\tau_0,2}+\frac{\eps^2}{\delta}\||\pt_t P+\nabla_x\cdot(PU)\||_{\tau_0,2}+\sup_{\substack{\al\in \mathbb N_0^3\\|\al|\ge 1}}\frac{\tau_0^{|\al|}}{\al!}\nl\pt_{x_3}\pt^\al U
\nr_{L^\infty},
\enda
\eeq
which will appear in the $L^2$ analytic estimate of the Boltzmann solution.
\begin{lemma}\label{decom-b3}
Fix $\al\in \mathbb N^3$. Then
\[
b_{\al,3}=\mathcal T^b_{\al}-\eps U_3a_\al-\frac \eps{2\delta}\pt^\al\{U_3^2\}-\frac\eps\delta \pt^\al (PU_3)
\]
where
\[\bega 
\nl \frac{\mathcal T^b_{\al}}{x_3}\nr_{L^2_x}&\lesssim  \nl f_\al\nr_{L^2}+\sum_{\pt\in\{\pt_{x_1},\pt_{x_2}\}}\nl \pt f_\al\nr_{L^2}+(\eps\delta)^{-1}\nl\pt^\al (\nabla_x\cdot U)\nr_{L^2}+\eps\delta^{-1}\sum_{i=1}^2\nl \pt^\al(U_i^2)\nr_{L^2}\\
&\quad+\eps^2\delta^{-1}\nl\pt^\al\lw(\pt_tP+\nabla_\parallel\cdot (PU_\parallel)\rw)
\nr_{L^2}+\delta^{-1}\eps^2\kappa \nl \pt^\al\lw(\pt_i U_j(\nabla_x\cdot U)
\rw)
\nr_{L^2} \\
&\quad+\sum_{\substack{\beta+\gamma=\al\\|\gamma|\ge1}}\frac{\al!}{\beta!\gamma!}\sum_{\pt\in \{\eps\pt_t,\pt_{x_1},\pt_{x_2}\}}\nl \pt f_\beta
\nr_{L^2}\cdot\nl e^{o(1)|\xi|^2}(1,\eps\pt_t,\nabla_x) \pt^{\gamma}\sqrt\mu
\nr_{L^\infty}.
\enda
\]
\end{lemma}
\begin{proof}
We write 
\[\bega 
b_{3,\al}(t,x_\parallel,x_3)&=\int_0^{x_3}\pt_3 b_{\al,3}dx_3'.
\enda 
\]
Using the conservation law \eqref{law1}, we get 
\[\bega 
&\pt_3 b_{\al,3}=-\eps \pt_t a_\al-\nabla_\parallel\cdot b_{\al,\parallel}-\eps\pt_j (U_j a_\al)-(\eps\delta)^{-1}\pt^\al (\div U)-\frac 1 2 \eps\delta^{-1}\pt^\al \pt_i(U_i^2)\\
&-\eps^2\delta^{-1}\pt^\al\lw(\pt_t P+\nabla_x\cdot (UP)
\rw)\\
&-\sum_{\substack{\beta+\gamma=\al\\|\beta|\ge 1
}}\frac{\al!}{\beta!\gamma!}(\eps\pt_t+\xi\cdot\nabla_x)f_\gamma \cdot\pt^\beta \sqrt\mu d\xi-\sum_{\substack{\beta+\gamma=\al\\|\beta|\ge 1
}}\frac{\al!}{\beta!\gamma!} f_\gamma(\eps\pt_t+\xi\cdot\nabla_x)(\pt^\beta \sqrt\mu)d\xi\\
&=\sum_{i=1}^4T_{\al,i}.
\enda 
\]
where
\[\bega
T_{\al,1}&=-\eps\sum_{j=1}^2\pt_j (U_j a_\al)-\frac 1{\eps\delta}\pt^\al (\nabla_x\cdot U)-\frac{\eps}{2\delta}\sum_{j=1}^2\pt^\al (U_j^2)\\
&\quad-\frac{\eps^2}{\delta}\pt^\al(\pt_t P+\nabla_\parallel\cdot (PU_\parallel))+\delta^{-1}\kappa\eps^2 c_{ij}\pt^\al(\pt_i U_j\div U),\\
T_{\al,2}&=-\sum_{\substack{\beta+\gamma=\al\\|\gamma|\ge 1} }\frac{\al!}{\beta!\gamma!}\int_{\R^3}(\eps\pt_t+\xi_\parallel\cdot\nabla_{x_\parallel})f_\beta\cdot \pt^{\gamma}\sqrt\mu d\xi,\\
T_{\al,3}&=-\sum_{\substack{\beta+\gamma=\al\\|\gamma|\ge 1} }\frac{\al!}{\beta!\gamma!}\int_{\R^3}f_\beta (\eps\pt_t+\xi\cdot\nabla_x)(\pt^{\gamma}\sqrt\mu) d\xi,\\
T_{\al,4}&=-\sum_{\substack{\beta+\gamma=\al\\|\gamma|\ge 1} }\frac{\al!}{\beta!\gamma!}\int_{\R^3}\lw\{\xi_3\pt_{x_3}\lw(f_\beta \pt^{\gamma}\sqrt\mu
\rw) -\lw(\xi_3 f_\beta \pt_{x_3}\pt^\gamma \sqrt\mu \rw) \rw\}d\xi\\
&\quad-\eps \pt_{x_3}(U_3a_{\al})-\frac {\eps}{2\delta}\pt_3(\pt^\al(U_3^2))-\eps\delta^{-1}\pt^\al \pt_3(U_3P).\\
\enda
\]
As for $T_{\al,1}$, we have 
\[\bega 
\nl T_{\al,1}
\nr_{L^2_x}&\lesssim  \nl f_\al\nr_{L^2}+\sum_{\pt\in\{\pt_{x_1},\pt_{x_2}\}}\nl \pt f_\al\nr_{L^2}+(\eps\delta)^{-1}\nl\pt^\al (\nabla_x\cdot U)\nr_{L^2}+\eps\delta^{-1}\sum_{i=1}^2\nl \pt^\al(U_i^2)\nr_{L^2}\\
&\quad+\eps^2\delta^{-1}\nl\pt^\al\lw(\pt_tP+\nabla_\parallel\cdot (PU_\parallel)\rw)
\nr_{L^2}+\delta^{-1}\eps^2\kappa \nl \pt^\al\lw(\pt_i U_j(\nabla_x\cdot U)
\rw)
\nr_{L^2}.
\enda
\]
Now for $T_{\al,2}$, we have 
\[
\nl T_{\al,2}\nr_{L^2}\lesssim \sum_{\substack{\beta+\gamma=\al\\|\gamma|\ge1}}\frac{\al!}{\beta!\gamma!}\sum_{\pt\in \{\eps\pt_t,\pt_{x_1},\pt_{x_2}\}}\nl \pt f_\beta
\nr_{L^2}\cdot\nl e^{o(1)|\xi|^2} \pt^{\gamma}\sqrt\mu
\nr_{L^\infty}.
\]
For $T_{\al,3}$, we have 
\[ 
\nl T_{\al,3}\nr_{L^2_x}\lesssim \sum_{\substack{\beta+\gamma=\al\\|\gamma|\ge1}}\frac{\al!}{\beta!\gamma!}\nl f_\beta
\nr_{L^2_{x,\xi}}
\lw\{ \nl (\eps\pt_t) \pt^\gamma \sqrt\mu\nr_{L^\infty_x L^2_\xi}
+\nl \la \xi\ra\nabla_x\pt^{\gamma}\sqrt\mu
\nr_{L^\infty_x L^2_\xi}\rw\}.\]
For $T_{4,\al}$, using the fact that $U$ vanishes on the boundary, we have 
\[\bega 
\int_0^{x_3} T_{\al,4}dx_3'&=-\sum_{\substack{\beta+\gamma=\al\\|\gamma|\ge1}}\frac{\al!}{\beta!\gamma!}\int_{\R^3}\xi_3 f_\beta\pt^{\gamma}\sqrt\mu d\xi\\
&\quad+\frac 1 2 \sum_{\substack{\beta+\gamma=\al\\|\gamma|\ge1}}\frac{\al!}{\beta!\gamma!}\int_{\R^3}\int_0^{x_3}\xi_3 f_\beta \pt^{\gamma}\lw\{ \lw(\eps \pt_{x_3'}U\cdot \vp
\rw)\sqrt\mu\rw\} dx_3'd\xi\\
&\quad-\eps U_3a_\al-\frac \eps{2\delta}\pt^\al\lw\{U_3^2\rw\}-\frac\eps\delta \pt^\al (PU_3).
\enda 
\]
Using the Hardy inequality, we have 
\[\bega 
&\nl \frac 1{x_3}\int_0^{x_3}T_{\al,4}dx_3'
\nr_{L^2}\\
&\lesssim \sum_{\substack{\beta+\gamma=\al\\|\gamma|\ge1}}\frac{\al!}{\beta!\gamma!}\nl  f_\beta\nr_{L^2}\nl e^{o(1)|\xi|^2}\pt^{\gamma}\sqrt\mu\nr_{L^\infty}+\eps \nl  f_\al
\nr_{L^2}+\eps^2 \delta^{-1}+\eps\delta^{-1}\nl\pt^\al(PU_3)
\nr_{L^2}.
\enda 
\]
The proof is complete.
\end{proof}
\begin{proposition}\label{bulk-est}
 There holds
 \[\bega 
&\int_{\Omega\times\R^3}\frac{(\pt_t+\frac 1 \eps \xi\cdot\nabla_x)\sqrt\mu}{\sqrt\mu}|f_\al|^2dxd\xi\\
&\lesssim   o(1)\frac{1}{\kappa\e^2}\nl\sqrt\nu \IP f_\al
\nr_{L^2_{x,\xi}}^2+e^{-\frac{\rho}{2\eps^2}} \nl\nabla_x U\nr_{L^2_x}\nl e^{\rho|\xi|^2} f_\al
\nr_{L^\infty_{x,\xi}}\nl  f_\al
\nr_{L^2}
+\nl f_\al
\nr_{L^2_{x,\xi}}^2\\
&\quad+ \nl \frac{\mathcal T_\al^b}{x_3}\nr_{L^2}\nl f_\al\nr_{L^2} +\eps\kappa^{-\frac 12}\delta^{-1} 
\lw( \nl\pt^\al (U_3)^2\nr_{L^2}+\nl \pt^\al (PU_3)\nr_{L^2}
\rw)\nl f_\al\nr_{L^2},
\enda\]
where $\mathcal T_\al^b$ satisfies the bound in Lemma \ref{decom-b3}.
\end{proposition}
\begin{proof}
We have
\beq\label{ran-55}
\bega 
&\int_{\Omega\times\R^3}\frac{(\pt_t+\frac 1 \eps \xi\cdot\nabla_x)\sqrt\mu}{\sqrt\mu}|f_\al |^2dx d\xi=\mathcal I_1+\mathcal I_2,\\
\enda 
\eeq where
\[\bega 
\mathcal I_1&=\int_{\Omega\times\R^3}\frac{(\pt_t+\frac 1 \eps \xi\cdot\nabla_x)\sqrt\mu}{\sqrt\mu}|\IP f_\al |^2dx d\xi,\\
\mathcal I_2&=\int_{\Omega\times\R^3}\frac{(\pt_t+\frac 1 \eps \xi\cdot\nabla_x)\sqrt\mu}{\sqrt\mu}\lw(|\P f_\al |^2+2\P f _\al \cdot \IP f_\al 
\rw) dx d\xi.\\
\enda \]
\underline{Bounding $\mathcal I_1$}:
We have 
\beq\label{bulk-mu}
\frac{(\pt_t+\frac 1 \eps \xi\cdot\nabla_x)\sqrt\mu}{\sqrt\mu}=\frac 12 \pt_i U_j \vp_i\vp_j+\eps(\pt_t U+U\cdot\nabla_x U)\cdot\vp.
\eeq
This implies 
\[\bega 
\mathcal I_1&\lesssim \int_{\Omega\times\R^3}|\nabla_x U||\vp|^2 |\IP f_\al |^2dxd\xi\\
&\quad+\int_{\Omega\times\R^3}|\eps (\pt_t U+U\cdot\nabla_x U)| |\vp| |\IP f_\al |^2dxd\xi\\
&\lesssim\int_{\Omega\times\R^3}|\nabla_x U||\vp|^2 |\IP f_\al |^2dxd\xi+o(1)\frac{1}{\kappa\eps^2}\nl \sqrt\nu \IP f_\al
\nr_{L^2_{x,\xi}}^2.
\enda 
\]
Now we bound the term
\[\int_{\Omega\times\R^3}|\nabla_x U||\vp|^2 |\IP f_\al|^2dxd\xi.
\]
We have 
\[\bega 
&\int_{\Omega\times\R^3}|\nabla_x U||\vp|^2 |\IP f_\al |^2dxd\xi\\
&=\int_{|\vp|\le \eps^{-1}}|\nabla_x U||\vp|^2 |\IP f_\al |^2dxd\xi+\int_{|\vp|\ge \eps^{-1}}|\nabla_x U||\vp|^2 |\IP f_\al|^2dxd\xi\\
&\lesssim o(1)\frac{1}{\kappa\e^2}\nl\sqrt\nu \IP f_\al
\nr_{L^2_{x,\xi}}^2+\int_{|\vp|\ge \eps^{-1}}|\nabla_x U||\vp|^2 |\IP f_\al|^2dxd\xi.\enda 
\]
It suffices to bound the second term in the above. 
We have 
\[\bega 
&\int_{|\vp|\ge \eps^{-1}}|\nabla_x U||\vp|^2 |\IP f_\al |^2dxd\xi \\
&\lesssim \nl e^{\rho|\xi|^2} f
\nr_{L^\infty_{x,\xi}}\int_{|\vp|\ge \eps^{-1}}|\nabla_x U||\vp|^2 e^{-\rho|\xi|^2} |\IP f_\al |dx d\xi\\
&\lesssim e^{-\frac{\rho}{2\eps^2}} \nl e^{\rho|\xi|^2} f
\nr_{L^\infty_{x,\xi}}\nl e^{-\frac \rho {10}|\xi|^2}\nabla_x U 
\nr_{L^2_{x,\xi}} \nl  f_\al
\nr_{L^2_{x,\xi}}\\
&\lesssim e^{-\frac{\rho}{2\eps^2}} \nl\nabla_x U\nr_{L^2_x}\nl e^{\rho|\xi|^2} f_\al
\nr_{L^\infty_{x,\xi}}\nl  f_\al
\nr_{L^2}.
\enda
\]
Hence we obtain 
\[\bega 
\mathcal I_1&\lesssim o(1)\frac{1}{\kappa\e^2}\nl\sqrt\nu \IP f_\al
\nr_{L^2_{x,\xi}}^2+e^{-\frac{\rho}{2\eps^2}} \nl\nabla_x U\nr_{L^2_x}\nl e^{\rho|\xi|^2} f_\al
\nr_{L^\infty_{x,\xi}}\nl  f_\al
\nr_{L^2}.
\enda 
\]
\underline{Bounding $\mathcal I_2$}:
By a direct calculation, we have 
\[\bega 
&\mathcal I_1 =\pt_i U_j b_{i,\al} b_{j,\al}\\
&+\eps a_\al  b_\al\cdot (\pt_t U+ U\cdot\nabla_x U)+\eps c_\al b_\al\cdot (\pt_t U+U\cdot \nabla_x U)+2a_\al \pt_i U_j \la \IP f_\al,\hat A_{ij}\ra_{L^2_{\xi}}\\
&+b_{m,\al}\lw\{ \eps(\pt_t U_k+U_i\pt_i U_k)\la \hat A_{mk},(\II-\P)f_\al \ra_{L^2_\xi}+\pt_iU_k\la \vp_i\vp_k\vp_m\sqrt\mu, (\II-\P)f_\al \ra_{L^2_\xi}
\rw\}\\
&+\frac 1 2 c_\al \lw\{\eps(\pt_t U_k+U\cdot\nabla U_k)\la \vp_k(|\vp|^2-3)\sqrt\mu,\IP f_\al \ra_{L^2_\xi}\rw\}\\
&+\frac 1 2c_\al \lw\{\pt_i U_k \la \vp_i\vp_k(|\vp|^2-3)\sqrt\mu,(\II-\P)f_\al 
\ra_{L^2_\xi}\rw\}\\
&+\lw\la
\frac{(\pt_t+\frac 1 \eps \xi\cdot\nabla_x)\sqrt\mu}{\sqrt\mu},|\IP f_\al|^2
\rw\ra_{L^2_\xi}.
\enda 
\]
This implies that \[\bega 
\mathcal I_2&
&\lesssim \sum_{j=1}^2 \lw|\int_\Omega \pt_3 U_j b_{\al,3}b_{\al,j}\rw|+\nl f_\al
\nr_{L^2_{x,\xi}}^2+o(1)\frac 1 {\kappa\eps^2}\nl\sqrt\nu \IP f_\al
\nr_{L^2_{x,\xi}}^2.\enda
\]
It suffices to bound the term $\int_\Omega \pt_3 U_j b_{\al,3}b_{\al,j}$ for $j\in \{1,2\}$. Using Lemma \ref{decom-b3}, we have
\[\bega 
&\int_\Omega \pt_3 U_j b_{\al,3}b_{\al,j}=\int_\Omega \pt_3 U_j(\mathcal T_{\al}^b-\eps U_3a_\al-\frac \eps{2\delta}\pt^\al(U_3^2)-\frac\eps\delta \pt^\al (PU_3))b_{\al,j}\\
&\lesssim\nl x_3\pt_3 U_j\nr_{L^\infty} \nl \frac{\mathcal T_\al^b}{x_3}\nr_{L^2}\nl f_\al\nr_{L^2} +\eps \nl f_\al\nr_{L^2}^2\\
&+\eps\kappa^{-\frac 12}\delta^{-1} 
\lw( \nl\pt^\al (U_3)^2\nr_{L^2}+\nl \pt^\al (PU_3)\nr_{L^2}
\rw)\nl f_\al\nr_{L^2}.
\enda 
\]
The proof is complete.
\end{proof}As a corollary, we have 
\begin{corollary}\label{bulk-est1}
 \[\bega
&\lw|\sum_{\al\in \mathbb N_0^3}\AA_\al(t)^2 \int_{\Omega\times \R^3}\frac{(\pt_t+\frac 1 \eps \xi\cdot\nabla_x)\sqrt\mu}{\sqrt\mu}|f_\al|^2dxd\xi
\rw|\\
&\lesssim o(1)\mathcal D_f(t)^2+\mathcal E_f(t)^2+e^{-\frac{\rho}{4\e^2}}\mathcal H_f(t)\mathcal E_f(t)+o(1)\mathcal Y_f(t)^2.
\enda 
\]
\end{corollary} 
\begin{proof} From the estimate in the above lemma \ref{bulk-est}, we have
\[\bega 
&\sum_\al A_\al^2 \int_{\Omega\times\R^3}\frac{(\pt_t+\frac 1 \eps \xi\cdot\nabla_x)\sqrt\mu}{\sqrt\mu}|f_\al|^2dxd\xi\\
&\lesssim o(1)\mathcal D_f(t)^2+\mathcal E_f(t)^2+e^{-\frac{\rho}{4\e^2}}\mathcal H_f(t)\mathcal E_f(t)+\sum_\al \AA_\al^2 \nl f_\al\nr_{L^2} \nl \frac{\mathcal T_\al^b}{x_3}\nr_{L^2}.\enda 
\]
Hence it suffices to estimate 
\[
\mathcal I=\sum_\al \AA_\al^2 \nl f_\al\nr_{L^2} \nl \frac{\mathcal T_\al^b}{x_3}\nr_{L^2}.\]
Now using Lemma \ref{decom-b3}, we get
\[
\mathcal I\lesssim \mathcal I_1+\mathcal I_2+\mathcal E_f(t)^2,
\]
where
\[\bega 
\mathcal I_1&=\sum_\al \AA_\al^2 \nl f_\al\nr_{L^2}\sum_{\pt\in\{\pt_{x_1},\pt_{x_2}\}}\nl \pt f_\al\nr_{L^2},\\
\mathcal I_2&=\sum_\al \AA_\al^2 \nl f_\al\nr_{L^2}\sum_{\substack{\beta+\gamma=\al\\|\gamma|\ge1}}\frac{\al!}{\beta!\gamma!}\sum_{\pt\in \{\eps\pt_t,\pt_{x_1},\pt_{x_2}\}}\nl \pt f_\beta
\nr_{L^2}\cdot\nl e^{o(1)|\xi|^2}(1,\eps\pt_t,\nabla_x) \pt^{\gamma}\sqrt\mu
\nr_{L^\infty}.
\enda 
\]
\underline{Bounding $\mathcal I_1$}:
Using the fact that 
\[
\frac{\tau}{\sqrt{(|\al|+1)(|\al|+2)
}}\max\lw\{\frac{\AA_{\al}}{\AA_{(\al_0+1,\al_1,\al_2)}},\frac{\AA_{\al}}{\AA_{(\al_0,\al_1+1,\al_2)}},\frac{\AA_{\al}}{\AA_{(\al_0,\al_1,\al_2+1)}}\rw\}\le 1,
\]
we obtain 
$
\mathcal I_1\lesssim \mathcal Y_f(t)^2.
$~\\
\underline{Bounding $\mathcal I_2$}: First, we note that for any non-negative sequences $\{x_\al\},\{y_\al\},\{z_\al\}$, we have 
\[\sum_\al \AA_\al^2 x_\al \sum_{\substack{\beta+\gamma=\al\\|\gamma|\ge1}}\frac{\al!}{\beta!\gamma!}y_\beta z_\gamma\lesssim \lw(\sum_\al \frac{|\al|+1}{\tau} \AA_\al^2 x_\al^2 \rw)^{\frac12}\lw(\sum_\al \frac{|\al|+1}{\tau} \AA_\al^2 y_\al^2
\rw)^{\frac 12}\sum_{|\gamma|\ge 1} \frac{\tau^{|\gamma|}}{\gamma!}z_\gamma\]
by the discrete Young inequality and the fact that 
\[
\frac{\al!\AA_\al}{\gamma!\AA_\gamma}\lesssim \tau^{-1}\beta! \min\lw\{\AA_{(\beta_0+1,\beta_1,\beta_2)}, \AA_{(\beta_0,\beta_1+1,\beta_2)}, \AA_{(\beta_0,\beta_1,\beta_2+1)}\rw\}.
\]
This, combining with the bound of $\mu$ in Lemma \eqref{local-1}, gives us the bound $\mathcal I_2\lesssim o(1)\mathcal Y_f(t)^2$. The proof is complete.
\end{proof}
\begin{proposition}\label{com-est} [Commutator terms] There holds 
\[
\bega
\frac 1  {\kappa\eps^2}\sum_{\al\in \mathbb N_0^3}\AA_\al^2(t)\la [L,\pt^\al]f, f_\al
\ra_{L^2_{x,\xi}}&\lesssim o(1)\mathcal D_f(t)^2+\mathcal E_f(t)\mathcal D_f(t),\\
\frac{1}{\kappa\e^2}\nl \nl \nu^{-\frac 1 2}\AA_\al [L,\pt^\al]f
\nr _{L^2_{x,\xi}}\nr_{\ell_\al^2}&\lesssim \mathcal D_f(t)+  \mathcal E_f(t),\\
\frac 1  {\kappa\eps^2}\sum_{\al\in \mathbb N_0^3}\AA_\al^2(t)\la [L,\pt_t \pt^\al]f, \pt_tf_\al
\ra_{L^2_{x,\xi}}&\lesssim \kappa^{-\frac 12} \mathcal D_f(t)^2+\kappa^{-\frac 12}\mathcal E_f(t)\mathcal D_f(t)\\
&\quad+o(1)\mathcal D_{\pt_tf}(t)^2+\mathcal E_{\pt_t}(t)\mathcal D_{\pt_tf}(t),\\
\frac{1}{\kappa\e^2}\nl \nl \nu^{-\frac 1 2}\AA_\al [L,\pt_t\pt^\al]f
\nr _{L^2_{x,\xi}}\nr_{\ell_\al^2}&\lesssim\kappa^{-\frac 12} \mathcal D_f(t)+  \kappa^{-\frac 12} \mathcal E_f(t)+\mathcal D_{\pt_tf}(t)+  \mathcal E_{\pt_tf}(t).\\
\enda
\]
\end{proposition}
\begin{proof}Recalling \eqref{L-def} and the fact that $\pt^\al \IP =\IP \pt^\al+[\P,\pt^\al]$, we decompose 
$
[L,\pt^\al]f=\sum_{i=1}^8 T_{\al,i},
$
where 
\beq\label{L-alpha-com}
\bega 
T_{\al,1}&=\sum_{\substack{\beta+\gamma+\vr =\al\\\vr<\al
}
}\frac{\al!}{\beta!\gamma!\vr!}\int_{\R^3}\int_{|\w|=1}|(\xi-\xi_\star)\cdot \w|\pt^\beta\sqrt\mu(\xi_\star)\pt^\gamma\sqrt\mu(\xi') \IP f_\vr(\xi_\star')d\w d\xi_\star,\\
T_{\al,2}&=\sum_{\substack{\beta+\gamma+\vr =\al\\\vr<\al
}
}\frac{\al!}{\beta!\gamma!\vr!}\int_{\R^3}\int_{|\w|=1}|(\xi-\xi_\star)\cdot \w|\pt^\beta\sqrt\mu(\xi_\star)\pt^\gamma\sqrt\mu(\xi') [\P,\pt^\vr] f(\xi_\star')d\w d\xi_\star,\\
T_{\al,3}&=\sum_{\substack{\beta+\gamma+\vr =\al\\\vr<\al
}
}\frac{\al!}{\beta!\gamma!\vr!}\int_{\R^3}\int_{|\w|=1}|(\xi-\xi_\star)\cdot \w|\pt^\beta \sqrt\mu(\xi_\star)\pt^\gamma \sqrt\mu(\xi_\star')\IP f_\vr(\xi')d\w d\xi_\star,\\
T_{\al,4}&=\sum_{\substack{\beta+\gamma+\rho =\al\\\vr<\al
}
}\frac{\al!}{\beta!\gamma!\vr!}\int_{\R^3}\int_{|\w|=1}|(\xi-\xi_\star)\cdot \w|\pt^\beta \sqrt\mu(\xi_\star)\pt^\gamma \sqrt\mu(\xi_\star')[\P,\pt^\vr]f(\xi')d\w d\xi_\star,\\
T_{\al,5}&=-\sum_{\substack{\beta+\gamma+\vr =\al\\\vr<\al
}
}\frac{\al!}{\beta!\gamma!\vr!}\int_{\R^3}\int_{|\w|=1}|(\xi-\xi_\star)\cdot \w|\pt^\beta\sqrt\mu(\xi_\star) \pt^\gamma \sqrt\mu(\xi)\IP f_\vr (\xi_\star)d\w d\xi_\star,\\
T_{\al,6}&=-\sum_{\substack{\beta+\gamma+\vr =\al\\\vr<\al
}
}\frac{\al!}{\beta!\gamma!\vr!}\int_{\R^3}\int_{|\w|=1}|(\xi-\xi_\star)\cdot \w|\pt^\beta\sqrt\mu(\xi_\star) \pt^\gamma \sqrt\mu(\xi)[\P,\pt^\vr] f(\xi_\star)d\w d\xi_\star,\\
T_{\al,7}&= -\sum_{\substack{\beta+\gamma=\al\\\gamma<\al
}
}\frac{\al!}{\beta!\gamma!}\int_{\R^3}\int_{|\w|=1}|(\xi-\xi_\star)\cdot \w|\pt^\beta\mu(\xi_\star) \IP   f_\gamma (\xi)d\w d\xi_\star,\\
T_{\al,8}&= -\sum_{\substack{\beta+\gamma=\al\\\gamma<\al
}
}\frac{\al!}{\beta!\gamma!}\int_{\R^3}\int_{|\w|=1}|(\xi-\xi_\star)\cdot \w|\pt^\beta\mu(\xi_\star) [\P,   \pt^\gamma] f (\xi)d\w d\xi_\star.
\enda 
\eeq
We have 
\[\bega 
&\AA_\al^2 \la [L,\pt^\al]f ,f_\al
\ra_{L^2_{x,\xi}}\\
&=\sum_{i=1}^8\la \AA_\al T_{\al,i},\AA_\al f_\al
\ra_{L^2_{x,\xi}}\\
&\lesssim \sum_{i=1}^8\lw\{ \nl\nu^{-\frac 1 2}\AA_\al T_{\al,i}\nr_{L^2_{x,\xi}} \nl \sqrt\nu \AA_\al \IP  f_\al
\nr_{L^2_{x,\xi}}+\nl\nu^{-\frac 1 2}\AA_\al T_{\al,i}\nr_{L^2_{x,\xi}} \nl  \AA_\al  f_\al
\nr_{L^2_{x,\xi}}\rw\}.
\enda 
\]
This implies 
\[\bega 
&\frac{1}{\kappa\eps^2}\sum_\al \AA_\al^2 \la [L,\pt^\al]f , f_\al
\ra_{L^2_{x,\xi}}\\
&\lesssim\max_{1\le i\le 8} \lw\{\sum_\al  \nl\nu^{-\frac 1 2}\AA_\al T_{\al,i}\nr_{L^2_{x,\xi}}^2
\rw\}^{\frac 1 2}\lw( \eps^{-1}\kappa^{-\frac 1 2} \mathcal D_f(t)+\eps^{-2}\kappa^{-1}\mathcal E_f(t)
\rw).
\enda 
\]
We now bound each of the term $T_{\al,i}$. Using Proposition \ref{BGR}, we have 
\[\bega 
\nl \nu^{-\frac 12}T_{\al,1}
\nr _{L^2_{x,\xi}}&\lesssim \sum_{\substack{\gamma+\vr=\al\\|\gamma|\ge 1
}}\AA_\gamma \lw\|e^{p_0|\xi|^2}\pt^\gamma\sqrt\mu
\rw\|_{L^\infty}\lw\|\sqrt\nu \IP \AA_\vr f_{\vr}
\rw\|_{L^2_{x,\xi}}\\
&\quad+\sum_{\substack{\beta+\gamma+\vr=\al\\
|\beta|\ge 1
}
}\AA_\beta \nl e^{p_0|\xi|^2}\pt^\beta\sqrt\mu\nr _{L^\infty}\cdot \AA_\gamma\nl e^{p_0|\xi|^2}\pt^\gamma\sqrt\mu\nr _{L^\infty}\cdot\lw\|\sqrt\nu \IP \AA_\vr f_{\vr}
\rw\|_{L^2_{x,\xi}},\\
\nl \nu^{-\frac 12}T_{\al,2} \nr _{L^2_{x,\xi}}&\lesssim \sum_{\substack{\gamma+\vr=\al\\|\gamma|\ge 1
}
}\AA_\gamma \nl e^{p_0|\xi|^2}\pt^\gamma\sqrt\mu\nr _{L^\infty}\AA_\vr \lw\|\sqrt\nu [\P,\pt^\vr] f
\rw\|_{L^2_{x,\xi}}\\
&\quad+\sum_{\substack{\beta+\gamma+\vr=\al\\
|\beta|\ge 1
}
}\AA_\beta \nl e^{p_0|\xi|^2}\pt^\beta\sqrt\mu\nr _{L^\infty}\cdot \AA_\gamma\nl e^{p_0|\xi|^2}\pt^\gamma\sqrt\mu\nr _{L^\infty} \AA_\vr \lw\|\sqrt\nu [\P,\pt^\vr] f
\rw\|_{L^2_{x,\xi}}.
\enda \]
Now using the discrete Young inequality, Lemma \ref{local-1} and Proposition \ref{PD} , we have 
\[\bega 
\nl 
\nl \nu^{-\frac 12}T_{\al,1}
\nr _{L^2_{x,\xi}}
\nr_{\ell^2_\al}+\nl 
\nl \nu^{-\frac 12}T_{\al,2}
\nr _{L^2_{x,\xi}}
\nr_{\ell^2_\al}&\lesssim \eps^2\kappa\mathcal D_f(t)+ \eps^2 \kappa \mathcal E_f(t).\\
\enda 
\]
Similarly, using Proposition \ref{BGR1}, Lemma \ref{local-1},  and Proposition \ref{PD}, we obtain
\[\bega 
\nl 
\nl \nu^{-\frac 12}T_{\al,1}
\nr _{L^2_{x,\xi}}
\nr_{\ell^2_\al}+\nl 
\nl \nu^{-\frac 12}T_{\al,2}
\nr _{L^2_{x,\xi}}
\nr_{\ell^2_\al}&\lesssim  \eps^2\kappa\mathcal D_f(t)+ \eps^2 \kappa \mathcal E_f(t).
\enda \] 
As for $T_{\al,5}$ and $T_{\al,6}$, we use Proposition \ref{BGR2} to get 
\[
\bega 
&\nl T_{\al,5}
\nr _{L^2_{x,\xi}}+\nl T_{\al,6}
\nr _{L^2_{x,\xi}}\\
&\lesssim\sum_{\substack{\beta+\gamma+\vr=\al\\|\beta|+|\gamma|\ge 1
}
}
\lw\{
\AA_\vr \nl \sqrt\nu \IP f_\vr 
 \nr_{L^2_{x,\xi}}+\AA_\vr \nl \sqrt\nu[\P,\pt^\vr] f
 \nr_{L^2_{x,\xi}} \rw\}\\
 &\times \nl\AA_\beta\lw(\int_{\xi_\star}\nu(\xi_\star)^{-1}|\xi-\xi_\star|^2 |\pt^\beta\sqrt\mu(\xi_\star)|^2d\xi_\star\rw)^{1/2}\AA_\gamma\pt^\gamma\sqrt\mu 
 \nr_{L^\infty_x L^2_\xi}.\\
   \enda 
\]
Now using the discrete Young inequality, Lemma \ref{local-1}, Proposition \ref{PD} and Proposition \ref{BGR2}, we have 
\[
\bega 
\max_{5\le m\le 8}\nl 
\nl \nu^{-\frac 12}T_{\al,m}
\nr _{L^2_{x,\xi}}
\nr_{\ell^2_\al}&\lesssim   \eps^2\kappa\mathcal D_f(t)+ \eps^2 \kappa \mathcal E_f(t).\enda
\]
Hence we get 
\[\bega 
\frac{1}{\kappa\eps^2}\sum_\al \AA_\al^2 \la [L,\pt^\al]f , f_\al
\ra_{L^2_{x,\xi}}&\lesssim \lw(\eps^2\kappa\mathcal D_f(t)+ \eps^2 \kappa \mathcal E_f(t)\rw)\lw( \eps^{-1}\kappa^{-\frac 1 2} \mathcal D_f(t)+\eps^{-2}\kappa^{-1}\mathcal E_f(t)
\rw)\\
&\lesssim \eps\kappa^{\frac 12}\mathcal D_f(t)^2+\eps \kappa
\mathcal D_f(t)\mathcal E_f(t)+ \mathcal E_f(t)^2.
\enda
\] The commutative estimate for time derivative is similar, where there are two cases: $\pt_t$ hits $\sqrt\mu$, which gives a scale of $\eps\kappa^{-\frac12}$, and when $\pt_t$ hits $f$, which gives us the same bound for $f$ without any derivatives.
This completes the proof.
\end{proof}

\begin{lemma}\label{bulk-ran-1}
There holds, for $\eps\ll \kappa$
\[\pt^\al\lw\{
\frac{(\pt_t+\eps^{-1}\xi\cdot\nabla_x)\sqrt\mu}{\sqrt\mu}
\rw\}=\frac 1 2 \sum_{i,\ell}
\pt_i \pt^\al U_\ell \vp_i \vp_\ell+J_\al\]
where 
\[
\sum_{\al}\frac{\tau_0^{|\al|}}{\al!} |J_\al|\lesssim o(1)\la\vp\ra.
\]
\end{lemma}

\begin{proof} Recalling \eqref{bulk-mu}, we get \[\bega 
&\pt^\al\lw\{
\frac{(\pt_t+\eps^{-1}\xi\cdot\nabla_x)\sqrt\mu}{\sqrt\mu}
\rw\}\\
&=\frac 1 2 \pt_i \pt^\al U_\ell \vp_i \vp_\ell +\frac 1 2\sum_{\substack{
\beta+\gamma+\vr=\al\\
|\gamma|\ge 1\\|\vr|\ge 1
}
}\frac{\al!}{\beta!\gamma!\vr!}\pt^\beta \pt_i U_\ell \pt^\gamma \vp_i\pt^\vr \vp_\ell \\
&\quad+\frac 1 2 \sum_{\substack{
\beta+\gamma=\al\\
|\gamma|\ge 1\\
}
}\frac{\al!}{\beta!\gamma!}\pt^\beta \pt_i U_\ell  \pt^\gamma \vp_i\vp_\ell +\frac 1 2 \sum_{\substack{
\beta+\gamma=\al\\
|\gamma|\ge 1\\
}
}\frac{\al!}{\beta!\gamma!}\pt^\beta \pt_i U_\ell  \pt^\gamma \vp_\ell \vp_i\\
&\quad+\eps\sum_{\substack{\beta+\gamma=\al\\|\gamma|\ge 1
}
}\frac{\al!}{\beta!\gamma!}\pt^\beta(\pt_t U+U\cdot\nabla_x U)\cdot \pt^\gamma \vp+\eps \pt^\al (\pt_t U+U\cdot\nabla_x U)\cdot \vp\\
&=\frac 1 2 \pt_i \pt^\al U_\ell\vp_i \vp_\ell +J_\al.
\enda
\]
where \[\bega 
J_\al&\lesssim \sum_{\substack{
\beta+\gamma+\vr=\al\\
|\gamma|\ge 1\\|\vr|\ge 1
}
}\frac{\al!}{\beta!\gamma!\vr!}|\pt^\beta \pt_i U_\ell |\cdot  \eps |\pt^\gamma U_i| \cdot \eps |\pt^\vr U_\ell |+\sum_{\substack{
\beta+\gamma=\al\\
|\gamma|\ge 1\\
}
}\frac{\al!}{\beta!\gamma!}|\pt^\beta \pt_i U_\ell | \cdot\eps | \pt^\gamma U_i| |\vp|\\
&\quad+\sum_{\substack{
\beta+\gamma=\al\\
|\gamma|\ge 1\\
}
}\frac{\al!}{\beta!\gamma!}|\pt^\beta \pt_i U_\ell | \cdot \eps| \pt^\gamma U_\ell | |\vp|+\eps \sum_{\substack{\beta+\gamma=\al\\|\gamma|\ge 1
}
}\frac{\al!}{\beta!\gamma!}|\pt^\beta(\pt_t U+U\cdot\nabla_x U)| \cdot \eps |\pt^\gamma U|\\
&\quad+\eps |\pt^\al (\pt_t U+U\cdot\nabla_x U)| |\vp|.\enda 
\]

This implies 
\[\bega 
\sum_{|\al|\ge 1}\frac{\tau^{|\al|}}{\al!}J_\al &\lesssim  \eps^2 \nl U
\nr_{\tau_0,\infty}^2\nl \nabla_x U
\nr_{\tau_0,\infty}+\eps \nl \nabla_x U
\nr_{\tau_0,\infty}\nl U\nr_{\tau_0,\infty}\\
&+\eps \lw(
\nl \nabla_x U
\nr_{\tau_0,\infty} \nl U
\nr _{\tau_0,\infty}+\nl \pt_t U
\nr _{\tau_0,\infty}
\rw)
|\vp| \\
&\quad+\eps^2 \nl U
\nr _{\tau_0,\infty}\nl \pt_t U+U\cdot\nabla_x U
\nr _{\tau_0,\infty}\lesssim \eps^{\frac 12} \la\vp\ra.\enda 
\]
This completes the proof.\end{proof}

\begin{proposition} \label{bulk-com}There holds 
\[\bega
&\sum_{\al\in \mathbb N_0^3}\AA_\al(t)^2 \lw\la \frac{(\pt_t+\frac 1 \eps \xi\cdot\nabla_x)\sqrt\mu}{\sqrt\mu}f_\al-\pt^\al\lw\{\frac{(\pt_t+\frac 1 \eps \xi\cdot\nabla_x)\sqrt\mu}{\sqrt\mu}f\rw\},f_\al\rw\ra_{L^2_{x,\xi}}\\
&\lesssim \mathcal E_f(t)^2+o(1)\mathcal D_f(t)^2+\sum_{\al} \AA_\al^2 \nl f_\al\nr_{L^2} \sum_{\substack{\beta+\gamma=\al\\|\gamma|\ge 1
}}\frac{\al!}{\beta!\gamma!}\nl x_3^{-1}T_\beta^b\nr_{L^2} \nl x_3\pt_3\pt^\gamma U_\parallel\nr_{L^\infty}\\
&\quad+\sum_\al \AA_\al^2\nl x_3^{-1}T_\al^b
\nr_{L^2}\sum_{\substack{\beta+\gamma=\al\\|\gamma|\ge 1
}}\frac{\al!}{\beta!\gamma!} \nl f_\beta\nr_{L^2} \nl x_3\pt_3\pt^\gamma U_\parallel\nr_{L^\infty},\enda 
\]
where $T_\al^b$ satisfies the bound in Lemma \ref{decom-b3}. 
\end{proposition}

\begin{proof}In the proof of this proposition, let us denote 
\[
S_\al=\AA_\al^2 \lw\la \frac{(\pt_t+\frac 1 \eps \xi\cdot\nabla_x)\sqrt\mu}{\sqrt\mu}f_\al-\pt^\al\lw\{\frac{(\pt_t+\frac 1 \eps \xi\cdot\nabla_x)\sqrt\mu}{\sqrt\mu}f\rw\},f_\al\rw\ra_{L^2_{x,\xi}}.
\]
We have 
\beq\label{ran-888}
\bega
&\frac{(\pt_t+\frac 1 \eps \xi\cdot\nabla_x)\sqrt\mu}{\sqrt\mu}f_\al-\pt^\al\lw\{\frac{(\pt_t+\frac 1 \eps \xi\cdot\nabla_x)\sqrt\mu}{\sqrt\mu}f\rw\}\\
&=-\sum_{\substack{\beta+\gamma=\al\\\beta<\al}}\frac{\al!}{\beta!\gamma!}\pt^\gamma\lw\{\frac{(\pt_t+\frac 1 \eps \xi\cdot\nabla_x)\sqrt\mu}{\sqrt\mu}\rw\}\pt^\beta f.\\
\enda
\eeq
Combining the above with Lemma \ref{bulk-ran-1}, we have 
\[S_\al=S_{\al,1}+S_{\al,2}\]
where 
\beq\label{ran-213}
\bega 
S_{\al,1}&=-\AA_\al^2\sum_{\substack{\beta+\gamma=\al\\\beta<\al}}\frac{\al!}{\beta!\gamma!}\lw\la J_\gamma,
f_\beta f_\al
\rw\ra_{L^2_{x,\xi}},\\
S_{\al,2}&=-\frac 1 2\AA_\al^2\sum_{\substack{\beta+\gamma=\al\\\beta<\al}}\frac{\al!}{\beta!\gamma!}\sum_{i,k}\la \pt_i \pt^\gamma U_k\vp_i\vp_k,f_\beta f_\al \ra_{L^2_{x,\xi}}.\enda \eeq
First we have 
\[\bega 
\la J_\gamma, f_\beta f_\al\ra_{L^2_{x,\xi}}&\lesssim \nl \nu^{-\frac 1 2}J_\gamma f_\beta\nr_{L^2_{x,\xi}}\nl \sqrt\nu f_\al \nr_{L^2_{x,\xi}}\\
&\lesssim \nl \nu^{-1}J_\gamma\nr_{L^\infty_{x,\xi}}\nl \sqrt\nu f_\beta\nr_{L^2_{x,\xi}}\nl \sqrt\nu f_\al\nr_{L^2_{x,\xi}}.
\enda 
\]
Hence
\[
\nl S_{\al,1}\nr_{\ell_\al^1}\lesssim \nl \AA_\al\sqrt\nu  f_\al\nr_{\ell_\al^\infty} \nl \nu^{-1}\AA_\al J_\al \nr_{\ell_\al^2}\nl\AA_\al\sqrt\nu  f_\al\nr_{\ell_\al^2}\lesssim  \mathcal E_f(t)^2+o(1)\mathcal D_f(t)^2.
\]
Next we bound $S_{\al,2}$. We have 
\[\bega 
\la \vp_i\vp_k,f_\beta f_\al\ra_{L^2_\xi}&=\la \vp_i\vp_k, \P f_\beta \P f_\al\ra_{L^2_\xi}+\la \vp_i\vp_k, \P f_\beta \IP f_\al\ra_{L^2_\xi}\\
&\quad+\la \vp_i\vp_k, \P f_\al \IP f_\beta\ra_{L^2_\xi}+\la \vp_i\vp_k, \IP f_\beta \IP f_\al\ra_{L^2_\xi}.\\
\enda 
\]
This gives us the decomposition
\[
S_{\al,2}=-\frac 1 2 \sum_{k=1}^4 T_{\al,k}
\]
where 
\[\bega
T_{\al,1}&= \AA_\al^2 \sum_{\substack{\beta+\gamma=\al\\|\gamma|\ge 1}}\frac{\al!}{\beta!\gamma!}\la\pt_i\pt^\gamma U_k,\vp_i\vp_k \P f_\al \P f_\beta\ra_{L^2_{x,\xi}},\\
T_{\al,2}&=\AA_\al^2 \sum_{\substack{\beta+\gamma=\al\\|\gamma|\ge 1}}\frac{\al!}{\beta!\gamma!}\la\pt_i\pt^\gamma U_k,\vp_i\vp_k \P f_\beta \IP f_\al\ra_{L^2_{x,\xi}},\\
T_{\al,3}&=\AA_\al^2 \sum_{\substack{\beta+\gamma=\al\\|\gamma|\ge 1}}\frac{\al!}{\beta!\gamma!}\la\pt_i\pt^\gamma U_k,\vp_i\vp_k \P f_\al \IP f_\beta\ra_{L^2_{x,\xi}},\\
T_{\al,4}&=\AA_\al^2 \sum_{\substack{\beta+\gamma=\al\\|\gamma|\ge 1}}\frac{\al!}{\beta!\gamma!}\la \pt_i\pt^\gamma U_k \vp_i\vp_k \IP f_\beta \IP f_\al\ra_{L^2_\xi}.
\enda
\]
First we bound $\nl T_{\al,1}\nr_{\ell_\al^1}$. To this end, we find 
\[\bega 
&\la\vp_i\vp_k, \P f_\al \P f_\beta\ra_{L^2_{\xi}}=a_\al a_\beta \delta_{ik}+3(a_\al c_\beta+a_\beta c_\al) \delta_{ik}\\
&\quad+b_{\al,m}b_{\beta,n}\lw(\delta_{ik}\delta_{mn}+\delta_{im}\delta_{kn}+\delta_{in}\delta_{km}
\rw)+c_\al c_\beta \delta_{ik}\lw\la |\vp|^2\mu,\lw(\frac{|\vp|^2-3}{2}\rw)^2\rw\ra.
\enda 
\]
This implies
\[\bega 
&\sum_{i,k}\pt_i\pt^\gamma U_k\la \vp_i\vp_k,\P f_\al,\P f_\beta\ra_{L^2_\xi}\\
&=a_\al a_\beta \pt^\gamma(\nabla_x\cdot U)+O(1) (a_\al c_\beta+a_\beta c_\al)\pt^\gamma(\nabla_x\cdot U) \\
&\quad+\pt^\gamma(\nabla_x\cdot U)(b_\al\cdot b_\beta)+\pt^\gamma \lw(\pt_k U_i+\pt_k U_i\rw) b_{\al,i}b_{\beta,k}\\
&\quad+O(1)c_\al c_\beta \pt^\gamma(\nabla_x\cdot U).\\
\enda\]
Hence
\[\bega 
\sum_{i,k}\pt_i\pt^\gamma U_k\la \vp_i\vp_k,\P f_\al,\P f_\beta\ra_{L^2_{x,\xi}}&\lesssim \nl  f_\al
\nr_{L^2}\nl f_\beta\nr_{L^2}\nl \pt^\gamma(\div U)\nr_{L^\infty}\\
&\quad+\lw|\pt^\gamma \lw(\pt_k U_i+\pt_i U_k\rw) b_{\al,i}b_{\beta,k}
\rw|.
\enda \]
Hence we get 
\[
\nl T_{\al,1}\nr_{\ell_\al^1}\lesssim \mathcal E_f(t)^2+\nl \AA_\al^2
\sum_{\substack{\beta+\gamma=\al\\|\gamma|\ge 1
}}\max_{1\le i\le 2}\lw|\int_\Omega \pt^\gamma \pt_3 U_i(b_{\al,i}b_{\beta,3}+b_{\al,3}b_{\beta,i})dx
\rw|
\nr_{\ell_\al^2}.
\]
Now for $T_{\al,2},T_{\al,3}$ and $T_{\al,4}$, it is straight forward that 
\[
\max_{1\le i\le 2}\nl T_{\al,i}\nr_{\ell_\al^2} \lesssim o(1)\mathcal D_f(t)\mathcal E_f(t),\qquad \nl T_{\al,4}\nr_{\ell_\al^2}\lesssim o(1)\mathcal D_f(t)^2.
\]
Here, the key is that we have used the dissipation gain of scale $\eps\sqrt\kappa$ to absorb the growth in $\kappa$ for the analytic norm of $\pt_{x_3}U$, which can only grow at most $\kappa^{-\frac 12}$ for a general Prandtl flow.
Now we will bound the most difficult term, which is 
\[\nl \AA_\al^2
\sum_{\substack{\beta+\gamma=\al\\|\gamma|\ge 1
}} \lw|
\int_{\Omega} \pt_3\pt^\gamma  U_1(b_{\al,1}b_{\beta,3}+b_{\al,3}b_{\beta,1})dx
\rw|
\nr_{\ell_\al^2}.\]
\begin{remark}
We note that since $|\gamma|\ge 1$ and $U$ is a shear flow, we have $\pt^\gamma \pt_{x_3}U_1\sim \sqrt\kappa$, which is a good term. For a general flow that is not a shear flow, this is not the case. However, using the same technique as in the proof of Proposition \ref{bulk-est}, we will show that we do not shear condition for this term, by making use of the analyticity and conservation laws. This observation will be important for our future works.
\end{remark}
Now to treat the above term, we use Lemma \ref{decom-b3} and decompose
\[\bega 
\pt_3\pt^\gamma U_1\lw(b_{\al,1}b_{\beta,3}+b_{\al,3}b_{\beta,1}
\rw)&=(\pt_3 \pt^\gamma U_1) b_{\al,1}\lw\{\mathcal T^b_{\beta}-\eps U_3a_\beta-\frac \eps{2\delta}\pt^\beta\{U_3^2\}-\frac\eps\delta \pt^\beta(PU_3)
\rw\}\\
&\quad+(\pt_3 \pt^\gamma U_1) b_{\beta,1}\lw\{\mathcal T^b_{\al}-\eps U_3a_\al-\frac \eps{2\delta}\pt^\al\{U_3^2\}-\frac\eps\delta \pt^\al (PU_3)\rw\}.\\
 \enda \]
 Hence
 \[\bega 
 &\int_\Omega\pt_3\pt^\gamma U_1\lw(b_{\al,1}b_{\beta,3}+b_{\al,3}b_{\beta,1}
\rw) \\
&\lesssim \nl \pt^\gamma \pt_3 U_1\nr_{L^\infty}\nl f_\al\nr_{L^2}\lw(\eps \nl f_\beta\nr_{L^2}+\eps \delta^{-1}\nl\pt^\beta(U_3^2)
\nr_{L^2}+\eps\delta^{-1}\nl \pt^\beta(PU_2)\nr_{L^2}
\rw)\\
&\quad+\nl \pt^\gamma \pt_3 U_1\nr_{L^\infty}\nl f_\beta\nr_{L^2}\lw(\eps \nl f_\al\nr_{L^2}+\eps \delta^{-1}\nl\pt^\al(U_3^2)
\nr_{L^2}+\eps\delta^{-1}\nl \pt^\al(PU_2)\nr_{L^2}
\rw)\\
&\quad+\nl x_3\pt_3\pt^\gamma U_1\nr_{L^\infty}\lw\{\nl f_\al\nr_{L^2}\nl x_3^{-1}T_\beta^b
\nr_{L^2}+\nl f_\beta\nr_{L^2}\nl x_3^{-1}T_\al^b
\nr_{L^2}\rw\}.
\enda \]
The proof is complete.\end{proof}

\begin{theorem} \label{l2-ana}Recall \eqref{EDH}. Let $f$ be the solution to the equation \eqref{fgeq} with the boundary condition \eqref{bdr-r} and \eqref{r-def}. There exists a constant $T>0$ independent of $\kappa$ and $\eps$ such that, for all $0\le t\le T$, there holds
\[\bega
&\frac{d}{dt}\mathcal E_f(t)^2+O(1)\mathcal D(t)^2+O(1)\frac 1 \eps\sum_{\al\in \mathbb N_0^3}\AA_\al^2\nl (\II-P_{\gamma_+})\pt^\al f
\nr_{L^2_{x,\xi}}^2\\
&\lesssim\eps\delta^{-2}\sum_{\al\in \mathbb N_0^3}\AA_\al^2\nl \pt^\al r
\nr_{L^2_{\gamma_-}}^2+ \sum_{\al\in \mathbb N_0^3}\AA_\al^2\lw|\la g_\al,f_\al \ra_{L^2_{x,\xi}}\rw|+\lw\{\frac{1}{\eps\delta}\mathcal E_{\nabla_x\cdot U}(t)+1\rw\}\mathcal E_f(t).\\
\enda 
\]
Here the radius of analyticity is given by 
\[\tau(t)=\tau_0-2\tau_0 T t.\]
\end{theorem}
\begin{proof} 
Recalling \eqref{EDH}, we have
\beq\label{ran-70}
\frac 1 2 \frac{d}{dt}\mathcal E_f(t)^2=\sum_{\al\in \mathbb N_0^3} \AA_\al(t)^2
\la f_\al,\pt_t f_\al\ra_{L^2}+\tau'(t)\mathcal Y_f(t)^2.
\eeq
Here we used the fact that 
\[
\frac 1 2\frac d {dt}\AA_\al(t)^2=\tau'(t)\frac{|\al|+1}{\tau} \AA_\al(t)^2.
\]
Using the equation \eqref{f-al-eq} and \eqref{ran-70}, we have
 \[\bega 
&\frac 1 2\frac{d}{dt}\mathcal E_f(t)-\tau'(t)\mathcal Y_f(t)^2 \\
&=-\sum_{\al\in \mathbb N_0^3}\AA_\al(t)^2 \frac 1 \eps \la \xi\cdot\nabla_x f_\al,f_\al\ra_{L^2_{x,\xi}}-\sum_{\al\in \mathbb N_0^3}\AA_\al(t)^2 \frac 1 {\kappa\e^2}\la Lf_\al,f_\al\ra_{L^2_{x,\xi}}\\
&\quad-\sum_{\al\in \mathbb N_0^3}\AA_\al(t)^2 \int_{\Omega\times \R^3}\frac{(\pt_t+\frac 1 \eps \xi\cdot\nabla_x)\sqrt\mu}{\sqrt\mu}|f_\al|^2dxd\xi\\
&\quad+\frac 1 {\kappa\eps^2}\sum_{\al\in \mathbb N_0^3}\AA_\al(t)^2 \la [L,\pt^\al]f,f_\al
\ra_{L^2_{x,\xi}}+\sum_{\al\in \mathbb N_0^3}\AA_\al(t)^2\la g_\al,f_\al \ra_{L^2_{x,\xi}}\\
&\quad+\sum_{\al\in \mathbb N_0^3}\AA_\al(t)^2 \lw\la \frac{(\pt_t+\frac 1 \eps \xi\cdot\nabla_x)\sqrt\mu}{\sqrt\mu}f_\al-\pt^\al\lw\{\frac{(\pt_t+\frac 1 \eps \xi\cdot\nabla_x)\sqrt\mu}{\sqrt\mu}f\rw\},f_\al\rw\ra.\\
\enda
\]
It suffices to bound the boundary term, all the other terms are treated in Propositions \ref{bulk-com}, \ref{com-est}, Corollary \ref{bulk-est1}.
We have 
\[\bega 
&-\frac 1 \eps \la \xi\cdot\nabla_x f_\al,f_\al\ra_{L^2_{x,\xi}}\\
&=\frac 1 {2\eps}\int_{\pt\Omega\times\R^3}|f_\al|^2 d\gamma =\frac 1{2\e}\int_{\gamma_+}|f_\al|^2d\gamma-\frac 1 {2\e}\int_{\gamma_-}|f_\al |^2d\gamma\\
&=\frac 1{2\e}\int_{\gamma_+}|f_\al|^2d\gamma-\frac 1{2\e}\int_{\gamma_-}\lw|\P_{\gamma_+}f_\al -\frac{\e}{\delta}\pt^\al r
\rw|^2d\gamma\\
&=\frac 1 {2\e}\int_{\gamma_+}\lw(|f_\al |^2-|\P_{\gamma_+}f_\al |^2
\rw)d\gamma-\frac{\eps}{2\delta^2}\|\pt^\al r\|_{L^2_{\gamma_-}}^2+\frac{1}{\delta}\int_{\gamma_-}\lw(\P_{\gamma_+}f_\al \rw) \pt^\al rd\gamma\\
&=\frac 1 {2\e}\int_{\gamma_+}\lw|(\II-\P_{\gamma_+})f_\al \rw|^2
d\gamma-\frac{\eps}{2\delta^2}\|\pt^\al r\|_{L^2_{\gamma_-}}^2+\frac{1}{\delta}\int_{\gamma_-}\lw(\P_{\gamma_+}f_\al \rw) \pt^\al rd\gamma.\enda 
\]
Now we will show that the third term in the above is zero:
\[
 \int_{\gamma_-}\P_{\gamma_+}f_\al \cdot \pt^\al r d\gamma=0.
\]
To this end, we note that
\[
\int_{\gamma_-}\P_{\gamma_+}f_\al \cdot \pt^\al  r d\gamma=\int_{\xi_3>0}z_{f_\al}(x)c_\mu \sqrt{\mu_0(\xi)}\cdot (\pt^\al \ff(t,x,\xi)-z_{\pt^\al \ff}(x)c_\mu\sqrt{\mu_0(\xi)}) |\xi_3| dS(x)d\xi,
\]
where 
\[
z_{f_\al}(x)=\int_{\xi_3<0}f_\al (t,x,\xi)\sqrt{\mu_0(\xi)}|\xi_3|d\xi,\quad z_{\pt^\al \ff}(x)=\int_{\xi_3<0}\pt^\al \ff(t,x,\xi)\sqrt{\mu_0(\xi)}|\xi_3|d\xi.\]
Hence
\beq\label{ran-71}
\bega
&\int_{\gamma_-}\P_{\gamma_+}f_\al \cdot \pt^\al r d\gamma\\
&=\int_{\pt\Omega}c_\mu z_{f_\al} (x)\lw\{ 
\int_{\xi_3>0} \sqrt{\mu_0(\xi)}\pt^\al \ff(t,x,\xi) |\xi_3|d\xi-z_{\pt^\al \ff}(x)c_\mu\int_{\xi_3>0}\mu_0(\xi)|\xi_3|d\xi\rw\} dS(x).
\enda 
\eeq
Since $\pt^\al \ff(t,x,\xi)|_{x_3=0}=\pt^\al P|_{x_3=0}\sqrt\mu_0(\xi)-\kappa\pt^\al \pt_i u_j A_{ij}(\xi).$
We note that $A_{i3}(\xi)$ is odd in $\xi_1,\xi_2$ for $i\neq 3$ and $A_{33}(\xi)$ is even in $\xi_3$. Therefore 
\[
z_{\pt^\al \ff}(x)=\int_{\xi_3<0}\pt^\al \ff (t,x,\xi)\sqrt{\mu_0(\xi)}|\xi_3|d\xi=\int_{\xi_3>0}\pt^\al \ff (t,x,\xi)\sqrt{\mu_0(\xi)}|\xi_3|d\xi.
\]
Hence from \eqref{ran-71} and above identity, we get
\[
\int_{\gamma_-}\lw(\P_{\gamma_+}f_\al\rw) \pt^\al rd\gamma=0.
\]
In the end, we have
\beq\label{ran-72}\bega
-\sum_{\al\in \mathbb N_0^3}\AA_\al(t)^2 \frac 1 \eps \la \xi\cdot\nabla_x f_\al,f_\al\ra_{L^2_{x,\xi}}&\ge \frac 1 {2\e}\sum_{\al\in \mathbb N_0^3}\AA_\al(t)^2
\int_{\gamma_+}\lw|(\II-\P_{\gamma_+})f_\al \rw|^2
d\gamma\\
&\quad -\frac{\e}{2\delta^2}\sum_{\al\in \mathbb N_0^3}\AA_\al(t)^2\|\pt^\al r\|_{L^2_{\gamma_-}}^2.\enda
\eeq
The proof is complete.
\end{proof}
\subsection{The $L^2$ nonlinear energy estimate
}
In this section, we estimate the term $ \sum_{\al\in \mathbb N_0^3}\AA_\al^2\lw|\la g_\al,f_\al \ra_{L^2_{x,\xi}}\rw|$ appearing in Theorem \ref{l2-ana}. We estimate the term $-\frac 1{\eps\delta\sqrt\mu}\textbf{R}_a$ in Proposition \ref{Ra-term}, the term $\frac 2 {\kappa}\Gamma(\ff,f)$ in Proposition \ref{est-Gammaf2f} 
nonlinear term in Proposition \ref{G-ff}.

     \begin{proposition}\label{G-ff}[Estimate of the nonlinear term] Recalling the norms  in \eqref{EDH}, 
     there holds 
     \beq\bega 
&\frac{\delta}{\kappa\eps}\int_0^t \sum_\al \AA_\al(t)^2 \la\pt^\al \Gamma(f,f),\pt^\al f\ra_{L^2_{x,\xi}}ds\\
&\lesssim \delta\kappa^{-\frac 12}\lw(\nl \mathcal D_f\nr_{L^2_t}+\mathcal E_f(t)\rw) \nl \nl \AA_\al \P f_\al
\nr_{L^\infty_tL^6_x}
\nr_{\ell_\al^{2}} \nl \nl \AA_\al \P f_\al
\nr_{L^2_t L^3_x} 
\nr_{\ell_\al^2}\\
&\quad+\delta \eps\nl \mathcal D_f\nr_{L^2_t}^2\sup_{0\le s\le t} \mathcal H_{\rho,f}(s).\enda
          \eeq
     \end{proposition}
     \begin{proof}     
Recalling the definition of $\Gamma(f,f)$, we have 
   \[\bega 
 \pt^\al \Gamma(f,f)
 &=\sum_{\beta+\gamma=\al}\frac{\al!}{\beta!\gamma!}\Gamma(f_\beta, f_\gamma)\\
 &\quad+\sum_{\substack{\beta+\gamma+\vr=\al\\|\beta|\ge 1}}\frac{\al!}{\beta!\gamma!\vr!}\int_{\R^3\times\S^2}|(\xi-\xi_\star)\cdot\w|\pt^\beta \sqrt\mu(\xi_\star)\lw( f_\gamma(\xi')f_\vr(\xi_\star')+f_\gamma(\xi') f_\vr(\xi_\star')
 \rw)d\w d\xi_\star\\
 &\quad+\sum_{\substack{\beta+\gamma+\vr=\al\\|\beta|\ge 1}}\frac{\al!}{\beta!\gamma!\vr!}\int_{\R^3\times\S^2}|(\xi-\xi_\star)\cdot\w|\pt^\beta \sqrt\mu(\xi_\star)\lw( f_\gamma (\xi)f_\vr(\xi_\star)+f_\gamma(\xi)f_\vr(\xi_\star)
 \rw)d\w d\xi_\star.\\
   \enda 
 \]
 This implies
\[\bega
&\la \pt^\al \Gamma(f,f),f_\al\ra_{L^2_{t,x,\xi}}\\
&\lesssim \sum_{\beta+\gamma=\al}\frac{\al!}{\beta!\gamma!}\nl \P f_\beta
 \nr_{L^\infty_tL^6_x}\nl \P f_\gamma\nr_{L^2_tL^3_x}\nl \sqrt\nu \IP f_\al\nr_{L^2_{t,x,\xi}}\\
 &+\nl \sqrt\nu \IP f_\al\nr_{L^2_{t,x,\xi}}\sum_{\beta+\gamma=\al}\frac{\al!}{\beta!\gamma!}\nl\sqrt\nu \IP f_\beta
 \nr_{L^2_{t,x,\xi}}\nl e^{\rho|\xi|^2}f_\gamma\nr_{L^\infty_{t,x,\xi}}
 \\
 &+ \nl f_\al\nr_{L^2}\sum_{\substack{\beta+\gamma+\vr=\al\\|\beta|\ge 1}}\frac{\al!}{\beta!\gamma!\vr!}\nl e^{o(1)|\xi|^2}\pt^\beta\sqrt\mu\nr_{L^\infty_{t,x,\xi}} \nl \P f_\gamma 
 \nr_ {L^\infty_t L^6_x}\nl \P f_\vr\nr_{L^2_t L^3_x}\\
 &+\nl \sqrt\nu \IP f_\al
 \nr_{L^2_{t,x,\xi}}\sum_{\substack{\beta+\gamma+\vr=\al\\|\beta|\ge 1}}\frac{\al!}{\beta!\gamma!\vr!}\nl e^{o(1)|\xi|^2}\pt^\beta\sqrt\mu\nr_{L^\infty_{t,x,\xi}} \nl\sqrt\nu \IP f_\gamma\nr_{L^2_{t,x,\xi}}\nl e^{\rho|\xi|^2}f_\vr\nr_{L^\infty_{t,x,\xi}}.
   \enda 
 \]
 Using the discrete Young inequality, the estimate of the local Maxwellian in Lemma \ref{local-1} and the fact that
    \beq\label{algebra-est}
   \begin{cases}
   \frac{\al!\AA_\al}{\beta!\AA_\beta\gamma!\AA_\gamma}&\lesssim \la \gamma\ra^{-9}1_{\{|\beta|\ge  |\gamma|\}}+\la \beta\ra^{-9}1_{\{|\gamma|\ge |\beta|\}},\\
   \frac{\al!\AA_\al}{\beta!\AA_\beta\gamma!\AA_\gamma \vr!\AA_\vr}&\lesssim \la \gamma\ra^{-9}1_{\{|\vr|\ge |\gamma|
   \}}+\la \vr\ra^{-9}1_{\{|\gamma|\ge |\vr|
   \}
   },
   \end{cases}   \eeq
 we obtain 
  \[\bega 
&\nl  \la \AA_\al^2(t)\pt^\al \Gamma(f,f),f_\al\ra_{L^2_{t,x,\xi}}\nr_{\ell_\al^1}\\
&\lesssim \kappa^{\frac 12}\eps \mathcal D_f(t)\nl \nl \AA_\al \P f_\al
\nr_{L^\infty_tL^6_x}
\nr_{\ell_\al^2} \nl \nl \AA_\al \P f_\al
\nr_{L^2_t L^3_x} 
\nr_{\ell_\al^2}+\eps^2\kappa\mathcal D_f(t)^2 \mathcal H_{\rho,f}(t)\\
&\quad+\eps \kappa^{\frac12} \mathcal E_f(t)\nl \nl \AA_\al \P f_\al
\nr_{L^\infty_tL^6_x}
\nr_{\ell_\al^2} \nl \nl \AA_\al \P f_\al
\nr_{L^2_t L^3_x} 
\nr_{\ell_\al^2}.
\enda
 \] The proof is complete.
    \end{proof}

     Next we bound the error term in both the $L^2$ analytic and $L^\infty$ analytic estimates. To simplify the statement of the theorem, we define following functions: 
     \[\bega
    & d_D(t)=\frac{\eps^2\kappa^{-\frac 12}}{\delta} 
\lw(\||P
\||_{\tau_0,\infty}+\kappa \||\nabla_x U
\||_{\tau_0,\infty}
\rw) ^2+\frac{\eps\kappa^{\frac 32}}{\delta}
\||\nabla_x^2 U
\||_{\tau_0,\infty},\\
&d_{E}(t)=\frac{\kappa\e^3}{\delta}\nl \NS(U,P)
\nr_{\tau_0,2}+\frac{1}{\eps\delta}\nl \div U
\nr_{\tau_0,2}+\frac{\kappa\eps}{\delta} \|| \nabla_x U
\||_{\tau_0,2}\\
&+\frac{\eps}{\delta}\lw\{\||\pt_t P+U\cdot\nabla_x P\||_{\tau_0,\infty}+\eps \|| P(\pt_t U+U\cdot\nabla_x U)\||_{\tau_0,\infty}+\|| P\pt_i U_j \||_{\tau_0,\infty}\rw\}\\
&+\frac{\eps\kappa}{\delta}\lw(\||  (\pt_t+U\cdot\nabla_x)(\pt_i U_j)\||_{\tau_0,2}+\eps\||  (\pt_t U+U\cdot\nabla_x U)(\pt_i U_j)\||_{\tau_0,2}+\|| \nabla_x U
\||_{\tau_0,\infty}\|| \nabla_x U
\||_{\tau_0,2}\rw).
\enda
     \]

\begin{proposition} [Estimate of the forcing term]  \label{Ra-term} Recalling the formula of $\textbf{R}_a$  in \eqref{R-a-def}, as long as $|T_M^{-1}-1|\ll 1$, we have 
\[\bega 
&\sum_{\al\in\mathbb N_0^3}\int_0^t\AA_\al^2(t)\lw
\la \pt^\al \lw(\frac{1}{\eps\delta\sqrt\mu}\textbf{R}_a
\rw),f_\al
\rw\ra_{L^2_{x,\xi}}ds\lesssim \nl \mathcal E_f\nr_{L^2_t}^2+\nl\mathcal D_f\nr_{L^2_t}^2,
\enda 
\]
under the assumption
\[
d_E(t)+d_D(t)\lesssim 1.
\]
\end{proposition}
\begin{proof}
Now we bound $L^2$ norm. From \eqref{R-a-def}, we have  \beq\label{ran-765}
\bega 
\frac {1}{\eps\delta\sqrt\mu}\textbf{R}_a&=\frac{1}{\delta} \NS(U,P)\cdot\vp \sqrt\mu-\frac{\eps}{\kappa\delta}\Gamma(\ff,\ff)+\frac 1 {3\eps\delta} \div U |\vp|^2\sqrt \mu-\frac \kappa \delta \pt_{i\ell}U_j\IP (\vp_\ell A_{ij})\\
&\quad+\frac \eps \delta \sqrt\mu\lw\{(\pt_t P+U\cdot\nabla_x P)+\eps P(\pt_t U+U\cdot\nabla_x U)\cdot\vp+\frac 1 2 P \pt_i U_\ell \vp_i\vp_\ell
\rw\} \\ 
&\quad-\frac{ \eps \kappa}{\delta} \lw\{(\pt_t+U\cdot\nabla_x)(\pt_i U_j)+\eps (\pt_tU+U\cdot\nabla_x U)\cdot \vp \pt_i U_j+\pt_i U_j \pt_k U_\ell \vp_k\vp_\ell
\rw\}A_{ij} \\
&\quad-\frac{\kappa\eps}{\delta}  \pt_i U_j\lw(\eps^2 \pt_t+(\eps^2 U+\vp)\cdot\nabla_x
\rw)A_{ij}.\\
\enda
\eeq
First we see that
\[\bega 
&\frac{1}{\delta}\sum_{\al}\AA_\al^2(t)\la \pt^\al\lw\{\NS(U,P)\cdot\vp \sqrt\mu\rw\},f_\al\ra_{L^2_{x,\xi}}\\
&\lesssim \frac{1}{\delta} \lw\{\sum_{\al}\AA_\al^2(t)\nl \NS(U,P)\cdot\vp \sqrt\mu
\nr_{L^2_{x,\xi}}^2\rw\}^{\frac12}\mathcal E_f(t).
\enda 
\]
Using Lemma \ref{local-1} for the analytic pointwise bound of $\mu$, we get
\[
\sum_{\al\in\mathbb N_0^3}\frac{\tau_0^{|\al|}}{\al!}\nl
\pt^\al\lw\{ \NS(U,P)\cdot\vp \sqrt\mu \rw\}
\nr_{L^2_x}
\lesssim  \nl \NS(U,P)
\nr_{\tau_0,2} e^{-\frac 1 {10} |\xi|^2}.
\]
Combining the above two inequalities, we get
\[
\frac 1 \delta \sum_\al\AA_\al^2 \la \pt^\al\lw\{\NS(U,P)\cdot\vp \sqrt\mu\rw\},f_\al\ra_{L^2_{x,\xi}}\lesssim \frac 1 \delta
 \nl \NS(U,P)
\nr_{\tau_0,2}\mathcal E_f(t).\]
Next we bound the second term coming from \eqref{ran-765}. Using the trilinear estimate in Proposition \ref{bilinear-1} and then Lemma \ref{local-1}, we have 
\[\bega
&\frac{\eps}{\kappa\delta}\int_0^t\sum_\al\AA_\al^2 \la\pt^\al \Gamma(\ff,\ff),f_\al\ra_{L^2_{x,\xi}}ds\\
&\lesssim \eps^2 \delta^{-1}\kappa^{-\frac 12} \sum_\al \frac{\tau_0^{2|\al|}}{\al!^2}\nl \nu \pt^\al \ff\nr_{L^\infty_xL^2_\xi}^2\lw(\nl \mathcal D_f\nr_{L^2_t}+\nl \mathcal E_f\nr_{L^2_t}\rw)\\
&\lesssim \eps^2\delta^{-1}\kappa^{-\frac 12}\lw(\||P
\||_{\tau_0,\infty}+\kappa \||\nabla_x U
\||_{\tau_0,\infty}
\rw) ^2\lw(\nl \mathcal D_f\nr_{L^2_t}+\nl \mathcal E_f\nr_{L^2_t}\rw).\enda
\]
Similarly, we get 
\[
\frac 1 {3\eps\delta} \int_0^t\sum_{\al}\AA_\al^2
\la
\pt^\al \lw\{
\div U |\vp|^2\sqrt \mu
\rw\}
,f_\al
\ra_{L^2_{x,\xi}}ds\lesssim \frac{1}{\eps\delta}\nl \div U
\nr_{\tau_0,2}\nl \mathcal E_f\nr_{L^2_t}.
\]
Next we bound 
\[
\frac \kappa \delta   \sum_{\al}\AA_\al^2
\la
\pt^\al \lw\{
\pt_{i\ell}U_j\IP (\vp_\ell A_{ij})
\rw\}
,f_\al
\ra_{L^2_{x,\xi}}.
\]
Applying Proposition \ref{VIPG} to the above term and then use  Corollary \ref{corol-Aij}, we get
\[\bega
\frac \kappa \delta   \sum_{\al}\AA_\al^2
\la
\pt^\al \lw\{
\pt_{i\ell}U_j\IP (\vp_\ell A_{ij})
\rw\}
,f_\al
\ra_{L^2_{x,\xi}}\lesssim \frac{\eps\kappa^{\frac 32}}{\delta}
\||\nabla_x^2 U
\||_{\tau_0,\infty}\lw(\mathcal E_f(t)+\mathcal D_f(t)
\rw).\enda 
\]
Finally, the last three terms can be bounded by the same argument as in the above terms. We skip the details. The proof is complete.

\end{proof}

\begin{proposition}\label{est-Gammaf2f} There holds 
\[
\int_0^t \sum_{\al\in\mathbb N_0^3}\AA_\al^2
\lw
\la \pt^\al
\lw(\frac 2 \kappa \Gamma(\ff,f)
\rw) ,f_\al
\rw\ra_{L^2_{x,\xi}}ds \lesssim o(1)\mathcal D_f(t)^2+\nl \mathcal E_f\nr_{L^2_t}^2+\frac{\eps^2}{\kappa}\mathcal H_{\rho,f}(t)^2.\]
\end{proposition}
\begin{proof}Using Proposition \ref{bilinear-1}, discrete Young inequality and then Lemma \ref{local-1}, we have 
\[\bega&\sum_{\al\in\mathbb N_0^3}\AA_\al^2
\lw
\la \pt^\al
\lw(\frac 2 \kappa \Gamma(\ff,f)
\rw) ,f_\al
\rw\ra_{L^2_{x,\xi}}\\
&\lesssim \frac 1 \kappa \sup_{\al}\frac{\tau_0^{|\al|}}{\al!}\nl \nu\pt^\al \ff
\nr_{L^\infty_xL^2_\xi} \cdot \lw\{\sqrt\kappa\eps \mathcal D_f(t)\mathcal E_f(t)+\kappa\eps^2 \mathcal D_f(t)^2+\eps\kappa^{\frac 12}\mathcal E_f(t)^2+\eps\kappa^{\frac 12}\mathcal H_{\rho,f}(t)\mathcal E_f(t)
\rw\}.\\
\enda
\]
It suffices to bound the terms involving $\ff$ in the above. We note that
\[\bega
&\frac{\tau_0^{|\al|}}{\al!}\nl \nu \pt^\al \ff\nr_{L^\infty_x L^2_\xi}\\
&\lesssim \sum_{\beta+\gamma=\al}\lw(\frac{\tau_0^{|\beta|}}{\beta!}\nl \pt^\beta P\nr_{L^\infty}\frac{\tau_0^{|\gamma|}}{\gamma!}\nl\nu \pt^\gamma\sqrt\mu\nr_{L^\infty_xL^2_\xi}+\frac{\tau_0^{|\beta|}}{\beta!}\kappa\nl \pt^\beta\nabla_x U
\nr_{L^\infty}\frac{\tau_0^{|\gamma|}}{\gamma!}\nl\nu\pt^\gamma A_{ij}
\nr_{L^\infty_x L^2_\xi}
\rw)\\
&\lesssim 1
\enda 
\]
by Proposition \ref{corol-Aij}.
The proof is complete.
   \end{proof}
 Combining the above bounds on the forcing term and  Theorem \ref{l2-ana}, we get
   
   \[\bega 
   &\mathcal E_f(t)^2+O(1)\int_0^t\mathcal D_f(s)^2ds +O(1)\eps^{-1}\nl \nl \AA_\al (\II-\P_{\gamma_+})f_\al
   \nr _{L^2_{t,x,\xi}}\nr_{\ell_\al^2}^2\\
   &\lesssim \sum_{\al}\int_0^t \AA_\al^2 \nl\pt^\al r\nr_{L^2_{\gamma_-}}^2+\lw\{\frac 1{\eps\delta}\nl|\nabla_x\cdot U\nr|_{\tau_0,2}+\delta^{-1}\||\NS(U)\||_{\tau_0,2}\rw\}\mathcal E_f(t) \\
   &\quad+\eps\kappa^{\frac 12}\lw(\nl \mathcal D_f\nr_{L^2_t}+\mathcal E_f(t)\rw) \nl \nl \AA_\al \P f_\al
\nr_{L^\infty_tL^6_x}
\nr_{\ell_\al^2} \nl \nl \AA_\al \P f_\al
\nr_{L^2_t L^3_x} 
\nr_{\ell_\al^2}+\eps^2\kappa\nl \mathcal D_f\nr_{L^2_t}^2 \mathcal H_{\rho,f}(t) .  \enda 
   \]
 \subsection{$L^\infty$ estimates}
In this section, we derive the $L^\infty$ for $f$ solving \eqref{f-eq}, or $F$ solving \eqref{F-eq}. 
\[
f=\frac{F-\mu-\eps^2 \ff\sqrt\mu}{\eps\delta\sqrt\mu}=\frac{\sqrt{\mu_M}}{\sqrt\mu} h.
\]
We now focus on bounding the analytic $L^\infty$ norm of $h$. We define
\beq\label{Hh-norm}
\mathcal H_{h}(t)=\lw\{\sum_{\al\in\mathbb N_0^3}\AA_\al^2(t)\nl \pt^\al h\nr_{L^\infty_{t,x,\xi}}^2\rw\}^{\frac12}.
\eeq
We note that 
\beq\label{ran-76}
\begin{cases}
\frac{\sqrt{\mu_M}}{\sqrt\mu}&= T_M^{-\frac32}e^{-\frac{|\xi|^2}{4T_M}+\frac{|\xi|^2}{4}-\frac 1 2\eps U\cdot\xi +\frac 1 4 \eps^2 |U|^2}=T_M^{-\frac 3 2} e^{-\frac{|\xi|^2}{4}(T_M^{-1}-1)}e^{\frac 1 4\eps^2|U|^2}e^{-\frac 1 2\eps U\cdot\xi}
\\
\frac{\sqrt{\mu_M}}{\sqrt\mu_0}&=T_M^{-\frac 3 2}e^{-\frac{|\xi|^2}{4}(T_M^{-1}-1)}.\end{cases}
\eeq
Here we note that $T_M^{-1}>1$.
Recalling the definition of the analytic $L^\infty$ norm of $f$ in \eqref{EDH}, we show the following lemma:

\begin{lemma}\label{fh-lem} For $\rho>0$ and $T_M\in (0,1)$ such that
\[
\rho\ll |T_M^{-1}-1|,
\]
there holds 
\[
\mathcal H_{ \rho,f}(t)\lesssim \mathcal H_{h}(t).
\]
\end{lemma}
\begin{proof}
 We recall that $f=\frac{\sqrt\mu_M}{\sqrt\mu} h$. Using \eqref{ran-76}, we have, for each $\al\in\mathbb N_0^3$:
 \beq\label{ran-77}
 \bega 
e^{\rho|\xi|^2}\frac 1{\al!}\pt^\al f
&\lesssim e^{-\frac{|\xi|^2}{4}(T_M^{-1}-1-4\rho_0)}\sum_{\beta+\gamma+\vr=\al}\frac{1}{\beta!}\lw|\pt^\beta\lw\{e^{\frac 1 2 \eps U\cdot\xi}\rw\}\rw|\cdot \frac{1}{\gamma!}\lw|\pt^\gamma\lw(e^{\frac 1 4 \eps^2|U|^2}
\rw) \rw|\cdot \frac{1}{\vr!}\lw|\pt^{\vr}h\rw|.\\
\enda \eeq
Using Lemma \ref{exp-ana} we get
\beq\label{ran-77-2}
\sup_{\beta\in\mathbb N_0^3}\frac{\tau^{|\beta|}}{\beta!}\lw|\pt^\beta\lw\{e^{\frac 1 2 \eps U_\ell\xi_\ell}\rw\}\rw|\lesssim e^{\frac 1 2 \eps U_\ell\xi_\ell}+\eps\kappa^{\frac 12}e^{\frac 1 2 \eps U_\ell\xi_\ell+4\eps\kappa^{\frac12}|\xi_\ell|}\lesssim e^{\eps |U|_{L^\infty}|\xi|}.
\eeq
Combining \eqref{ran-77}, \eqref{ran-77-2} and the fact that $e^{\frac 1 4\eps^2|U|^2}$ is analytic in $x_\parallel$, we get the result, as long as \[
\rho+\eps |U|_{L^\infty}\ll |T_M^{-1}-1|.
\]
The proof is complete.
\end{proof}
We recall the equation for $h_\al=\pt^\al h$ in \eqref{h-al-eq}. We define the backward exit time and position as
\[
\tb(x,\xi)=\sup\lw\{\tau\ge 0:x-\tau \frac{\xi}\eps\in \bar \Omega\rw\},\qquad \xb(x,\xi)=x-\frac{\tb(x,\xi)}{\eps}\xi.\]
When $\Omega=\mathbb T^2\times\R_+$, one gets the explicit formula
\[
\tb(x,\xi)=\eps\frac{x_3}{\xi_3},\qquad \xb(x,\xi)=x-\frac{x_3}{\xi_3}\xi. \]
We also define the following notations
\[\begin{cases}
-\int_s^t \nu_M&=-\int_s^t 
\nu_M(\tau,x-\frac{t-\tau}{\e}\xi,\xi)d\tau, \\
-\int_s^{t-\tb}\nu_M&=-\int_s^{t-\tb}\nu_M(\tau,\xb-\frac{t-\tb-\tau}{\e}\xi_\star,\xi_\star)d\tau.\end{cases}
\]
In the next proposition, we give a complete trajectory presentation for $h_\al$ solving the equation \eqref{h-al-eq} with the boundary condition \eqref{h-al-bdr}.
To this end, we define the following operators
\beq\label{i-0}
\mathcal I(h_{\al,0})=\mathcal I_1(h_{\al,0})+\mathcal I_2(h_{\al,0}),
\eeq
where
\beq\label{i-1}
\mathcal I_1(h_{\al,0})=1_{\{t\le \tb
\}
}
e^{-\frac 1 {\kappa\eps^2}\int_0^t 
\nu_M
}h_{\al,0}\lw(x-\frac t \e \xi,\xi
\rw),
\eeq
\beq\label{i-2}
\bega 
\mathcal I_2(h_{\al,0})&=1_{\{t\ge \tb
\}
}e^{-\frac 1 {\kappa\eps^2}\int_{t-\tb}^t \nu_M
}w_1(\xi)\\
&\times\int_{\R^3}e^{-\frac 1 {\kappa\eps^2}
\int_0^{t-\tb}
\nu_M 
} h_{\al,0}\lw(\xb-\frac{t-\tb}{\e}\xi_\star,\xi_\star\rw)w_2(\xi_\star)d\xi_\star. \enda 
\eeq
The operator $\mathcal M(h_\al)$ is defined as 
\beq\label{M0}
\mM (h_\al)=\mM_1(h_\al)+\mM_2(h_\al)
\eeq
where
\beq \label{M1}
\mM_1(h_\al)=\frac{1}{\kappa\eps^2}\int_{\max\{0,t-\tb
\}
}^t e^{-\frac 1 {\kappa\eps^2}\int_s^t 
\nu_M
}
K_M h_\al\lw(s,x-\frac{t-s}{\e}\xi,\xi
\rw)ds,
\eeq 
\beq\label{M2}
\bega 
\mM_2(h_\al)&=1_{\{t\ge \tb
\}
}e^{-\frac 1 {\kappa\eps^2}\int_{t-\tb}^t\nu_M 
}w_1(\xi)\\
&\quad\times \frac 1{\kappa\eps^2}\int_{\R^3}\int_0^{t-\tb}
e^{-\frac 1 {\kappa\eps^2}\int_s^{t-\tb}\nu_M
} K_{M}h_\al\lw(s,\xb-\frac{t-\tb-s}{\e}\xi_\star,\xi_\star
\rw)w_2(\xi_\star) ds d\xi_\star.
\enda
\eeq
The operator $\mF$ is defined as
\beq\label{F0}
\mathcal F(\wtd g_\al)=\mathcal F_1(\wtd g_\al)+\mathcal F_2(\wtd g_\al),
\eeq
where 
\beq\label{f1}
\mathcal F_1 (\wtd g_\al)=\int_{\max\{0,t-\tb
\}
}^t e^{-\frac 1 {\kappa\eps^2}\int_s^t 
\nu_M
}
 \wtd g_\al\lw(s,x-\frac{t-s}{\e}\xi,\xi
\rw)ds,
\eeq
\beq\label{f2}
\bega
&\mathcal F_2(\wtd g_\al)=1_{\{t> \tb
\}
}e^{-\frac 1 {\kappa\eps^2}\int_{t-\tb}^t \nu_M
}
w_1(\xi)\\
&\quad\times \int_{\xi_\star}
\int_0^{t-\tb}e^{-\frac 1 {\kappa\eps^2}\int_s^{t-\tb}\nu_M}\wtd g_\al\lw(s,\xb-\frac{t-\tb-s}{\e}\xi_\star,\xi_\star
\rw)w_2(\xi_\star)ds d\xi_\star.\\
\enda 
\eeq
The boundary term is given by 
\beq\label{B0}
\mB(r_\al)=1_{\{t\ge \tb 
\}
}w_3(\xi)^{-1}e^{-\frac 1 {\kappa\eps^2}\int_{t-\tb}^t \nu_M 
} \pt^\al r (t-\tb,\xb,\xi).
\eeq

\begin{proposition} \label{ran-790} Let $h_\al$ solves \eqref{h-al-eq} with the boundary condition \eqref{h-al-bdr} and initial data $h_{\al,0}$. There holds 
\[
h_\al(t,x,\xi)=\mathcal I(h_{\al,0})+\mathcal M (h_\al)+\mathcal F(\wtd g_
\al)-\frac{\eps}\delta T_M^{\frac 3 2} \mathcal B(r_\al)
\]
where $\mI,\mM,\mF$ are defined in \eqref{i-0},\eqref{M0} and \eqref{F0}.
\end{proposition}
\begin{proof} Denoting $\tb=\tb(x,\xi),\xb=\xb(x,\xi)$. Using the Duhamel principle for the equation
\[\bega
&\pt_t h_\al+\frac 1 \eps \xi\cdot\nabla_x h_\al+\frac1 {\eps^2\kappa} \nu_Mh_\al =\frac 1 {\kappa\eps^2}K_M h_\al+\wtd g_\al \\
&h_\al|_{\xi_3>0}=c_\mu w_1(\xi)\int_{\R^3}h_\al (t,x,\xi_\star)w_2(\xi_\star)d\xi_\star-\frac \eps \delta w_3(\xi)^{-1}\pt^\al r,
\enda
\]
we get
\beq\label{Duh}
\bega
h_\al(t,x,\xi)&=1_{\{t<\tb
\}
}
e^{-\frac 1 {\kappa\eps^2}\int_0^t 
\nu_M
}h_{\al,0}\lw(x-\frac t \e \xi,\xi
\rw)\\
&\quad+\frac{1}{\kappa\eps^2}\int_{\max\{0,t-\tb
\}
}^t e^{-\frac 1 {\kappa\eps^2}\int_s^t 
\nu_M 
}
K_{M}h_\al\lw(s,x-\frac{t-s}{\e}\xi,\xi
\rw)ds\\
&\quad+\int_{\max\{0,t-\tb
\}
}^t e^{-\frac 1 {\kappa\eps^2}\int_s^t 
\nu_M
}
\wtd g_\al\lw(s,x-\frac{t-s}{\e}\xi,\xi
\rw)ds\\
&\quad+1_{\{t>\tb
\}
}e^{-\frac 1 {\kappa\eps^2}\int_{t-\tb}^t \nu_M
}
h_\al\lw(t-\tb,\xb,\xi
\rw).
\enda
\eeq
From the equation \eqref{Duh}, we obtain the term $\mI_{1}$ and $\mF_1$ directly. We now write  the boundary term on the forth line in \eqref{Duh}, using the boundary condition \eqref{h-al-bdr}. We have 
\beq\label{h-al-bdr1h-al-bdr}
\bega
&h_\al(t-\tb,\xb,\xi)=w_1(\xi)\int_{\xi_{\star}}h_\al(t-\tb,\xb,\xi_\star)w_2(\xi_\star)d\xi_\star-\frac \eps\delta w_3(\xi)^{-1} \pt^\al r(t-\tb,\xb,\xi).
\enda 
\eeq
Now using \eqref{Duh},  we have 
\[
\bega 
h_\al(t-\tb,\xb,\xi_\star)&=
e^{-\frac 1 {\kappa\eps^2}
\int_0^{t-\tb}
\nu_M
}
h_{\al,0}\lw(\xb-\frac{t-\tb}{\e}\xi_\star,\xi_\star\rw)\\
&\quad+\frac 1{\kappa\eps^2}\int_0^{t-\tb}
e^{-\frac 1 {\kappa\eps^2}\int_s^{t-\tb}\nu_M
}K_{M}h_\al
\lw(s,\xb-\frac{t-\tb-s}{\e}\xi_\star,\xi_\star
\rw)ds\\
&\quad+\int_0^{t-\tb}e^{-\frac 1 {\kappa\eps^2}\int_s^{t-\tb}\nu_M}\wtd g_\al\lw(s,\xb-\frac{t-\tb-s}{\e}\xi_\star,\xi_\star
\rw)ds.\\
\enda 
\]
Plugging the above into the last line of the equation \eqref{Duh}, we get 
the term $\mI_2$ and $\mF_2$ in the decomposition of $\mI$ and $\mF$.
The third term in the above corresponds to $\mM_2$ in the decomposition of $\mM$. The proof is complete.
\end{proof}

\begin{proposition}\label{ran-801} Recall $\mM_1,\mM_2$ in \eqref{M1},  \eqref{M2}, and $\mF,\mF_1$, $\mF_2$ in \eqref{F0}, \eqref{f1}, \eqref{f2}. 
There hold
\[\bega 
|\mM_1(t)|+|\mM_2(t)|&\lesssim o(1)\sup_{0\le s\le t} \nl h_\al(s)
\nr _{L^\infty_{x,\xi}}+\nl h_\al|_{t=0}
\nr _{L^\infty_{x,\xi}}\\
&\quad+(\eps\kappa)^{-\frac 12} \sup_{0\le s\le t}\nl \P f_\al(s)
     \nr_{L^6_x}+(\eps\kappa)^{-\frac 3 2}\sup_{0\le s\le t} \nl \IP f_\al
     \nr_{L^2_{x,\xi}}\\
     &\quad+\sum_{\substack{\beta+\gamma=\al\\|\beta|\ge 1
}}\frac{\al!}{\beta!\gamma!}\sup_{|\xi_{\star\star}|\lesssim 1}\nl \pt^\beta\lw\{e^{-\frac 1 4 \eps^2 |U|^2+\frac 1 2 \eps U\cdot\xi_{\star\star}}
\rw\}\nr _{L^\infty_x}\sup_{0\le s\le t} \nl h_\gamma( s)
\nr_{L^\infty_{x,\xi}}\\
&\quad+\la\xi\ra^{-1}  \kappa\e^2 \sup_{0\le s\le t}\nl \la\xi\ra^{-1}\wtd g_\al(s)
\nr _{L^\infty_{x,\xi}}+\la\xi\ra^{-1} \frac\eps\delta 
\nl w_3^{-1}\pt^\al r
\nr_{L^\infty},
   \enda 
\]
and 
\[
\mF(t)\lesssim \kappa\eps^2\la\xi\ra^{-1}\sup_{0\le s\le t}\nl \la\xi\ra^{-1}\wtd g_\al(s)
\nr _{L^\infty_{x,\xi}}.
\]
\end{proposition}
\begin{proof} Using Lemma \ref{LM-lem}, we get 
\beq\label{ran-79}\bega 
\mM_1&\lesssim \eps_0 \nl h_\al\nr_{L^\infty_{t,x,\xi}}+ C_{\eps_0}\frac 1{\kappa\eps^2}\int_{\max\{0,t-\tb\}}^t  e^{-c_0\la\xi\ra\frac{t-s}{\kappa\eps^2}}\int_{\xi_\star}k_{\rho_0}(\xi,\xi_\star)
|h_\al| \lw(s,\wtd x,\xi_\star
\rw)d\xi_\star ds\\
\enda
\eeq
where $\wtd x=x-\frac{t-s}{\eps}\xi$. 
Now using \eqref{Duh} again,  we have 
\beq\label{ran-261}
\bega 
h_\al\lw(s,\wtd x,\xi_\star\rw)
&=1_{\lw\{s<\tb\lw(\wtd x,\xi_\star\rw)\rw\}}e^{-\int_0^s \frac 1 {\kappa\eps^2}\nu_M} h_{\al,0}\lw(\wtd x-\frac{s}{\eps}\xi_\star,\xi_\star\rw)\\
&\quad+\frac{1}{\kappa\eps^2}\int_{\max\{0,s-\tb (\wtd x,\xi_\star)\}}^s e^{-\frac 1{\kappa\eps^2}\int_{\wtd s}^s \nu_M
}K_{M}h_\al\lw(\wtd s,\wtd x-\frac{s-\wtd s}{\e}\xi_\star,\xi_\star\rw)d\wtd s\\
&\quad+\int_{\max\{0,s-\tb(\wtd x,\xi_\star
)
\}
}^s
e^{-\frac{1}{\kappa\eps^2}\int_{\wtd s}^s \nu_M
}\wtd g_\al\lw(\wtd s,\wtd x-\frac{s-\wtd s}{\e}\xi_\star,\xi_\star
\rw)d\wtd s\\
&\quad+1_{\{s\ge \tb(\wtd x,\xi_\star)
\}
}
e^{-\frac{1}{\kappa\eps^2}
\int^s_{s-\tb\lw(\wtd x,\xi_\star\rw)} \nu_M} h_\al\lw(s-\tb\lw(\wtd x,\xi_\star
\rw),\xb\lw(\wtd x,\xi_\star\rw),\xi_\star
\rw),
\enda 
\eeq
where 
\[
-\int_{\wtd s}^s \nu_M=-\int_{\wtd s}^s \nu_M(\tau,\wtd x-\frac{s-\tau}{\eps}\xi_\star,\xi_\star)d\tau.
\]
Now we rewrite the last term in the above, using the boundary condition \eqref{h-al-bdr}.
Now we note that using the identity \eqref{h-al-bdr} for $x\to \wtd x$ and $\xi\to\xi_\star$ and $t\to s$, the forth term in the above right hand side can be written as
\beq\label{ran-7999}
\bega 
h_\al(s-\tb(\wtd x,\xi_\star),\xb(\wtd x,\xi_\star),\xi_\star)&=w_1(\xi_\star)\int_{\xi_{\star\star}}h_\al(s-\tb(\wtd x,\xi_\star),\xb(\wtd x,\xi_\star),\xi_{\star\star})w_2(\xi_{\star\star})d\xi_{\star\star}\\
&\quad- \frac \eps\delta  w_3(\xi_\star)^{-1}\pt^\al r(s-\tb(\wtd x,\xi_\star),\xb(\wtd x,\xi_\star),\xi_\star).\enda
\eeq
Combining \eqref{ran-261} and \eqref{ran-7999}, we can bound $h_\al(s, \wtd x,\xi_\star)$ as follows
\beq\label{ran-88}
\bega
|h_\al|(s, \wtd x,\xi_\star)
&\lesssim e^{-c_0\la \xi_\star \ra \frac{s}{\kappa\eps^2}}\|h_{\al,0}\|_{L^\infty_{x,\xi}}+\eps_0 \nl h_{\al,0}\nr_{L^\infty_{x,\xi}}\\
&\quad +\frac 1 {\kappa\e^2}\int_0^s  e^{-c_0\la \xi_\star\ra \frac{s-\wtd s}{\kappa\eps^2}
}\int_{\xi_{\star\star}}k_{\rho_0}(\xi_\star,\xi_{\star\star})
|h_\al|\lw(\wtd s,\wtd x-\frac{s-\wtd s}{\e}\xi_\star,\xi_{\star\star}\rw)d\xi_{\star\star} d\wtd s  \\
&\quad+\int_0^s e^{-c_0\la \xi_\star\ra\frac{s-\wtd s}{\kappa\eps^2}
}| \wtd g_\al|\lw(\wtd s,\wtd x-\frac{s-\wtd s}{\e}\xi_\star,\xi_\star\rw)d\wtd s\\
&\quad+w_1(\xi_\star)\int_{\xi_{\star\star}}|h_\al|\lw(s-\tb\lw(\wtd x,\xi_\star\rw),\xb\lw(\wtd x,\xi_\star\rw),\xi_{\star\star}\rw)w_2(\xi_{\star\star})d\xi_{\star\star}\\
&\quad+\frac \eps\delta w_3(\xi_\star)^{-1}|\pt^\al r|\lw(s-\tb\lw(\wtd x,\xi_\star\rw),\xb\lw(\wtd x,\xi_\star\rw),\xi_\star\rw).
\enda 
\eeq
Combining \eqref{ran-79} and \eqref{ran-88}, we have 
\[\bega 
&\mM_1\lesssim \sum_{i=1}^5 \mathcal S_i,
\enda 
 \]
 where 
 \beq\label{ran-617}
\bega 
\mathcal S_1&= \frac 1{\kappa\e^2}\int_0^t e^{-c_0\la \xi\ra \frac{t-s}{\kappa\eps^2}}\int_{\xi_\star}k_{\rho_0}(\xi,\xi_\star)
 e^{-c_0\la \xi_\star\ra \frac{s}{\kappa\eps^2}}\|h_{\al,0}\|_{L^\infty_{x,\xi}}d\xi_\star ds,\\
 \mathcal S_2&=\frac 1 {\kappa\eps^2}\int_0^t e^{-c_0\la \xi\ra \frac{t-s}{\kappa\eps^2}} \int_{\xi_\star}k_{\rho_0}(\xi,\xi_\star)\frac 1 {\kappa\e^2}\int_0^s  e^{-c_0\la\xi_\star\ra \frac{s-\wtd s}{\kappa\eps^2}
}\\
&\times\int_{\xi_{\star\star}}k_{\rho_0}(\xi_\star,\xi_{\star\star})|h_\al|\lw(\wtd s,\wtd x-\frac{s-\wtd s}{\e}\xi_\star,\xi_{\star\star}\rw)d\xi_{\star\star} d\wtd s d\xi_\star ds, \\
\enda\eeq
\beq\bega
\mathcal S_3&=\frac 1 {\kappa\eps^2}\int_0^t e^{-c_0\la\xi\ra \frac{t-s}{\kappa\eps^2}} \int_{\xi_\star}k_{\rho_0} (\xi,\xi_\star)\int_0^s e^{-c_0\la \xi_\star\ra \frac{s-\wtd s}{\kappa\eps^2}
}|\wtd g_\al|\lw(\wtd s,\wtd x-\frac{s-\wtd s}{\e}\xi_\star,\xi_\star\rw)d\wtd s d\xi_\star ds,\\
\mathcal S_4&=\frac 1 {\kappa\eps^2}\int_{\max\{0,t-\tb
\}
}^t e^{-c_0\la\xi\ra \frac{t-s}{\kappa\eps^2}} \int_{\xi_\star}k_{\rho_0}(\xi,\xi_\star)1_{\{s\ge \tb(\wtd x,\xi_\star)
\}
}\\
&\quad\times e^{-c_0\la \xi_\star\ra \frac{\tb(\wtd x,\xi_\star)}{\kappa\eps^2}
} \lw\{w_1(\xi_\star)\int_{\xi_{\star\star}}|h_\al|\lw(s-\tb(\wtd x,\xi_\star),\xb(\wtd x,\xi_\star),\xi_{\star\star}\rw)w_2(\xi_{\star\star})\rw\}d\xi_{\star\star} d\xi_\star ds,\\
\mathcal S_5&= \frac 1{\kappa\e^2}\int_0^t e^{-c_0\la \xi\ra \frac{t-s}{\kappa\eps^2}}\int_{\xi_\star}k_{\rho_0}(\xi,\xi_\star)\cdot \frac \eps \delta w_3(\xi_\star)^{-1}|\pt^\al r|\lw(s-\tb(\wtd x,\xi_\star),\xb(\wtd x,\xi_\star),\xi_\star\rw)d\xi_\star ds.
\enda
  \eeq
  \textbf{Treating $\mathcal S_1$.}
  We have 
  \[
  \mathcal S_1\lesssim \nl h_{\al,0}\nr_{L^\infty_{x,\xi}}\cdot  \frac 1{\kappa\e^2} \int_0^t e^{-c_0\frac{t-s}{\kappa\e^2}}ds\lesssim \nl h_{\al,0}
  \nr_{L^\infty_{x,\xi}}.
   \]
    \textbf{Treating $\mathcal S_2$.}  We have 
   \beq\label{ran-821}
    \bega
     \mathcal S_2&=\frac 1 {(\kappa\eps^2)^2}\int_0^t \int_0^s e^{
     -c_0\la \xi\ra \frac{t-s}{\kappa\eps^2}-c_0\la\xi_\star\ra \frac{s-\wtd s}{\kappa\eps^2} }\\
     &\times\int_{\R^6}k_{\rho_0}(\xi,\xi_\star) k_{\rho_0}(\xi_\star,\xi_{\star\star})|h_\al|\lw(\wtd s,\wtd x-\frac{s-\wtd s}{\e}\xi_\star,\xi_{\star\star}\rw)d\xi_{\star\star}  d\xi_\star d\wtd s ds.
    \enda
    \eeq
     We take $\delta_1>0$ to be small and $\bar N>0$ to be large. Define \[
     U_{\bar N}=\lw\{(\xi_\star,\xi_{\star\star} )\in\R^6:\quad \frac 1 {\bar N}<|\xi_\star-\xi_{\star\star} |<\bar N,\quad |\xi_\star|<\bar N\rw\}.
     \]
      We decompose the integral $ \int_0^s \int_{\xi_\star}\int_{\xi_{\star\star} } d\xi_{\star\star} d\xi_\star d\wtd s$ in $\mmA_2$ as follows
    \[\bega 
  \int_0^s \int_{\xi_\star}\int_{\xi_{\star\star} } d\xi_{\star\star}  d\xi_\star d\wtd s&=\int_{|\wtd s-s|\le \delta_1\kappa\e^2}\int_{\xi_\star}\int_{\wtd \xi_\star} d\xi_{\star\star}  d\xi_\star d\wtd s+ \int_{|\wtd s-s|\ge \delta_1\kappa\e^2}\int_{U_{\bar N}} d\xi_{\star\star} d\xi_\star d\wtd s\\
    &\quad+\int_{|\wtd s-s|\ge \delta_1\kappa\e^2}\int_{U_{\bar N}^c} d\xi_{\star\star}  d\xi_\star d\wtd s ds.       \enda
            \]
    and we have the corresponding decomposition 
    \[
  \mathcal S_2=\mathcal S_{2,1}+\mathcal S_{2,2}+\mathcal S_{2,3}.
    \]
       Using the fact that 
           \[
           \int_0^t \frac{1}{\kappa\e^2}e^{-c_0\la\xi\ra \frac{t-s}{\kappa\e^2}}ds\lesssim 1,\quad \frac 1{\kappa\e^2}\int_{|\wtd s-s|\le \delta_1\kappa\eps^2}d\wtd s\lesssim \delta_1,\quad \int_{\R^6 }k_{\rho}(\xi,\xi_\star) k_{\rho}(\xi_\star,\xi_{\star\star} ) d\xi_{\star\star}  d\xi_\star\lesssim 1,         \]
            we get the upper bound 
            \[
          \mathcal S_{2,1}\lesssim \delta_1\nl h_\al\nr_{L^\infty_{t,x,\xi}}.
            \]
 As for $\mathcal S_{2,3}$, using the fact that, as $\bar N\to\infty$: 
    \[
    \int_{|\xi_\star|\ge \bar N}k_{\rho}(\xi_\star,\xi_{\star\star} )d\xi_{\star\star} \lesssim \frac 1{\bar N},\quad \sup_{\xi_\star\in \R^3}\int_{\{|\xi_\star-\xi_{\star\star} |\le \frac 1{\bar N}\}\cup\{|\xi_\star-\xi_{\star\star} |\ge \bar N
    \}}k_{\rho}(\xi_\star,\xi_{\star\star} )d\xi_{\star\star} \to 0,    \]
    We have
    \[
 \mathcal S_{2,3}\le C_{\delta_1} c_{\bar N}\nl h_\al\nr_{L^\infty_{t, x,\xi}}.
    \]
    where $c_{\bar N}\to 0$ as $\bar N\to \infty$.
    Finally for $\mathcal A_{2,2}$, first if $|\xi-\xi_\star|\ge \bar N$ or $|\xi|\ge \bar N$, we have 
    \[
    \int_{\xi_\star}k_{\rho_0}(\xi,\xi_\star)d\xi_\star\lesssim \frac 1{\bar N},\qquad \sup_{\xi\in\R^3}\int_{|\xi-\xi_\star|\ge \bar N}k_{\rho_0}(\xi,\xi_\star)d\xi_\star \to 0\qquad\text{as}\quad \bar N\to \infty,
    \]
    giving an upper bound $C_{\delta_1}c_{\bar N}\nl h_\al\nr_{L^\infty_{t,x,\xi}}$.
We now consider $|\xi|\le \bar N$ and $|\xi-\xi_\star|\le \bar N$.  We have\[\bega 
&\int_{\{\frac 1 {\bar N}<|\xi_\star-\xi_{\star\star} |<\bar N,|\xi_\star|<\bar N\}}k_{\rho}(\xi,\xi_\star) k_{\rho}(\xi_\star,\xi_{\star\star} )|h_\al|\lw(\wtd s,\wtd x-\frac{s-\wtd s}{\e}\xi_\star,\xi_{\star\star} \rw)d\xi_{\star\star} d\xi_\star \\
&\lesssim \int_{\{|\xi_{\star\star} |< 2\bar N, |\xi_\star|<\bar N,|\xi_\star-\xi_{\star\star} |>\bar N^{-1}\}}k_{\rho_0}(\xi,\xi_\star)k_{\rho_0}(\xi_\star,\xi_{\star\star} )|h_\al|\lw(\wtd s,\wtd x-\frac{s-\wtd s}{\e}\xi_\star,\xi_{\star\star} \rw)d\xi_{\star\star}  d\xi_\star \\
&\lesssim \int_{\{|\xi_{\star\star} |< 2\bar N, |\xi_\star|<\bar N,|\xi_\star-\xi_{\star\star} |>\bar N^{-1}\}}k_{\rho_0}(\xi,\xi_\star)k_{\rho_0}(\xi_\star,\xi_{\star\star} )\\
&\quad 
\lw|\pt^\al \lw(e^{\frac{|\xi_{\star\star}|^2}{4}(T_M^{-1}-1)}e^{-\frac{1}{4}\eps^2|U|^2}e^{\frac 1 2 \eps U\cdot\xi_{\star\star}} f
\rw)
\rw|\lw(\wtd s,\wtd x-\frac{s-\wtd s}{\e}\xi_\star,\xi_{\star\star} \rw)d\xi_{\star\star}  d\xi_\star \\
&\lesssim  \int_{\{|\xi_{\star\star} |< 2\bar N, |\xi_\star|<\bar N,|\xi_\star-\xi_{\star\star} |>\bar N^{-1}\}}k_{\rho_0}(\xi,\xi_\star)k_{\rho_0}(\xi_\star,\xi_{\star\star} )|f_\al|  \lw(\wtd s,\wtd x-\frac{s-\wtd s}{\e}\xi_\star,\xi_{\star\star} \rw)d\xi_{\star\star}  d\xi_\star\\
&\quad+\sum_{\substack{\beta+\gamma=\al\\|\beta|\ge 1
}}\frac{\al!}{\beta!\gamma!} \int_{\{|\xi_{\star\star} |< 2\bar N, |\xi_\star|<\bar N,|\xi_\star-\xi_{\star\star} |>\bar N^{-1}\}}k_{\rho_0}(\xi,\xi_\star)k_{\rho_0}(\xi_\star,\xi_{\star\star} )\\
&\quad\times \lw|\pt^\beta\lw\{e^{-\frac 1 4 \eps^2 |U|^2+\frac 1 2 \eps U\cdot\xi_{\star\star}}
\rw\}(\wtd s,\wtd x-\frac{s-\wtd s}{\eps}\xi_\star)
\rw||f_\gamma|  \lw(\wtd s,\wtd x-\frac{s-\wtd s}{\e}\xi_\star,\xi_{\star\star} \rw)d\xi_{\star\star}  d\xi_\star\\
&\lesssim (\eps\kappa)^{-\frac 12}\nl \P f_\al(\wtd s)
\nr_{L^6_x}+(\eps\kappa)^{-\frac 32}\nl \IP f_\al(\wtd s)
\nr _{L^2_{x,\xi}}\\
&\quad+\sum_{\substack{\beta+\gamma=\al\\|\beta|\ge 1
}}\frac{\al!}{\beta!\gamma!}\sup_{|\xi_{\star\star}|\le 2\bar N}\nl \pt^\beta\lw\{e^{-\frac 1 4 \eps^2 |U|^2+\frac 1 2 \eps U\cdot\xi_{\star\star}}
\rw\}\nr _{L^\infty_x}\nl h_\gamma(\wtd s)
\nr_{L^\infty_{x,\xi}}.
\enda\]
Here we use the change of variables 
$
y=x-\frac{t-s}{\eps}\xi-\frac{s-\wtd s}{\eps}\xi_\star
$ with $$|dy|=|\frac{\eps}{s-\wtd s}   |^3 |d\xi_\star|\lesssim \frac{\eps^3}{(\delta_1\kappa\eps^2)^3}|d\xi_\star|.$$
Hence we get
\beq\label{ran-116}
\bega 
\mathcal S_{2,2}
     &\lesssim (\eps\kappa)^{-\frac 12} \sup_{0\le s\le t}\nl \P f_\al(s)
     \nr_{L^6_x}+(\eps\kappa)^{-\frac 3 2}\sup_{0\le s\le t} \nl \IP f_\al
     \nr_{L^2_{x,\xi}}\\
     &\quad+\sum_{\substack{\beta+\gamma=\al\\|\beta|\ge 1
}}\frac{\al!}{\beta!\gamma!}\sup_{|\xi_{\star\star}|\le 2\bar N}\nl \pt^\beta\lw\{e^{-\frac 1 4 \eps^2 |U|^2+\frac 1 2 \eps U\cdot\xi_{\star\star}}
\rw\}\nr _{L^\infty_x}\sup_{0\le s\le t} \nl h_\gamma( s)
\nr_{L^\infty_{x,\xi}}  .   \enda 
     \eeq
\textbf{Treating $\mathcal S_3$.}
Recalling the definition of $\mathcal S_3$ in \eqref{ran-617}, we obtain
\[
\mathcal S_3\lesssim \la\xi\ra^{-1}  \kappa\e^2 \sup_{0\le s\le t}\nl \la\xi\ra^{-1}\wtd g_\al(s)
\nr _{L^\infty_{x,\xi}}.
\]
\textbf{Treating $\mathcal S_4$.} We recall from \eqref{ran-617} that
\beq\label{ran-84}
\bega
\mathcal S_4&=\frac 1 {\kappa\eps^2}\int_{\max\{0,t-\tb
\}
}^t e^{-c_0\la\xi\ra \frac{t-s}{\kappa\eps^2}} \int_{\xi_\star}k_{\rho_0}(\xi,\xi_\star)1_{\{s\ge \tb(\wtd x,\xi_\star)
\}
} \\
&\quad\times e^{-c_0\la \xi_\star\ra \frac{\tb(\wtd x,\xi_\star)}{\kappa\eps^2}
} \lw\{w_1(\xi_\star)\int_{\xi_{\star\star}}|h_\al|\lw(s-\tb(\wtd x,\xi_\star),\xb(\wtd x,\xi_\star),\xi_{\star\star}\rw)w_2(\xi_{\star\star})\rw\}d\xi_{\star\star} d\xi_\star ds.\\
\enda\eeq
Now using \eqref{Duh}, we get
\beq\label{ran-85}
\bega 
&|h_\al|(s-\tb(\wtd x,\xi_\star),\xb(\wtd x,\xi_\star),\xi_{\star\star})\\
&\lesssim \nl h_{\al,0}\nr_{L^\infty_{x,\xi}}\\
&+\frac{1}{\kappa\eps^2}\int_0^{s-\tb(\wtd x,\xi_\star)}e^{-c_0\la\xi_{\star\star}\ra \frac{s-\tb(\wtd x,\xi_\star)-\wtd s}{\kappa\eps^2}}|K_{M}h_\al|\lw(\wtd s,\xb(\wtd x,\xi_\star)-\frac{s-\tb(\wtd x,\xi_\star)-\wtd s}{\eps}\xi_{\star\star},\xi_{\star\star}\rw) d\wtd s\\
&+\int_0^{s-\tb(X,\xi_\star)}e^{-c_0\la\xi_{\star\star}
\ra \frac{s-\tb(X,\xi_\star)-\wtd s}{\kappa\eps^2}}| \wtd g_\al|\lw(\wtd s,\xb(\wtd x,\xi_\star)-\frac{s-\tb(\wtd x,\xi_\star)-\wtd s}{\eps}\xi_{\star\star},\xi_{\star\star}\rw) d\wtd s.\enda 
\eeq
Combining \eqref{ran-84} and \eqref{ran-85}, we get 
\[
\mathcal S_{4}\lesssim \nl h_{\al,0}
\nr _{L^\infty_{x,\xi}}+\mathcal S_{4,1}+\mathcal S_{4,2},
\]
where 
\beq\label{ran-86}
\bega
&\mathcal S_{4,1}=\frac 1 {\kappa\eps^2}\int_{\max\{0,t-\tb
\}
}^t e^{-c_0\la\xi\ra \frac{t-s}{4\kappa\eps^2}} \int_{\xi_\star}k_{\rho_0}(\xi,\xi_\star)e^{-c_0\la\xi_\star\ra \frac{ \tb(\wtd x,\xi_\star)}{\kappa\e^2}}w_1(\xi_\star)\frac{1}{\kappa\eps^2}\\
&\int_{\xi_{\star\star}}\int_0^{s-\tb(X,\xi_\star)}e^{-c_0\la \xi_{\star\star}\ra \frac{s-\tb(X,\xi_\star)-\wtd s}{\kappa\eps^2}}|K_M h_\al|\lw(\wtd s,\xb(\wtd x,\xi_\star)-\frac{s-\tb(\wtd x,\xi_\star)-\wtd s}{\eps}\xi_{\star\star},\xi_{\star\star}\rw) d\wtd s\\
&\times w_2(\xi_{\star\star})d\xi_{\star\star} d\xi_\star ds\\
\mathcal S_{4,2}&=\frac 1 {\kappa\eps^2}\int_{\max\{0,t-\tb\}
}^t e^{-c_0\la\xi\ra \frac{t-s}{\kappa\eps^2}} \int_{\xi_\star}k_{\rho_0}(\xi,\xi_\star)e^{-c_0\la \xi_\star\ra \frac{ \tb(\wtd x,\xi_\star)}{\kappa\e^2}}w_1(\xi_\star),\\
&\times\int_{\xi_{\star\star}}\int_0^{s-\tb(\wtd x,\xi_\star)}e^{-c_0\la\xi_{\star\star}\ra\frac{s-\tb(\wtd x,\xi_\star)-\wtd s}{\kappa\eps^2}}| \wtd g_\al|\lw(\wtd s,\xb(\wtd x,\xi_\star)-\frac{s-\tb(\wtd x,\xi_\star)-\wtd s}{\eps}\xi_{\star\star},\xi_{\star\star}\rw) d\wtd s\\
&\times w_2(\xi_{\star\star})d\xi_{\star\star}d\xi_\star ds.
\enda
\eeq
We bound $\mathcal S_{4,1}$. Using Lemma \ref{LM-lem}, we get
\[\bega 
&K_Mh_\al\lw(\wtd s,\xb(\wtd x,\xi_\star)-\frac{s-\tb(\wtd x,\xi_\star)-\wtd s}{\eps}\xi_{\star\star},\xi_{\star\star}\rw)\\
&\lesssim \eps_0\la \xi_{\star\star}\ra \nl h_\al(\wtd s)\nr_{L^\infty_{x,\xi}}\\
&\quad+\int_{\xi_{\star\star\star} }k_{\rho_0}(\xi_{\star\star},\xi_{\star\star\star})|h_\al|\lw(\wtd s,\xb(\wtd x,\xi_\star)-\frac{s-\tb(\wtd x,\xi_\star)-\wtd s}{\eps}\xi_{\star\star},\xi_{\star\star\star}\rw)d\xi_{\star\star\star}.
\enda 
\]
Combining the above with the definition of $\mmA_{4,2}$ in \eqref{ran-86}, we get
\[\bega 
&\mathcal S_{4,1}\lesssim \frac 1 {\kappa\eps^2}\int_{\max\{0,t-\tb
\}
}^t e^{-c_0\la \xi\ra \frac{t-s}{\kappa\eps^2}} \int_{\xi_\star}\int_{\xi_{\star\star}}k_{\rho_0}(\xi,\xi_\star)k_{\rho_0}(\xi_{\star\star},\xi_{\star\star\star})e^{-c_0\la \xi_\star\ra \frac{\tb(\wtd x,\xi_\star)}{\kappa\e^2}}w_1(\xi_\star)\\
&\times\frac{1}{\kappa\eps^2}\int_{\xi_{\star\star\star}}\int_0^{s-\tb(X,\xi_\star)}e^{-c_0\la \xi_{\star\star}\ra \frac{s-\tb(X,\xi_\star)-\wtd s}{\kappa\eps^2}}h_\al\lw(\wtd s,\xb(\wtd x,\xi_\star)-\frac{s-\tb(\wtd x,\xi_\star)-\wtd s}{\eps}\xi_{\star\star},\xi_{\star\star\star}\rw)\\
&\times w_2(\xi_{\star\star})d\xi_{\star\star\star}d\xi_{\star\star} d\xi_\star d\wtd sds,\qquad \wtd x =x-\frac{t-s}{\eps}\xi.\\
\enda\]
Similarly as in the estimate of the term $\mathcal S_2$ in \eqref{ran-821}, we obtain
\[\bega 
&\mathcal S_{4,1}\lesssim o(1)\sup_{0\le s\le t}\nl h_\al(s)\nr_{L^\infty_{x,\xi}}+(\eps\kappa)^{-\frac 12} \sup_{0\le s\le t}\nl \P f_\al(s)
     \nr_{L^6_x}+(\eps\kappa)^{-\frac 3 2}\sup_{0\le s\le t} \nl \IP f_\al
     \nr_{L^2_{x,\xi}}\\
     &+\sum_{\substack{\beta+\gamma=\al\\|\beta|\ge 1
}}\frac{\al!}{\beta!\gamma!}\sup_{|\xi_{\star\star}|\le 2\bar N}\nl \pt^\beta\lw\{e^{-\frac 1 4 \eps^2 |U|^2+\frac 1 2 \eps U\cdot\xi_{\star\star}}
\rw\}\nr _{L^\infty_x}\sup_{0\le s\le t} \nl h_\gamma( s)
\nr_{L^\infty_{x,\xi}}.   \enda 
\]
As for the term $\mathcal S_{4,2}$, we get 
\[
\mathcal S_{4,3} \lesssim  \la\xi\ra^{-1}  \kappa\e^2 \sup_{0\le s\le t}\nl \la\xi\ra^{-1}\wtd g_\al(s)
\nr _{L^\infty_{x,\xi}}.
\]
\textbf{Treating $\mathcal S_5$.} Recalling $\mathcal S_5$ in \eqref{ran-617}, we have
\[
\mathcal S_5\lesssim \la\xi\ra^{-1} \frac\eps\delta 
\nl w_3^{-1}\pt^\al r
\nr_{L^\infty}.
\]
The bounds for $\mM_2$ is similar and we skip the details.  The proof is complete.
\end{proof}
Combining the propositions below, we have

\begin{theorem} Let $h_\al$ solve \eqref{h-al-eq} with the boundary condition \eqref{h-al-bdr}. Recall the analytic norms in \eqref{EDH}. There holds 
\[\bega 
\nl h_\al(t)\nr_{L^\infty_{x,\xi}}&\lesssim \|h_{\al,0}\|_{L^\infty_{x,\xi}}+(\eps\kappa)^{-\frac 12} \sup_{0\le s\le t}\nl \P f_\al(s)
     \nr_{L^6_x}+(\eps\kappa)^{-\frac 3 2}\sup_{0\le s\le t} \nl \IP f_\al
     \nr_{L^2_{x,\xi}}\\
     &\quad+\sum_{\substack{\beta+\gamma=\al\\|\beta|\ge 1
}}\frac{\al!}{\beta!\gamma!}\sup_{|\xi_{\star\star}|\lesssim 1}\nl \pt^\beta\lw\{e^{-\frac 1 4 \eps^2 |U|^2+\frac 1 2 \eps U\cdot\xi_{\star\star}}
\rw\}\nr _{L^\infty_x}\sup_{0\le s\le t} \nl h_\gamma( s)
\nr_{L^\infty_{x,\xi}}\\
&\quad+\kappa\e^2 \sup_{0\le s\le t}\nl \la\xi\ra^{-1}\wtd g_\al(s)
\nr _{L^\infty_{x,\xi}}+\frac\eps\delta 
\nl w_3^{-1}\pt^\al r
\nr_{L^\infty_{x,\xi}},
\enda\]
and we have the analytic estimate 
\[\bega 
\mathcal H_h(t)&\lesssim \mathcal H_h(0)+(\eps\kappa)^{-\frac 12}\nl \nl \AA_\al \P f_\al
\nr_{L^\infty(0,t,L^6_x)}
\nr _{\ell_\al^2}+(\eps\kappa)^{-\frac 3 2}\nl 
\nl \AA_\al \IP f_\al
\nr _{L^\infty(0,t,L^2_{x,\xi})}
\nr _{\ell_\al^2} \\
&\quad+\kappa\e^2\nl 
\nl
\la\xi\ra^{-1} \AA_\al \wtd g_\al
\nr_{L^\infty((0,t)\times\Omega\times\R^3)} 
\nr _{\ell_\al^2}+\frac\eps\delta\nl \nl w_3^{-1}\AA_\al \pt^\al r
\nr_{L^\infty(0,t,L^\infty_{x,\xi})}
\nr_{\ell_\al^2} .
\enda 
\]
\end{theorem}
We now recall from \eqref{h-al-eq} that 
\[
\wtd g_\al=\frac 1 {\kappa\eps^2}[L_M,\pt^\al]h+ \pt^\al\lw(-\frac{1}{ \eps\delta\sqrt\mu_M}\textbf{R}_a
+\frac{\delta}{\kappa\e}\Gamma_M (h,h)+\frac 2\kappa \frac{Q(\ff\sqrt\mu,h\sqrt{\mu_M})}{\sqrt{\mu_M}}\rw).
\]
\begin{definition} We define the following quantity
\[\bega
&d_{\infty}(t)=\frac{\kappa \eps^2}{\delta}\nl \NS(U,P)
\nr_{\tau_0,\infty}+\frac {\kappa\e} {\delta} \nl \div U
\nr_{\tau_0,\infty}+\frac {\e^3}{\delta}(\|| P\||_{\tau_0,\infty}+\kappa\nl \nabla_x U
\nr_{\tau_0,\infty})^2\\
&+\frac{\eps ^2\kappa}\delta \lw\{
 \||\pt_t P+U\cdot\nabla_x P\||_{\tau_0,\infty}+\eps \|| P(\pt_t U+U\cdot\nabla_x U)\||_{\tau_0,\infty}+\|| P\pt_i U_j \||_{\tau_0,\infty}
 \rw\}\\
 &+\frac{\eps^3 \kappa^2 }\delta \lw(\||  (\pt_t+U\cdot\nabla_x)(\pt_i U_j)\||_{\tau_0,\infty}+\eps\||  (\pt_t U+U\cdot\nabla_x U)(\pt_i U_j)\||_{\tau_0,\infty}+\|| \nabla_x U
\||_{\tau_0,\infty}^2\rw)\\
&+ \frac{\kappa^2\eps^3 }\delta \|| \pt_i U_j
\||_{\tau_0,\infty}\lw(1+\eps^2 \|| U
\||_{\tau_0,\infty}
\rw)+ \frac{\eps^2 \kappa^2}\delta \nl \nabla_x^2 U
\nr_{\tau_0,\infty}
 \\
&+ \frac{\eps}{\delta} \lw\{\nl \AA_\al 
\nl \pt^\al P
\nr_{L^\infty(\pt\Omega)}
\nr_{\ell_\al^2} 
+\sum_{j=1}^2 \nl \AA_\al 
\nl \pt^\al \pt_{x_3}U_j
\nr_{L^\infty(\pt\Omega)}
\nr_{\ell_\al^2}
\rw\}.
\enda\]
\end{definition}
We have $d_\infty(t)\lesssim \eps^{0.99}$ for suitable fluid solutions.
Next we bound each term in the above. Next, we show the bound for $[L_M,\pt^\al]$.

\begin{lemma} There holds 
\[
\nl
\nl \la\xi\ra^{-1}\AA_\al [L_M,\pt^\al] h
\nr_{L^\infty(0,t,L^\infty_{x,\xi})}
\nr_{\ell_\al^2}\lesssim \eps \kappa^{\frac 12} \mathcal H_h(t),
\]
and 
\beq\label{ran-120}
\frac{\delta}{\kappa\eps}\nl 
\nl \AA_\al 
\la \xi\ra^{-1}\pt^\al \Gamma_M(h,h)
\nr _{L^\infty(0,t,L^\infty_{x,\xi})}
\nr_{\ell_\al^2}\lesssim \frac{\delta}{\kappa\eps}\mathcal H_h(t)^2.
\eeq
\end{lemma}
\begin{proof} From \eqref{GammaM}, we obtain 
\[\bega 
&\frac 12 L_Mh\\
&=
\frac{\mu}{\sqrt{\mu_M}}(\xi)\int_{\R^3\times\S^2}|(\xi-\xi_\star)\cdot\w|\sqrt{\mu_M(\xi_\star)}h(\xi_\star)d\xi_\star d\w\\
&+h(\xi)\int_{\R^3\times\S^2}|(\xi-\xi_\star)\cdot\w|\mu(\xi_\star)d\w d\xi_\star\\
&-\int_{\R^3\times\S^2}|(\xi-\xi_\star)\cdot\w|\sqrt{\mu_M(\xi_\star)}\lw(\frac{\mu}{\sqrt{\mu_M}}(\xi')h(\xi_\star')+h(\xi')\frac{\mu}{\sqrt{\mu_M}}(\xi_\star')\rw)d\w d\xi_\star.\\
\enda
\]
Since $[L,\pt^\al]f=0$ for $|\al|=0$, we only need to consider $|\al|\ge 1$. 
We have
\[
\bega
[L_M,\pt^\al]h
&=2\sum_{\substack{\beta+\gamma =\al\\\gamma<\al
}
}\frac{\al!}{\beta!\gamma!}\int_{\R^3\times\S^2}|(\xi-\xi_\star)\cdot \w|\sqrt{\mu_M(\xi_\star)}\pt^\beta\lw\{\frac{ \mu}{\sqrt{\mu_M}}\rw\}(\xi')\pt^\gamma h (\xi'_\star)d\w d\xi_\star\\
&\quad+ 2\sum_{\substack{\beta+\gamma =\al\\\gamma<\al
}
}\frac{\al!}{\beta!\gamma!}\int_{\R^3\times\S^2}|(\xi-\xi_\star)\cdot \w|\sqrt{\mu_M(\xi_\star)}\pt^\beta\lw\{\frac{\mu}{\sqrt{\mu_M}}\rw\}(\xi'_\star)\pt^\gamma h (\xi')d\w d\xi_\star\\
&\quad-2\sum_{\substack{\beta+\gamma =\al\\\gamma<\al
}
}\frac{\al!}{\beta!\gamma!} \pt^\gamma h(\xi)\int_{\R^3\times\S^2}|(\xi-\xi_\star)\cdot\w|\pt^\beta\mu(\xi_\star)d\w d\xi_\star\\
&\quad-2\sum_{\substack{\beta+\gamma=\al\\\gamma<\al
}
}\frac{\al!}{\beta!\gamma!}\pt^\beta\lw\{\frac{\mu}{\sqrt{\mu_M}}\rw\}(\xi)\int_{\R^3\times\S^2}|(\xi-\xi_\star)\cdot\w|\sqrt{\mu_M(\xi_\star)}\pt^\gamma h(\xi_\star)d\xi_\star d\w.\\
\enda \]
By a direct calculation, we get 
\beq\label{ran-120}
\frac{\mu}{\sqrt{\mu_M}}=T_M^{3/4}e^{-\frac{1}{2}|\xi-\eps U|^2+\frac{|\xi|^2}{4}}=T_M^{\frac 3 4}e^{-\frac 1 4|\xi|^2+\eps U\cdot\xi-\frac 1 2 \eps^2|U|^2}.
\eeq
Hence we get 
\[
\la\xi\ra^{-1}[L_M,\pt^\al]h\lesssim \sum_{\substack{\beta+\gamma =\al\\\gamma<\al
}
}\frac{\al!}{\beta!\gamma!}\nl \pt^\gamma h
\nr_{L^\infty_\xi}\lw\{ \nl \frac{\pt^\beta\mu}{\sqrt\mu_M}\nr_{L^\infty_{\xi}}+\nl \la\xi\ra\pt^\beta \mu\nr_{L^1_\xi}\rw\}.
\]
Using Proposition \eqref{local-1} and the discrete Young inequality, we have the result.
Next we show \eqref{ran-120}. We have 
\[
\la\xi\ra^{-1}\pt^\al \Gamma_M(h,h)=\sum_{\beta+\gamma=\al}\frac{\al!}{\beta!\gamma!}\la\xi\ra^{-1}\Gamma_M(h_\beta ,h_\gamma)\lesssim \sum_{\beta+\gamma=\al}\frac{\al!}{\beta!\gamma!} \nl h_\beta 
\nr_{L^\infty_\xi} \nl h_\gamma\nr_{L^\infty_\xi} .
\]
Using the discrete Young inequality and \eqref{algebra-est}, we obtain \eqref{ran-120}.
\end{proof}

\begin{proposition} Recalling $\textbf{R}_a$ in \eqref{R-a-def},  there holds 
\[
\kappa\eps^2\cdot \frac 1{\eps\delta} \nl
\nl \frac 1 \nu \pt^\al \lw(\mu_M^{-\frac12}\textbf{R}_a
\rw)
\nr_{L^\infty_{x,\xi}}
\nr
_{\ell_\al^2} \lesssim d_\infty(t).
\]
\end{proposition}
\begin{proof} We have 
\[\bega 
&\frac{\kappa\e^2}{\eps\delta}\mu_M^{-\frac12}\textbf{R}_a\\
&=\frac{\kappa \eps^2}{\delta} \NS(U,P)\cdot\vp \mu \mu_M^{-\frac12}-\frac {\e^3}{\delta}\mu^{\frac12}\mu_M^{-\frac12}\Gamma(\ff,\ff)+\frac {\kappa\e} {3\delta} \div U |\vp|^2\mu\mu_M^{-\frac12}\\
&-\frac{\eps^2 \kappa^2}\delta \pt_{i\ell}U_j\IP (\vp_\ell A_{ij})\mu^{\frac12}\mu_M^{-\frac12}\\
&+\frac{\eps ^2\kappa}\delta \cdot  \mu\mu_M^{-\frac12}\lw\{(\pt_t P+U\cdot\nabla_x P)+\eps P(\pt_t U+U\cdot\nabla_x U)\cdot\vp+\frac 1 2 P \pt_i U_\ell \vp_i\vp_\ell
\rw\} \\ 
&-\frac{\eps^3 \kappa^2 }\delta\mu^{\frac12}\mu_M^{-\frac12}\lw\{(\pt_t+U\cdot\nabla_x)(\pt_i U_j)+\eps (\pt_tU+U\cdot\nabla_x U)\cdot \vp \pt_i U_j+\pt_i U_j \pt_k U_\ell \vp_k\vp_\ell
\rw\}A_{ij} \\
&-\frac{\kappa^2\eps^3 }\delta \mu^{\frac 12}\mu_M^{-\frac12} \pt_i U_j\lw(\eps^2 \pt_t+(\eps^2 U+\vp)\cdot\nabla_x
\rw)A_{ij}.\\
\enda
\]
We note that 
\beq\label{MMM}
\begin{cases}
\mu^{\frac 12}\mu_M^{-\frac12}&=T_M^{\frac 3 4}e^{\frac{|\xi|^2}{4}(T_M^{-1}-1)}e^{\frac 1 2\eps U\cdot\xi}e^{-\frac 1 4\eps^2|U|^2},\\
\mu \mu_M^{-\frac 12}&=T_M^{\frac 34}(2\pi)^{-\frac 34}e^{-\frac 1 4 (2-T_M^{-1})|\xi|^2}e^{-\frac 1 {2T_M}\eps U\cdot\xi}e^{\frac {\eps^2|U|^2}{4T_M}}.
\end{cases}
\eeq
First we bound 
\[
\frac{\kappa \eps^2}{\delta} \NS(U,P)\cdot\vp \mu \mu_M^{-\frac12}=T_M^{-\frac 3 2}(2\pi)^{-\frac 3 4}\frac{\kappa \eps^2}{\delta}\NS(U,P)\cdot \vp e^{-\frac 1 4 (2-T_M^{-1})|\xi|^2}e^{-\frac 1 {2T_M}\eps U\cdot\xi}e^{\frac {\eps^2|U|^2}{4T_M}}.
\]
Using the product rule and Lemma \ref{exp-ana} and Lemma \ref{exp-ana-ele} for the function $e^{-\frac 1 4 (2-T_M^{-1})|\xi|^2}e^{-\frac 1 {2T_M}\eps U\cdot\xi}$, we get 
\[
\nl
\AA_\al \lw|\pt^\al\lw\{\kappa\e^3 \NS(U,P)\cdot\vp \mu \mu_M^{-\frac12}\rw\}
\rw|
\nr_{\ell_\al^2}\lesssim \frac{\kappa \eps^2}{\delta}\nl \NS(U,P)
\nr_{\tau_0,\infty} e^{-\frac 1 {10} (2-T_M^{-1})|\xi|^2}.
\]
Using Proposition \ref{ana-gamma} and recalling \eqref{ff-def}, we have, for $|T^{-1}_M-1|\ll 1$, we have 
\[\bega
&\nl
\AA_\al \lw|e^{\frac 1 {10} |\xi|^2(T_M^{-1}-1)} \pt^\al \Gamma(\ff,\ff)
\rw|
\nr_{\ell_\al^2}\\
&\lesssim\nl 
\AA_\al \nl e^{o(1)|T_M^{-1}-1||\xi|^2} \pt^\al \ff
\nr _{L^\infty_\xi}\nr_{\ell_\al^2}^2
\lesssim \lw\{\|| P\||_{\tau_0,\infty}+\kappa\nl \nabla_x U
\nr_{\tau_0,\infty} \rw\}^2.
\enda 
\]
Hence
\[\frac {\e^3}{\delta}\nl\AA_\al
 \lw|
 \pt^\al\lw\{
\mu^{\frac12}\mu_M^{-\frac12}\Gamma(\ff,\ff)
\rw\}
\rw|
\nr_{\ell_\al^2}\lesssim\frac {\e^3}{\delta} e^{-\frac 1{10} |\xi|^2(T_M^{-1}-1)}\lw\{\|| P\||_{\tau_0,\infty}+\kappa\nl \nabla_x U
\nr_{\tau_0,\infty} \rw\}^2.\] 
Now we bound 
\[\frac {\kappa\e} {3\delta} \div U |\vp|^2\mu\mu_M^{-\frac12}=\frac 1 3 T_M^{\frac 34}(2\pi)^{-\frac 34}\cdot \frac {\kappa\e} {\delta} \div U |\vp|^2e^{-\frac 1 4 (2-T_M^{-1})|\xi|^2}e^{-\frac 1 {2T_M}\eps U\cdot\xi}e^{\frac {\eps^2|U|^2}{4T_M}}.
\]
Similarly as above, using Lemma \ref{exp-ana-ele} and Lemma \ref{exp-ana} for the function $e^{-\frac 1 4 (2-T_M^{-1})|\xi|^2}e^{-\frac 1 {2T_M}\eps U\cdot\xi}$, we have
\[
\frac {\kappa\e} {\delta}
\nl
\AA_\al
 \lw| \pt^\al
\lw\{\div U |\vp|^2\mu\mu_M^{-\frac12}
\rw\}
\rw|
\nr_{\ell_\al^2}\lesssim\frac {\kappa\e} {\delta} \nl \div U
\nr_{\tau_0,\infty} e^{-\frac 1 {10}(2-T_M^{-1})|\xi|^2}.
\]
Next we bound 
\[\bega 
&\frac{\eps^2 \kappa^2}\delta \pt_{i\ell}U_j\IP (\vp_\ell A_{ij})\mu^{\frac12}\mu_M^{-\frac12}\\
&=T_M^{\frac 3 4}\cdot \frac{\eps^2 \kappa^2}\delta\pt_{i\ell}U_j e^{\frac{|\xi|^2}{4}(T_M^{-1}-1)}e^{\frac 1 2\eps U\cdot\xi}e^{-\frac 1 4\eps^2|U|^2} \IP (\vp_\ell A_{ij}(\vp)).\enda 
\]
Since the function $e^{-\frac 14\eps^2|U|^2}$ is analytic, it suffices to bound\[
\frac{\eps^2 \kappa^2}\delta \nl\AA_\al
\lw|e^{\frac{|\xi|^2}{4}(T_M^{-1}-1)}\pt^\al\lw\{e^{\frac 1 2\eps U\cdot \xi} \pt_{i\ell}U_j \vp_\ell A_{ij}\rw\}
\rw|
\nr_{\ell_\al^2}.
\]
As in the estimate \eqref{ran-128} and the product rule, we have, for $|T_M^{-1}-1|\ll 1$: 
\[
\frac{\eps^2 \kappa^2}\delta \nl\AA_\al
\lw|e^{\frac{|\xi|^2}{4}(T_M^{-1}-1)}\pt^\al\lw\{e^{\frac 1 2\eps U\cdot \xi} \pt_{i\ell}U_j \vp_\ell A_{ij}\rw\}
\rw|
\nr_{\ell_\al^2}\lesssim \frac{\eps^2 \kappa^2}\delta \nl \nabla_x^2 U
\nr_{\tau_0,\infty} e^{-\frac 1{10}(T_M^{-1}-1)|\xi|^2}.
\]
Here we use the fact that $A_{ij}$ is analytic decay exponentially, proven in Corollary \ref{corol-Aij}.
Next, we consider 
\[\frac{\eps ^2\kappa}\delta   \mu\mu_M^{-\frac12}\lw\{(\pt_t P+U\cdot\nabla_x P)+\eps P(\pt_t U+U\cdot\nabla_x U)\cdot\vp+\frac 1 2 P \pt_i U_\ell \vp_i\vp_\ell
\rw\}.  \]
Similarly, recalling \eqref{MMM},  we have
\[
\bega 
&\frac{\eps ^2\kappa}\delta 
\nl \AA_\al
\lw|
\pt^\al\lw(
\mu\mu_M^{-\frac12}\lw\{(\pt_t P+U\cdot\nabla_x P)+\eps P(\pt_t U+U\cdot\nabla_x U)\cdot\vp+\frac 1 2 P \pt_i U_\ell \vp_i\vp_\ell
\rw\}
\rw)
\rw|
\nr_{\ell_\al^2} \\
&\lesssim\frac{\eps ^2\kappa}\delta \lw\{\||\pt_t P+U\cdot\nabla_x P\||_{\tau_0,\infty}+\eps \|| P(\pt_t U+U\cdot\nabla_x U)\||_{\tau_0,\infty}+\|| P\pt_i U_j \||_{\tau_0,\infty}\rw\}\\
&\times e^{-\frac 1 {10}(2-T_M^{-1})|\xi|^2}.
\enda 
\]
Next we bound \[\bega 
&\frac{\eps^3 \kappa^2 }\delta \mu^{\frac12}\mu_M^{-\frac12}\lw\{(\pt_t+U\cdot\nabla_x)(\pt_i U_j)+\eps (\pt_tU+U\cdot\nabla_x U)\cdot \vp \pt_i U_j+\pt_i U_j \pt_k U_\ell \vp_k\vp_\ell
\rw\}A_{ij}\\
&=T_M^{\frac 3 4}\cdot \frac{\eps^3 \kappa^2 }\delta e^{\frac{|\xi|^2}{4}(T_M^{-1}-1)}e^{\frac 1 2\eps U\cdot\xi}A_{ij}\\
&\times e^{-\frac 1 4\eps^2|U|^2} \lw\{(\pt_t+U\cdot\nabla_x)(\pt_i U_j)+\eps (\pt_tU+U\cdot\nabla_x U)\cdot \vp \pt_i U_j+\pt_i U_j \pt_k U_\ell \vp_k\vp_\ell
\rw\},\enda 
\]
where we recall \eqref{MMM}. Similarly, we get
 \[\bega 
&\frac{\eps^3 \kappa^2 }\delta
\nl 
\AA_\al \pt^\al \lw\{
\mu^{\frac12}\mu_M^{-\frac12}\lw\{(\pt_t+U\cdot\nabla_x)(\pt_i U_j)+\eps (\pt_tU+U\cdot\nabla_x U)\cdot \vp \pt_i U_j+\pt_i U_j \pt_k U_\ell \vp_k\vp_\ell
\rw\}A_{ij}\rw\}
\nr_{\ell_\al^2}\\
&\lesssim\frac{\eps^3 \kappa^2 }
\delta  e^{-\frac 1 {10}(T_M^{-1}-1)|\xi|^2}\\
&\quad\times 
\lw(\||  (\pt_t+U\cdot\nabla_x)(\pt_i U_j)\||_{\tau_0,\infty}+\eps\||  (\pt_t U+U\cdot\nabla_x U)(\pt_i U_j)\||_{\tau_0,\infty}+\|| \nabla_x U
\||_{\tau_0,\infty}^2\rw).\enda
\]
Next we bound the analytic norm of 
\[\bega 
&\frac{\kappa^2\eps^3 }\delta  \mu^{\frac 12}\mu_M^{-\frac12} \pt_i U_j\lw(\eps^2 \pt_t+(\eps^2 U+\vp)\cdot\nabla_x
\rw)A_{ij}\\
&=T_M^{\frac 3 4}\cdot \frac{\kappa^2\eps^3 }\delta e^{\frac{|\xi|^2}{4}(T_M^{-1}-1)}e^{\frac 1 2\eps U\cdot\xi}e^{-\frac 1 4\eps^2|U|^2}\pt_i U_j\lw(\eps^2 \pt_t+(\eps^2 U+\vp)\cdot\nabla_x
\rw)A_{ij}.
\enda
\]
Using Corollary \ref{corol-Aij}, we get
\[\bega 
&\frac{\kappa^2\eps^3 }\delta \nl\AA_\al
\lw|
\pt^\al \lw\{ \mu^{\frac 12}\mu_M^{-\frac12} \pt_i U_j\lw(\eps^2 \pt_t+(\eps^2 U+\vp)\cdot\nabla_x
\rw)A_{ij}
\rw\}
\rw|
\nr_{\ell_\al^2}\\
&\lesssim \frac{\kappa^2\eps^3 }\delta \|| \pt_i U_j
\||_{\tau_0,\infty}\lw(1+\eps^2 \|| U
\||_{\tau_0,\infty}
\rw)e^{-\frac{1}{10}(T_M^{-1}-1)|\xi|^2}.
\enda 
\]
The proof is complete.\end{proof}

\begin{lemma} Recalling $r$ in \eqref{r-def} and \eqref{ff-def}.  There holds 
\[\frac\eps\delta \nl \nl w_3^{-1}\AA_\al \pt^\al r
\nr_{L^\infty(0,t,L^\infty_{x,\xi})}
\nr_{\ell_\al^2}\lesssim \frac{\eps}{\delta} \lw\{\nl \AA_\al 
\nl \pt^\al P
\nr_{L^\infty(\pt\Omega)}
\nr_{\ell_\al^2} 
+\sum_{j=1}^2 \nl \AA_\al 
\nl \pt^\al \pt_{x_3}U_j
\nr_{L^\infty(\pt\Omega)}
\nr_{\ell_\al^2}
\rw\}.
\]
\end{lemma} 

\begin{proof}
We have 
\[
\ff|_{x_3=0}=P|_{x_3=0}\sqrt\mu_0-\sum_{j=1}^2 \pt_{x_3}U_j|_{x_3=0}A_{3j}(\xi),
\]
and 
\[
\P_{\gamma_+}\ff=c_\mu\sqrt\mu_0 \lw\{P|_{x_3=0}\int_{\xi_3<0}\mu_0(\xi)|\xi_3|d\xi -\sum_{j=1}^2 \pt_{x_3}U_j|_{x_3=0}\int_{\xi_3<0}A_{3j}(\xi)\sqrt\mu_0(\xi)|\xi_3|d\xi\rw\}.
\]
This implies 
\[
\pt^\al r\lesssim \lw(|\pt^\al P(t,x_1,x_2,0)|+\sum_{j=1}^2 |\pt_{x_3}\pt^\al U_j(t,x_1,x_2,0)|\rw)\sqrt\mu_0 .
\]
The result follows.
\end{proof}

Combining all of the above estimates, we obtain the theorem:

\begin{theorem}\label{l-inf} Let $f$ solve \eqref{f-eq} with the boundary condition \eqref{bdr-r}. There holds
\[\bega 
\mathcal H_{\rho,f}(t)&\lesssim \mathcal H_h|_{t=0}+(\eps\kappa)^{-\frac 12}\nl \nl \AA_\al \P f_\al
\nr_{L^\infty(0,t,L^6_x)}
\nr _{\ell_\al^2}\\
&\quad+(\eps\kappa)^{-\frac 3 2}\nl 
\nl \AA_\al \IP f_\al
\nr _{L^\infty(0,t,L^2_{x,\xi})}
\nr _{\ell_\al^2}+\eps^{0.99}.\enda 
\]
\end{theorem} 
\subsection{Analytic $L^\infty_t L^6_x$ estimates}\label{L6-sec}
In this section we derive the $L^6$ analytic estimate for\[
\nl \nl \AA_\al \P f_\al
\nr_{L^\infty(0,t,L^6_x)}
\nr _{\ell_\al^2},
\] where 
 $f$ solves \eqref{f-eq} with the boundary condition \eqref{bdr-r}. 
We recall the equation \eqref{f-al-eq} with the boundary condition \eqref{f-al-bdr-con}.
We have 
\[
\pt_t f_\al +\frac 1 \e \xi \cdot\nabla_x f_\al +\frac 1 {\kappa\eps^2}Lf_\al =G_\al,
\]
where 
\beq\label{G-al}\bega
G_\al&=\frac 1 {\kappa\eps^2}[L,\pt^\al]f-\pt^\al\lw\{\frac{(\pt_t+\frac 1 \eps \xi\cdot\nabla_x)\sqrt\mu}{\sqrt\mu}f\rw\}\\
&+\pt^\al\lw(-\frac{1}{\eps\delta \sqrt\mu}\textbf{R}_a+\frac{\delta}{\kappa\eps}\Gamma(f,f)+\frac{2}{\kappa}\Gamma(\ff,f)\rw),
\enda 
\eeq
along with the boundary condition 
\[
f_\al|_{\gamma_-}=\P_{\gamma_+} f_\al-\frac \eps\delta \pt^\al r.
\]
We recall 
\[
\P f_\al =a_\al(t,x)\sqrt\mu+b_\al(t,x)\cdot \vp\sqrt\mu+c_\al(t,x)\frac{|\vp|^2-3}{2}\sqrt\mu.
\]
We have, for any test function $\psi$:
\beq\label{test-L6}
\bega
\la \xi\cdot\nabla_x\psi, \P f_\al\ra_{L^2_{x,\xi}}-\la f_\al,\psi\ra_{L^2_\gamma}&=-\la \xi\cdot\nabla_x \psi,\IP f_\al\ra_{L^2_{x,\xi}}+\la \eps\pt_t f_\al ,\psi\ra_{L^2_{x,\xi}}\\
&\quad+\frac 1 {\kappa\eps}\la L\psi,\IP f_\al\ra_{L^2_{x,\xi}}-\la \eps G_\al,\psi\ra_{L^2_{x,\xi}},\enda \eeq
where
\[
\la f_\al,\psi\ra_{L^2_\gamma}=-\int_{\T^2}\xi_3 \lw(f_\al \psi\rw)|_{x_3=0}dx_1dx_2d\xi.
\]

\begin{proposition}\label{L6-P} There holds 
\[\bega 
\|\P f_\al \|_{L^6_x}&\lesssim o(1)\nl \P f_\al 
\nr_{L^2_x}+ 
\nl e^{\rho|\xi|^2}f_\al\nr_{L^\infty_{x,\xi}}^{\frac 2 3}\nl\IP f_\al\nr_{L^2_{x,\xi}}^{\frac 13}+\nl \eps\pt_t f_\al 
\nr_{L^2_{x,\xi}}\\
&+(\kappa\e)^{-1}\nl \IP f_\al
\nr _{L^2_{x,\xi}}+\nl (\II-\P_{\gamma_+})f_\al\nr_{L^2_{\gamma_+}}^{\frac 12}\nl e^{\rho|\xi|^2}f_\al\nr_{L^\infty_{x,\xi}}^{\frac 12}\\
&+\frac \eps \delta\|\pt^\al r \|_{L^4_{\gamma_-}}+\eps \nl \la\xi\ra^{-9}G_\al 
\nr_{L^2_{x,\xi}}.\enda
\]
\end{proposition}
\begin{proof} 
We will estimate $\nl a_\al
\nr _{L^6}, \nl b_\al
\nr _{L^6}$ and $\nl c_\al
\nr _{L^6}$ respectively. ~\\
\underline{Estimating $\nl a_\al\nr_{L^6}$ and $\nl c_\al
\nr_{L^6_x}$ }:~\\
We choose the test function $\psi$ to be 
\[\begin{cases}
\psi_a(t,x,\xi)&=(|\vp |^2-\beta_a)\sqrt\mu
\sum_{j=1}^3\vp_j\pt_{x_j} \phi_a,\qquad \beta_a=10,\\
\psi_c(t,x,\xi)&=(|\vp|^2-\beta_c)\sqrt\mu \sum_{j=1}^3 \vp_j\pt_{x_j}\phi_c,\qquad \beta_c=5,
\end{cases}
\]
where $\phi_a=\phi_a(t,x),\phi_c=\phi_c(t,x)$ solve the elliptic equation 
\beq\label{ran-95}
\begin{cases}
&-\triangle\phi_a=a_\al^5,\quad \pt_{x_3}\phi_a|_{x_3=0}=0,\\
&-\triangle\phi_c=c_\al^5,\quad \phi_c|_{x_3=0}=0.\\
\end{cases}
\eeq
It is standard that  $\phi\in\{\phi_a,\phi_c\}$ satisfy the following standard elliptic estimates 
\beq\label{elliptic-phi-ac}
\nl \nabla_x\phi \nr_{L^{4/3}(\pt\Omega)}+\nl \nabla_x^2 \phi \nr_{L^{6/5}(\Omega)}+\nl \nabla_x\phi\nr_{L^2(\Omega)}+\nl\phi \nr_{L^6(\Omega)}\lesssim 
\nl \P f_\al \nr_{L^6(\Omega)}^5.
\eeq
Now we have, for $\psi\in \{\psi_a,\psi_c\}$ and $\beta\in \{\beta_a,\beta_c\}$, we have 
\[
\xi\cdot\nabla_x \psi=\sum_{i,j}\vp_i\vp_j (|\vp|^2-\beta)\sqrt\mu \pt_{x_ix_j}\phi+R_{\phi},\]
where
\[R_{\phi}=
\sum_{i,j}\xi_i\pt_{x_j}\lw((|\vp|^2-\beta)\sqrt\mu \vp_j
\rw)\pt_{x_i}\phi+\eps \sum_{i,j}U_i\lw\{\vp_j(|\vp|^2-\beta)\sqrt\mu\rw\}\pt_{x_ix_j}\phi.
\]
This implies, for $\psi\in \{\psi_a,\psi_c\}$:
\beq\label{ran-96}
\bega 
\la\xi\cdot\nabla_x \psi,\P f_\al\ra_{L^2_{x,\xi}}&=\frac 1 3 \lw\{ \la |\vp|^2(|\vp|^2-\beta ),\mu\ra_{L^2_\vp}\rw\}\la a_\al, \triangle \phi \ra_{L^2_x}\\
&\quad+\frac 1 {6} \la |\vp|^2(|\vp|^2-3)(|\vp|^2-\beta),\mu\ra_{L^2_\vp}\la c_\al,\triangle\phi\ra_{L^2_x}+\la R_{\phi}, \P f_\al\ra_{L^2_{x,\xi}}.
\enda 
\eeq
Combining the above with \eqref{ran-95}, 
\eqref{test-L6} and the fact that
\[
\begin{cases}
&\la 
|\vp|^2(|\vp|^2-\beta_a),\mu\ra_{L^2_\vp}\neq 0,\qquad \la |\vp|^2(|\vp|^2-3)(|\vp|^2-\beta_a),\mu\ra_{L^2_\vp}=0,\\
&\la 
|\vp|^2(|\vp|^2-\beta_c),\mu\ra_{L^2_\vp}= 0,\qquad \la |\vp|^2(|\vp|^2-3)(|\vp|^2-\beta_c),\mu\ra_{L^2_\vp}\neq 0,\\\end{cases}
\]
we get
\[
\bega 
&\nl a_\al(t)
\nr _{L^6_x}^6+\nl c_\al(t)
\nr _{L^6_x}^6\\
&\lesssim \max_{\psi\in\{\psi_a,\psi_c\}}\lw\{
 \lw|\la R_{\phi}, \P f_\al\ra_{L^2_{x,\xi}}\rw|+\lw|\la f_\al,\psi\ra_{L^2_\gamma}\rw|+\lw| \la \xi\cdot\nabla_x \psi,\IP f_\al\ra_{L^2_{x,\xi}}\rw|+\la \eps\pt_t f_\al ,\psi \ra_{L^2_{x,\xi}}
\rw\}\\
&\quad+ \max_{\psi\in\{\psi_a,\psi_c\}}\lw\{
\frac 1 {\kappa\eps}\lw|\la L\psi,\IP f_\al\ra_{L^2_{x,\xi}}\rw|+\lw|\la \eps G_\al,\psi\ra_{L^2_{x,\xi}}\rw|
\rw\}\\
&\lesssim o(1)\nl \P f_\al
\nr_{L^2_x}\max_{\phi\in\{\phi_a,\phi_c\}}\nl \nabla_x\phi
\nr _{L^2_x}+
\nl \P f_\al
\nr_{L^6_x}\max_{\phi\in\{\phi_a,\phi_c\}}\nl \nabla_x^2\phi
\nr _{L^{\frac 6 5}_{x}}
+\max_{\psi\in\{\psi_a,\psi_c\}}\lw|\la f_\al,\psi\ra_{L^2_\gamma}\rw|\\
&\quad+\nl \IP f_\al 
\nr _{L^6_{x,\xi}}\max_{\psi\in\{\psi_a,\psi_c\}}\nl \xi\cdot\nabla_x\psi
\nr _{L^{\frac 65}_{x,\xi}}\\
&\quad+\lw\{\nl \eps\pt_t f_\al
\nr _{L^2_{x,\xi}}+\nl\eps G_\al
\nr _{L^2_{x,\xi}}
\rw\}\max_{\psi\in\{\psi_a,\psi_c\}}\nl \psi
\nr _{L^2_{x,\xi}}+\frac 1 {\kappa\eps}\nl \IP f_\al\nr_{L^2_{x,\xi}}\max_{\psi\in\{\psi_a,\psi_c\}}\nl L\psi\nr_{L^2_{x,\xi}}\\
&\lesssim o(1)\nl \P f_\al 
\nr _{L^2_x} \nl \P f_\al 
\nr _{L^6_x}^5+o(1)\nl \P f_\al 
\nr _{L^6_x}^6+\max_{\psi\in\{\psi_a,\psi_c\}}\lw|\la f_\al,\psi\ra_{L^2_\gamma}\rw|+\nl \IP f_\al 
\nr _{L^6_{x,\xi}}\nl \nabla_x^2 \phi \nr_{L^{\frac 65}_x}\\
&\quad+\lw\{\nl \eps\pt_t f_\al
\nr _{L^2_{x,\xi}}+\nl\eps G_\al
\nr _{L^2_{x,\xi}}+(\kappa\e)^{-1} \nl \IP f_\al\nr_{L^2_{x,\xi}}
\rw\}\max_{\phi\in \{\phi_a,\phi_c\}}\nl
\nabla_x\phi\nr_{L^2_x}\\
&\lesssim o(1)\nl \P f_\al 
\nr _{L^2_x} \nl \P f_\al 
\nr _{L^6_x}^5+o(1) \nl \P f_\al 
\nr _{L^6_x}^6+\max_{\psi\in\{\psi_a,\psi_c\}}\lw|\la f_\al,\psi\ra_{L^2_\gamma}\rw|\\
&\quad +\lw\{\nl \eps\pt_t f_\al
\nr _{L^2_{x,\xi}}+\nl\eps G_\al
\nr _{L^2_{x,\xi}}+(\kappa\e)^{-1} \nl \IP f_\al\nr_{L^2_{x,\xi}}+\nl \IP f_\al 
\nr _{L^6_{x,\xi}}\rw\} \nl \P f_\al 
\nr _{L^6_x}^5,
\enda
\]
we used \eqref{elliptic-phi-ac}. We shall now treat the boundary term 
\[
\la f_\al ,\psi\ra_{L^2_\gamma}=\int_{\pt\Omega\times\R_\xi^3}f(t,x,v)\psi(t,x)(\xi \cdot n(x))dS(x)d\xi.
\]
To this end, we decompose 
\[\bega 
f_\al |_{\gamma}
=1_{\gamma_+}(\II-\P_\gamma)f_\al+\P_\gamma f_\al-\frac{\eps}{\delta} 1_{\gamma_-}\pt^\al r.
\enda
\]
We recall 
\[
\P_\gamma f_\al=c_\mu \sqrt\mu_0(\xi)\lw\{\int_{\gamma_+} f_\al(t,x,\wtd\xi)\sqrt\mu_0(\wtd\xi)(n(x)\cdot\wtd \xi)d\wtd\xi
\rw\}.
\]
Hence we have, for $\psi\in \{\psi_a,\psi_c\}$ and $\beta\in \{\beta_a,\beta_c\}$
\beq \label{ran-971}
\bega 
&\la \P_\gamma f_\al ,\psi\ra_{L^2_\gamma}\\
&=c_\mu \int_{\pt\Omega\times \R_\xi^3} \lw\{\int_{\gamma_+} f_\al(t,x,\wtd\xi)\sqrt\mu_0(\wtd\xi)(n(x)\cdot\wtd \xi)d\wtd\xi \rw\}{\mu_0}(\xi)(|\xi|^2-\beta)\xi \cdot\nabla_x \phi(\xi\cdot n(x))d\xi dS(x)\\
&=\int_{\T^2\times\R_\xi^3} \lw\{\int_{\gamma_+} f_\al(t,x,\wtd\xi)\sqrt\mu_0(\wtd\xi)(n(x)\cdot\wtd \xi)d\wtd\xi \rw\}\mu_0(\xi)(|\xi |^2-\beta)(\xi_{\parallel}\pt_{x_\parallel} \phi+\xi_3 \pt_{x_3}\phi)(-\xi_3) d\xi dx_{\parallel}\\
&=0.
\enda 
\eeq
Here we used the fact that $\pt_{x_3}\phi_a=0$ on $\pt\Omega$ and 
$\int_{\R^3 }\xi_\parallel \xi_3 (|\xi |^2-\beta)\mu_0(\xi)d\xi=\int_{\R^3}\xi_3^2 (|\xi|^2-\beta_c)\mu_0(\xi)d\xi=0.$
This implies, for $\psi\in\{\psi_a,\psi_c\}$:
\[\bega 
\la f_\al ,\psi\ra_{L^2_\gamma}&=\la 1_{\gamma_+}(\II-\P_\gamma)f_\al -\frac\eps \delta 1_{\gamma_-}\pt^\al r,\psi\ra_{L^2_\gamma}\\
&\lesssim \nl (\II-\P_\gamma)f_\al\nr_{L^4_{\gamma_+}}\|\nabla_x\phi\|_{L^{4/3}(\pt\Omega)}+\frac \eps\delta  \|\pt^\al r\|_{L^4_{\gamma_-}}\|\nabla_x\phi\|_{L^{4/3}(\pt\Omega)}\\
&\lesssim\|\P f_\al \|_{L^6_x}^5\lw(\|(\II-\P_\gamma)f_\al \|_{L^4_{\gamma_+}}+\frac \eps \delta\|\pt^\al r \|_{L^4_{\gamma_-}}\rw),
\enda 
\]
where we used \eqref{elliptic-phi-ac}.~\\
\underline{Estimating $\nl b_\al 
\nr  _{L^6_x}$:}~\\
To bound $b_\al=(b_{\al,1},b_{\al,2},b_{\al,3}) $, we first choose the test function, for $i,j\in \{1,2,3\}$
\[
\psi=\psi_b^{i,j}=(\vp_i^2-1)\sqrt\mu \pt_j\phi_b^j,\qquad -\triangle \phi_b^j=b_{\al,j}^5,\quad \phi_b^j|_{\pt\Omega}=0.
\]
Then we have the elliptic estimate
\[
\|\nabla_x\phi_b^j\|_{L^{4/3}(\pt\Omega)}+\|\nabla_x^2 \phi_b^j\|_{L^{6/5}}+\|\nabla_x\phi_b^j\|_{L^2}+\|\phi_b^j\|_{L^6}\lesssim \nl Pf_\al\nr_{L^6(\Omega)}^5.
\]
We have 
\[
\xi\cdot\nabla_x \psi_b^{i,j}=\vp_k(\vp_i^2-1)\sqrt\mu \cdot \pt_{x_jx_k}\phi_b^j+R_{\phi_b^j},
\]
where \[
R_{\phi_b^j}=\xi_k\pt_{x_k}\lw((\vp_i^2-1)\sqrt\mu
\rw)\pt_{x_j}\phi_b^j+\eps U_k(\vp_i^2-1)\sqrt\mu\pt_{x_jx_k}\phi_b^j.
\]
This implies 
\beq\label{ran-98}
\bega 
\la \xi\cdot\nabla_x\psi_b^{i,j},\P f_\al \ra_{L^2_{x,\xi}}&=\sum_{k=1}^3 \la \vp_k^2(\vp_i^2-1),\mu\ra_{L^2_\vp}\la b_{\al,k},\pt_{x_jx_k}\phi_b^j\ra_{L^2_x}+\la R_{\phi_b^j},\P f_\al\ra_{L^2_{x,\xi}}\\
&=-2\la b_{\al,i},\pt_{ij} \phi_b^j\ra_{L^2_x}+\la R_{\phi_b^j},\P f_\al\ra_{L^2_{x,\xi}},\enda 
\eeq
where we used the fact that 
\[ \la \vp_k^2(\vp_i^2-1),\mu\ra_{L^2_\vp}=\begin{cases} &-2\qquad\text{if}\quad i=k\\
&0\qquad \text{if}\quad i\neq k.
\end{cases}
\]
Combining \eqref{ran-98} with \eqref{test-L6}, we get
\[\bega
&\la \pt_{ij}(-\triangle_D)^{-1}b_{\al,j},b_{\al,i}\ra_{L^2_x}\\
&\lesssim \lw|\la R_{\phi_b^j}, \P f_\al\ra_{L^2_{x,\xi}}\rw|+\lw|\la f_\al, \psi_b^{i,j}\ra_{L^2_\gamma}\rw|+\lw| \la \xi\cdot\nabla_x \psi_b^{i,j},\IP f_\al\ra_{L^2_{x,\xi}}\rw|+\la \eps\pt_t f_\al , \psi_b^{i,j}\ra_{L^2_{x,\xi}}\\
&\quad+\frac 1 {\kappa\eps}\lw|\la L\psi_b^{i,j},\IP f_\al\ra_{L^2_{x,\xi}}\rw|+\lw|\la \eps G_\al, \psi_b^{i,j}\ra_{L^2_{x,\xi}}\rw|\\
&\lesssim o(1)\nl \P f_\al 
\nr _{L^2_x} \nl \P f_\al 
\nr _{L^6_x}^5+o(1) \nl \P f_\al 
\nr _{L^6_x}^6+\lw|\la f_\al,\psi_b^{i,j}\ra_{L^2_\gamma}\rw|\\
&\quad +\lw\{\nl \eps\pt_t f_\al
\nr _{L^2_{x,\xi}}+\nl\eps G_\al
\nr _{L^2_{x,\xi}}+(\kappa\e)^{-1} \nl \IP f_\al\nr_{L^2_{x,\xi}}+\nl \IP f_\al 
\nr _{L^6_{x,\xi}}\rw\} \nl \P f_\al 
\nr _{L^6_x}^5.
\enda \]
For the boundary term, as in the above, it suffices to see that 
\[
\la \P_\gamma f_\al,(\xi_i^2-1)\sqrt\mu_0\pt_j \phi_b^j
\ra _{L^2_\gamma}=0.
\]
To this end, using a similar argument as in \eqref{ran-971}, we get
\[\bega 
&\la \P_\gamma f_\al,(\xi_i^2-1)\sqrt\mu_0\pt_j \phi_b^j
\ra _{L^2_\gamma}\\
&=-c_\mu \int_{\pt\Omega\times \R_\xi^3} \lw\{\int_{\gamma_+} f_\al(t,x,\wtd\xi)\sqrt\mu_0(\wtd\xi)(n(x)\cdot\wtd \xi)d\wtd\xi \rw\}{\mu_0}(\xi)\xi_3(\xi_i^2-1) \pt_j \phi_b^jd\xi dS(x)=0.\\\enda 
\]
Finally, we get 
\beq\label{ran-99}
\bega
&\lw|\la \pt_{ij}(-\triangle_D)^{-1}\lw(b_{\al,j}^5\rw),b_{\al,i}\ra_{L^2_x}\rw|\\
&\lesssim o(1)\nl \P f_\al 
\nr _{L^2_x} \nl \P f_\al 
\nr _{L^6_x}^5+o(1) \nl \P f_\al 
\nr _{L^6_x}^6+\|\P f_\al \|_{L^6_x}^5\lw(\|(\II-\P_\gamma)f_\al \|_{L^4_{\gamma_+}}+\frac \eps \delta\|\pt^\al r \|_{L^4_{\gamma_-}}\rw)\\
&\quad +\lw\{\nl \eps\pt_t f_\al
\nr _{L^2_{x,\xi}}+\nl\eps G_\al
\nr _{L^2_{x,\xi}}+(\kappa\e)^{-1} \nl \IP f_\al\nr_{L^2_{x,\xi}}+\nl \IP f_\al 
\nr _{L^6_{x,\xi}}\rw\} \nl \P f_\al 
\nr _{L^6_x}^5.
\enda\eeq
Next, for $i\neq j$, we choose the test function to be 
\[
\psi=|\vp|^2\vp_i\vp_j \sqrt\mu \pt_j \phi_b^j.
\]
We have 
\[
\xi\cdot\nabla_x \psi=\sum_{k=1}^3\vp_k\vp_i\vp_j |\vp|^2\sqrt\mu \pt_{jk}\phi_b^j +R_{\psi}^{i,j},
\]
where 
\[
R_\psi^{i,j}=\sum_{k=1}^3\xi_k\pt_{x_k}\lw(\vp_i\vp_j|\vp|^2\sqrt\mu
\rw)\pt_j\phi_b^j+\eps \sum_{k=1}^3 U_k \vp_i\vp_j|\vp|^2\sqrt\mu \pt_{jk}\phi_b^j.
\]
This implies 
\[\bega 
\la \xi\cdot\nabla_x \psi,\P f_\al\ra_{L^2_{x,\xi}}&=\sum_{k,m}\la \vp_i\vp_j\sqrt\mu ,\vp_k\vp_m\sqrt\mu
\ra _{L^2_\vp}\la b_{\al,m},\pt_{jk}\phi_b^j\ra_{L^2_x}+\la R_\psi^{i,j},\P f_\al\ra_{L^2_{x,\xi}}\\
&=\sum_{k,m}(\delta_{ij}\delta_{km}+\delta_{ik}\delta_{jm}+\delta_{im}\delta_{jk})\la \pt_{jk}\phi_b^j,b_{\al,m}\ra_{L^2_x}+\la R_\psi^{i,j},\P f_\al\ra_{L^2_{x,\xi}}\\
&=\la\pt_{ij}\phi_b^j,b_{j,\al}\ra_{L^2_x}+\la \pt_j^2\phi^j_b,b_{\al,i} \ra_{L^2_x}+\la R_\psi^{i,j},\P f_\al\ra_{L^2_{x,\xi}},\enda 
\]
where we use the fact that $i\neq j$. Combining the above with \eqref{test-L6}, we have 
\beq\label{ran-991}
\bega 
&1_{j\neq i}\lw|\la \pt_j^2\triangle_D^{-1} \lw(b_{\al,j}^5\rw),b_{i,\al} \ra_{L^2_x}\rw|\\
& \lesssim \lw|\la \pt_{ij}\triangle_D^{-1}\lw(b_{i,\al}^5\rw),b_{i,\al}\ra_{L^2_x}
\rw|+ o(1)\nl \P f_\al 
\nr _{L^2_x} \nl \P f_\al 
\nr _{L^6_x}^5+o(1) \nl \P f_\al 
\nr _{L^6_x}^6+\lw|\la f_\al,\psi\ra_{L^2_\gamma}\rw|\\
&\quad +\lw\{\nl \eps\pt_t f_\al
\nr _{L^2_{x,\xi}}+\nl\eps G_\al
\nr _{L^2_{x,\xi}}+(\kappa\e)^{-1} \nl \IP f_\al\nr_{L^2_{x,\xi}}+\nl \IP f_\al 
\nr _{L^6_{x,\xi}}\rw\} \nl \P f_\al 
\nr _{L^6_x}^5.\enda 
\eeq
Here we used the fact that 
\[\bega 
\la \P _\gamma f_\al,\psi\ra_{L^2_\gamma}=-c_\mu  \int_{\pt\Omega\times \R_\xi^3} \lw\{\int_{\gamma_+} f_\al(t,x,\wtd\xi)\sqrt\mu_0(\wtd\xi)(n(x)\cdot\wtd \xi)d\wtd\xi \rw\}{\mu_0}(\xi)|\xi|^2\xi_3\xi_i\xi_j \pt_j \phi_b^jd\xi dS(x)=0.
\enda
\]
Combining \eqref{ran-991} and \eqref{ran-99}, we get the result. Finally, we note that 
\[
\nl \IP f_\al\nr_{L^6_{x,\xi}}\lesssim \nl e^{\rho|\xi|^2}f_\al\nr_{L^\infty_{x,\xi}}^{\frac 2 3}\nl\IP f_\al\nr_{L^2_{x,\xi}}^{\frac 13},
\]
and 
\[
\nl (\II-\P_{\gamma_+})f_\al\nr_{L^4_{\gamma_+}}\lesssim\nl (\II-\P_{\gamma_+})f_\al\nr_{L^2_{\gamma_+}}^{\frac 12}\nl e^{\rho|\xi|^2}f_\al\nr_{L^\infty_{x,\xi}}^{\frac 12}.
\]
The proof is complete.
\end{proof}
We now focus on estimating $\nl
\eps \la\xi\ra^{-9}G_\al\nr_{L^2_{x,\xi}}$ where $G_\al$ is given in \eqref{G-al}.

\begin{lemma} There holds 
\[\bega 
&\eps \nl \AA_\al \nl\la\xi\ra^{-9}G_\al 
\nr_{L^2_{x,\xi}}
\nr_{\ell_\al^2}\\
&\lesssim \eps^{0.99}+\eps\kappa^{-1}\mathcal E_f(t)+\eps\kappa^{-\frac 12}\mathcal D_f(t) +\delta\eps\mathcal D_f(t)\mathcal H_{\rho,f}(t)+\delta\kappa^{-\frac 12}\nl \AA_\al \nl \P f_\al
\nr_{L^6_x}
\nr_{\ell^2_\al}\nl \AA_\al \nl \P f_\al
\nr_{L^3_x}
\nr _{\ell^{2}_\al}.
\enda\]
\end{lemma}

\begin{proof}
First, we get
\[
\frac {1}{\kappa\e}\nl 
\nl
\la\xi\ra^{-2}\AA_\al [L,\pt^\al]f
\nr_{L^2_{x,\xi}}
\nr _{\ell^2_\al}\lesssim \eps\kappa^{-1}\mathcal E_f(t)+\eps\kappa^{-\frac 12}\mathcal D_f(t),
\]
where we used \eqref{com-est}.
Next we use Lemma \ref{bulk-ran-1} to get
\[\eps \nl
\AA_\al \nl
\pt^\al\lw\{\frac{(\pt_t+\frac 1 \eps \xi\cdot\nabla_x)\sqrt\mu}{\sqrt\mu}f\rw\}
\nr_{L^2_{x,\xi}}\nr
_{\ell_\al^2}\lesssim \mathcal E_f(t).
\]
Next we have 
\[\bega 
&\frac \delta \kappa\nl \AA_\al\nl \pt^\al \Gamma(f,f)\nr_{L^2_{x,\xi}}\nr_{\ell_\al^2}\\
&\lesssim \delta\eps\mathcal D_f(t)\mathcal H_{\rho,f}(t)+\delta\kappa^{-\frac 12}\nl \AA_\al \nl \P f_\al
\nr_{L^6_x}
\nr_{\ell^2_\al}\nl \AA_\al \nl \P f_\al
\nr_{L^3_x}
\nr _{\ell^{2}_\al}.\enda 
\]
Finally for ${\bf{R}}_a$ term, we bound exactly as in Proposition \ref{Ra-term}. 
The proof is complete.
\end{proof}
From the Proposition \ref{L6-P} and the above analytic bound on the forcing term, we obtain the following theorem

\begin{theorem}\label{L6-final} Let $f$ solve \eqref{f-eq}. There holds 
\[\bega
&\nl
\AA_\al \nl \P f_\al
\nr_{L^\infty_t L^6_x}
\nr_{\ell_\al^2}\\
& \lesssim \eps^{\frac 1 3}\kappa^{\frac 1 6} \nl \mathcal H_{\rho,f}\nr_{L^\infty_t}^{\frac 2 3} \nl \mathcal D_f
\nr _{L^\infty_t}^{\frac 1 3} +\nl \mathcal Y_f
\nr_{L^\infty_t}+\kappa^{-\frac 12}\nl \mathcal D_f 
\nr_{L^\infty_t} \\
&\quad+\eps^{-\frac 12}\lw\{ \nl \AA_\al \nl (\II-\P_{\gamma_+})\pt_t f_\al
\nr_{L^2_tL^2_{\gamma_+}} 
\nr _{\ell_\al^2}+\nl \AA_\al \nl (\II-\P_{\gamma_+})f_\al
\nr_{L^2_tL^2_{\gamma_+}} 
\nr _{\ell_\al^2}
\rw\}\nl \mathcal H_{\rho,f}\nr_{L^\infty_t}\\
&\quad+\eps^{0.99}+\eps\kappa^{-\frac 12} \nl \mathcal D_f\nr_{L^\infty_t}+\delta\eps \nl\mathcal D_f \mathcal H_{\rho,f}\nr_{L^\infty_t}\\
&\quad+\delta\kappa^{-\frac12} \nl\AA_\al \nl \P f_\al
\nr_{L^\infty_tL^6_x} 
\nr_{\ell_\al^2}\lw\{
\nl \AA_\al \nl \P f_\al\nr_{L^2_tL^3_x}
\nr _{\ell_\al^2}+ \nl \AA_\al \nl \P \pt_t  f_\al\nr_{L^2_tL^3_x}
\nr _{\ell_\al^2}\rw\}.
\enda 
\]
\end{theorem}

\subsection{Analytic $L^2_tL^p_x$ estimate for $p\in (2,3)$}
In this section, we estimate \[\nl \nl \AA_\al \P f_\al
\nr_{L^2_t L^3_x} 
\nr_{\ell_\al^2},\]
appearing in the bound of the nonlinear term in Proposition \ref{G-ff}.
To this end, we recall the equation for $f_\al$ in \eqref{f-al-eq}
\[
\bega 
&\pt_t f_\al+\frac 1 \eps \xi\cdot\nabla_x f_\al+\frac 1 {\kappa\eps^2}Lf_\al+\frac{(\pt_t+\frac 1 \eps \xi\cdot\nabla_x)\sqrt\mu}{\sqrt\mu}f_\al\\
&=\frac 1 {\kappa\eps^2}[L,\pt^\al]f+g_\al+\frac{(\pt_t+\frac 1 \eps \xi\cdot\nabla_x)\sqrt\mu}{\sqrt\mu}f_\al-\pt^\al\lw\{\frac{(\pt_t+\frac 1 \eps \xi\cdot\nabla_x)\sqrt\mu}{\sqrt\mu}f\rw\},
\enda 
\]
where 
\[
g_\al=\pt^\al\lw(-\frac{1}{\eps\delta \sqrt\mu}\textbf{R}_a+\frac{\delta}{\kappa\eps}\Gamma(f,f)+\frac{2}{\kappa}\Gamma(\ff,f)\rw),
\]
along with the boundary condition 
\[
f_\al|_{\gamma_-}=\P_\gamma f_\al-\frac \eps\delta \pt^\al r,
\]
where $r$ is given in \eqref{r-def}.
 We recall the following Proposition, which the averaging in $\xi$ lemma for the transport equation with boundaries, proved in \cite{JangKimAPDE}:
\begin{proposition}\label{JK-Lp}Let $p\in (2,3)$. There holds 
\[
\nl \P f\nr_{L^2_tL^p_x}\lesssim \nl f_0\nr_{L^2_\gamma}+\nl f
\nr_{L^\infty_t L^2_{x,\xi}}+\nl (\eps\pt_t+\xi\cdot\nabla_x)f\nr_{L^2_{t,x,\xi}}.
\]
\end{proposition}

\begin{theorem}\label{thmLp} Let $p\in (2,3)$ and $f$ solves the equation \eqref{f-eq}, there holds 
\[\bega 
\nl \AA_\al \nl\P f_\al\nr_{L^2_tL^p_x}
\nr_{\ell_\al^2}& \lesssim \nl \AA_\al \nl \pt^\al f_0
\nr_{L^2_\gamma}
\nr_{\ell_\al^2}+\sup_{0\le s\le t}\mathcal E_f(t)+\kappa^{-\frac 12}\nl \mathcal D_f(t)\nr_{L^2_t}+\kappa^{-\frac 12} \nl \mathcal E_f\nr_{L^2_t}\\
&\quad+\eps\kappa^{\frac12}\nl  \mathcal D_f\nr_{L^2_t}\nl \mathcal H_{\rho,f}
\nr_{L^\infty_t}
+\eps^{0.99}\\
&\quad+\delta\kappa^{-\frac 12}\nl\AA_\al \nl \P f_\al
\nr_{L^\infty_t L^6_x}
\nr_{\ell_\al^2} \nl\AA_\al \nl \P f_\al
\nr_{L^2_t L^3_x}
\nr_{\ell_\al^2}.
\enda
\]
\end{theorem}

\begin{proof} Applying Proposition \ref{JK-Lp} for $f_\al$, we get
\[\bega 
&\nl \P f_\al\nr_{L^2_t L^p_x}\\
&\lesssim \nl \pt^\al f_0
\nr_{L^2_\gamma}+\nl f_\al\nr_{L^\infty_t L^2_{x,\xi}}\\
&+(\kappa\eps)^{-1}\lw\{\nl Lf_\al\nr_{L^2_{t,x,\xi}}+\nl [L,\pt^\al]f
\nr_{L^2_{t,x,\xi}}\rw\}+\delta^{-1} \nl\pt^\al\lw\{\mu^{-\frac 12}{\bf{R}}_a
\rw\}
\nr_{L^2_{t,x\xi}}\\
&+\delta\kappa^{-1}\nl\pt^\al \Gamma(f,f)
\nr_{L^2_{t,x,\xi}}+\eps\kappa^{-1} \nl\pt^\al\Gamma(\ff,f)
\nr_{L^2_{t,x,\xi}}+\eps \nl \pt^\al\lw\{\frac{(\pt_t+\frac 1 \eps \xi\cdot\nabla_x)\sqrt\mu}{\sqrt\mu}f\rw\}
\nr_{L^2_{t,x,\xi}}.\\
\enda 
\]
As in the $L^2$ analytic estimates, we obtain
\[
\bega 
(\kappa\e)^{-1}\nl \nl \AA_\al [L,\pt^\al]f
\nr_{L^2_{t,x,\xi}}\nr_{\ell_\al^2}&\lesssim \kappa^{-\frac 12}\nl \mathcal E_f\nr_{L^2_t}+\eps\nl\mathcal D_f\nr_{L^2_t}\\
\nl 
\nl\AA_\al \pt^\al \Gamma(f,f)
\nr_{L^2_{t,x,\xi}}
\nr_{\ell_\al^2}&\lesssim \nl\AA_\al \nl \P f_\al
\nr_{L^\infty_t L^6_x}
\nr_{\ell_\al^2} \nl\AA_\al \nl \P f_\al
\nr_{L^2_t L^3_x}
\nr_{\ell_\al^2}\\
&\quad +\eps\kappa^{\frac12}\nl  \mathcal D_f\nr_{L^2_t}\nl \mathcal H_{\rho,f}
\nr_{L^\infty_t}\\
\eps\kappa^{-1}\nl
\AA_\al 
\nl\pt^\al \Gamma(\ff,f)
\nr _{L^2_{t,x,\xi}}
\nr_{\ell_\al^2} &\lesssim \eps^{0.1}\nl\mathcal E_f
\nr _{L^2_t}+o(1)\nl\mathcal D_f\nr_{L^2_t}\\
\eps\kappa^{\frac12}\nl \AA_\al \nl \pt^\al\lw\{\mu^{-\frac12}(\pt_t+\eps^{-1} \xi\cdot\nabla_x)\sqrt\mu)f\rw\}
\nr_{L^2_{t,x,\xi}}\nr_{\ell_\al^2}
&\lesssim \eps\kappa^{\frac 12}\nl\mathcal H_{\rho,f}
\nr_{L^\infty_t} +\eps\kappa^{\frac 12}
\nl \mathcal E_f\nr_{L^2_t}+\eps^2\kappa\nl \mathcal D_f\nr_{L^2_t}.
\enda 
\]
The proof is complete.
\end{proof}
\section{Proof of the main theorem}
Fix $p\in (2,3)$ and denote $\lambda\in (0,1)$ such that $\frac 1 3=\frac \lambda 6+\frac {1-\lambda}{p}$. By the interpolation and the discrete Holder inequality, we have
\[\nl \AA_\al
\nl
\P f_\al\nr_{L^2_t L^3_x}\nr_{\ell_\al^2}\lesssim \nl\AA_\al \nl\P f_\al\nr_{L^\infty_tL^6_x}\nr_{\ell_\al^2}^\lambda \nl \AA_\al\nl \P f_\al
\nr_{L^2_tL^p_x}
\nr_{\ell_\al^2}^{1-\lambda}.
\] We define the quantity
\[
\zeta_f(t)=\nl \nl \AA_\al \P f_\al
\nr_{L^\infty_tL^6_x}
\nr_{\ell_\al^{2}}  +\nl \nl \AA_\al \P f_\al
\nr_{L^2_t L^p_x} 
\nr_{\ell_\al^2}.\]
From the above inequality, we see that $\nl \AA_\al
\nl
\P f_\al\nr_{L^2_t L^3_x}\nr_{\ell_\al^2}\lesssim \zeta_f(t)
$.
From Theorem \ref{l2-ana} and Proposition \ref{G-ff}, we get

\[\bega
&\mathcal E_f(t)^2+\nl\mathcal D_f\nr_{L^2_t}^2+\eps^{-1} \nl \AA_\al\nl (\II-\P_{\gamma_+})f_\al\nr_{L^2_{t,x,\xi}}
\nr _{\ell_\al^2}^2\\
&\lesssim\delta\kappa^{-\frac 12}\lw(\nl \mathcal D_f\nr_{L^2_t}+\mathcal E_f(t)\rw)\nl \nl \AA_\al \P f_\al
\nr_{L^\infty_tL^6_x}
\nr_{\ell_\al^{2}} \nl \nl \AA_\al \P f_\al
\nr_{L^2_t L^3_x} 
\nr_{\ell_\al^2}\\
&\quad+\delta \eps\nl \mathcal D_f\nr_{L^2_t}^2\sup_{0\le s\le t} \mathcal H_{\rho,f}(s)\\
&\quad+\nl\AA_\al\nl \pt^\al \IP\ff\nr_{L^2_tL^2_{\gamma_-}}\nr_{\ell_\al^2}^2+\frac 1{\eps\delta}
\nl|\div U\nr|_{\tau_0,2}\int_0^t \mathcal E_f(s)ds+\frac{\eps^2}{\kappa}\mathcal H_{\rho,f}(t)^2.\\\enda
\]
This implies, under the assumption $(\eps\delta)^{-1}\||\div U\||_{\tau_0,2}+\delta^{-1}\||\NS(U)\||_{\tau_0,2}\lesssim \eps^{0.99}$ that 
\[\bega
&\mathcal E_f(t)^2+\nl\mathcal D_f\nr_{L^2_t}^2\lw\{1-(\delta\kappa^{-\frac 12} \zeta_f(t)^2)^2-\delta\eps \nl \mathcal H_{\rho,f}\nr_{L^\infty_t}
\rw\}\\
&+\frac 1 \eps \nl \AA_\al\nl (\II-\P_{\gamma_+})f_\al\nr_{L^2_{t,x,\xi}}
\nr _{\ell_\al^2}^2+\nl\mathcal Y_f
\nr_{L^2_t}\\
&\lesssim O(1) \nl \mathcal E_f
\nr_{L^2_t}^2+\delta^2\kappa^{-1}\zeta_f(t)^4+\eps^2\kappa^{-1}\mathcal H_{\rho,f}^2+\eps^{0.99}.
\enda\]
Under the assumption that 
\beq\label{boot-2}
\delta\kappa^{-\frac 12}\zeta_f(t)+ \eps \nl \mathcal H_{\rho,f}
\nr_{L^\infty_t} +\eps^{-\frac 12}\nl \nl \AA_\al(\II-\P_{\gamma_+})f_\al
\nr_{L^2_{t,x,\xi}}
\nr_{\ell_\al^2}
\ll \delta,
\eeq
we will have 
\beq\label{boot-1}
\mathcal E_f(t)+\nl \mathcal D_f
\nr_{L^2_t}+\nl\mathcal Y_f
\nr_{L^2_t}\lesssim \delta.
\eeq
It suffices to check the compatibility of the above conditions, and the result follows by bootstrap argument.
As for $L^\infty$ analytic norm $\mathcal H_{\rho,f}$, we recall Theorem \ref{l-inf}
\[\bega 
\mathcal H_{\rho,f}(t)&\lesssim  \mathcal H_h|_{t=0}+\eps^{-\frac 12} \kappa^{-1}\lw\{\nl\mathcal D_f 
\nr_{L^2_t}+\nl\pt_t \mathcal D_f\nr_{L^2_t}
\rw\}
+\eps^{0.99}.\enda 
\]
It suffices to have estimate on $\zeta_f(t)$. From Theorem \ref{thmLp}, we obtain 
\beq\label{ran-299}
\bega 
\nl \nl \AA_\al \P f_\al
\nr_{L^2_t L^p_x} 
\nr_{\ell_\al^2}&\lesssim \nl \AA_\al \nl \pt^\al f_0
\nr_{L^2_\gamma}
\nr_{\ell_\al^2}+\nl \mathcal E_f\nr_{L^\infty_t} +\kappa^{-\frac 12}\nl \mathcal D_f\nr_{L^2_t}+\kappa^{-\frac 12} \nl \mathcal E_f\nr_{L^2_t}\\
&\quad+\eps\kappa^{\frac12}\nl  \mathcal D_f\nr_{L^2_t}\nl \mathcal H_{\rho,f}
\nr_{L^\infty_t}
+\eps^{0.99}+\delta\kappa^{-\frac 12}\zeta_f(t)^2.\enda 
\eeq
At the same time, from Theorem \ref{L6-final}, we obtain  
\[\bega
&\nl
\AA_\al \nl \P f_\al
\nr_{L^\infty_t L^6_x}
\nr_{\ell_\al^2}\\
& \lesssim \eps^{\frac 1 3}\kappa^{\frac 1 6} \nl \mathcal H_{\rho,f}\nr_{L^\infty_t}^{\frac 2 3} \lw\{\nl\mathcal D_f 
\nr_{L^2_t}+\nl\pt_t \mathcal D_f\nr_{L^2_t}
\rw\}
^{\frac 1 3} +\nl \mathcal Y_f
\nr_{L^2_t}+\nl \pt_t \mathcal Y_f
\nr_{L^2_t}+\kappa^{-\frac 12}\nl \mathcal D_f 
\nr_{L^\infty_t} \\
&\quad+\eps^{-\frac 12}\lw\{ \nl \AA_\al \nl (\II-\P_{\gamma_+})\pt_t f_\al
\nr_{L^2_tL^2_{\gamma_+}} 
\nr _{\ell_\al^2}+\nl \AA_\al \nl (\II-\P_{\gamma_+})f_\al
\nr_{L^2_tL^2_{\gamma_+}} 
\nr _{\ell_\al^2}
\rw\}\cdot\eps^{\frac 12}\nl \mathcal H_{\rho,f}\nr_{L^\infty_t}\\
&\quad+\eps^{0.99}+\eps\kappa^{-\frac 12}\lw\{\nl \mathcal D_f\nr_{L^2_t}+\nl \pt_t\mathcal D_f\nr_{L^2_t}\rw\}
+\delta\eps\lw\{ \nl \mathcal D_f \nr_{L^2_t}+\nl \pt_t\mathcal D_f \nr_{L^2_t}\rw\}\nl \mathcal H_{\rho,f}\nr_{L^\infty_t}\\
&\quad+\delta\kappa^{-\frac12}\zeta_f(t)\lw\{
\zeta_f(t)+ \nl \AA_\al \nl \P \pt_t  f_\al\nr_{L^2_tL^3_x}
\nr _{\ell_\al^2}\rw\}.
\enda 
\]
Combining the above with \eqref{ran-299}, the assumptions \eqref{boot-1} and \eqref{boot-2}, we have 
\[\bega 
&\zeta_f(t)\lesssim  \nl \AA_\al \nl \pt^\al f_0
\nr_{L^2_\gamma}
\nr_{\ell_\al^2}+\delta\kappa^{-\frac 12}+\eps^{0.99}+\delta\kappa^{-\frac 12}\zeta_f(t)^2+\nl \pt_t \mathcal Y_f
\nr_{L^2_t}+\kappa^{-\frac 12}\nl\pt_t \mathcal D_f
\nr_{L^2_t} \\
&+\eps^{-\frac 12} \nl \AA_\al \nl (\II-\P_{\gamma_+})\pt_t f_\al
\nr_{L^2_tL^2_{\gamma_+}} 
\nr _{\ell_\al^2}\cdot \eps^{-\frac 12}\delta+\eps\kappa^{-\frac 12} \nl\pt_t\mathcal D_f
\nr_{L^2_t}
\\
&+\delta\kappa^{-\frac 12}\zeta_f(t)^2+\delta\kappa^{-\frac 12}\zeta_f(t) \zeta_{\pt_tf}(t).
\enda 
\]
The estimate \eqref{boot-2} for $\zeta_f(t)$ is satisfied, given 
\[\bega 
&\kappa^{-\frac 12}\nl \AA_\al \nl \pt^\al f_0
\nr_{L^2_\gamma}
\nr_{\ell_\al^2}+\delta\kappa^{-1}+\kappa^{-\frac 12}\nl\pt_t \mathcal Y\nr_{L^2_t}+\kappa^{-1}\nl\pt_t \mathcal D_f\nr_{L^2_t}+\delta\kappa^{-\frac 12}\zeta_{\pt_t f}(t)\\
&\quad+\delta^{-1}\eps^{-\frac 12} \nl \AA_\al \nl (\II-\P_{\gamma_+})\pt_t f_\al
\nr_{L^2_tL^2_{\gamma_+}} 
\nr _{\ell_\al^2}+\delta^2\eps^{-\frac 12}\lesssim 1.
\enda \]
We can take $\eps=\kappa^{2M}$ and $\delta=\kappa^M$ for $M$ large. The fluid solution is selected so that 
\[
\nl \div U\nr_{\tau_0,2}+\nl\NS (U)\nr_{\tau_0,2}\lesssim \kappa^{100M}.
\]
We see that we need control on the analytic norm of $\pt_t f$. The idea is that the equation for $\pt_t f$ now is a linear equation, given the analytic bound of $f$. 
We have 
\[
\bega 
&\pt_t \pt_t f_\al +\frac 1 \eps \xi\cdot\nabla_x \pt_t f_\al +\frac 1{\kappa\eps^2}L\pt_t f_\al +\frac{(\pt_t +\frac 1\eps \xi\cdot\nabla_x)\sqrt\mu}{\sqrt\mu}\pt_t f_\al\\
&=\lw\{\frac{(\pt_t +\frac 1\eps \xi\cdot\nabla_x)\sqrt\mu}{\sqrt\mu}\pt_t f_\al-\pt_t\pt^\al \lw\{\frac{(\pt_t +\frac 1\eps \xi\cdot\nabla_x)\sqrt\mu}{\sqrt\mu}\rw\} f\rw\}\\
&\quad-\frac 1{\kappa\e^2}[\pt_t\pt^\al ,L]f-\frac{1}{\eps\delta }\pt_t\pt^\al\lw\{\mu^{-\frac 12}
\textbf{R}_a
\rw\}+\frac{2\delta}{\kappa\eps}\pt^\al \Gamma(\pt_t f,f)+\frac{2}{\kappa}\pt^\al \Gamma(\ff,\pt_t f)+\frac2\kappa\pt^\al\Gamma(\ff,\pt_tf).\\
\enda
\]
Using Theorem \ref{l2-ana}, we get 
\[\bega 
&\nl\AA_\al \nl \pt_t f_\al\nr_{L^2_{t,x,\xi}}\nr_{\ell_\al^2}^2+O(1)\nl \mathcal D_{\pt_tf}\nr_{L^2_t}^2+\eps^{-1}\nl\AA_\al \nl (\II-\P_\gamma)\pt_t f_\al
\nr_{L^2_tL^2_\gamma}
\nr_{\ell_\al^2}^2\\
&\lesssim o(1)\mathcal E_f(t) \nl\AA_\al \nl \pt_t f_\al\nr_{L^2_{t,x,\xi}}\nr_{\ell_\al^2}+\kappa^{-\frac12}\nl \mathcal D_f\nr_{L^2_t}^2+\kappa^{-\frac12}\nl \mathcal E_f\nr_{L^2_t}\nl \mathcal D_f\nr_{L^2_t}+o(1)\nl \mathcal D_{\pt_tf}\nr_{L^2_t}^2+\nl \mathcal E_{\pt_tf}\nr_{L^2_t}^2\\
&\quad+\delta\kappa^{-\frac12} \lw(\nl \mathcal D_{\pt_t}f\nr_{L^2_t}+\mathcal E_{\pt_tf}
\rw) 
\nl \nl \AA_\al \P \pt_t f_\al\nr_{L^2_tL^3_x}
\nr_{\ell_\al^2}\nl \nl\AA_\al\P f_\al\nr_{L^2_tL^3_x}\nr_{\ell_\al^2}.
\enda
\]
The proof is complete, since there is no loss of derivative coming from the nonlinear term of the right hand side. This gives us full control on $\nl \mathcal D_{\pt_t f}\nr_{L^2_t}$. Finally, we note that 
\[
\nl \mathcal D_f\nr_{L^\infty_t}\lesssim \nl \mathcal D_f\nr_{L^2_t}+\nl\pt_t\mathcal D_f\nr_{L^2_t}\lesssim \nl \mathcal D_f\nr_{L^2_t}+\nl\mathcal D_{\pt_tf}\nr_{L^2_t}+\eps^{0.99}\nl \mathcal E_f\nr_{L^2_t}.
\]
The proof is complete.
\section{Appendix}
\subsection{Elementary lemmas}
\begin{lemma}[\cite{MR1182792}]\label{lem-Faa}
Let $R>0$ and $n\ge 1,n\in\mathbb N$. There holds
\beq\label{faa-iden1}
\sum_{k_1+2k_2+\cdots+nk_n=n}\frac{(k_1+k_2+\cdots+k_n)!}{k_1!k_2!\cdots k_n!}R^{k_1+k_2+\cdots+k_n}=R(R+1)^{n-1}.
\eeq
\end{lemma}
\begin{proof} We first show \eqref{faa-iden1}. We define 
\[
g(x)=\frac 1{1-x},\qquad f(x)=\frac{1}{1-R(x-1)},\qquad h(x)=f(g(x))=\frac {1-x}{1-(R+1)x}.
\]
It is clear that
\beq\label{ran-100-1}
\bega
g(x)&=\sum_{n=1}^\infty x^n=\sum_{n=1}^\infty \frac{g^{(n)}(0)}{n!}x^n\\
 f(x)&=\sum_{n=1}^\infty R^nx^n=\sum_{n=1}^\infty \frac{f^{(n)}(0)}{n!}x^n\\
  h(x)&=1+\sum_{n=1}^\infty R(R+1)^{n-1}x^n.
\enda 
\eeq
At the same time, by the chain rule, we have 
\[
h^{(n)}(x)=\sum_{k_1+2k_2+\cdots+nk_n=n}\frac{n!}{k_1!k_2!\cdots k_n!}f^{(k_1+k_2+\cdots+k_n)}(g(x))\prod_{i=1}^n \lw(\frac{g^{(i)}(x)}{i!}
\rw)^{k_i}.
\]
Putting $x=0$ in the above equation and using \eqref{ran-100-1}, we have, for $n\ge 1$:
\[
R(R+1)^{n-1}=\sum_{k_1+2k_2+\cdots+nk_n=n}\frac{(k_1+k_2+\cdots+k_n)!}{k_1!k_2!\cdots k_n!}R^{k_1+k_2+\cdots+k_n}.
\]
This proves \eqref{faa-iden1}. 
\end{proof}
\begin{lemma}\label{lem-Faa2} Let $d\in\mathbb N, d\ge 2$ and $R>0$. There exists a constant $C_d>0$ only depending on $d$ such that
\beq\label{faa-iden2}
\sup_{\substack{\al\in\mathbb N^d\\|\al|\ge 1
}
}\sum_{\substack{k_1,k_2,\cdots k_s\in \mathbb N\\
\beta_1,\beta_2,\cdots, \beta_s\in\mathbb N^d\\
k_1 \beta_1+k_2 \beta_2+\cdots+k_s  \beta_s= \al
}
}\frac{(k_1+k_2+\cdots+k_s)!}{k_1!k_2!\cdots k_s!}R^{k_1+k_2+\cdots+k_s}\le C_d \min\{R,1\} e^{2^d R}.
\eeq
\end{lemma}
\begin{proof}
We define the functions $g:\{x\in\mathbb \R^d:\quad \|x\|<1\}\to \R$ and $f:\mathbb R-\{1\}\to \mathbb R$ such that
\[
g(x)=\sum_{\al\in\mathbb N_0^d} x^\al=\prod_{i=1}^d \frac{1}{1-x_i},\qquad f(t)=\frac 1{1-R(t-1)}=(1+R-Rt)^{-1},
\]
where 
\[
x^\al=x_1^{\al_1}x_2^{\al_2}\cdots x_d^{\al_d},\qquad x\in\R^d,\quad \al\in \mathbb N_0^d.
\]
For $n\ge 1$, we have 
\[
f^{(n)}(t)=n!R^n(1+R-Rt)^{-(n+1)}.
\]
Define \[
h(x)=f(g(x))=\frac{1}{1-R(g(x)-1)}\]
 Now using the Faà di Bruno's formula, we have
\[\bega 
&\pt^\al h(x) =\sum_{\substack{k_1,k_2,\cdots k_s\in \mathbb N\\
\beta_1,\beta_2,\cdots, \beta_s\in\mathbb N^d\\
k_1 \beta_1+k_2 \beta_2+\cdots+k_s  \beta_s= \al
}
}\frac{\alpha!}{k_1!k_2!\cdots k_s!}R^{k_1+k_2+\cdots+k_s}(k_1+k_2+\cdots+k_s)!\\
&\times(1+R-Rg(x))^{-(k_1+k_2+\cdots+k_s+1)} \lw\{\frac{1}{\beta_1!}\pt^{\beta_1}g
\rw\}^{k_1}\lw\{\frac{1}{\beta_1!}\pt^{\beta_2} g
\rw\}^{k_2}\cdots\lw\{\frac{1}{\beta_s!}\pt^{\beta_1} g
\rw\}^{k_s}.\\
\enda 
\]
Now putting $x=0$ in the above equation and using the fact that $\pt^\al g|_{x=0}=\al!$, we get
\beq\label{ran-101-2}
\sum_{\substack{k_1,k_2,\cdots k_s\in \mathbb N\\
\beta_1,\beta_2,\cdots, \beta_s\in\mathbb N^d\\
k_1 \beta_1+k_2 \beta_2+\cdots+k_s  \beta_s= \al
}
}\frac{(k_1+k_2+\cdots+k_s)!}{k_1!k_2!\cdots k_s!}R^{k_1+k_2+\cdots+k_s}=\frac{1}{\al!}\pt^\al h|_{x=0}.\eeq
Now we compute the left right side of the above, by expanding $h$ around $x=0$. Denote $$\vec c_0=(c_0,c_0,\cdots,c_0)\in\mathbb R^d.$$
We have 
\beq\label{ran-102-1}
h(\vec c_0)=\sum_{\substack{\al\in\mathbb N_0^3\\|\al|\ge 1}} \pt^\al h|_{x=0} \frac{c_0^{|\al|}}{\al!}
\eeq
Since $g(\vec c_0)=(1-c_0)^{-d}$, we have 
\beq
h(\vec c_0)=f(g(\vec c_0))=f((1-c_0)^{-d})=\frac{1}{1-R((1-c_0)^{-d}-1)}=\wtd f(\wtd g(c_0)),\eeq
where
\[
\wtd f(t)=\frac{1}{1-Rt},\qquad \wtd g(t)=(1-t)^{-d}-1.
\]
Now we have
\[
\wtd f(\wtd g(c_0))=\sum_{n=0}^\infty \frac{1}{n!} (f(\wtd g(c_0)))^{(n)}|_{c_0=0}c_0^n.\]
By the Faà di Bruno's formula, we have 
\[
(\wtd f(\wtd g(c_0))^{(n)}=\sum_{k=1}^n\sum_{
\substack{b_1+b_2+\cdots+b_n=k\\
b_1+2b_2+\cdots+nb_n=n}}\wtd f^{(k)}(\wtd g(c_0))\prod_{i=1}^n\lw\{\frac{\wtd g^{(i)}(c_0)}{i!}
\rw\}^{b_i}.
\]
Plugging $c_0=0$ in the above and using the fact that $\wtd f^{(n)}(0)=R^n, \wtd g^{(n)}(0)=d(d+1)\cdots (d+n-1)$, we have 
\[
(\wtd f(\wtd g(c_0))^{(n)}|_{c_0=0}=\sum_{k=1}^n \sum_{
\substack{b_1+b_2+\cdots+b_n=k\\
b_1+2b_2+\cdots+nb_n=n}}R^k \prod_{i=1}^n \lw\{\frac{d(d+1)\cdots (d+i-1)}{i!}
\rw\}^{b_i}.
\]
Hence 
\[h(\vec c_0)=\wtd f(\wtd g(c_0))=\sum_{n=1}^\infty \lw(\frac 1{n!} \sum_{k=1}^n \sum_{
\substack{b_1+b_2+\cdots+b_n=k\\
b_1+2b_2+\cdots+nb_n=n}}R^k \prod_{i=1}^n \lw\{\frac{d(d+1)\cdots (d+i-1)}{i!}
\rw\}^{b_i}
\rw)c_0^n.
\]
Combining the above with \eqref{ran-102-1}, we have
\[
\sum_{|\al|\ge 1}\frac{\pt^\al h|_{x=0}}{\al!}x^{|\al|}=\sum_{n=1}^\infty \lw(\frac 1{n!} \sum_{k=1}^n \sum_{
\substack{b_1+b_2+\cdots+b_n=k\\
b_1+2b_2+\cdots+nb_n=n}}R^k \prod_{i=1}^n \lw\{\frac{d(d+1)\cdots (d+i-1)}{i!}
\rw\}^{b_i}
\rw)x^n.
\]
Equation $x^n$ on both sides, we have 
\[
\sum_{|\al|=n}\frac{\pt^\al h|_{x=0}}{\al!}=\frac 1{n!} \sum_{k=1}^n \sum_{
\substack{b_1+b_2+\cdots+b_n=k\\
b_1+2b_2+\cdots+nb_n=n}}R^k \prod_{i=1}^n \lw\{\frac{d(d+1)\cdots (d+i-1)}{i!}
\rw\}^{b_i}.\]
Now combining the above with \eqref{ran-101-2}, we have, for $\al\in\mathbb N_0^d$ with $|\al|=n$:
\beq\label{exp-ana-1}
\bega 
&\sum_{\substack{k_1,k_2,\cdots k_s\in \mathbb N\\
\beta_1,\beta_2,\cdots, \beta_s\in\mathbb N^d\\
k_1 \beta_1+k_2 \beta_2+\cdots+k_s  \beta_s= \al
}
}\frac{(k_1+k_2+\cdots+k_s)!}{k_1!k_2!\cdots k_s!}R^{k_1+k_2+\cdots+k_s}\\
&\le \frac 1{n!} \sum_{k=1}^n \sum_{
\substack{b_1+b_2+\cdots+b_n=k\\
b_1+2b_2+\cdots+nb_n=n}}R^k \prod_{i=1}^n \lw\{\frac{d(d+1)\cdots (d+i-1)}{i!}
\rw\}^{b_i}.\enda
\eeq 
Now we note that
\[
\frac{d(d+1)\cdots (d+i-1)}{i!}=\binom{d+i-1}{i}\le \sum_{i=0}^d\binom{d+i-1}{i}= 2^{d+i-1}.
\]
Hence 
\[
\prod_{i=1}^n \lw\{\frac{d(d+1)\cdots (d+i-1)}{i!}
\rw\}^{b_i}\le 2^{dk+n-k}=2^{k(d-1)}2^n.
\]
Combining the above with \eqref{exp-ana-1}, we get
\[\bega 
&\sum_{\substack{k_1,k_2,\cdots k_s\in \mathbb N\\
\beta_1,\beta_2,\cdots, \beta_s\in\mathbb N^d\\
k_1 \beta_1+k_2 \beta_2+\cdots+k_s  \beta_s= \al
}
}\frac{(k_1+k_2+\cdots+k_s)!}{k_1!k_2!\cdots k_s!}R^{k_1+k_2+\cdots+k_s}\\
&\le \frac {2^n} {n!}  \sum_{
\substack{
b_1+2b_2+\cdots+nb_n=n}}\lw( 2^{d-1} R\rw)^{b_1+b_2+\cdots+b_n}\\
&\le \frac {2^n}{n!}\sum_{b_1+2b_2+\cdots+nb_n=n}\frac{(b_1+b_2+\cdots+b_n)!}{b_1!b_2!\cdots b_n!}(2^{d-1}R)^{b_1+b_2+\cdots+b_n}\\
&=\frac{2^n}{n!} 2^{d-1}R(1+2^{d-1}R)^{n-1}\lesssim \min\{R,1\} \frac{2^n}{n!} (1+2^{d-1}R)^n\\
&\lesssim \min\{R,1\}\frac{1}{n!}(2+2^d R)^n\lesssim \min\{ R,1\} e^{2^d R},
\enda \]
where we used \eqref{faa-iden1} in the above. The proof is complete.
\end{proof}
\begin{lemma}\label{composite}[Analyticity estimate for composite functions]. Let $H:\mathbb R^d\to \R$ and $G:\mathbb \R\to\R$ satisfy 
\beq\label{GH}
\sup_{\substack{\al\in  \mathbb N_0^d\\|\al|\ge 1
}}\frac{\tau_H^{|\al|}}{\al !}\nl \pt^\al H\nr_{L^\infty_x}\lesssim R,\qquad \sup_{\al\in\mathbb N_0^3}\frac{\tau_G^{|\al|}}{\al!}\nl \pt^\al G
\nr_{L^\infty_x}\lesssim 1.\eeq
There holds
\[
\sum_{\substack{\al\in\mathbb N_0^d\\|\al|\ge 1}}\frac{\tau_H^{|\al|}}{\al!} \nl \pt^\al  (G(H(x))\nr_{L^\infty_x}\lesssim \min\{R,1\}e^{\tau_H+2^d\tau_G^{-1}R}.
\]
\end{lemma}
\begin{proof}By the Faà di Bruno's formula, we have 
\[\bega 
\pt^\al \lw(G(H(x))\rw)&=\sum_{\substack{k_1,k_2,\cdots k_s\in \mathbb N\\
\beta_1,\beta_2,\cdots, \beta_s\in\mathbb N^d\\
k_1 \beta_1+k_2 \beta_2+\cdots+k_s  \beta_s= \al
}
}\frac{\alpha!}{k_1!k_2!\cdots k_s!}\pt_X^{k_1+k_2+\cdots+k_s}G|_{X=H(x)}\\
&\quad\times\lw\{\frac{1}{\beta_1!}\pt^{\beta_1} H
\rw\}^{k_1}\lw\{\frac{1}{\beta_1!}\pt^{\beta_2} H
\rw\}^{k_2}\cdots\lw\{\frac{1}{\beta_s!}\pt^{\beta_1} H
\rw\}^{k_s}.\\
\enda
\]
Combining the above with \eqref{GH}, we have
\[
\frac{\tau_H^{|\al|}}{\al !}|\pt^\al G(H(x))|\le \sum_{\substack{k_1,k_2,\cdots k_s\in \mathbb N\\
\beta_1,\beta_2,\cdots, \beta_s\in\mathbb N^d\\
k_1 \beta_1+k_2 \beta_2+\cdots+k_s  \beta_s= \al
}
}\frac{(k_1+k_2+\cdots+k_s)!}{k_1!k_2!\cdots k_s!}\lw(\frac{R}{\tau_G}
\rw)^{k_1+k_2+\cdots+k_s}.
\]
Now using Lemma \ref{lem-Faa2} for the right hand side in the above, we get
\[
\sup_{\substack{\al\in\mathbb N_0^3\\|\al|\ge 1}}\frac{\tau_H^{|\al|}}{\al !}|\pt^\al \lw\{G(H(x)\rw\}|\lesssim \min\{R,1\} e^{2^d \tau_G^{-1}R}.
\]
The proof is complete, by summing over $\al$.
 \end{proof}
\begin{lemma}\label{exp-ana-ele} Let $A>0$, $R_0\in (0,1)$. There holds  
\[
\sup_{n\ge 0}\frac{R_0^n}{n!} \lw|\pt_s^n (e^{-As^2})
\rw|\le  e^{-\frac{A}{2}s^2},
\]
and 
\[
\sup_{n\ge 0}\frac{R_0^n}{n!}|\pt_s^n(e^{As^2})|\lesssim e^{3As^2}.
\]
\end{lemma}
\begin{proof} We have 
\[\bega 
|\pt_s^n (e^{-As^2})|&=\lw|\sum_{k=1}^n \sum_{\substack{b_1+2b_2=n\\b_1+b_2=k}}e^{-As^2}\frac{n!}{b_1!b_2!}(-2As)^{b_1}(-A)^{b_2}\rw|\\
&\le e^{-As^2} \sum_{k=1}^n \frac {n!} {k!}\sum_{b_1+b_2=k}\frac{k!}{b_1!b_2!}(2A|s|+A)^{b_1+b_2}\\
&\le e^{-As^2}\sum_{k=1}^n \frac{n!}{k!}(2A|s|+2A)^k.
\enda\]
Hence 
\[
\frac{R_0^n}{n!}|\pt_s^n(e^{-As^2})|\lesssim R_0^n e^{-As^2+2A|s|}\lesssim e^{-\frac A 2s^2}.
\]
The proof is complete. The second inequality is similar.
\end{proof}

\begin{lemma}\label{exp-ana}
Let $d\ge 3$ and $f:\mathbb R^d\to \mathbb R$ such that
\[
\sup_{|\al|\ge 1}\frac{\tau_f^{|\al|}}{\al!}|\pt^\al f(x)|\lesssim A_f\qquad \text{for all}\quad x\in\mathbb R^d
\]
There holds 
\[
\sup_{|\al|\ge 1}\frac{\tau_f^{|\al|}}{\al!}|\pt^\al (e^{sf(x)})|\lesssim A_f \exp\lw(sf(x)+2^{d-1}A_f|s|
\rw).
\]
\end{lemma}

\begin{proof} By the Faà di Bruno's formula, we have 
\[\bega 
\pt^\al &(e^{sf(x)})=\sum_{\substack{k_1,k_2,\cdots,k_\ell\in\mathbb N\\\beta_1,\beta_2,\cdots, \beta_\ell\in \mathbb N^d\\k_1\beta_1+k_2\beta_2+\cdots+k_\ell \beta_\ell=\al
}}\frac{\al!}{k_1!k_2!\cdots k_\ell!}s^{k_1+k_2+\cdots+k_\ell}e^{sf(x)}\prod_{i=1}^\ell \lw\{\frac{1}{\beta_i!}\pt^{\beta_i}f
\rw\}^{k_i}\\
&\lesssim \al! e^{sf(x)}\sum_{\substack{k_1,k_2,\cdots,k_\ell\in\mathbb N\\\beta_1,\beta_2,\cdots, \beta_\ell\in \mathbb N^d\\k_1\beta_1+k_2\beta_2+\cdots+k_\ell \beta_\ell=\al
}}\frac{(k_1+k_2+\cdots+k_\ell)!}{k_1!k_2!\cdots k_\ell!}(|s|A_f)^{k_1+k_2+\cdots+k_\ell}\tau_f^{-|\al|}.\\
\enda 
\]  
Now combining the above with Lemma \ref{lem-Faa2}, we have
\[
\sup_{|\al|\ge 1}\frac{\tau_f^{|\al|}}{\al!}|\pt^\al(e^{sf(x)})|\lesssim  e^{sf(x)}\min\{|s|A_f,1\}e^{2^d |s|A_f}\lesssim e^{sf(x)}A_f e^{2^{d-1} A_f|s|}.
\]
This completes the proof.
\end{proof}

\subsection{Bilinear estimates}
\begin{lemma}\label{bili-appen1}
There holds
\[\bega 
&\nl\nu^{-1/2}\int_{\R^3\times\S^2 }|(\xi-\xi_\star)\cdot\w|\lw|f(\xi_\star)\rw|\lw|g(\xi')\rw|\lw|h(\xi_\star')\rw|d\w d\xi_\star\nr_{L^2_{\xi}}\\
&\lesssim\nl \nu(\xi)^{-1}\lw(\int_{\R^3_{\xi_\star}}|\xi-\xi_\star|^2 |f(\xi_\star)|^2d\xi_\star
\rw)^{1/2}\nr_{L^\infty_\xi} \lw(\nl\sqrt\nu g
\nr_{L^2_\xi}\nl h
\nr_{L^2_\xi}+\nl g
\nr_{L^2_\xi}\nl \sqrt\nu h
\nr_{L^2_\xi}\rw),\enda 
\]
\[\bega 
&\nl\int_{\R^3\times\S^2}|(\xi-\xi_\star)\cdot\w|\lw|f(\xi_\star)\rw|\lw|g(\xi')\rw|\lw|h(\xi_\star')\rw|d\w d\xi_\star\nr_{L^2_{\xi}}\\
&\lesssim \nl\lw(\int_{\R^3_{\xi_\star}}|\xi-\xi_\star|^2 |f(\xi_\star)|^2d\xi_\star
\rw)^{1/2}\nr_{ L^\infty_\xi}\| g\|_{L^2_{\xi}}\| h\|_{L^2_{\xi}}.\enda 
\]
\[
\nl\int_{\R^3\times\S^2}|(\xi-\xi_\star)\cdot\w|f(\xi_\star)g(\xi)h(\xi_\star)d\w d\xi_\star\nr_{L^2_{\xi}}\lesssim \|h\|_{L^2_{\xi}}\nl \lw(\int_{\xi_\star}|\xi-\xi_\star|^2|f(\xi_\star)|^2d\xi_\star\rw)^{1/2} g
\nr_{ L^2_{\xi}}.
\]
\end{lemma}
\begin{proof}
We have 
\[\bega 
&\nu(\xi)^{-1/2}\int_{\R^3}\int_{|\w|=1}|(\xi-\xi_\star)\cdot\w|\lw|f(\xi_\star)\rw|\lw|g(\xi')\rw|\lw|h(\xi_\star')\rw|d\w d\xi_\star\\
&\lesssim \nu(\xi)^{-1}\lw(\int_{\R^3_{\xi_\star}}|\xi-\xi_\star|^2 |f(\xi_\star)|^2d\xi_\star
\rw)^{1/2}  \lw\|\int_{|\w|=1}\sqrt\nu(\xi)|g(\xi')||h(\xi_\star')|d\w 
\rw\|_{L^2_{\xi_\star}}.\\
\enda
\]
Hence 
\[\bega 
&\nl \nu(\xi)^{-1/2}\int_{\R^3}\int_{|\w|=1}|(\xi-\xi_\star)\cdot\w|\lw|f(\xi_\star)\rw|\lw|g(\xi')\rw|\lw|h(\xi_\star')\rw|d\w d\xi_\star\nr_{L^2_\xi}\\
&\lesssim\sup_{\xi\in \R^3}\lw\{\nu(\xi)^{-1}\lw(\int_{\R^3_{\xi_\star}}|\xi-\xi_\star|^2 |f(\xi_\star)|^2d\xi_\star
\rw)^{1/2}\rw\}\lw(\nl\sqrt\nu g
\nr_{L^2_\xi}\nl h
\nr_{L^2_\xi}+\nl g
\nr_{L^2_\xi}\nl \sqrt\nu h
\nr_{L^2_\xi}\rw).
\enda
\]
Here we used the fact that $\nu(\xi)\lesssim \nu(\xi')+\nu(\xi_\star')$ and
$d\xi d\xi_\star=d\xi' d\xi_\star'$. 
Next we have 
\[\bega 
&\int_{\R^3\times \S^2}|(\xi-\xi_\star)\cdot\w||f(\xi_\star)||g(\xi)||h(\xi_\star)|d\w d\xi_\star\\
&\lesssim |g(\xi)|\lw(\int_{\R^3}|\xi-\xi_\star|^2|f(\xi_\star)|^2d\xi_\star\rw)^{1/2}\|h\|_{L^2_{\xi}},\\
\enda\]
The proof is complete.
\end{proof}

\begin{lemma}\label{Gamma-bilinear}
Recall the definition of $\Gamma_{\pm}$ in \eqref{Gamma-pm}. There holds 
\[\bega 
\nl \nu^{-1/2}\Gamma_+(f,g)\nr_{L^2_\xi}&\lesssim \nl\sqrt\nu f
\nr_{L^2_\xi}\nl g
\nr_{L^2_\xi}+\nl f
\nr_{L^2_\xi}\nl \sqrt\nu g
\nr_{L^2_\xi},\\
\nl\nu^{-1/2}\Gamma_-(f,g)
\nr_{L^2_\xi}&\lesssim \|\sqrt\nu f\|_{L^2_\xi}\|g\|_{L^2_\xi},\\
\nl \nu^{-1/2}\Gamma(f,g)
\nr_{L^2_\xi}&\lesssim \nl\sqrt\nu f
\nr_{L^2_\xi}\nl g
\nr_{L^2_\xi}+\nl f
\nr_{L^2_\xi}\nl \sqrt\nu g
\nr_{L^2_\xi}.\enda 
\]
\end{lemma}
\begin{proof}
Using Lemma \eqref{bili-appen1}, we get 
\[
\nl\nu^{-1/2}\Gamma_+(f,g)
\nr_{L^2_\xi}\lesssim \nl\sqrt\nu f
\nr_{L^2_\xi}\nl g
\nr_{L^2_\xi}+\nl f
\nr_{L^2_\xi}\nl \sqrt\nu g
\nr_{L^2_\xi}.
\]
Now we apply Lemma \ref{bili-appen1}, we have 
\[\bega 
&\nl\int_{\R^3\times\S^2}|(\xi-\xi_\star)\cdot\w|\sqrt\mu(\xi_\star)\frac{f}{\sqrt\nu}(\xi)g(\xi_\star)d\w d\xi_\star\nr_{L^2_{\xi}}\\
&\lesssim \|g\|_{L^2_{\xi}}\nl \lw(\int_{\xi_\star}|\xi-\xi_\star|^2|\sqrt\mu(\xi_\star)|^2d\xi_\star\rw)^{1/2} \frac{f}{\sqrt\nu}
\nr_{ L^2_{\xi}}\lesssim \|\sqrt\nu f\|_{L^2_\xi}\|g\|_{L^2_\xi}. 
\enda 
\]
The proof is complete.
\end{proof}

 \begin{proposition}\label{bilinear-1}
 Recall the definition of $\Gamma$ in \eqref{Gamma-pm}
 There holds 
 \[\bega 
&\la\pt^\al \Gamma(f,g),h\ra_{L^2_{\xi}}\\
&\lesssim  \sum_{\beta+\gamma=\al}\frac{\al !}{\beta !\gamma!}\lw(\nl \sqrt\nu \pt^\beta  f
   \nr_{L^2_\xi}\nl  \pt^\gamma g
   \nr_{L^2_\xi}+\nl  \pt^\beta f
   \nr_{L^2_\xi}\nl \sqrt\nu \pt^\gamma g
   \nr_{L^2_\xi}\rw)\nl\sqrt\nu \IP h\nr_{L^2_\xi}\\
   &+\|h\|_{L^2_\xi}\sum_{\substack{\beta+\gamma+\vr=\al\\|\beta|\ge 1}}\frac{\al!}{\beta!\gamma!\rho!} \sup_\xi \lw\{
   \nu(\xi)^{-1}
   \nl |\xi-\xi_\star|\pt^\beta \sqrt\mu(\xi_\star)
   \nr _{L^2_{\xi_\star}} \rw\}\\
   &\times
\lw(\|\nu \pt^\gamma f\|_{L^2_\xi}\|\pt^\vr g\|_{L^2_\xi}+\|\pt^\gamma f\|_{L^2_\xi}\|\nu \pt^\vr g\|_{L^2_\xi}
 \rw).
   \enda 
 \]
 \end{proposition}
 \begin{proof}
 We have 
  \[\bega
\pt^\al\Gamma_+(f,g)&=\sum_{\beta+\gamma+\vr=\al}\frac{\al!}{\beta!\gamma!\vr !}\int_{\R^3\times\S^2}|(\xi-\xi_\star)\cdot\w|\pt^\beta\sqrt\mu(\xi_\star)\lw(\pt^\gamma f(\xi')\pt^\rho g(\xi_\star')+\pt^\gamma g(\xi')\pt^\vr f(\xi_\star')
 \rw)d\w d\xi_\star\\
\pt^\al \Gamma_-(f,g)&=\sum_{\beta+\gamma+\vr=\al}\frac{\al!}{\beta!\gamma!\vr!}\int_{\R^3\times\S^2}|(\xi-\xi_\star)\cdot\w|\pt^\beta\sqrt\mu(\xi_\star)\lw(\pt^\gamma f(\xi)\pt^\vr g(\xi_\star)+\pt^\gamma g(\xi)\pt^\vr f(\xi_\star)
 \rw)d\w d\xi_\star.
\enda
 \]
This implies 
 \[\bega 
& \pt^\al \Gamma(f,g)\\
 &=\sum_{\beta+\gamma=\al}\frac{\al!}{\beta!\gamma!}\Gamma(\pt^\beta f,\pt^\gamma g)\\
 &\quad+\sum_{\substack{\beta+\gamma+\vr=\al\\|\beta|\ge 1}}\frac{\al!}{\beta!\gamma!\vr!}\int_{\R^3\times\S^2}|(\xi-\xi_\star)\cdot\w|\pt^\beta \sqrt\mu(\xi_\star)\lw(\pt^\gamma f(\xi')\pt^\vr g(\xi_\star')+\pt^\gamma g(\xi')\pt^\vr f(\xi_\star')
 \rw)d\w d\xi_\star\\
 &\quad+\sum_{\substack{\beta+\gamma+\vr=\al\\|\beta|\ge 1}}\frac{\al!}{\beta!\gamma!\vr!}\int_{\R^3\times\S^2}|(\xi-\xi_\star)\cdot\w|\pt^\beta \sqrt\mu(\xi_\star)\lw(\pt^\gamma f(\xi)\pt^\vr g(\xi_\star)+\pt^\gamma g(\xi)\pt^\vr f(\xi_\star)
 \rw)d\w d\xi_\star.\\
   \enda 
 \]
 For the first term, we have 
 \[
 \la \Gamma(\pt^\beta  f,\pt^\gamma g),h\ra_{L^2_{\xi}}=\la  \Gamma(\pt^\beta f,\pt^\gamma g),\IP h\ra_{L^2_{\xi}}\lesssim \nl\nu^{-1/2}\Gamma(\pt^\beta f,\pt^\gamma g)
 \nr_{L^2_{\xi}}\nl\sqrt\nu \IP h
 \nr_{L^2_{\xi}}.
   \]
   Now using Lemma \ref{Gamma-bilinear}, we have 
   \[
   \nl\nu^{-1/2}\Gamma(\pt^\beta f,\pt^\gamma g)
   \nr_{L^2_{\xi}}\lesssim \nl \sqrt\nu \pt^\beta f
   \nr_{L^2_\xi}\nl  \pt^\gamma g
   \nr_{L^2_\xi}+\nl  \pt^\beta f
   \nr_{L^2_\xi}\nl \sqrt\nu \pt^\gamma g
   \nr_{L^2_\xi}.
   \]
   Next we bound 
   \[
\int_{\R^3\times\S^2}|(\xi-\xi_\star)\cdot\w|\pt^\beta \sqrt\mu(\xi_\star)\lw(\pt^\gamma f(\xi')\pt^\rho g(\xi_\star')+\pt^\gamma g(\xi')\pt^\rho f(\xi_\star')
 \rw)d\w d\xi_\star   .
   \]
   Using similar ideas as in Lemma \ref{bili-appen1}, we obtain
      \[\bega 
&\nl\int_{\R^3\times\S^2}|(\xi-\xi_\star)\cdot\w|\pt^\beta\sqrt\mu(\xi_\star)\lw(\pt^\gamma f(\xi')\pt^\rho g(\xi_\star')+\pt^\gamma g(\xi')\pt^\rho f(\xi_\star')
 \rw)d\w d\xi_\star  \nr_{L^2_\xi}\\
  &\lesssim  \nl\nu(\xi)^{-1}\lw(\int_{\xi_\star}|\xi-\xi_\star|^2|\pt^\beta\sqrt\mu(\xi_\star)|^2d\xi_\star
 \rw)\nr_{L^\infty_\xi}\lw(\|\nu\pt^\gamma f\|_{L^2_\xi}\|\pt^\rho g\|_{L^2_\xi}+\|\pt^\gamma f\|_{L^2_\xi}\|\nu \pt^\rho g\|_{L^2_\xi}
 \rw).
 \enda
   \] 
   The proof is complete.
     \end{proof}

\begin{proposition}\label{BGR}
 Let $\beta,\gamma,\rho\in \mathbb N_0^3$ and $\al\in \mathbb N_0^3$ such that $\beta+\gamma+\vr=\al$ and $\vr<\al$. There holds, for a sequence of functions $\{T_\vr\}_{\vr\in\mathbb N_0^3}$, \[\bega 
&\lw\| \nu^{-\frac 12}
\int_{\R^3\times\S^2}|(\xi-\xi_\star)\cdot\w|\lw|\pt^\beta\sqrt\mu(\xi_\star)\rw|\lw|\pt^\gamma \sqrt\mu(\xi')\rw|\lw|T_\vr (\xi_\star')\rw|d\w d\xi_\star
\rw\|_{L^2_{\xi}}\\
&\lesssim \lw\|\nu^{\frac12}\pt^\gamma\sqrt\mu
\rw\|_{L^2_{\xi}}\cdot 1_{\{|\gamma|\ge 1,\gamma+\rho=\al\}} \cdot \lw\|\nu^{\frac12}T_\vr
\rw\|_{L^2_{\xi}}\\
&+\lw\|\nu^{-1}\lw(\int_{\R^3_{\xi_\star}}|\xi-\xi_\star|^2 |\pt^\beta\sqrt\mu(\xi_\star)|^2d\xi_\star\rw)^{1/2}
\rw\|_{L^\infty_{\xi}}\cdot 1_{\{|\beta|\ge 1\}}\cdot \lw\|\nu^{\frac12}\pt^\gamma \sqrt\mu
\rw\|_{L^2_{\xi}}\cdot\lw\|\nu^{\frac12}T_\vr
\rw\|_{L^2_{\xi}}.
\enda 
\]
\end{proposition}
\begin{proof} Using Lemma \ref{bili-appen1} and then consider $\beta\neq 0$ or $\beta=0$, we have 
\[\bega 
&\lw\| \nu^{-\frac 12}
\int_{\R^3\times\S^2}|(\xi-\xi_\star)\cdot\w|\lw|\pt^\beta\sqrt\mu(\xi_\star)\rw|\lw|\pt^\gamma\sqrt\mu(\xi')\rw|\lw|\IP \pt^\rho f(\xi_\star')\rw|d\w d\xi_\star
\rw\|_{L^2_{\xi}}\\
&\lesssim  \lw\|\nu^{-1}\lw(\int_{\R^3_{\xi_\star}}|\xi-\xi_\star|^2 |\pt^\beta\sqrt\mu(\xi_\star)|^2d\xi_\star\rw)^{1/2}
\rw\|_{L^\infty_{\xi}} \cdot 
1_{\{|\beta|\ge 1\}
}
\lw\|\nu^{\frac12} \pt^\gamma\sqrt\mu(\xi)\rw\|_{L^2_\xi}\lw\|\sqrt\nu \IP \pt^\rho f
\rw\|_{L^2_\xi}\\
&\quad+ \lw\|\nu^{\frac12}\pt^\gamma\sqrt\mu
\rw\|_{L^2_{\xi}}\cdot 1_{\{|\gamma|\ge 1\}
}
\lw\|\nu^{\frac12}\IP \pt^\rho f
\rw\|_{L^2_{\xi}}.\\
\enda
\]
Here, for $|\beta|=0$, we use lemma \ref{local-1} to get
\[
\lw(\int_{\R^3_{\xi_\star}}|\xi-\xi_\star|^2 |\sqrt\mu(\xi_\star)|^2d\xi_\star\rw)^{1/2}\lesssim \nu_0(\vp)=\nu(\xi).
\]
The second inequality is shown similarly. This completes the proof. \end{proof}

\begin{proposition}\label{BGR1} Let $\beta,\gamma,\vr\in \mathbb N_0^3$ and $\al\in \mathbb N_0^3$ such that $\beta+\gamma+\vr=\al$ and $\vr<\al$. There holds, for a sequence of functions $\{T_\vr\}_{\vr\in\mathbb N_0^3}$
\[\bega 
 &\lw\| \nu^{-\frac 12}
\int_{\R^3\times\S^2}|(\xi-\xi_\star)\cdot\w|\lw|\pt^\beta \sqrt\mu(\xi_\star)\rw|\lw|\pt^\gamma \sqrt\mu(\xi'_\star)\rw|\lw|T_\vr(\xi')\rw|d\w d\xi_\star
\rw\|_{L^2_{\xi}}\\
&\lesssim 1_{\{\gamma+\rho=\al,|\gamma|\ge 1
\}
}
\nl\nu^{\frac 12} \pt^\gamma\sqrt\mu
\nr_{L^2_\xi}\nl\nu^{\frac 12} T_\vr 
\nr_{L^2_{\xi}}\\
&\quad+1_{\{|\beta|\ge 1\}
}\lw\|\nu^{-1}\lw(\int_{\R^3_{\xi_\star}}|\xi-\xi_\star|^2 |\pt^\beta\sqrt\mu(\xi_\star)|^2d\xi_\star\rw)^{1/2}
\rw\|_{L^\infty_{\xi}}\lw\|\nu^{\frac 12} \pt^\gamma\sqrt\mu\rw\|_{L^2_\xi}\lw\|\nu^{\frac 12} T_\vr\rw\|_{L^2_{\xi}}.\enda 
\]
\end{proposition}
\begin{proof} Applying Lemma \ref{bili-appen1} and then consider $\beta\neq 0$ or $\beta=0$, we have  
\[\bega
& \lw\| \nu^{-1/2}
\int_{\R^3}\int_{|\w|=1}|(\xi-\xi_\star)\cdot\w|\lw|\pt^\beta \sqrt\mu(\xi_\star)\rw|\lw|\pt^\gamma \sqrt\mu(\xi'_\star)\rw|\lw|T_\vr(\xi')\rw|d\w d\xi_\star
\rw\|_{L^2_{\xi}}\\
&\lesssim 1_{\{\gamma+\rho=\al,|\gamma|\ge 1
\}
}
\nl\sqrt\nu \pt^\gamma\sqrt\mu
\nr_{L^2_\xi}\nl\sqrt\nu T_\vr\nr_{L^2_\xi}\\
&\quad+1_{\{|\beta|\ge 1\}
}\lw\|\nu^{-1}\lw(\int_{\R^3_{\xi_\star}}|\xi-\xi_\star|^2 |\pt^\beta\sqrt\mu(\xi_\star)|^2d\xi_\star\rw)^{1/2}
\rw\|_{L^\infty_{\xi}}\lw\|\sqrt\nu \pt^\gamma\sqrt\mu\rw\|_{L^2_\xi}\lw\|\sqrt\nu T_\vr
\rw\|_{L^2_\xi}.
\enda
\]
Here, we used Lemma \ref{local-1} for the case $|\beta|=0$. This completes the proof.
\end{proof}

\begin{proposition}\label{BGR2} Let $\beta,\gamma,\vr,\al\in \mathbb N_0^3$ such that $\beta+\gamma+\vr=\al$ and $\rho<\al$. There holds, for any sequence $\{T_{\vr}\}_{\vr\in\mathbb N_0^3}$, we have 
\[\bega 
&\nl
\nu^{-\frac12} \int_{\R^3}\int_{|\w|=1}|(\xi-\xi_\star)\cdot \w|\lw|\pt^\beta\sqrt\mu(\xi_\star)\rw|\lw|\pt^\gamma\sqrt\mu(\xi) \rw|\lw|T_\vr (\xi_\star)\rw|d\w d\xi_\star\nr_{L^2_{\xi}}\\
 &\lesssim  \nl \nu^{\frac12} T_\vr  \nr_{L^2_{\xi}}\nl\nu(\xi)^{-1}\lw(\int_{\xi_\star}\nu(\xi_\star)^{-1}|\xi-\xi_\star|^2 |\pt^\beta\sqrt\mu(\xi_\star)|^2d\xi_\star
 \rw)^{1/2} \nr_{L^\infty_{\xi}}\nl \pt^\gamma \sqrt\mu
 \nr_{ L^2_\xi}.\\
 \enda 
\]
and 
\[
\bega
&\nl\nu^{-\frac 12} \int_{\R^3\times\S^2}|(\xi-\xi_\star)\cdot \w|\pt^\beta\mu(\xi_\star)T_\gamma(\xi)d\w d\xi_\star\nr_{L^2_{\xi}}\\
&
\lesssim \nl \nu(\xi)^{-1}\int_{\R^3}\int_{|\w|=1}|(\xi-\xi_\star)\cdot \w|\pt^\beta\mu(\xi_\star)d\w d\xi_\star\nr_{L^\infty_\xi}
\nl \nu^{\frac 12}T_\gamma
\nr_{L^2_\xi}.\enda
 \]
 \end{proposition} 
\begin{proof}
Applying Lemma \ref{bili-appen1}, we have 
\[\bega 
&\nl
\nu^{-1/2} \int_{\R^3}\int_{|\w|=1}|(\xi-\xi_\star)\cdot \w|\pt^\beta\sqrt\mu(\xi_\star)\pt^\gamma\sqrt\mu(\xi) T_\vr (\xi_\star)d\w d\xi_\star\nr_{L^2_{\xi}}\\
 &\lesssim \nl \sqrt\nu T_\vr 
 \nr_{L^2_\xi}\nl\lw(\int_{\xi_\star}\nu(\xi_\star)^{-1}|\xi-\xi_\star|^2 |\pt^\beta\sqrt\mu(\xi_\star)|^2d\xi_\star
 \rw)^{1/2} \nu^{-1/2}\pt^\gamma\sqrt\mu 
 \nr_{L^2_\xi}\\
 &\lesssim  \nl \sqrt\nu T_\vr 
 \nr_{L^2_\xi}\nl\nu(\xi)^{-1}\lw(\int_{\xi_\star}\nu(\xi_\star)^{-1}|\xi-\xi_\star|^2 |\pt^\beta\sqrt\mu(\xi_\star)|^2d\xi_\star
 \rw)^{1/2} \nr_{L^\infty_\xi}\nl \pt^\gamma \sqrt\mu
 \nr_{L^2_\xi}.\\
  \enda 
\]
As for the second inequality, we have 
\[\bega 
&\nl \nu(\xi)^{-\frac12}\int_{\R^3}\int_{|\w|=1}|(\xi-\xi_\star)\cdot \w|\pt^\beta\mu(\xi_\star)d\w d\xi_\star\cdot T_\gamma  (\xi)\nr_{L^2_\xi}\\
&\lesssim \nl \nu(\xi)^{-1}\int_{\R^3}\int_{|\w|=1}|(\xi-\xi_\star)\cdot \w|\pt^\beta\mu(\xi_\star)d\w d\xi_\star\nr_{L^\infty_\xi}
\nl \nu^{\frac 12}T_\gamma\nr_{L^2_\xi}.
\enda
\]
The proof is complete.\end{proof}

\begin{lemma} \label{bili-2}There holds 
\[\bega 
&\nl\nu^{-1}\int_{\R^3\times\S^2 }|(\xi-\xi_\star)\cdot\w|\lw|f(\xi_\star)\rw|\lw|g(\xi')\rw|\lw|h(\xi_\star')\rw|d\w d\xi_\star\nr_{L^\infty_{\xi}}\\&\lesssim \nl \la \xi\ra f
\nr _{L^2_\xi} \min\{\nl g\nr_{L^\infty_\xi}\nl f\nr_{L^2_\xi},\nl f\nr_{L^\infty_\xi}\nl g\nr_{L^2_\xi},\}
\enda 
\]
and \[
\nl
\nu^{-1}\int_{\R^3\times\S^2}|(\xi-\xi_\star)\cdot\w|f(\xi_\star)g(\xi)d\w d\xi_\star 
\nr_{L^\infty_\xi}\lesssim \nl g
\nr_{L^\infty_\xi}\nl f\nr_{L^1_\xi}+\nl\nu^{-1}g
\nr_{L^\infty_\xi}\nl \la\xi\ra f\nr_{L^1_\xi}.\]
\end{lemma}

\begin{proof}
We have 
\[\bega 
&\nu(\xi)^{-1}\int_{\R^3\times\S^2 }|(\xi-\xi_\star)\cdot\w|\lw|f(\xi_\star)\rw|\lw|g(\xi')\rw|\lw|h(\xi_\star')\rw|d\w d\xi_\star\\
&\lesssim \nu(\xi)^{-1}\lw(\int_{\R^3}|\xi-\xi_\star|^2 |f(\xi_\star)|^2d\xi_\star\rw)^{\frac12}\lw(\int_{\R^3}\int_{|\w|=1}|g(\xi')|^2|h(\xi_\star')|^2 d\w d\xi_\star\rw)^{\frac12}\\
&\lesssim \nl \la \xi\ra f
\nr _{L^2_\xi} \min\{\nl g\nr_{L^\infty_\xi}\nl f\nr_{L^2_\xi},\nl f\nr_{L^\infty_\xi}\nl g\nr_{L^2_\xi}\}.
\enda 
\]
Next, we have 
\[\bega 
&\nu(\xi)^{-1}g(\xi) \int_{\R^3\times\S^2}|(\xi-\xi_\star)\cdot\w|f(\xi_\star) d\w d\xi_\star\\
& \lesssim \nu(\xi)^{-1}|g(\xi)|\int_{\R^3}(|\xi|+|\xi_\star|)|f(\xi_\star)|d\xi_\star \\
&\lesssim |g(\xi)|\nl f
\nr_{L^1_\xi}+\nu^{-1}(\xi)g(\xi)\nl\la\xi\ra f
\nr _{L^1_\xi}.
\enda 
\]
The proof is complete.
\end{proof}

\begin{proposition}\label{ana-gamma}
There holds, for any $p_0\ll 1$
\[\nl
\AA_\al \nl
 \nu^{-1} \pt^\al \Gamma(f,g)
\nr_{L^\infty_\xi}
\nr_{\ell_\al^2}\lesssim\nl\AA_\al  \nl e^{p_0|\xi|^2}\pt^\al f 
\nr_{L^\infty_\xi}
\nr_{\ell_\al^2}\cdot\nl\AA_\al  \nl e^{p_0|\xi|^2}\pt^\al g
\nr_{L^\infty_\xi}
\nr_{\ell_\al^2}.\]
\end{proposition}
\begin{proof} We recall from \eqref{Gamma-def} and then use Lemma \ref{bili-2} to get 
 \[\bega 
\nu^{-1}\pt^\al \Gamma_+(f,g)&\lesssim\sum_{\substack{\beta+\gamma+\vr=\al
}
}
\frac{\al!}{\beta!\gamma!\vr!}\nl\la \xi\ra\pt^\beta\sqrt\mu
\nr_{L^2_\xi}\min\lw\{\nl \pt^\gamma f\nr_{L^2_\xi}\nl \pt^\vr g
\nr _{L^\infty_\xi},\nl \pt^\gamma f\nr_{L^\infty_\xi}\nl \pt^\vr g
\nr _{L^2_\xi}\rw\}\\
&\lesssim \sum_{\substack{\beta+\gamma+\vr=\al
}
}
\frac{\al!}{\beta!\gamma!\vr!}\nl e^{\rho|\xi|^2}\pt^\beta\sqrt\mu
\nr_{L^\infty_\xi}\nl e^{\rho|\xi|^2} \pt^\gamma f\nr_{L^\infty_\xi}\nl e^{\rho|\xi|^2}\pt^\vr g
\nr _{L^\infty_\xi}.
\enda 
 \]
Similarly, we have 
 \[\bega
& \nu^{-1}\pt^\al \Gamma_-(f,g)\\
&\lesssim \sum_{\substack{\beta+\gamma=\al
}
}
\frac{\al!}{\beta!\gamma!} \lw\{\nl
\pt^\beta f
\nr _{L^\infty_\xi}\nl\pt^\gamma(\sqrt\mu g)
\nr_{L^1_\xi}+\nl\nu^{-1} \pt^\beta f
\nr_{L^\infty_\xi}\nl \la\xi\ra \pt^\gamma(\sqrt\mu g)
\nr _{L^1_\xi}
\rw\}\\
&\quad+ \sum_{\substack{\beta+\gamma=\al
}
}
\frac{\al!}{\beta!\gamma!} \lw\{\nl
\pt^\beta (\sqrt\mu g)\nr _{L^\infty_\xi}\nl\pt^\gamma f
\nr_{L^1_\xi}+\nl\nu^{-1} \pt^\beta (\sqrt\mu g)
\nr_{L^\infty_\xi}\nl \la\xi\ra \pt^\gamma f
\nr _{L^1_\xi}
\rw\}\\\enda
\]
The proof is complete, by using Proposition \ref{local-1} and the discrete Young inequality. 
\end{proof}
\begin{proposition}  \label{VIPG}
Recalling \eqref{EDH}, 
There holds 
\[\bega 
&\sum_\al \AA_\al^2\lw\la \pt^\al (V(x)\IP f(x,\xi)),\pt^\al g
\rw\ra_{L^2_{x,\xi}}\lesssim\eps\kappa^{\frac 12} \|| V
\||_{\tau_0,\infty} \mathcal E_f\lw(
\mathcal E_g+\mathcal D_g
\rw).
\enda 
\]
\end{proposition}

\begin{proof}
We have 
\[\bega 
&\la \pt^\al (V(x)\IP f),\pt^\al g
\ra _{L^2_{x,\xi}}=\sum_{\beta+\gamma=\al}\frac{\al!}{\beta!\gamma!}\la \pt^\beta V\pt^\gamma \lw\{\IP f
\rw\},\pt^\al g
\ra_{L^2_{x,\xi}}\\
&=\la \pt^\al V \IP,\pt^\al g\ra_{L^2_{x,\xi}}+\sum_{\substack{\beta+\gamma=\al\\|\gamma|\ge 1
}}\frac{\al!}{\beta!\gamma!}\la \pt^\beta V\pt^\gamma \lw\{\IP f
\rw\},\pt^\al g
\ra_{L^2_{x,\xi}}.\\\enda 
\]
Now for $|\gamma|\ge 1$, we have 
\[
\pt^\gamma (\IP f)=\pt^\gamma f-\pt^\gamma \P f=\pt^\gamma f-\P\pt^\gamma f+[\P,\pt^\gamma]f=\IP \pt^\gamma f+[\P,\pt^\gamma]f.
\]
This implies 
\[\bega 
&\la \pt^\al (V(x)\IP f),\pt^\al g
\ra _{L^2_{x,\xi}}\\
&=\la \pt^\al V \IP f, \IP \pt^\al g\ra_{L^2_{x,\xi}}+\sum_{\substack{\beta+\gamma=\al\\|\gamma|\ge 1
}}\frac{\al!}{\beta!\gamma!}\la\pt^\beta V \IP \pt^\gamma f,\IP \pt^\al g\ra_{L^2_{x,\xi}}\\
&\quad+\sum_{\substack{\beta+\gamma=\al\\|\gamma|\ge 1
}}\frac{\al!}{\beta!\gamma!}\la \pt^\beta V [\P,\pt^\gamma]f,\pt^\al g
\ra_{L^2_{x,\xi}}.
\enda 
\]
Hence we get 
\[\bega 
&\la \pt^\al (V(x)\IP f),\pt^\al g
\ra _{L^2_{x,\xi}}\\
&\lesssim \nl \pt^\al V
\nr_{L^\infty_x}\nl \IP f
\nr_{L^2_{x,\xi}}\nl \nu^{\frac12} \IP \pt^\al g \nr_{L^2_{x,\xi}}\\
&\quad+\sum_{\substack{\beta+\gamma=\al\\|\gamma|\ge 1
}}\frac{\al!}{\beta!\gamma!}\nl \pt^\beta V
\nr _{L^\infty_x}\nl \IP \pt^\gamma f
\nr_{L^2_{x,\xi}}\nl \nu^{\frac12}\IP\pt^\al g \nr_{L^2_{x,\xi}}\\
&\quad+\sum_{\substack{\beta+\gamma=\al\\|\gamma|\ge 1
}}\frac{\al!}{\beta!\gamma!}
\nl \pt^\beta V
\nr_{L^\infty_x}
\nl  [\P,\pt^\gamma]f\nr_{L^2_{x,\xi}}\nl \pt^\al g
\nr_{L^2_{x,\xi}}.\\\enda
\]
The proof is complete, by applying Lemma \ref{PD} and the discrete Young inequality.
\end{proof}
\subsection{Commutator estimates}
\begin{lemma} \label{com-1}Assume 
\eqref{Pf-def}. 
Then for all $m>0$, we have 
\beq\label{com-1-1}
\bega
&\nl\la\vp\ra^m [\pt^\al, \P]f\nr_{L^2_{x,\xi}}\\
&\lesssim \sum_{\substack{\beta+\gamma=\al\\|\beta|\ge 1}}\frac{\al!}{\beta!\gamma!}
\nl \la\vp\ra^m\pt^\beta \lw\{(1,\vp,|\vp|^2) \sqrt\mu\rw\}\nr_{L^\infty_x L^2_\xi} \nl \pt^\gamma f\nr_{L^2_{x,\xi}}\\
&\quad+\sum_{\substack{\beta+\gamma=\al\\|\beta|\ge 1}}\frac{\al!}{\beta!\gamma!}\nl\la \vp\ra^m\pt^\beta \lw\{(1,\vp,|\vp|^2) \sqrt\mu\rw\}
\nr_{L^\infty_x L^2_\xi}\nl \pt^\gamma (a,b,c) \nr_{L^2}.
\enda 
\eeq
As for derivatives, there hold 
\beq\label{com-1-2}
\bega
&\nl\la\vp\ra^m \pt_{x_3}[\pt^\al, \P]f\nr_{L^2_{x,\xi}}\\
&\lesssim \sum_{\substack{\beta+\gamma=\al\\|\beta|\ge 1}}\frac{\al!}{\beta!\gamma!} 
\lw(
\eps\kappa^{-\frac12}
\nl \pt^\beta \sqrt\mu \nr_{L^\infty_x L^2_\xi}+\nl \pt^\beta\pt_{x_3}\sqrt\mu\nr_{L^\infty_x L^2_\xi}\rw)
\nl \pt^\gamma f\nr_{L^2_{x,\xi}}\\
&\quad+ \sum_{\substack{\beta+\gamma=\al\\|\beta|\ge 1}}\frac{\al!}{\beta!\gamma!}\nl \la\vp\ra^m\pt^\beta \lw\{(1,\vp,|\vp|^2) \sqrt\mu\rw\}\nr_{L^\infty_x L^2_\xi}\nl \pt^\gamma \pt_{x_3} f\nr_{L^2_{x,\xi}}\\
&\quad+\sum_{\substack{\beta+\gamma=\al\\|\beta|\ge 1}}\frac{\al!}{\beta!\gamma!}\nl\la \vp\ra^m\pt^\beta \pt_{x_3} \lw\{(1,\vp,|\vp|^2) \sqrt\mu\rw\}
\nr_{L^\infty_x L^2_\xi}\nl \pt^\gamma (a,b,c) \nr_{L^2}\\
&\quad+\sum_{\substack{\beta+\gamma=\al\\|\beta|\ge 1}}\frac{\al!}{\beta!\gamma!}\nl\la \vp\ra^m\pt^\beta \pt_{x_3} \lw\{(1,\vp,|\vp|^2) \sqrt\mu\rw\}
\nr_{L^\infty_x L^2_\xi}\nl \pt^\gamma \pt_{x_3} (a,b,c) \nr_{L^2}.\enda 
\eeq
For $i\in \{1,2\}$, there holds 
\beq\label{com-1-3}
\bega
&\nl\la\vp\ra^m \pt_{x_i}[\pt^\al, \P]f\nr_{L^2_{x,\xi}}\\
&\lesssim \sum_{\substack{\beta+\gamma=\al\\|\beta|\ge 1}}\frac{\al!}{\beta!\gamma!} 
\lw(
\eps
\nl \pt^\beta \sqrt\mu \nr_{L^\infty_x L^2_\xi}+\nl \pt^\beta\pt_{x_i}\sqrt\mu\nr_{L^\infty_x L^2_\xi}\rw)
\nl \pt^\gamma f\nr_{L^2_{x,\xi}}\\
&\quad+ \sum_{\substack{\beta+\gamma=\al\\|\beta|\ge 1}}\frac{\al!}{\beta!\gamma!}\nl \la\vp\ra^m\pt^\beta \lw\{(1,\vp,|\vp|^2) \sqrt\mu\rw\}\nr_{L^\infty_x L^2_\xi}\nl \pt^\gamma \pt_{x_i} f\nr_{L^2_{x,\xi}}\\
&\quad+\sum_{\substack{\beta+\gamma=\al\\|\beta|\ge 1}}\frac{\al!}{\beta!\gamma!}\nl\la \vp\ra^m\pt^\beta \pt_{x_i} \lw\{(1,\vp,|\vp|^2) \sqrt\mu\rw\}
\nr_{L^\infty_x L^2_\xi}\nl \pt^\gamma (a,b,c) \nr_{L^2}\\
&\quad+\sum_{\substack{\beta+\gamma=\al\\|\beta|\ge 1}}\frac{\al!}{\beta!\gamma!}\nl\la \vp\ra^m\pt^\beta \pt_{x_i} \lw\{(1,\vp,|\vp|^2) \sqrt\mu\rw\}
\nr_{L^\infty_x L^2_\xi}\nl \pt^\gamma \pt_{x_i} (a,b,c) \nr_{L^2}.\enda 
\eeq
\end{lemma}
\begin{proof}
Recalling $a_\al,b_\al,c_\al$ defined in \eqref{PFal}. We have
\[\bega 
&[\pt^\al ,\P]f=\pt^\al (\P f)-\P(\pt^\al f)\\
&=\lw\{\pt^\al(a\sqrt\mu)-a_\al \sqrt\mu\rw\}+\lw\{\pt^\al (b\cdot\vp\sqrt\mu)-b_\al \cdot\vp\sqrt\mu\rw\}+\lw\{\pt^\al\lw(c\frac{|\vp|^2-3}{2}\sqrt\mu\rw)-c_\al \frac{|\vp|^2-3}{2}\sqrt\mu\rw\}\\
&=I_a+I_b+I_c.
\enda 
\]
where we denote, in this proof,
\[
\begin{cases}
I_a&=\pt^\al(a\sqrt\mu)-a_\al\sqrt\mu\\
I_b&=\pt^\al (b\cdot\vp\sqrt\mu)-b_\al \cdot\vp\sqrt\mu\\
I_c&=\pt^\al(c\frac{|\vp|^2-3}{2}\sqrt\mu)-c_\al \frac{|\vp|^2-3}{2}\sqrt\mu
\end{cases}.
\]
We will treat $I_a$ only, as the proof for $I_b, I_c$ is similar. We have 
\beq\label{ran-107}\bega 
I_a&=\pt^\al (a\sqrt\mu)-a_\al\sqrt\mu=\sum_{\substack{\beta+\gamma=\al\\
}}\frac{\al!}{\beta!\gamma!}\pt^\gamma a\cdot \pt^\beta \sqrt\mu-a_\al \sqrt\mu\\
&=(\pt^\al a-a_\al)\sqrt\mu+\sum_{\substack{\beta+\gamma=\al\\|\beta|\ge 1
}}\frac{\al!}{\beta!\gamma!}\pt^\gamma a\cdot\pt^\beta\sqrt\mu.\\
\enda 
\eeq
Now 
\beq\label{ran-107-1}
\pt^\al a-a_\al =\pt^\al\lw(\int_{\R^3}f\sqrt\mu d\xi\rw)-\int_{\R^3}\pt^\al f\sqrt\mu d\xi=\sum_{\substack{\beta+\gamma=\al\\|\beta|\ge 1}}\frac{\al!}{\beta!\gamma!}\int_{\R^3}\pt^\beta \sqrt\mu \cdot \pt^\gamma fd\xi .
\eeq
Combining \eqref{ran-107} and \eqref{ran-107-1}, we have 
\beq\label{ran-107-2}
I_a=\sum_{\substack{\beta+\gamma=\al\\|\beta|\ge 1}}\frac{\al!}{\beta!\gamma!}\lw\{\sqrt\mu\int_{\R^3}\pt^\beta \sqrt\mu \cdot \pt^\gamma fd\xi+\pt^\gamma a\cdot\pt^\beta\sqrt\mu\rw\}.
\eeq
This implies 
\[\bega 
\nl \la \vp \ra^m I_a\nr_{L^2_\xi}&\lesssim  \sum_{\substack{\beta+\gamma=\al\\|\beta|\ge 1}}\frac{\al!}{\beta!\gamma!}\lw\{
\nl \la\vp\ra^m\pt^\beta  \sqrt\mu\nr_{L^2_\xi}\nl \pt^\gamma f\nr_{L^2_\xi}+
\nl\la \vp\ra^m \pt^\beta \sqrt\mu
\nr_{L^2_\xi }\lw|\pt^\gamma a\rw|
\rw\},\\
\nl \la \vp \ra^m I_a\nr_{L^2_{x,\xi}}&\lesssim \sum_{\substack{\beta+\gamma=\al\\|\beta|\ge 1}}\frac{\al!}{\beta!\gamma!}\lw\{
\nl \la\vp\ra^m\pt^\beta  \sqrt\mu\nr_{L^\infty_x L^2_\xi} \nl \pt^\gamma f\nr_{L^2_{x,\xi}}+
\nl\la \vp\ra^m \pt^\beta \sqrt\mu
\nr_{ L^\infty_xL^2_\xi}\nl \pt^\gamma a\nr_{L^2}
\rw\}\\
\enda 
\]
The proof of \eqref{com-1-1} is complete. Now we show \eqref{com-1-2}. From \eqref{ran-107-2},  we have
\[\bega
\pt_{x_3}I_a&=\sum_{\substack{\beta+\gamma=\al\\|\beta|\ge 1}}\frac{\al!}{\beta!\gamma!}\lw\{\pt_{x_3}\sqrt\mu\int_{\R^3}\pt^\beta \sqrt\mu \cdot \pt^\gamma fd\xi+\pt^\gamma a\cdot\pt^\beta\pt_{x_3}\sqrt\mu\rw\}\\
&\quad+\sum_{\substack{\beta+\gamma=\al\\|\beta|\ge 1}}\frac{\al!}{\beta!\gamma!}\lw\{\sqrt\mu\int_{\R^3}\pt^\beta \sqrt\mu \cdot \pt^\gamma \pt_{x_3} fd\xi+\pt^\gamma \pt_{x_3}a\cdot\pt^\beta\sqrt\mu\rw\}\\
&\quad+\sum_{\substack{\beta+\gamma=\al\\|\beta|\ge 1}}\frac{\al!}{\beta!\gamma!}\sqrt\mu \int_{\R^3}\pt^\beta\pt_{x_3} \sqrt\mu \cdot \pt^\gamma fd\xi. \enda
\]
This implies 
\[\bega
&\nl\la\vp\ra^m \pt_{x_3}I_a
\nr_{L^2_\xi}\lesssim \sum_{\substack{\beta+\gamma=\al\\|\beta|\ge 1}}\frac{\al!}{\beta!\gamma!}\lw\{\nl \la\vp\ra^m\pt_{x_3}\sqrt\mu\nr_{L^2_\xi}
\nl \pt^\beta \sqrt\mu \nr_{L^2_\xi}\nl \pt^\gamma f\nr_{L^2_\xi}+\lw|\pt^\gamma a
\rw|\nl\la\vp\ra^m \pt^\beta \pt_{x_3}\sqrt\mu
\nr_{L^2_\xi}
\rw\}\\
&+ \sum_{\substack{\beta+\gamma=\al\\|\beta|\ge 1}}\frac{\al!}{\beta!\gamma!}\lw\{\nl\la\vp\ra^m
\sqrt\mu\nr_{L^2_\xi}\nl \pt^\beta \sqrt\mu \nr_{L^2_\xi}\nl  \pt^\gamma \pt_{x_3} f\nr_{L^2_\xi}
+\lw|\pt^\gamma \pt_{x_3}a\rw|\nl \pt^\beta\sqrt\mu\nr_{L^2_\xi}
\rw\}\\
&+\sum_{\substack{\beta+\gamma=\al\\|\beta|\ge 1}}\frac{\al!}{\beta!\gamma!}\nl \pt^\beta\pt_{x_3}\sqrt\mu\nr_{L^2_\xi}\nl\pt^\gamma f\nr_{L^2_\xi}.\enda
\]
Taking $L^2$ in $x\in\Omega$, we get
\[\bega
&\nl\la\vp\ra^m \pt_{x_3}I_a
\nr_{L^2_{x,\xi}}\\
&\lesssim \sum_{\substack{\beta+\gamma=\al\\|\beta|\ge 1}}\frac{\al!}{\beta!\gamma!}\lw\{
\lw(
\eps\kappa^{-\frac12}
\nl \pt^\beta \sqrt\mu \nr_{L^\infty_x L^2_\xi}+\nl \pt^\beta\pt_{x_3}\sqrt\mu\nr_{L^\infty_x L^2_\xi}\rw)
\nl \pt^\gamma f\nr_{L^2_{x,\xi}}\rw\}\\
&\quad+ \sum_{\substack{\beta+\gamma=\al\\|\beta|\ge 1}}\frac{\al!}{\beta!\gamma!}\lw\{\nl \pt^\gamma a
\nr_{L^2}\nl\la\vp\ra^m \pt^\beta \pt_{x_3}\sqrt\mu
\nr_{L^\infty_x L^2_\xi}\rw\}\\
&\quad+ \sum_{\substack{\beta+\gamma=\al\\|\beta|\ge 1}}\frac{\al!}{\beta!\gamma!}\lw\{
\nl \pt^\beta \sqrt\mu \nr_{L^\infty_x L^2_\xi}\nl  \pt^\gamma \pt_{x_3} f\nr_{L^2_{x,\xi}}
+\nl \pt^\gamma \pt_{x_3}a\nr_{L^2}\nl \pt^\beta\sqrt\mu\nr_{L^\infty_x L^2_\xi}
\rw\}.\\
\enda
\]
Here we use Proposition \ref{local-1} in the last estimate. The proof of \eqref{com-1-2} is complete. The proof of \eqref{com-1-3} is similar, and we skip the details.
\end{proof}
\begin{corollary} If $\P f=0$, there holds 
\beq\label{ran-109}
\bega
\nl\nu^{\frac 12} \pt_{x_3}[\pt^\al, \P]f\nr_{L^2_{\xi}}
&\lesssim \sum_{\substack{\beta+\gamma=\al\\|\beta|\ge 1}}\frac{\al!}{\beta!\gamma!} 
\lw(
\eps\kappa^{-\frac12}
\nl \pt^\beta \sqrt\mu \nr_{L^\infty_x L^2_\xi}+\nl \pt^\beta\pt_{x_3}\sqrt\mu\nr_{L^\infty_x L^2_\xi}\rw)
 \nl \pt^\gamma f\nr_{L^2_{\xi}}\\
&\quad+ \sum_{\substack{\beta+\gamma=\al\\|\beta|\ge 1}}\frac{\al!}{\beta!\gamma!}\nl \nu^{\frac 12}\pt^\beta \lw\{(1,\vp,|\vp|^2) \sqrt\mu\rw\}\nr_{L^\infty_x L^2_\xi}\nl \pt^\gamma \pt_{x_3} f\nr_{L^2_{\xi}},\\
\enda 
\eeq
\beq\label{ran-109-1}
\bega
\nl\nu^{\frac 12} \pt_{x_i}[\pt^\al, \P]f\nr_{L^2_{\xi}}
&\lesssim \sum_{\substack{\beta+\gamma=\al\\|\beta|\ge 1}}\frac{\al!}{\beta!\gamma!} 
\lw(
\eps
\nl \pt^\beta \sqrt\mu \nr_{L^\infty_x L^2_\xi}+\nl \pt^\beta\pt_{x_i}\sqrt\mu\nr_{L^\infty_x L^2_\xi}\rw)
 \nl \pt^\gamma f\nr_{L^2_{\xi}}\\
&\quad+ \sum_{\substack{\beta+\gamma=\al\\|\beta|\ge 1}}\frac{\al!}{\beta!\gamma!}\nl \nu^{\frac 12}\pt^\beta \lw\{(1,\vp,|\vp|^2) \sqrt\mu\rw\}\nr_{L^\infty_x L^2_\xi}\nl \pt^\gamma \pt_{x_i} f\nr_{L^2_{\xi}}.\\
\enda 
\eeq
There also hold the analytic estimates
\beq\label{P-com-pt3}
\bega
&\sum_{\al\in\mathbb N_0^3}\frac{\tau^{|\al|}}{\al!} \nl\nu^{\frac 12} \pt_{x_3}[\pt^\al, \P]f\nr_{L^2_{\xi}}\\
&\lesssim \eps\kappa^{-\frac 12}\sum_{\al\in\mathbb N_0^3}\frac{\tau^{|\al|}}{\al!}\nl \pt^\al f
\nr _{L^2_\xi}+\eps\kappa^{\frac 12}\sum_{\al\in\mathbb N_0^3}\frac{\tau^{|\al|}}{\al!}\nl \pt^\al \pt_{x_3} f
\nr _{L^2_\xi},
\enda 
\eeq
and 
\beq\label{P-com-pt12}
\bega
&\sum_{\al\in\mathbb N_0^3}\frac{\tau^{|\al|}}{\al!} \nl\nu^{\frac 12} \pt_{x_i}[\pt^\al, \P]f\nr_{L^2_{\xi}}\\
&\lesssim \eps \sum_{\al\in\mathbb N_0^3}\frac{\tau^{|\al|}}{\al!}\nl \pt^\al f
\nr _{L^2_\xi}+\eps\kappa^{\frac 12}\sum_{\al\in\mathbb N_0^3}\frac{\tau^{|\al|}}{\al!}\nl \pt^\al \pt_{x_i} f
\nr _{L^2_\xi}.
\enda \eeq
\end{corollary}
\begin{proof} The proof follows directly from the previous proposition, using the fact that $(a,b,c)\equiv 0$ and skipping the $x$ integration. The inequality \eqref{P-com-pt3} follows from Proposition \ref{local-1} and Lemma \eqref{M-local-2}. Similarly, \eqref{P-com-pt12}  follows from \eqref{ran-109-1} and \eqref{M-local-2}.
\end{proof}

\begin{proposition}\label{com-2} Assume \eqref{Pf-def}.  For $\tau\le \tau_0$ small defined in Proposition \ref{local-1}, we have 
\beq\label{ptP}
\|\pt\P f\|_{L^2_x}\lesssim \|\pt f\|_{L^2}+\eps\sqrt\kappa \|f\|_{L^2}\qquad\forall \pt\in \{\eps\pt_t,\pt_{x_1},\pt_{x_2}\}.
\eeq
Moreover, there holds the analytic estimate
\beq\label{ptP-a}
\sum_{\al\in \mathbb N_0^3}\frac{\tau^{|\al|}}{\al!}\nl \pt^\al \P f\nr_{L^2_x}\lesssim\sum_{\al\in \mathbb N_0^3}\frac{\tau^{|\al|}}{\al!}\nl\pt^\al f
\nr_{L^2_{x,\xi}},
\eeq
and 
\beq\label{ptP-a1}\lw(\sum_{\al\in \mathbb N_0^3}\AA_\al^2\nl \pt^\al \P f\nr_{L^2_x}^2\rw)^{1/2}\lesssim\lw(\sum_{\al\in \mathbb N_0^3}\AA_\al^2\nl\pt^\al f
\nr_{L^2}^2\rw)^{1/2}.
\eeq
\end{proposition}
\begin{proof} We show the bound for $a$ only, as the other cases are similar. We have \[
a(t,x)=\int_{\R^3}f(t,x,\xi)\sqrt\mu(\xi)d\xi\]
Take $\pt\in \{\eps\pt_t,\pt_{x_1},\pt_{x_2}\}$ and $\al\in \mathbb N_0^3$, we have 
\beq\label{ran-212}
\bega 
\pt^\al a&=\sum_{\beta\le \al} \binom{\al}{\beta}\int_{\R^3} \pt^\beta f\pt^{\al-\beta}\sqrt\mu d\xi\lesssim \sum_{\beta\le \al}\frac{\al!}{\beta!(\al-\beta)!}\nl \pt^\beta f
\nr_{L^2_\xi}\nl \pt^{\al-\beta}\sqrt\mu\nr_{L^2_\xi}\\
&\lesssim \sum_{\beta\le \al}\frac{\al!}{\beta!(\al-\beta)!}\nl \pt^\beta f
\nr_{L^2_\xi}\nl \pt^{\al-\beta}\sqrt\mu\nr_{L^\infty_xL^2_\xi}\\
\nl\pt^\al a\nr_{L^2_x}&\lesssim  \sum_{\beta\le \al}\frac{\al!}{\beta!(\al-\beta)!}\nl \pt^\beta f
\nr_{L^2}\nl \pt^{\al-\beta}\sqrt\mu\nr_{L^\infty_xL^2_\xi}.\\
\enda 
\eeq
Using Proposition \ref{local-1}, we get $\nl \pt^{\al-\beta}\sqrt\mu\nr_{L^\infty_xL^2_\xi}\lesssim 1_{\{|\al-\beta|\ge 1\}}\eps\sqrt\kappa+1_{\{\beta=\al\}}$. Combining this with \eqref{ran-212}, we get
\[
\nl \pt^\al a\nr_{L^2_x}\lesssim \|\pt^\al f\|_{L^2}+\eps\sqrt\kappa \sum_{\beta<\al}\|\pt^\beta f\|_{L^2}.
\]
This shows the inequality \eqref{ptP-a1}. As for the second inequality, we have 
\[\bega 
\sum_{\al\in\mathbb N_0^3}\frac{\tau^{|\al|}}{\al!}\nl \pt^\al a\nr_{L^2_x}&\lesssim \sum_{\al\in\mathbb N_0^3}\sum_{\beta\le \al}\frac{\tau^{|\beta|}}{\beta!}\|\pt^\beta f\|_{L^2}\cdot \frac{\tau^{|\al-\beta|}}{|\al-\beta|!}\nl\pt^{\al-\beta}\sqrt\mu
\nr_{L^\infty_xL^2_\xi}\cdot \frac{\al!}{\beta!(\al-\beta)!}\cdot\frac{\beta!|\al-\beta|!}{\al!}\\
&\lesssim \sum_{\al\in\mathbb N_0^3}\frac{\tau^{|\al|}}{\al!}\nl \pt^\al f\nr_{L^2},\enda 
\]
again, by using Proposition \ref{local-1}. This finishes the proof of \eqref{ptP-a}. The second inequality is done similarly. The proof is complete. 
\end{proof}
\begin{proposition}\label{PD} For all $m>0$, there hold the following estimates
\beq\label{PD1}
\sum_{\al\in \mathbb N_0^3}\frac{\tau^{|\al|}}{\al!}\nl\la\vp\ra^m[\pt^\al, \P]f\nr_{L^2_{x,\xi}}\lesssim\eps \kappa^{\frac12}\sum_{\al\in \mathbb N_0^3}\frac{\tau^{|\al|}}{\al!}\nl\pt^\al f
\nr_{L^2_{x,\xi}},
\eeq
and
\beq\label{PD2}
\lw\{\sum_{\al\in \mathbb N_0^3}\AA_\al^2\nl\la\vp\ra^m[\pt^\al, \P]f\nr_{L^2_{x,\xi}}^2\rw\}^{1/2}\lesssim\eps \kappa^{\frac12}\cdot \lw\{\sum_{\al\in \mathbb N_0^3}\AA_\al^2\|\pt^\al f\|_{L^2_{x,\xi}}^2\rw\}^{1/2}.
\eeq
\end{proposition}
\begin{proof} Using Lemma \ref{com-1}, we get
\[\bega
&\frac{\tau^{|\al|}}{\al!}\nl\la\vp\ra^m [\pt^\al, \P]f\nr_{L^2_{x,\xi}}\\
&\lesssim \sum_{\substack{\beta+\gamma=\al\\|\beta|\ge 1}}\frac{\tau^{|\beta|}}{\beta!}
\nl \la\vp\ra^m\pt^\beta \lw\{(1,\vp,|\vp|^2) \sqrt\mu\rw\}\nr_{ L^2_\xi L^\infty_x} \frac{\tau^{|\gamma|}}{\gamma!}\nl \pt^\gamma f\nr_{L^2_{x,\xi}}\\
&\quad+\sum_{\substack{\beta+\gamma=\al\\|\beta|\ge 1}}\frac{\tau^{|\beta|}}{\beta!}\nl\la \vp\ra^m\pt^\beta \lw\{(1,\vp,|\vp|^2) \sqrt\mu\rw\}
\nr_{L^\infty_x L^2_\xi }\frac{\tau^{|\gamma|}}{\gamma!}\nl \pt^\gamma (a,b,c) \nr_{L^2}.
\enda 
\]
The result follows directly from Proposition \ref{local-1} and Proposition \ref{com-2}.
\end{proof}
\begin{proposition}\label{com-1} There hold 
\beq\label{ran-112}
\sum_{\al\in\mathbb N_0^3}\frac{\tau^{|\al|}}{\al!}\nl \nu^{-\frac 1 2}[L,\pt^\al]f\nr_{L^2_\xi}\lesssim \eps\kappa^{\frac12}\sum_{\al\in\mathbb N_0^3}\frac{\tau^{|\al|}}{\al!} \nl\nu^{\frac 12}\pt^\al f
\nr_{L^2_\xi},
\eeq
and 
\beq\label{ran-112-1}
\sum_{\al\in\mathbb N_0^3}\frac{\tau^{|\al|}}{\al!}\nl \nu^{-1}[L,\pt^\al]f\nr_{L^\infty_\xi}\lesssim \eps\kappa^{\frac12}\sum_{\al\in\mathbb N_0^3}\frac{\tau^{|\al|}}{\al!} \lw(\nl\pt^\al f
\nr_{L^\infty_\xi}+\nl\pt^\al f
\nr_{L^2_\xi}
\rw).\eeq
There also hold 
\[\bega 
\nl\nu^{-\frac12} [L,\pt_{x_3}]f 
\nr_{L^2_\xi}
&\lesssim  \eps \kappa^{\frac12}\nl\nu^{\frac 12}\pt_{x_3} f
\nr_{L^2_\xi}+\eps\kappa^{-\frac 12} \nl \nu^{\frac 12} f\nr_{L^2_\xi},\\
\nl\nu^{-\frac12}[L, \pt_{x_i}]f 
\nr_{L^2_\xi}
&\lesssim  \eps \kappa^{\frac12}\nl\nu^{\frac 12}\pt_{x_i} f
\nr_{L^2_\xi}+\eps \nl\nu^{\frac 12} f\nr_{L^2_\xi},\qquad 1\le i\le 2\\
\nl\nu^{-1}[L,\pt_{x_3}]f
\nr_{L^\infty_\xi}&\lesssim \eps\kappa^{\frac12}\lw(\nl\pt_{x_3}f
\nr_{L^\infty_\xi}+\nl \pt_{x_3}f
\nr_{L^2_\xi}
\rw)+\eps \kappa^{-\frac12}\lw(\nl f
\nr_{L^\infty_\xi}+\nl f
\nr_{L^2_\xi}
\rw),\\
\nl\nu^{-1}[L,\pt_{x_i}]f
\nr_{L^\infty_\xi}&\lesssim \eps\kappa^{\frac12}\lw(\nl\pt_{x_i}f
\nr_{L^\infty_\xi}+\nl \pt_{x_i}f
\nr_{L^2_\xi}
\rw)+\eps \kappa^{-\frac12}\lw(\nl f
\nr_{L^\infty_\xi}+\nl f
\nr_{L^2_\xi}
\rw),\quad 1\le i\le 2.\enda 
\]
\end{proposition}
\begin{proof} We show the first inequality, and the others are similar. As in \eqref{L-alpha-com}, we have 
\[
\bega
&\nu^{-\frac12}[L,\pt^\al]f=L(\pt^\al f)-\pt^\al Lf\\
&=\nu^{-\frac12}\sum_{\substack{\beta+\gamma+\rho =\al\\\rho<\al
}
}\frac{\al!}{\beta!\gamma!\rho!}\int_{\R^3\times\S^2}|(\xi-\xi_\star)\cdot \w|\pt^\beta\sqrt\mu(\xi_\star)\pt^\gamma\sqrt\mu(\xi')\pt^\rho  f(\xi_\star')d\w d\xi_\star\\
&\quad +\nu^{-\frac12}\sum_{\substack{\beta+\gamma+\rho =\al\\\rho<\al
}
}\frac{\al!}{\beta!\gamma!\rho!}\int_{\R^3\times\S^2}|(\xi-\xi_\star)\cdot \w|\pt^\beta \sqrt\mu(\xi_\star)\pt^\gamma \sqrt\mu(\xi_\star')\pt^\rho  f(\xi')d\w d\xi_\star\\
&\quad-\nu^{-\frac12}\sum_{\substack{\beta+\gamma+\rho =\al\\\rho<\al
}
}\frac{\al!}{\beta!\gamma!\rho!}\int_{\R^3\times \S^2}|(\xi-\xi_\star)\cdot \w|\pt^\beta\sqrt\mu(\xi_\star) \pt^\gamma \sqrt\mu(\xi)\pt^\rho  f(\xi_\star)d\w d\xi_\star\\
&\quad-\nu^{-\frac12}\sum_{\substack{\beta+\gamma=\al\\\gamma<\al
}
}\frac{\al!}{\beta!\gamma!}\int_{\R^3\times \S^2}|(\xi-\xi_\star)\cdot \w|\pt^\beta\mu(\xi_\star) \pt^\gamma   f(\xi)d\w d\xi_\star\\
&=A_\al+B_\al+C_\al+D_\al.
\enda \]
Using Propositions \ref{BGR}, \ref{BGR1}, \ref{BGR2}, we have
\[\bega 
&\nl A_\al\nr_{L^2_\xi} \lesssim\sum_{\substack{\gamma+\rho=\al\\|\gamma|\ge 1}}\frac{\al!}{\rho!\gamma!}\lw\|\nu^{\frac12}\pt^\gamma\sqrt\mu
\rw\|_{L^2_{\xi}}\lw\|\nu^{\frac12} \pt^\rho f
\rw\|_{L^2_{\xi}}\\
&+
 \sum_{\substack{\beta+\gamma+\rho =\al\\|\beta|\ge 1
}
}\frac{\al!}{\beta!\gamma!\rho!}\lw\|\nu^{-1}\lw(\int_{\R^3_{\xi_\star}}|\xi-\xi_\star|^2 |\pt^\beta\sqrt\mu(\xi_\star)|^2d\xi_\star\rw)^{1/2}
\rw\|_{L^\infty_{\xi}}\lw\|\nu^{\frac12}\pt^\gamma \sqrt\mu
\rw\|_{L^2_{\xi}}\lw\|\nu^{\frac12} \pt^\rho f
\rw\|_{L^2_{\xi}},\\
\enda\]
\[\bega
&\nl B_\al\nr_{L^2_\xi}\lesssim \sum_{\substack{\gamma+\rho=\al\\|\gamma|\ge 1}}\frac{\al!}{\rho!\gamma!}\nl\nu^{\frac 12} \pt^\gamma\sqrt\mu
\nr_{L^2_\xi}\nl\nu^{\frac 12} \pt^\rho f
\nr_{L^2_{\xi}}\\
&+
 \sum_{\substack{\beta+\gamma+\rho =\al\\|\beta|\ge 1
}
}\frac{\al!}{\beta!\gamma!\rho!}\lw\|\nu(\xi)^{-1}\lw(\int_{\R^3_{\xi_\star}}|\xi-\xi_\star|^2 |\pt^\beta\sqrt\mu(\xi_\star)|^2d\xi_\star\rw)^{1/2}
\rw\|_{L^\infty_{\xi}}\lw\|\nu^{\frac 12} \pt^\gamma\sqrt\mu\rw\|_{L^2_\xi}\lw\|\nu^{\frac 12}  \pt^\rho f
\rw\|_{L^2_{\xi}},\\
\enda\]
\[\bega 
\nl C_\al\nr_{L^2_\xi}&\lesssim \sum_{\substack{\beta+\gamma+\rho =\al\\|\beta|+|\gamma|\ge 1
}
}\frac{\al!}{\beta!\gamma!\rho!} \nl\nu(\xi)^{-1}\lw(\int_{\xi_\star}\nu(\xi_\star)^{-1}|\xi-\xi_\star|^2 |\pt^\beta\sqrt\mu(\xi_\star)|^2d\xi_\star
 \rw)^{1/2} \nr_{L^\infty_{\xi}}\nl \pt^\gamma \sqrt\mu
 \nr_{ L^2_\xi}\nl \nu^{\frac12}  \pt^\rho f
 \nr_{L^2_{\xi}},\\
\nl D_\al\nr_{L^2_\xi}&\lesssim \sum_{\substack{\beta+\gamma+\rho =\al\\|\beta|\ge 1
}
}\frac{\al!}{\beta!\gamma!}\nl \nu(\xi)^{-1}\int_{\R^3}\int_{|\w|=1}|(\xi-\xi_\star)\cdot \w|\pt^\beta\mu(\xi_\star)d\w d\xi_\star\nr_{L^\infty_\xi}
\nl \nu^{\frac 12} \pt^\gamma f
\nr_{L^2_\xi}.  \enda
\]
The proof of \eqref{ran-112} is complete, using Lemma \ref{local-1}. The proof of \eqref{ran-112-1} is done similarly, by using Lemma \ref{bili-2} and Lemma \ref{local-1}. We skip the details.
\end{proof}

\begin{proposition}\label{com-2} There holds 
\beq\label{ran-113}
\sum_{\al\in\mathbb N_0^3}\frac{\tau^{|\al|}}{\al!}\nl\nu^{-\frac12} \pt_{x_3}[L,\pt^\al]f 
\nr_{L^2_\xi}
\lesssim  \eps \kappa^{\frac12}\sum_{\al\in\mathbb N_0^3}\frac{\tau^{|\al|}}{\al!} \nl\nu^{\frac 12}\pt^\al \pt_{x_3} f
\nr_{L^2_\xi}+\eps\kappa^{-\frac 12} \sum_{\al\in\mathbb N_0^3}\frac{\tau^{|\al|}}{\al!}\nl \pt^\al f\nr_{L^2_\xi}.\eeq
For $i\in \{1,2\}$, there hold
\beq\label{ran-113-1}
\sum_{\al\in\mathbb N_0^3}\frac{\tau^{|\al|}}{\al!}\nl\nu^{-\frac12} \pt_{x_i}[L,\pt^\al]f 
\nr_{L^2_\xi}
\lesssim  \eps \kappa^{\frac12}\sum_{\al\in\mathbb N_0^3}\frac{\tau^{|\al|}}{\al!} \nl\nu^{\frac 12}\pt^\al \pt_{x_i} f
\nr_{L^2_\xi}+\eps \sum_{\al\in\mathbb N_0^3}\frac{\tau^{|\al|}}{\al!}\nl \pt^\al f\nr_{L^2_\xi}.\eeq
\end{proposition}
\begin{proof} We show \eqref{ran-113} only, as the proof of \eqref{ran-113-1} similar.
Using \eqref{L-alpha-com}, we have
\[
\nu^{-\frac12}\pt_{x_3}[L,\pt^\al]f=\nu^{-\frac12}[L,\pt^\al](\pt_{x_3}f)+A_{1,\al}+A_{2,\al}+B_{1,\al}+B_{2,\al}-C_{1,\al}-C_{2,\al}-D_{\al},
\]
where 
\[\bega
A_{\al,1}&=\sum_{\substack{\beta+\gamma+\rho =\al\\\rho<\al
}
}\frac{\al!}{\beta!\gamma!\rho!}\int_{\R^3\times\S^2}|(\xi-\xi_\star)\cdot \w|\pt^\beta\pt_{x_3}\sqrt\mu(\xi_\star)\pt^\gamma\sqrt\mu(\xi')\pt^\rho f(\xi_\star')d\w d\xi_\star,\\
A_{\al,2}&=\sum_{\substack{\beta+\gamma+\rho =\al\\\rho<\al
}
}\frac{\al!}{\beta!\gamma!\rho!}\int_{\R^3\times\S^2}|(\xi-\xi_\star)\cdot \w|\pt^\beta\sqrt\mu(\xi_\star)\pt^\gamma\pt_{x_3}\sqrt\mu(\xi')\pt^\rho f(\xi_\star')d\w d\xi_\star,\\
B_{\al,1}&=\sum_{\substack{\beta+\gamma+\rho =\al\\\rho<\al
}
}\frac{\al!}{\beta!\gamma!\rho!}\int_{\R^3\times\S^2}|(\xi-\xi_\star)\cdot \w|\pt^\beta \pt_{x_3}\sqrt\mu(\xi_\star)\pt^\gamma \sqrt\mu(\xi_\star')\pt^\rho  f(\xi')d\w d\xi_\star,\\
B_{\al,2}&=\sum_{\substack{\beta+\gamma+\rho =\al\\\rho<\al
}
}\frac{\al!}{\beta!\gamma!\rho!}\int_{\R^3\times\S^2}|(\xi-\xi_\star)\cdot \w|\pt^\beta \sqrt\mu(\xi_\star)\pt^\gamma \pt_{x_3}\sqrt\mu(\xi_\star')\pt^\rho  f(\xi')d\w d\xi_\star,\\
C_{1,\al}&=\sum_{\substack{\beta+\gamma+\rho =\al\\\rho<\al
}
}\frac{\al!}{\beta!\gamma!\rho!}\int_{\R^3\times\S^2}|(\xi-\xi_\star)\cdot \w|\pt^\beta\pt_{x_3}\sqrt\mu(\xi_\star) \pt^\gamma \sqrt\mu(\xi)\pt^\rho  f(\xi_\star)d\w d\xi_\star,\\
C_{\al,2}&=\sum_{\substack{\beta+\gamma+\rho =\al\\\rho<\al
}
}\frac{\al!}{\beta!\gamma!\rho!}\int_{\R^3\times\S^2}|(\xi-\xi_\star)\cdot \w|\pt^\beta\sqrt\mu(\xi_\star) \pt^\gamma \pt_{x_3}\sqrt\mu(\xi)\pt^\rho  f(\xi_\star)d\w d\xi_\star,\\
D_{\al}&=\sum_{\substack{\beta+\gamma=\al\\\gamma<\al
}
}\frac{\al!}{\beta!\gamma!}\int_{\R^3}\int_{|\w|=1}|(\xi-\xi_\star)\cdot \w|\pt^\beta \pt_{x_3}\mu(\xi_\star) \pt^\gamma f(\xi)d\w d\xi_\star.
\enda
\]
Using Proposition \ref{BGR}, \ref{BGR1}, \ref{BGR2}, we obtain the following estimates
\[
\bega
\nl A_{\al,1}\nr_{L^2_\xi}&\lesssim \nl \nu(\xi)^{-1}\lw(\int_{\R^3_{\xi_\star}}|\xi-\xi_\star|^2 |\pt_{x_3}\sqrt\mu(\xi_\star)|^2d\xi_\star\rw)^{1/2} \nr_{L^\infty_\xi} \sum_{\substack{\gamma+\rho=\al\\|\gamma|\ge 1}}\frac{\al!}{\rho!\gamma!}\lw\|\nu^{\frac12}\pt^\gamma\sqrt\mu
\rw\|_{L^2_{\xi}}\lw\|\nu^{\frac12} \pt^\rho f
\rw\|_{L^2_{\xi}}\\
&\quad+
 \sum_{\substack{\beta+\gamma+\rho =\al\\|\beta|\ge 1
}
}\frac{\al!}{\beta!\gamma!\rho!}\lw\|\nu^{-1}\lw(\int_{\R^3_{\xi_\star}}|\xi-\xi_\star|^2 |\pt^\beta\pt_{x_3}\sqrt\mu(\xi_\star)|^2d\xi_\star\rw)^{1/2}
\rw\|_{L^\infty_{\xi}}\lw\|\nu^{\frac12}\pt^\gamma \sqrt\mu
\rw\|_{L^2_{\xi}}\lw\|\nu^{\frac12} \pt^\rho f
\rw\|_{L^2_{\xi}},\\
\nl A_{\al,2}\nr_{L^2_\xi}&\lesssim \sum_{\substack{\gamma+\rho=\al\\|\gamma|\ge 1}}\frac{\al!}{\gamma!\rho!}
\nl\nu^{\frac12} \pt^\gamma\pt_{x_3}\sqrt\mu
\nr_{L^2_\xi}\nl\nu^{\frac12}  \pt^\rho f
\nr_{L^2_{\xi}}\\
&\quad+\sum_{\substack{\beta+\gamma+\rho =\al\\|\beta|\ge 1
}
}\frac{\al!}{\beta!\gamma!\rho!}
\lw\|\nu^{-1}\lw(\int_{\R^3_{\xi_\star}}|\xi-\xi_\star|^2 |\pt^\beta\pt_{x_3}\sqrt\mu(\xi_\star)|^2d\xi_\star\rw)^{1/2}
\rw\|_{L^2_{\xi}}\lw\|\nu^{\frac12} \pt^\gamma\sqrt\mu\rw\|_{L^2_\xi}\lw\|\nu^{\frac12} \pt^\rho f
\rw\|_{L^2_{\xi}}.\\
\enda\]
Using Proposition \ref{BGR1}, we obtain 
\[\bega 
\nl B_{\al,1}\nr_{L^2_\xi}&\lesssim  \nl \nu(\xi)^{-1}\lw(\int_{\R^3_{\xi_\star}}|\xi-\xi_\star|^2 |\pt_{x_3}\sqrt\mu(\xi_\star)|^2d\xi_\star\rw)^{1/2} \nr_{L^\infty_\xi} \sum_{\substack{\gamma+\rho=\al\\|\gamma|\ge 1}}\frac{\al!}{\rho!\gamma!}\nl\nu^{\frac 12} \pt^\gamma\sqrt\mu
\nr_{L^2_\xi}\nl\nu^{\frac 12} \pt^\rho f
\nr_{L^2_{\xi}}\\
&\quad+
 \sum_{\substack{\beta+\gamma+\rho =\al\\|\beta|\ge 1
}
}\frac{\al!}{\beta!\gamma!\rho!}\lw\|\nu(\xi)^{-1}\lw(\int_{\R^3_{\xi_\star}}|\xi-\xi_\star|^2 |\pt^\beta\pt_{x_3}\sqrt\mu(\xi_\star)|^2d\xi_\star\rw)^{1/2}
\rw\|_{L^\infty_{\xi}}\lw\|\nu^{\frac 12} \pt^\gamma\sqrt\mu\rw\|_{L^2_\xi}\lw\|\nu^{\frac 12}  \pt^\rho f
\rw\|_{L^2_{\xi}},\\
\nl B_{\al,2}\nr_{L^2_\xi}&\lesssim \sum_{\substack{\gamma+\rho=\al\\|\gamma|\ge 1}}\frac{\al!}{\rho!\gamma!}\nl\nu^{\frac 12} \pt_{x_3}\pt^\gamma\sqrt\mu
\nr_{L^2_\xi}\nl\nu^{\frac 12} \pt^\rho f
\nr_{L^2_{\xi}}\\
&\quad+
 \sum_{\substack{\beta+\gamma+\rho =\al\\|\beta|\ge 1
}
}\frac{\al!}{\beta!\gamma!\rho!}\lw\|\nu(\xi)^{-1}\lw(\int_{\R^3_{\xi_\star}}|\xi-\xi_\star|^2 |\pt^\beta\sqrt\mu(\xi_\star)|^2d\xi_\star\rw)^{1/2}
\rw\|_{L^\infty_{\xi}}\lw\|\nu^{\frac 12} \pt^\gamma\pt_{x_3}\sqrt\mu\rw\|_{L^2_\xi}\lw\|\nu^{\frac 12}  \pt^\rho f
\rw\|_{L^2_{\xi}}.\\
\enda\]
Using Proposition \ref{BGR2}, we obtain 
\[\bega 
\nl C_{\al,1}\nr_{L^2_\xi}&\lesssim \sum_{\substack{\beta+\gamma+\rho =\al\\|\beta|+|\gamma|\ge 1
}
}\frac{\al!}{\beta!\gamma!\rho!} \nl\nu(\xi)^{-1}\lw(\int_{\xi_\star}\nu(\xi_\star)^{-1}|\xi-\xi_\star|^2 |\pt^\beta\pt_{x_3}\sqrt\mu(\xi_\star)|^2d\xi_\star
 \rw)^{1/2} \nr_{L^\infty_{\xi}}\nl \pt^\gamma \sqrt\mu
 \nr_{ L^2_\xi}\nl \nu^{\frac12}  \pt^\rho f
 \nr_{L^2_{\xi}},\\
 \nl C_{\al,2}\nr_{L^2_\xi}&\lesssim \sum_{\substack{\beta+\gamma+\rho =\al\\|\beta|+|\gamma|\ge 1
}
}\frac{\al!}{\beta!\gamma!\rho!} \nl\nu(\xi)^{-1}\lw(\int_{\xi_\star}\nu(\xi_\star)^{-1}|\xi-\xi_\star|^2 |\pt^\beta\sqrt\mu(\xi_\star)|^2d\xi_\star
 \rw)^{1/2} \nr_{L^\infty_{\xi}}\nl \pt^\gamma\pt_{x_3}\sqrt\mu
 \nr_{ L^2_\xi}\nl \nu^{\frac12}  \pt^\rho f
 \nr_{L^2_{\xi}}.\\ 
\enda
\]
Hence we get \[
\nl D_\al\nr_{L^2_\xi}\lesssim \sum_{\substack{\beta+\gamma+\rho =\al\\|\beta|\ge 1
}
}\frac{\al!}{\beta!\gamma!}\nl \nu(\xi)^{-1}\int_{\R^3\times\S^2}|(\xi-\xi_\star)\cdot \w|\pt^\beta\mu(\xi_\star)d\w d\xi_\star\nr_{L^\infty_\xi}
\nl \nu^{\frac 12} \pt^\gamma f
\nr_{L^2_\xi} .
\]
Using Propositions \ref{local-1} and \eqref{M-local-2}, we get
\[\bega 
&\sum_{|\al|\ge 1}\frac{\tau^{|\al|}}{\al!}\lw\{\nl A_{\al,1}\nr_{L^2_\xi}+
\nl A_{\al,2}\nr_{L^2_\xi}+\nl B_{\al,1}\nr_{L^2_\xi}+\nl B_{\al,2}\nr_{L^2_\xi}+\nl C_{\al,1}\nr_{L^2_\xi}+ \nl C_{\al,2}\nr_{L^2_\xi}+\nl D_\al\nr_{L^2_\xi}\rw\}\\
&\lesssim \eps\kappa^{-\frac 12} \sum_{\al\in\mathbb N_0^3}\frac{\tau^{|\al|}}{\al!}\nl \pt^\al f\nr_{L^2_\xi}. 
\enda 
\]
Lastly, applying Theorem \ref{com-1}, we get 
\[\sum_{\al\in\mathbb N_0^3}\frac{\tau^{|\al|}}{\al!}\nl \nu^{-\frac 1 2}[L,\pt^\al]\pt_{x_3} f\nr_{L^2_\xi}\lesssim \eps \kappa^{\frac12}\sum_{\al\in\mathbb N_0^3}\frac{\tau^{|\al|}}{\al!} \nl\nu^{\frac 12}\pt^\al \pt_{x_3} f
\nr_{L^2_\xi}.
\]The proof is complete.
\end{proof}

\subsection{Analytic estimates on $L^{-1}$ and bounds on $A_{ij}(\vp)$}

\begin{lemma}\label{ran-124}
Recall the definition of $\nu$ in \eqref{nu-def}. 
There exists $\tau_0>0$ such that
\[
\sum_{|\al|\ge 1}\frac{\tau_0^{|\al|}}{\al!}\nl\pt^\al (\nu^{-1}(\xi))
\nr_{L^\infty_\xi}\lesssim \eps \kappa^{\frac 12}.
\]
\end{lemma}

\begin{proof}
We have 
\[\bega 
|\pt^\al (\nu(\xi)^{-1})|&\lesssim \sum_{\substack{k_1,k_2,\cdots k_s\in \mathbb N\\
\beta_1,\beta_2,\cdots, \beta_s\in\mathbb N^d\\
k_1 \beta_1+k_2 \beta_2+\cdots+k_s  \beta_s= \al
}
}\frac{\alpha!}{k_1!k_2!\cdots k_s!}(k_1+k_2+\cdots+k_s)!\prod_{i=1}^s\lw\{\frac{1}{\beta_i!}\nu(\xi)^{-1}\pt^{\beta_i}\nu
\rw\}^{k_i}.\enda 
\]
Using Lemma \ref{local-1}, we have
 \[
 \sup_{|\beta_i|\ge 1
 }\frac{\tau_0^{|\beta_i|}}{\beta_i!}\nu^{-1}(\xi)\pt^{\beta_i}\nu \lesssim \eps\kappa^{\frac12}.
 \]
Hence 
\[
\frac{\tau_0^{|\al|}}{\al!}|\pt^\al(\nu(\xi)^{-1})|\lesssim \sum_{\substack{k_1,k_2,\cdots k_s\in \mathbb N\\
\beta_1,\beta_2,\cdots, \beta_s\in\mathbb N^d\\
k_1 \beta_1+k_2 \beta_2+\cdots+k_s  \beta_s= \al
}
}\frac{(k_1+k_2+\cdots+k_s)!}{k_1!k_2!\cdots k_s!} (\eps\kappa^{\frac12})^{k_1+k_2+\cdots+k_s} .
\]
Combining the above with Lemma \ref{lem-Faa}, we get, for any $c_0\in (0,1)$:
\[
\sum_{|\al|\ge 1}\frac{(c_0\tau_0)^{|\al|}}{\al!} |\pt^\al (\nu(\xi)^{-1})|\lesssim \frac{\eps\kappa^{\frac 12}}{(1-c_0)^d-\eps \kappa^{\frac 12}}.
\]
The proof is complete.
\end{proof}

\begin{proposition}For $g\in\mathcal N^\perp$,  there holds, for $\rho\ll 1$
\beq\label{ran-154}
\nl e^{\rho|\xi|^2}L^{-1}g
\nr_{L^\infty_\xi}\lesssim \nl \nu(\xi)^{-1}e^{\rho|\xi|^2}g
\nr_{L^\infty_\xi}+\nl \nu(\xi)^{-1}e^{\rho |\xi|^2}g
\nr_{L^2_\xi}.
\eeq
\end{proposition}
\begin{proof} We have, for $\rho\in (0,\frac 14)$:
\[
\nl \nu_0(\xi-\eps U)e^{\rho|\xi-\eps U|^2} f
\nr_{L^\infty_\xi}\lesssim \nl e^{\rho|\vp|^2}g
\nr_{L^\infty_\xi}+\nl \nu_0(\vp)^{-1}e^{\rho|\vp|^2} g
\nr_{L^2_\vp}.
\]
Hence 
\[
\nl e^{\rho|\xi|^2}f
\nr_{L^\infty_\xi}\lesssim \nl \nu(\xi)^{-1}e^{\rho|\xi|^2}g
\nr_{L^\infty_\xi}+\nl \nu(\xi)^{-1}e^{\rho |\xi|^2}g
\nr_{L^2_\xi},
\]
for $\rho\ll 1$.
\end{proof}

\begin{proposition} \label{Lfg-L2} Assume that $Lf=g$. There holds, for $\rho\ll 1$, the estimate 
\[
\sum_{\al\in\mathbb N_0^3}\frac{\tau^{|\al|}}{\al!}\nl e^{\rho |\xi|^2}\pt^\al f\nr_{L^2_\xi}\lesssim\sum_{\al\in\mathbb N_0^3}\frac{\tau^{|\al|}}{\al!}\nl e^{\rho|\xi|^2} \nu^{-1}\pt^\al g
\nr_{L^2_\xi},
\]
and
\[\bega 
\sum_{\al\in\mathbb N_0^3}\frac{\tau^{|\al|}}{\al!}\nl e^{\rho|\xi|^2}\pt^\al  f
\nr_{L^\infty_\xi}\lesssim \sum_{\al\in\mathbb N_0^3}\frac{\tau^{|\al|}}{\al!}\nl e^{\rho|\xi|^2} \nu^{-1}\pt^\al g
\nr_{L^2_\xi\cap L^2_\xi}.\\
\enda
\]
Let $\ell\in \{1,2,3\}$. There holds 
\[\bega 
\sum_{\al\in\mathbb N_0^3}\frac{\tau^{|\al|}}{\al!}\nl e^{\rho |\xi|^2}\pt^\al \pt_{x_\ell} f\nr_{L^\infty_\xi}&\lesssim\sum_{\al\in\mathbb N_0^3}\frac{\tau^{|\al|}}{\al!}\lw\{\nl e^{\rho|\xi|^2} \nu^{-1}\pt^\al g
\nr_{L^2_\xi\cap L^\infty_\xi }+\nl e^{\rho|\xi|^2} \nu^{-1}\pt^\al \pt_{x_\ell} g
\nr_{L^2_\xi\cap L^\infty_\xi }\rw\},
\enda 
\]
and
\[\bega 
\sum_{\al\in\mathbb N_0^3}\frac{\tau^{|\al|}}{\al!}\nl e^{\rho |\xi|^2}\pt^\al \pt_t f\nr_{L^\infty_\xi}&\lesssim\sum_{\al\in\mathbb N_0^3}\frac{\tau^{|\al|}}{\al!}\lw\{\nl e^{\rho|\xi|^2} \nu^{-1}\pt^\al g
\nr_{L^2_\xi\cap L^\infty_\xi }+\nl e^{\rho|\xi|^2} \nu^{-1}\pt^\al \pt_t g
\nr_{L^2_\xi\cap L^\infty_\xi }\rw\}.
\enda 
\]
\end{proposition}
\begin{proof} We recall that 
\beq\label{ran-123-0}
\bega 
-Lf&=\int_{\R^3\times\S^2}|(\xi-\xi_\star)\cdot \w|\sqrt\mu(\xi_\star) \lw\{\sqrt\mu(\xi')f(\xi_\star')+f(\xi')\sqrt\mu(\xi_\star')\rw\}d\w d\xi_\star\\
&\quad-\sqrt\mu(\xi)\int_{\R^3\times\S^2}|(\xi-\xi_\star)\cdot\w| \sqrt\mu(\xi_\star)f(\xi_\star)d\w d\xi_\star\\
&\quad-f(\xi)\int_{\R^3\times\S^2}|(\xi-\xi_\star)\cdot\w|\mu(\xi_\star)d\w d\xi_\star
\enda 
\eeq
Take $\al\in\mathbb N_0^3$ and apply $\pt^\al$ on both sides, we have
\beq\label{ran-123}\bega 
&-L(\pt^\al f)\\
&\quad+\sum_{\substack{\beta+\gamma+\vr=\al\\\vr<\al}}\frac{\al!}{\beta!\gamma!\vr!}\int_{\R^3\times\S^2}|(\xi-\xi_\star)\cdot \w|\pt^\beta\sqrt\mu(\xi_\star)\lw\{\pt^\gamma \sqrt\mu(\xi')\pt^\vr f(\xi_\star')+\pt^\vr  f(\xi')\pt^\gamma\sqrt\mu(\xi_\star')\rw\}d\w d\xi_\star
\\
&\quad-\sum_{\substack{\beta+\gamma+\vr=\al\\\vr<\al}}\frac{\al!}{\beta!\gamma!\vr!}\pt^\beta \sqrt\mu(\xi)\int_{\R^3\times\S^2}|(\xi-\xi_\star)\cdot \w| \pt^\gamma\sqrt\mu(\xi_\star)\pt^\vr f(\xi_\star)d\w d\xi_\star
\\
&\quad-\sum_{\substack{\beta+\gamma=\al\\\beta<\al
}}\frac{\al!}{\beta!\gamma!} \pt^\beta f(\xi)\int_{\R^3\times\S^2}|(\xi-\xi_\star)\cdot\w|\pt^\gamma\sqrt\mu(\xi_\star) d\w d\xi_\star\\
&=\pt^\al g .\enda 
\eeq
By the $L^2$ estimate \eqref{L-1}, we have 
\[\bega 
&\nl e^{\rho|\xi|^2}\pt^\al f
\nr_{L^2_\xi}\\
&\lesssim \nl e^{\rho|\xi|^2}\nu(\xi)^{-1}\pt^\al g
\nr_{L^2_\xi}\\
&\quad+\sum_{\substack{\beta+\gamma+\vr=\al\\\vr<\al}}\frac{\al!}{\beta!\gamma!\vr!}\nl e^{\rho|\xi|^2}\nu(\xi)^{-1}\int_{\R^3\times\S^2}|(\xi-\xi_\star)\cdot \w|\pt^\beta\sqrt\mu(\xi_\star)\pt^\gamma \sqrt\mu(\xi'_\star) \pt^\vr f(\xi')d\w d\xi_\star\nr_{L^2_\xi}\\
&\quad+\sum_{\substack{\beta+\gamma+\vr=\al\\\vr<\al}}\frac{\al!}{\beta!\gamma!\vr!}\nl e^{\rho|\xi|^2}\nu(\xi)^{-1}\int_{\R^3\times\S^2}|(\xi-\xi_\star)\cdot \w|\pt^\beta\sqrt\mu(\xi_\star)\pt^\gamma \sqrt\mu(\xi') \pt^\vr f(\xi_\star')d\w d\xi_\star\nr_{L^2_\xi}\\
&\quad+\sum_{\substack{\beta+\gamma+\vr=\al\\\vr<\al}}\frac{\al!}{\beta!\gamma!\vr!}\nl e^{\rho|\xi|^2}\nu(\xi)^{-1}\pt^\beta \sqrt\mu(\xi)\int_{\R^3\times\S^2}|(\xi-\xi_\star)\cdot \w| \pt^\gamma\sqrt\mu(\xi_\star)\pt^\vr f(\xi_\star)d\w d\xi_\star\nr_{L^2_\xi}\\
&\quad+ \sum_{\substack{\beta+\gamma=\al\\\beta<\al
}}\frac{\al!}{\beta!\gamma!} \nl e^{\rho|\xi|^2}\nu(\xi)^{-1} \pt^\beta f(\xi)\int_{\R^3\times\S^2}|(\xi-\xi_\star)\cdot\w|\pt^\gamma\sqrt\mu(\xi_\star) d\w d\xi_\star
\nr_{L^2_\xi}.
\enda
\]
Applying Lemma \ref{bili-appen1} and the fact that $|\xi|^2\le |\xi_\star'|^2+|\xi'|^2$, we get
  \[\bega 
  &\nl e^{\rho|\xi|^2}\nu(\xi)^{-1}\int_{\R^3\times\S^2}|(\xi-\xi_\star)\cdot \w|\pt^\beta\sqrt\mu(\xi_\star)\pt^\gamma \sqrt\mu(\xi'_\star) \pt^\vr f(\xi')d\w d\xi_\star\nr_{L^2_\xi}\\
  &\quad+\nl e^{\rho|\xi|^2}\nu(\xi)^{-1}\int_{\R^3\times\S^2}|(\xi-\xi_\star)\cdot \w|\pt^\beta\sqrt\mu(\xi_\star)\pt^\gamma \sqrt\mu(\xi') \pt^\vr f(\xi_\star')d\w d\xi_\star\nr_{L^2_\xi}  \\
  &\quad+\nl e^{\rho|\xi|^2}\nu(\xi)^{-1}\pt^\beta \sqrt\mu(\xi)\int_{\R^3\times\S^2}|(\xi-\xi_\star)\cdot \w| \pt^\gamma\sqrt\mu(\xi_\star)\pt^\vr f(\xi_\star)d\w d\xi_\star\nr_{L^2_\xi}\\
      &\lesssim \nl e^{\rho|\xi|^2}\pt^\vr f
  \nr_{L^2_\xi}\nl e^{2\rho|\xi|^2}\pt^\beta \sqrt\mu
  \nr_{L^\infty_\xi}\nl e^{2\rho |\xi|^2}\pt^\gamma\sqrt\mu
  \nr_{L^\infty_\xi}.
  \enda
    \]
    Similarly, we obtain 
    \[\bega 
    & \nl e^{\rho|\xi|^2}\nu(\xi)^{-1} \pt^\beta f(\xi)\int_{\R^3\times\S^2}|(\xi-\xi_\star)\cdot\w|\pt^\gamma\sqrt\mu(\xi_\star) d\w d\xi_\star
\nr_{L^2_\xi}\\
    &\lesssim \nl e^{\rho|\xi|^2} \pt^\beta f
    \nr_{L^2_\xi}\nl e^{2\rho|\xi|^2}  \pt^\gamma \sqrt\mu
    \nr _{L^\infty_\xi}.
    \enda 
        \]
       Combining all of the above inequalities with Lemma \ref{local-1}, we get 
\[
\sum_{\al\in\mathbb N_0^3}\frac{\tau^{|\al|}}{\al!}\nl e^{\rho|\xi|^2}\pt^\al f
\nr_{L^2_\xi}\lesssim \sum_{\al\in\mathbb N_0^3}\frac{\tau^{|\al|}}{\al!}\nl e^{\rho|\xi|^2}\nu(\xi)^{-1}\pt^\al g 
\nr_{L^2_\xi}+\eps\kappa^{\frac12} \sum_{\al\in\mathbb N_0^3}\frac{\tau^{|\al|}}{\al!}\nl e^{\rho|\xi|^2}\pt^\al f
\nr_{L^2_\xi}.\]
The proof is complete. 
Next we show the analytic $L^\infty$ estimate. Again from \eqref{ran-123}
and the inequality \eqref{ran-154}, we have 
\[\bega 
&\nl e^{\rho|\xi|^2}\pt^\al  f
\nr_{L^\infty_\xi}\\
&\lesssim \nl e^{\rho|\xi|^2}\nu(\xi)^{-1}\pt^\al g
\nr_{L^2_\xi\cap L^\infty_\xi}\\
&\quad+\sum_{\substack{\beta+\gamma+\vr=\al\\\vr<\al}}\frac{\al!}{\beta!\gamma!\vr!}\nl e^{\rho|\xi|^2}\nu(\xi)^{-1}\int_{\R^3\times\S^2}|(\xi-\xi_\star)\cdot \w|\pt^\beta\sqrt\mu(\xi_\star)\pt^\gamma \sqrt\mu(\xi'_\star) \pt^\vr f(\xi')d\w d\xi_\star\nr_{L^2_\xi\cap L^\infty_\xi}\\
&\quad+\sum_{\substack{\beta+\gamma+\vr=\al\\\vr<\al}}\frac{\al!}{\beta!\gamma!\vr!}\nl e^{\rho|\xi|^2}\nu(\xi)^{-1}\int_{\R^3\times\S^2}|(\xi-\xi_\star)\cdot \w|\pt^\beta\sqrt\mu(\xi_\star)\pt^\gamma \sqrt\mu(\xi') \pt^\vr f(\xi_\star')d\w d\xi_\star\nr_{L^2_\xi\cap L^\infty_\xi}\\
&\quad+\sum_{\substack{\beta+\gamma+\vr=\al\\\vr<\al}}\frac{\al!}{\beta!\gamma!\vr!}\nl e^{\rho|\xi|^2}\nu(\xi)^{-1}\pt^\beta \sqrt\mu(\xi)\int_{\R^3\times\S^2}|(\xi-\xi_\star)\cdot \w| \pt^\gamma\sqrt\mu(\xi_\star)\pt^\vr f(\xi_\star)d\w d\xi_\star\nr_{L^2_\xi\cap L^\infty_\xi}\\
&\quad+ \sum_{\substack{\beta+\gamma=\al\\\beta<\al
}}\frac{\al!}{\beta!\gamma!} \nl e^{\rho|\xi|^2}\nu(\xi)^{-1} \pt^\beta f(\xi)\int_{\R^3\times\S^2}|(\xi-\xi_\star)\cdot\w|\pt^\gamma\sqrt\mu(\xi_\star) d\w d\xi_\star
\nr_{L^2_\xi\cap L^\infty_\xi}.
\enda 
\]
Similarly as the above $L^2$ estimate above, we
use Lemma \ref{bili-2} and Lemma \ref{local-1} to get 
\[\bega 
&\sum_{\al\in\mathbb N_0^3}\frac{\tau^{|\al|}}{\al!}\nl e^{\rho|\xi|^2}\pt^\al  f
\nr_{L^\infty_\xi}\\
&\lesssim \sum_{\al\in\mathbb N_0^3} \nl e^{\rho|\xi|^2}\nu(\xi)^{-1}\pt^\al g
\nr_{L^2_\xi\cap L^\infty_\xi}\\
&\quad+\eps\kappa^{\frac 12}\lw\{\sum_{\al\in\mathbb N_0^3}\frac{\tau^{|\al|}}{\al!}\nl e^{\rho|\xi|^2}\pt^\al  f
\nr_{L^\infty_\xi}+\sum_{\al\in\mathbb N_0^3}\frac{\tau^{|\al|}}{\al!}\nl e^{\rho|\xi|^2}\pt^\al  f
\nr_{L^2_\xi}\rw\}.
\enda 
\]
The proof is complete.
Now we show the third inequality in Proposition \ref{Lfg-L2}. From \eqref{ran-123-0}, we have 
\[
\bega 
&-L (\pt_{x_\ell}f)+\int_{\R^3\times\S^2}|(\xi-\xi_\star)\cdot \w|\sqrt\mu(\xi_\star) \lw\{\pt_{x_\ell}\sqrt\mu(\xi')f(\xi_\star')+f(\xi')\pt_{x_\ell}\sqrt\mu(\xi_\star')\rw\}d\w d\xi_\star\\
&\quad-\pt_{x_\ell}\sqrt\mu(\xi)\int_{\R^3\times\S^2}|(\xi-\xi_\star)\cdot\w| \sqrt\mu(\xi_\star)f(\xi_\star)d\w d\xi_\star-\sqrt\mu(\xi)\int_{\R^3\times\S^2}|(\xi-\xi_\star)\cdot\w| \pt_{x_\ell}\sqrt\mu(\xi_\star)f(\xi_\star)d\w d\xi_\star\\
&\quad-f(\xi)\int_{\R^3\times\S^2}|(\xi-\xi_\star)\cdot\w|\pt_{x_\ell} \mu(\xi_\star)d\w d\xi_\star=\pt_{x_\ell} g.\enda 
\]
The proof proceeds exactly as above, by using Lemma \ref{M-local-2} and the fact that $\eps\kappa^{-\frac 12}\ll 1$. We skip the details.
\end{proof}

\begin{corollary}\label{corol-Aij} There exists universal constants $\tau_0,\rho_0>0$ such that 
\[\bega 
&\sum_{\al\in\mathbb N^3}\frac{\tau_0^{|\al|}}{\al!}
\nl e^{\rho_0|\xi|^2}\pt^\al (Id,\pt_{x_\ell},\pt_t)A_{ij}(\vp)
\nr_{L^2_\xi\cap L^\infty_\xi}
\lesssim 1.
\enda 
\]
\end{corollary}
\begin{proof}
Applying Proposition \ref{Lfg-L2} and using the fact that $LA_{ij}=\hat A_{ij}=\lw(\vp_i\vp_j-\frac 1 3 \delta_{ij}|\vp|^2
\rw)\sqrt\mu$, we need to check that 
\[
\sum_{\al\in\mathbb N^3}\frac{\tau_0^{|\al|}}{\al!}
\nl e^{\rho_0|\xi|^2}\pt^\al \hat A_{ij}
\nr _{L^2_\xi\cap L^\infty_\xi}\lesssim 1
\]
and 
\[\sum_{\al\in\mathbb N^3}\frac{\tau_0^{|\al|}}{\al!}
\nl e^{\rho_0|\xi|^2}\pt^\al  \pt_{x_\ell}\hat  A_{ij}
\nr _{L^2_\xi\cap L^\infty_\xi}\lesssim 1.
\]
The above inequality follows from \eqref{exp-ana-ele}.
\end{proof}
\begin{proposition} Recall $\tau_0>0$ to be the analyticity radius of $U$, and the definition of $\mu_M$ in \eqref{mu-M-def}. Let $T_M\in (0,1)$ so that $|T_M^{-1}-1|\ll 1$. 
 There holds, for $\tau_0\ll 1$:
\[
\sum_{|\al|\ge 1}\frac{\tau_0^{|\al|}}{\al!}\nl\pt^\al \lw(\frac{\sqrt\mu}{\sqrt\mu_M} A_{ij}
\rw)
\nr_{L^2_\xi} \lesssim 1.
\]
\end{proposition}
\begin{proof} We recall from \eqref{ran-76} that
\[
\frac{\sqrt\mu}{\sqrt\mu_M}=T_M^{\frac 32}e^{|T_M^{-1}-1|\frac{|\xi|^2}{4}}e^{-\frac 1 4 \eps^2|U|^2}e^{\frac 12 \eps U\cdot\xi}=T_M^{\frac 32} \prod_{\ell=1}^3 e^{|T_M^{-1}-1|\frac{\xi_\ell^2}{4}+\frac 12 \eps U_\ell \xi_\ell} \prod_{\ell=1}^3e^{-\frac 1 4 \eps^2U_\ell^2}.\]
Since the function $e^{-\frac 1 4\eps^2U_\ell^2}$ is analytic, it suffices to bound
\[
\sup_{|\al|\ge 1}\frac{\tau_0^{|\al|}}{\al!}\nl \pt^\al\lw(e^{|T_M^{-1}-1|\frac{\xi_\ell^2}{4}+\frac 12 \eps U_\ell \xi_\ell} A_{ij}(\vp)\rw)
\nr_{L^2_\xi}.
\]
Using Lemma \ref{exp-ana} and \ref{exp-ana-ele} and  we have 
\beq\label{ran-128}
\bega 
\sup_{|\al|\ge 1}\frac{\tau_0^{|\al|}}{\al!}
\lw|\pt^\al (e^{\frac 12\eps U_\ell \xi_\ell})
\rw|&\lesssim \eps\kappa^{\frac 12}e^{\frac 12\eps \xi_\ell U_\ell+4\eps\kappa^{\frac 12}|\xi_\ell|}\lesssim e^{C_0\eps (1+\|U\|_{L^\infty})|\xi|},\\
\sup_{|\al|\ge 1}\frac{\tau_0^{|\al|}}{\al!}
\lw|\pt^\al (e^{|T_M^{-1}-1|\frac{\xi_\ell^2}{4}})
\rw|&\lesssim e^{\frac{3|T_M^{-1}-1|}{4}\xi_\ell^2}\lesssim e^{\frac{3|T_M^{-1}-1|}{4}|\xi|^2}.\enda
\eeq
Now we have 
\[
\pt^\al\lw(e^{|T_M^{-1}-1|\frac{\xi_\ell^2}{4}+\frac 12 \eps U_\ell \xi_\ell} A_{ij}(\vp)
\rw)=\sum_{\beta+\gamma+\vr=\al}\frac{\al!}{\beta!\gamma!\vr!}\pt^\beta \lw(e^{|T_M^{-1}-1|\frac{\xi_\ell^2}{4}}\rw)\pt^\gamma\lw(e^{\frac 12 \eps U_\ell \xi_\ell}
\rw)\pt^\vr\lw(A_{ij}(\vp)
\rw).
\]
Combining the above with \eqref{ran-128} and Corollary \ref{corol-Aij}, we have
\[\bega 
\frac{\tau_0^{|\al|}}{\al!}\pt^\al\lw(e^{|T_M^{-1}-1|\frac{\xi_\ell^2}{4}+\frac 12 \eps U_\ell \xi_\ell} A_{ij}(\vp)
\rw)\lesssim e^{\frac{3|T_M^{-1}-1|}{4}|\xi|^2+C_0\eps (1+|U|_{L^\infty})|\xi|-\rho |\xi|^2} \cdot \frac{\tau_0^{|\vr|}}{\vr!} e^{\rho|\xi|^2}|\pt^\vr A_{ij}|.\\
\enda
\]
The proof is complete, under the assumption $|T_M^{-1}-1|\ll \rho$.
\end{proof}

\def\cprime{$'$} \def\cprime{$'$}

\end{document}